\documentclass{amsart}

\usepackage{amstext}
\usepackage{amsmath}
\usepackage[all]{xy}
\usepackage{mathrsfs}

\usepackage{amssymb}
\usepackage{amsmath}
\usepackage{mathrsfs}
\usepackage{epsfig}
\usepackage{amscd}

\usepackage[original]{imakeidx}
\makeindex 
\makeindex[name=syindex, title=List of notations]

\newcommand{\R}{\ifmmode{\mathbb R}\else{$\mathbb R$}\fi} 
\newcommand{\Q}{\ifmmode{\mathbb Q}\else{$\mathbb Q$}\fi} 
\newcommand{\C}{\ifmmode{\mathbb C}\else{$\mathbb C$}\fi} 
\newcommand{\Z}{\ifmmode{\mathbb Z}\else{$\mathbb Z$}\fi} 

\newcommand{\ben}{\begin{enumerate}}
\newcommand{\een}{\end{enumerate}}
\newcommand{\be}{\begin{equation}}
\newcommand{\ee}{\end{equation}}
\newcommand{\bea}{\begin{eqnarray}}
\newcommand{\eea}{\end{eqnarray}}
\newcommand{\beastar}{\begin{eqnarray*}}
\newcommand{\eeastar}{\end{eqnarray*}}
\newcommand{\bc}{\begin{center}}
\newcommand{\ec}{\end{center}}

\newcommand{\del}{\partial}
\newcommand{\delbar}{\overline \partial}
\newcommand{\supp}{\operatorname{supp}}

\newcommand{\Exp}{\operatorname{Exp}}
\newcommand{\Pal}{\text{\rm Pal}}

\newcommand{\MM}{{\mathcal M}}

\newcommand{\dudtau}{\frac{\partial u}{\partial \tau}}
\newcommand{\dudt}{\frac{\partial u}{\partial t}}
\newcommand{\Int}{\operatorname{Int}}

\newtheorem{thm}{Theorem}[section]

\newtheorem{lem}[thm]{Lemma}
\newtheorem{sublem}[thm]{Sublemma}
\newtheorem{prop}[thm]{Proposition}

\theoremstyle{definition}
\newtheorem{defn}[thm]{Definition}
\newtheorem{rem}[thm]{Remark}

\newtheorem{conds}[thm]{Condition}

\newtheorem{defnlem}[thm]{Definition-Lemma}
\newtheorem{assump}[thm]{Assumption}

\newtheorem{conv}[thm]{Convention}
\newtheorem{shitu}[thm]{Situation}

\numberwithin{equation}{section}
\numberwithin{figure}{section}

\makeindex

\begin{document}

\title[Exponential decay estimates and smoothness]
{Exponential decay estimates and smoothness of the moduli space of pseudoholomorphic curves}


\author{Kenji Fukaya}
\address{Simons Center for Geometry and Physics, State University of New York, Stony Brook, NY 11794-3636, USA}
\curraddr{}
\email{kfukaya@scgp.stonytbrook.edu}
\thanks{Supported partially by JSPS Grant-in-Aid for Scientific Research
No. 23224002, NSF grant \# 1406423 and 
Simons Collaboration on homological Mirror symmetry.}

\author{Yong-Geun Oh}
\address{Center for Geometry and Physics, Institute for Basic Sciences (IBS), Pohang 37673, Korea
\& Department of Mathematics, POSTECH, Pohang 37673, Korea}
\curraddr{}
\email{yongoh1@postech.ac.kr}
\thanks{Supported partially by IBS project \# IBS-R003-D1 in Korea}

\author{Hiroshi Ohta}
\address{Graduate School of Mathematics,
Nagoya University, Nagoya, Japan}
\curraddr{}
\email{ohta@math.nagoya-u.ac.jp}
\thanks{Supported partially by JSPS Grant-in-Aid
for Scientific Research Nos.19340017, 23340015, 15H02054, 21H00983,
21K18576,
21H00985,
23K20796.
}

\author{Kaoru Ono}
\address{Research Institute for Mathematical Sciences, Kyoto University, Kyoto, Japan}
\curraddr{}
\email{ono@kurims.kyoto-u.ac.jp}
\thanks{Supported partially by JSPS Grant-in-Aid for
Scientific Research, Nos. 18340014, 21244002, 26247006, 19H00636.}





\begin{abstract}
The compacified moduli space of  bordered stable maps carries a
Kuranishi structure with boundary. Smoothness of Kuranishi structure along the boundary 
requires smoothness of coordinate changes along the boundary.
The  proof of smoothness is written by the authors in [FOOO1], [FOOO3]
based on some uniform exponential decay estimates of the stable maps
with respect to a parameter $T$, the length of the gluing cylindrical neck,
near the boundary of the moduli space. In this paper we establish this exponential
decay estimates in a precise manner by carefully examining the dependence on the
parameter $T$ of the gluing construction in the setting of bordered Riemann surfaces with boundary punctures.
We also show that  the smoothness of the collar follows from the aforementioned
exponential estimates by taking $s = 1/T$ as the radial coordinate of the collar
\end{abstract}

\maketitle

\tableofcontents


\section{Introduction}

The theory of Kuranishi structures is introduced by the first and the {fourth} named authors in
\cite{FOn}, and further amplified by the authors in \cite{fooo:book1}, \linebreak
\cite{fooo:techI}, \cite{fooonewbook}.
To implement the abstract theory of Kuranishi structure in the study of concrete moduli problem,
especially that of the moduli space of pseudoholomorphic curves, one
issue is to establish smoothness of the Kuranishi map $\mathfrak s$ and of the coordinate change of the Kuranishi
neighborhoods.
\par
Let us first describe the problem in a bit more detail in the context of bordered
stable pseudoholomorphic curves attached to Lagrangian submanifolds.\footnote{We
can generalize it to the case of pseudoholomorphic curves attached to
totally real submanifolds for given almost complex structure $J$ on $X$.
We do not discuss it here since we do not know its application.}
{Note that an object representing a point near the boundary of our moduli space 
contains a neck region, which is a
long strip $[-T,T] \times [0,1]$, and the boundary point corresponds to the case where the source curve
is singular, that is, the case when $T = \infty$.}

Note $\infty$ is included in $(T_0,\infty]$.
As a topological space $(T_0,\infty]$ has unambiguous meaning.
On the other hand there is no obvious choice of its smooth structure as a manifold with boundary.
Moreover for several relevant maps such as the Kuranishi map in the definition of Kuranishi structure,
it is not obvious whether they are smooth for the given
coordinate of $(T_0,\infty]$ naturally arising from the standard gluing construction.
This issue has recently been raised and asked to the present
authors by several mathematicians.
\footnote{Among others, Y. Ruan, C.C. Liu, J. Solomon, I. Smith
and H. Hofer asked the question. We thank them for asking this question.}

There are several ways to resolve this problem. One approach is rather topological
which uses the fact that the gluing chart is smooth
in the $T$-slice on which the gluing parameter $T$ above is fixed.
This approach is strong enough to establish all the results of \cite{FOn} for which
the method used in \cite{McSa94} was good enough to work out the analytic details along this approach.
However it is not clear to the authors whether this approach is good enough to establish
smoothness of the Kuranishi map or of the coordinate
changes at $T=\infty$. This point was mentioned by the first and the fourth named
authors themselves in \cite[Remark 13.16]{FOn}.
To prove an existence of the Kuranishi structure that literally satisfies its axioms,
we take a gluing method different from
the ones employed in \cite{McSa94,FOn}, called the alternating method, in this article.
(See Remark \ref{rem12} for a discussion about several other differences on the
analytic points.)
\par
Using the alternating method described in \cite[Section A1.4]{fooo:book1},
we can find an appropriate coordinate chart at
$T=\infty$ so that the Kuranishi map and the coordinate changes of our Kuranishi neighborhoods are
of $C^{\infty}$ class. For this purpose, we take the parameter $s = 1/T$.
As we mentioned in \cite[page 771]{fooo:book1}
this parameter $s$ is different from the one taken in algebraic geometry when the
target $X$ is projective. The parameter used in algebraic geometry corresponds to $z = e^{-{\rm const}\cdot T}$.
See Chapter \ref{sec:smoothness of coordinate change}.
It seems likely that in our situation where either
almost complex structure is  non-integrable and/or the Lagrangian submanifold enters
as the boundary condition (the source being a bordered stable curve) the
Kuranishi map or the coordinate changes is {\it not} smooth with respect to the parameter $z = e^{-{\rm const \cdot}T}$.
But it is smooth in our parameter $s = 1/T$, as is proved in
\cite[Proposition A1.56]{fooo:book1}.
\par
The proof of this smoothness relies on some
exponential decay estimate of the solution of the equation
(\ref{mainequation00}) with respect to $T$, that is, the length of the neck.
An outline of the proof of this exponential decay is given
in \cite[Section A1.4, Lemma A1.58]{fooo:book1}.
Because of the popular demand of more details of this smoothness proof,
we provide such details in the present paper by polishing the presentation
given in  \cite[Part 3]{fooo:techI}.
The exponential decay estimate is proved as Theorem \ref{exdecayT}.
We then show how it enables us to prove the smoothness of the
coordinate change of Kuranishi structure in Chapter \ref{sec:smoothness of coordinate change}.
See {\bf Theorems \ref{them81} and \ref{thm822}}.
\par
In the present paper, we restrict ourselves to the case where we glue
two (bordered) stable maps whose domain curves
(without considering the maps) are already
stable. By restricting ourselves to this case we can
address all the analytic issues needed to work out the
general case also without making the notations so much heavy.
\par
In case when the element is a stable map from a curve which has a 
single nodal singularity, its neighborhood still contains a
stable map from a nonsingular curve.
So we need to study the problem of gluing or of stretching the neck.
Such a problem of gluing solutions of non-linear elliptic partial differential equation
has been studied extensively in gauge theory and in symplectic
geometry during the last decade of the 20th century.
Several methods had been developed to solve the problem which
are also applicable to our case. In this article, following \cite[Section A1.4]{fooo:book1},
we employ the alternating method, \index{alternating method} which was first exploited by Donaldson \cite{Don86I} in gauge theory.
\begin{rem}
In \cite{Don86I} Donaldson used {the} alternating method by inductively solving {\it nonlinear} equation
(ASD-equation) on each of the two pieces {into} which he divided {the} given 4 manifold.
In this paper we will solve {the} {\it linearized} version of Cauchy-Riemann equation
inductively on each of the divided pieces.
So our proof is a mixture of the alternating method and Newton's iteration.\index{Newton's iteration} 
\end{rem}
\par
In this method, one solves the linearization of the given equation on each
piece of the domain (that is the completion of the complement
of the neck region of the source of our pseudoholomorphic curve.)
Then we construct a sequence of maps that converges to a 
solution of a version of the Cauchy-Riemann equation, that is,
\begin{equation}\label{mainequation00}
\overline{\partial} u' \equiv 0  \mod \mathcal E(u')
\end{equation}
and which are parameterized by a manifold (or an orbifold).
Here $\mathcal E(u')$ is a family of finite dimensional vector spaces
of smooth sections of an appropriate vector bundle
depending on $u'$.
\par
We provide the relevant analytic details using the same induction scheme as
\cite[Section A1.4, page 773-776]{fooo:book1}. The only difference is that we use $L^2_{m}$ space (the space of maps
whose derivative up to order $m$ are of $L^2$ class) here,
while we used $L^p_{1}$ space following the tradition of
symplectic geometry community in \cite[Section A1.4]{fooo:book1}.
It is well-known that the shift of $T$ causes loss of differentiability of the maps
in terms of the order of Sobolev spaces. However by considering various {\it weighted}
Sobolev spaces with various $m$ simultaneously and using the definition of
$C^\infty$-topology, which is a Frech\^et topology, we can still get the differentiability
of $C^{\infty}$ order and its exponential decay.
See Remark \ref{Abremark}.
\par
In Chapter \ref{alternatingmethod} we provide the details of the estimate and show that
the induction scheme of \cite[Section A1.4]{fooo:book1} provides
a convergent family of solutions of our equation (\ref{mainequation00}).
This estimate is actually a fairly straightforward application of the
(improved) Newton's iteration scheme although its detail is tedious to write down.
\par
We then prove the exponential decay estimate of the derivative of the
gluing map with respect to the gluing parameter in Chapter \ref{subsecdecayT}.
See {\bf Theorem \ref{exdecayT}}.
\par
We remark that in this paper we intentionally avoid using
the framework of
Hilbert or Banach manifold in our gluing analysis. Certainly one
can use such a framework to interpret the proof given in Chapter
\ref{alternatingmethod} as a kind of implicit function theorem.
(Recall that Newton's iteration, which we use therein, is one of the ways to
prove the implicit function theorem.)
The reason  we avoid using such an infinite dimensional manifold framework is
because in the proof of Chapter \ref{subsecdecayT} where we study $T$-derivatives,
we need to deal with the situation where domains of the maps vary.
Handling with a family of Sobolev spaces of maps with
varying domain raises a nontrivial issue especially when one tries
to regard the total space of the whole family as a certain
version of Hilbert or Banach manifolds.  Such a fact has been recognized by many
researchers in various branches of mathematics. For example,
the domain rotation on the Sobolev loop space is just continuous not smooth \cite{Klingenberg}.
For an infinite dimensional representation, say the regular representation, of a Lie group
$G$, the representation space (which is a space of maps from $G$)
is not acted by the Lie algebra of $G$, since $G$ action is not differentiable.
This is the reason  the notion of ``smooth vector'' is introduced.
Another example is the issue of smoothness of the determinant line bundle for a smooth family of operators
of Dirac type, for which the relevant family of Sobolev spaces of sections do not make a smooth Banach bundle
over the parameter space (see, e.g., \cite{Quillen}, \cite{Bismut-Freed}).  It may be also worth mentioning
that the method of finite dimensional approximations is used in such situations.
\par
In the study of pseudoholomorphic curves  non-smoothness
of the action of the group of diffeomorphism
of the domain on the Sobolev spaces  is emphasized by
H. Hofer, K. Wysocki and E. Zehnder \cite{HWZ}.
\par
We also remark that this issue is the reason  we work
with $L^2_m$ space rather than $L^p_1$ space.
The proof of Theorem \ref{exdecayT} given in this paper
does {\it not} work if we replace $L^2_m$ by $L^p_1$.
See Remark \ref{Abremark}.
\par
We bypass this problem by specifying the function spaces of maps we use
(which are Hilbert spaces) and explicitly defining the maps between the function spaces, e.g.,
the cut-paste maps entering in the gluing construction.
Keeping track of the function spaces we use at all the
stages of the construction, rather than putting them in an
abstract framework or in a black box of functional analysis,
is crucial to perform such a cut-paste process precisely in the way
that the relevant derivative estimates can be obtained.
It is also useful for the study of the properties of the Kuranishi structure
we obtain, like its relationship to the forgetful map, to the target space group action and etc.
(Note we use a Hilbert manifold $\text{\rm Map}_{L^2_{m+1}}((K_i^S,K_i^S\cap\partial \Sigma_i),(X,L))$
in Chapter \ref{subsecdecayT}. There we carefully ensure that
the domain and {codomain} of the maps do not change.)
\par
There are several appendices which collect some parts of the proofs of Chapters \ref{alternatingmethod}
and \ref{subsecdecayT}. They comprise technical details which are
based on some straightforward estimates that do not break the mainstream of the
proofs. They mainly rely on the fact that for
a given smooth map $F : M \to N$ the assignment $u \mapsto F(u)$
for $u: \Sigma \to M$ defines a smooth map $L^2_m(\Sigma, M) \to L^2_{m}(\Sigma, N)$ if $m$ is sufficiently large.
The readers who have strong background in analysis may find those appendices
something obvious or at least well-known to them.
We include the details of those proofs only for completeness' sake. 
\par\smallskip
\thanks{
This paper grows up from \cite[Part 3]{fooo:techI} (The contents of Chapter
\ref{sec:smoothness of coordinate change}
is also taken from  \cite[Part 4]{fooo:techI}.),
which is a slightly modified version of the document
which the authors uploaded to the google group Kuranishi
in 2012 June.
We would like to thank all the participants
in the discussions of the `Kuranishi' google group for motivating us to
go through this painstaking labour.
The authors also thank  M. Abouzaid for making us
important comments on \cite[Part 3]{fooo:techI}.
We thank the anonymous referees for going through the painstaking details of the paper, 
and much appreciate her/him of making a valuable assesment on the significance and 
usefulness of the result of the paper.
\section{Preliminaries}
\label{sec:prelim}

In this chapter, we collect various known results that will be
used in relation to the gluing analysis, and provide basic setting {for}
the study of gluing, and {of} exponential decay of the elements near the boundary of the moduli space of
pseudoholomorphic maps from bordered Riemann surfaces.

\subsection{Choice of Riemannian metric}
We take the following choice of Riemannian metric for convenience.
The proof of the next  lemma can be  borrowed from
\cite[page 683]{ye}. (See Remark \ref{remarkA2} for a relevant remark.)

\begin{lem}\label{lem:g0}
Let $(X,J)$ be an almost complex manifold and
$L$ a (maximally) totally real submanifold with respect to $J$.
There exists a Hermitian metric $g$ on $X$ such
that $L$ is totally geodesic and satisfies
\be\label{eq:J-perpen}
JT_p L \perp T_pL
\ee
for every $p \in L$.
\end{lem}
\begin{proof} Let $\{e_1,\ldots, e_n\}$ be a local orthonormal frame on
$L$. Since $L$ is totally real with respect to $J$, $\{e_1,\ldots,e_n,
Je_1,\ldots, Je_n\}$ is a local frame of $TX$ on $L$. By a partition of unity
we obtain a metric $\widetilde g$ on the vector bundle
for which it satisfies $\widetilde g(e_i,Je_j) = 0$ for all $i,\, j$.
We then extend $\widetilde g$ on $TX|_N$ to a tubular neighborhood of $N$.
Then we set $\overline g = \widetilde g(J,J) + \widetilde g$.

Then we note that any metric on a vector bundle $E \to N$ (e.g., $E = TN$) can be
extended to a metric on $E$ as a manifold so that the latter becomes reflection
invariant on $E$. In particular the zero section of $E$ becomes totally geodesic.
Using this, it is easy to enhance the above metric $g$ so that $L$ becomes
totally geodesic in addition with respect to the resulting metric $g$.
\end{proof}

\subsection{Exponential map and its inverse}

Let $\nabla$ be the Levi-Civita connection of the above chosen metric $g$ in Lemma \ref{lem:g0}.
We also use an exponential map.\index{exponential map} (The same map was used in \cite[pages 410-411]{fooo:book1}.)
Denote by $\Exp: TX \to  X$ the (global) exponential map. \index[syindex]{Exp@$\Exp$} We put
\be\label{eq:Exp}
\widetilde{\Exp}(x,v) = (x,\Exp(x,v))
\ee
and denote by $\widetilde{\rm E} : U \to TX$ its inverse, \index[syindex]{Etilde@$\widetilde{\rm E}$}
defined on a neighborhood of the diagonal
$\Delta \subset X \times X$. 
We define ${\rm E} : U \to TX$ by
\be\label{eq:E}
\widetilde{\rm E}(x,y) = (x,{\rm E}(x,y))
\ee
Note $U$ can be defined to be\index[syindex]{E@${\rm E}$}
$$
U = \{(x,y)\in X \times X \mid d(x,y) < \iota_X \}
$$
where  $\iota_X$ is the injectivity radius of $(X,g)$.\index[syindex]{iota@$\iota_X$}
Here and hereafter $d(x,y)$ is the Riemannian distance between $x$ and $y$.

\begin{rem}\label{rem2222}
 When the metric is flat and in flat coordinates, the maps
introduced above also can be expressed as
\beastar
\Exp(x,v)  =  x + v, \quad {\rm E}(x,y) = y-x.
\eeastar
Readers may find it useful to compare the
invariant expression via the exponential maps in the present paper with these coordinate
calculations used in \cite{fooo:techI} to get some help for visualization of exponential maps.
Because of Lemma \ref{expest}, the difference between the two estimates
will be exponentially small for the pseudoholomorphic maps $u(\tau,t)$
in the various circumstances we are  looking at in the present paper.
\end{rem}

\subsection{Parallel transportation}

We take and fix a Riemannian metric on $X$ that satisfies the
conclusion of Lemma \ref{lem:g0}.
We denote by $\iota_X$ the injectivity radius. \index{injectivity radius} Namely
if $x,y \in X$ with $d(x,y) < \iota_X$ then there exists a
unique geodesic of length $d(x,y)$ joining $x$ and $y$.
\par
For two points $x,y \in X$ with $d(x,y) < \iota_X$ we denote by
\begin{equation}
{\rm Pal}_x^y : T_xX \to T_yX
\end{equation}
the parallel transport \index{parallel transport} along the unique minimal geodesic.
(We use the Levi-Civita connection to define a parallel
transport.)
Suppose $X$ has an almost complex structure $J$.
We then denote by \index[syindex]{Pal@${\rm Pal}_x^y$}\index[syindex]{PalJ@$({\rm Pal}_x^y)^{J}$}
\begin{equation}\label{newform2525}
({\rm Pal}_x^y)^{J} : T_xX \to T_yX,
\end{equation}
the complex linear part of ${\rm Pal}_x^y$.
Namely we decompose (uniquely) ${\rm Pal}_x^y$ to
$$
{\rm Pal}_x^y = ({\rm Pal}_x^y)^{J}
+ (({\rm Pal}_x^y)^{{J}})'
$$
such that
$$
( {\rm Pal}_x^y)^{J}\circ J_x = J_y \circ  ({\rm Pal}_x^y)^{J},
 \quad
(({\rm Pal}_x^y)^{J})'\circ J_x = -J_y \circ  (({\rm Pal}_x^y)^{J})'
$$
In other words
$$
( {\rm Pal}_x^y)^{J} = \frac{1}{2}\Big(
{\rm Pal}_x^y - J_y \circ {\rm Pal}_x^y\circ J_x
\Big).
$$
We choose and fix a constant
$\iota'_X > 0$ such that the following holds.\index[syindex]{iota'@$\iota'_X$}
\begin{conds}\label{cond23}
If $d(x,y) \le \iota'_X$ then $( {\rm Pal}_x^y)^{J}$ is a linear isomorphism.
\end{conds}

Let $\Sigma$ be a two dimensional complex manifold and $u,v : \Sigma \to X$.
We assume
$$
\sup\{ d(u(z),v(z)) \mid z \in \Sigma \} \le \iota'_X.
$$
Then using  pointwise maps $ {\rm Pal}_x^y$ and $( {\rm Pal}_x^y)^{J}$ we obtain
the maps of sections \index[syindex]{paluv@${\rm Pal}_u^v$}
\begin{equation}\label{paluv}
{\rm Pal}_u^v : \Gamma(\Sigma,u^*TX) \to \Gamma(\Sigma,v^*TX),
({\rm Pal}_u^v)^J : \Gamma(\Sigma,u^*TX) \to \Gamma(\Sigma,v^*TX).
\end{equation}
We also note that by composing with the $(0,1)$-projections with respect to
$J: TM \to TM$, $( {\rm Pal}_x^y)^{J}$ also induces the map\index[syindex]{paluv01@$({\rm Pal}_u^v)^{(0,1)}$}
$$
({\rm Pal}_{u(x)}^{v(x)})^{(0,1)} : T_{u(x)}X\otimes \Lambda_x^{0,1} \to T_{v(x)}X
\otimes \Lambda_x^{0,1}
$$
which in turn induces the map of sections
\begin{equation}\label{paluv01}
({\rm Pal}_u^v)^{(0,1)} : \Gamma(\Sigma,u^*TX\otimes \Lambda^{0,1}) \to \Gamma(\Sigma,v^*TX
\otimes \Lambda^{0,1}).
\end{equation}

\subsection{Exponential decay of pseudoholomorphic maps with small image}

We now consider a pseudoholomorphic map $u: [-S, S] \times [0,1] \to (X,L)$
$$
\frac{\del u}{\del \tau} + J\frac{\del u}{\del t} = 0
$$
with finite energy $\mathfrak E(u) < \infty$.  \index[syindex]{Eu@$\mathfrak E(u)$} We also recall the definition of
energy
$$
\mathfrak E (u) = \frac{1}{2} \int \left(\left|\dudtau\right|^2 + \left|\dudt\right|^2 \right)\, dt\, d\tau
$$
where the norm $|\cdot|$ is measured in terms of a metric $g$ that is compatible with
$\omega$, $g = \omega(\cdot, J \cdot)$ for a fixed compatible almost complex structure $J$.
We use the following well-known uniform exponential decay estimate
for pseudoholomorphic curve with small diameter.

\begin{lem}\label{lem:C1expdecay}
There exists $T_{(\ref{eq2929})} > 0$ and $\delta_1 > 0, \epsilon_{1} > 0,
C_{(\ref{eq2929})} >0$ depending only on  $X, L$,
and $\frak E_0$
with the following properties. If $T>T_{(\ref{eq2929})}$, $u:
[-T-1,T+1] \times [0,1] \to X$ is pseudoholomorphic, $\mathfrak E(u) \le \mathfrak E_0$,
$u([-T-1,T+1] \times \{0,1\}) \subset L$ and $\operatorname{Diam}({\rm Im}u) \leq \epsilon_{1}$, then
\begin{equation}\label{eq2929}
\left|\frac{\del u}{\del \tau} (\tau,t)\right| +
\left|\frac{\del u}{\del t} (\tau,t)\right| \leq C_{(\ref{eq2929})} e^{-\delta_1
d(\tau, \partial [-T-1,T+1])}
\end{equation}
for $(\tau,t) \in [-T,T] \times [0,1]$.
\end{lem}
We refer to \cite[Lemma 11.2]{FOn} or to  \cite[Lemma B.1]{oh:book} for its proof, for example.

Using this $C^1$-exponential decay,  $\epsilon$-regularity theorem and the uniform local a priori estimates, we also
obtain

\begin{lem}\label{lem:Ckexpdecay}
There exists $C_{k,(\ref{eq2930})}$ depending only on the $(X,J,g), L$,
$\frak E_0$ and $k\ge 0$
 such that if
$u: [-T-1,T+1] \times [0,1] \to X$ is as in Lemma \ref{lem:C1expdecay}  then
\begin{equation}\label{eq2930}
\left\vert\frac{\partial u}{\partial \tau}(\tau,t)\right\vert_{C^k}
\leq C_{k,(\ref{eq2930})} e^{-\delta_1 d(\tau, \partial [-T-1,T+1])}.
\end{equation}
In particular, if $u: [0, \infty)\times [0,1] \to X$ (or $u:(-\infty,0]
\times [0,1]\to X$) is pseudoholomorphic $u( [0, \infty)\times \{0,1\}) \subset L$ (or
$u((-\infty,0]\times \{0,1\}) \subset L$),
$\operatorname{Diam}({\rm Im}u) \leq \epsilon_{1}$
and $\mathfrak E(u) \le \mathfrak E_0$, and $k\ge 0$ then
we have
\begin{equation}\label{approestu}
\left\vert\frac{\partial u}{\partial \tau}(\tau,t)\right\vert_{C^k} \leq  C_{k,(\ref{eq2930})} e^{-\delta_1 |\tau|}.
\end{equation}
\end{lem}
\begin{rem}
Actually we can take $\delta_1 = \pi$. We do not need this explicit value in this paper.
\end{rem}
\begin{conv}
Hereafter we assume that {\it all}  pseudoholomorphic curves $u$ we consider
satisfy the inequality
$$
\frak E(u) \le \frak E_0,
$$
for some fixed $\frak E_0$. \index[syindex]{Ezero@$\frak E_0$} So the $\frak E_0$ dependence of various
constants will not be mentioned. 
\end{conv}

\begin{rem}
Since many constants appear in this paper we use the following  enumeration convention.\index{enumeration of the constants}
The constant such as, {$C_{m,(10.12)}$, for example,} is the one appearing in formula (10.12)
which may depend on $m$.
\par
For some constants appearing repeatedly, we do {\it not} apply this rule
of enumerating by the number of the associated formulae.
The list of such constants is:
$\iota'_X$ appears in Condition \ref{cond23},
$\epsilon_1$, $\delta_1$ appear in Lemma \ref{lem:C1expdecay},
$\delta$
appears in (\ref{form310310}), $T_{1,\epsilon(1)}$, $T_{2,\epsilon(1),\epsilon(2)}$
appear in Theorem \ref{gluethm1}, 
$C_{1,m}$    appear
Lemma \ref{lem18}, $\epsilon_{2}$ appears in Lemma \ref{surjlineastep11}, $C_{2,m}$ appears in Lemma \ref{lem516}, $C_{3,m}$
appears in Lemma \ref{mainestimatestep13}, $C_{4,m}$ appears in
(\ref{form0184a}), $\epsilon_{3,m}$ and $T_{3,m,\epsilon(4)}$,
$T_{4,m}$ appear
 in Proposition \ref{prop:kappatokappa+1},
 $C_{5,m}$, $C_{6,m}$, $C_{7,m}$, $C_{8,m}$, $C_{9,m}$
$T_{5,m,\epsilon(6)}$, $T_{6,m}$, $\epsilon_{4,m}$ appear in Proposition \ref{prop:inequalitieskappa},
$\epsilon_5$ appears in Proposition \ref{neckaprioridecay}.
$\delta_2$ appears in Theorem \ref{them816}, $\delta_3$ appears in Proposition \ref{prop816}
and $\delta_4$ appears in Proposition \ref{prof825},
$\epsilon_6$, $\epsilon_7$, $\nu_1$, $\nu_2$ appear
around (\ref{lastform845}),
$\epsilon_{8}$,
$\epsilon_{9}$,
$\epsilon_{10}$,
$\nu_{3,m}$,
$\nu_{4,m}$
appear in Lemma \ref{lem824new}.
\par
The positive numbers $\epsilon(1), \epsilon(2),\dots$ are {\it not} fixed but may vary.
\end{rem}
\section{Statement of the gluing theorem}
\label{sec:statement}

For simplicity and clarity of the exposition on the main analytical points,
we consider the case where we glue pseudoholomorphic maps from two stable bordered Riemann surfaces
to $(X,L)$ in the present paper.
The general case  is proved in the same way in \cite[Part IV]{fooo:techI}.

\subsection{{The weighted Sobolev spaces for the gluing analysis}}
Let $\Sigma_i$ be a bordered Riemann surface for $i=1,\,2$.
We assume that each $\Sigma_i$ has strip-like ends such that $\Sigma_i \setminus K_i$
becomes semi-infinite strip.
Recall that $\Sigma_i$ is conformally a bordered Riemann surface with one puncture.

More specifically, we fix a K\"ahler metric on $\Sigma_i$ whose end is isometric to the flat strip.
We fix  isothermal coordinates $(\tau,t)$ so that {$\Sigma_i$ is decomposed} into \index[syindex]{Sigma1@$\Sigma_1$}
\begin{equation}\label{eq:Sigma12}
\aligned
\Sigma_1 &= K_1 \cup ([-5T,\infty)\times [0,1]), \\
\Sigma_2 &= ((-\infty, 5T] \times [0,1]) \cup K_2
\endaligned
\end{equation}
where $K_i$ {is a}  \index[syindex]{Ki@$K_i$} compact subset of $\Sigma_i$ respectively.
See Figure \ref{Figure1}.
\begin{figure}
\centering
\includegraphics{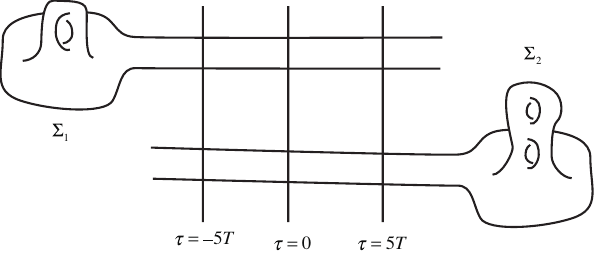}
\caption{$\Sigma_1$ and $\Sigma_2$.}
\label{Figure1}
\end{figure}

Here $(\tau,t)$ is the standard coordinates on the strip $\R \times [0,1]$ and
restricted to $[-5T, \infty) \times [0,1]$ (resp.
$(-\infty, 5T] \times [0,1]$) on $\Sigma_1 \setminus K_1$ (resp. $\Sigma_2 \setminus K_2$).
We would like to note that the coordinate $\tau$
depends on $T$. For each $T > 0$, we put \index[syindex]{SigmaT@$\Sigma_T$}
\begin{equation}\label{form32}
\Sigma_T = K_1 \cup ([-5T,5T]\times [0,1]) \cup  K_2.
\end{equation}
\begin{figure}[h]
\centering
\includegraphics[scale=0.5]{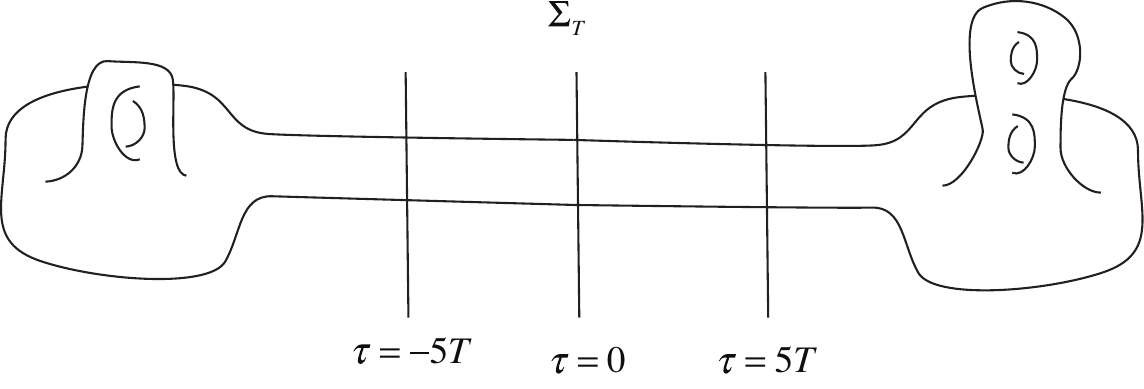}
\caption{$\Sigma_T$.}
\label{Figure2}
\end{figure}
See Figure \ref{Figure2}. \index[syindex]{SigmaT@$\Sigma_T$}

Let $X$ be a compact symplectic manifold with compatible almost complex structure
and $L$  its Lagrangian submanifold.
\par
Let
$
u_i : (\Sigma_i,\partial \Sigma_i)  \to  (X,L)$, $i=1,2$
be pseudoholomorphic maps of finite energy.
Then, by the removable singularity theorem \cite{oh:removal}, the maps $u_i$ smoothly
extend to the associated puncture of $\Sigma_i$ and so the asymptotic limits
\begin{equation}\label{tauinf}
\lim_{\tau\to \infty} u_1(\tau,t) \in L
\end{equation}
and
\begin{equation}\label{tauminf}
\lim_{\tau\to -\infty} u_2(\tau,t) \in L
\end{equation}
are uniquely determined and do not depend on $t \in [0,1]$.
\par
We will consider a pair $(u_1,u_2)$ of the maps for which the limits (\ref{tauinf}),
(\ref{tauminf}) coincide which we denote by $p_0 \in L$.
\par
\begin{conds}\label{conds:smalldiam}
We also assume that
\begin{equation}\label{form3535}
\text{Diam }  u_i(\Sigma_i \setminus K_i) \le \epsilon_{1}
\end{equation}
where $\epsilon_1$ is as in Lemma \ref{lem:C1expdecay}.
\end{conds}
Then the exponential decay (\ref{approestu}) holds by Lemma \ref{lem:Ckexpdecay}.
\begin{rem}
If $u_1$ (resp. $u_2$) satisfies (\ref{tauinf}) (resp. (\ref{tauminf})) then we may replace
$K_1$ (resp. $K_2$) by $K_1 \cup [-5T, -5T + S] \times [0,1]$
(resp. $K_2 \cup [-S + 5T,5T] \times [0,1]$)
and change the $T$ coordinate to $T - S$ (resp. $T+S$)
for an appropriate $S$ so that (\ref{form3535}) is satisfied.
In other words, we can assume (\ref{form3535}) without loss of generality.
\end{rem}
For the clarity  of our exposition, we consider the case of fixed complex structures of the
{domains} $\Sigma_i$, in Chapters \ref{sec:statement}-\ref{surjinj}.
We then include the deformation parameter of
the complex structures of $\Sigma_i$ in Chapter \ref{sec:smoothness of coordinate change}.
In order to handle the case of genus 0, we also need to stabilize and so
add marked points on the domain curve $\Sigma_i$.
We denote the ordered set of marked points by $\vec z_i = (z_{i,1}, \cdots, z_{i,k_i})$
for $i = 1, \, 2$.
For the simplicity of notations, we only consider
boundary marked points.
\par
We may add one point $\tau = \infty$ or $-\infty$ and compactify the source curve
$\Sigma_i$ and extend our map $u_i$ thereto. We regard
the added point as the $0$-th marked point $z_{i,0}$.
\begin{assump}\label{assm34}
In this paper we always assume that $(\Sigma_i,\vec z_i)$ is stable.
\par
Note there is one more marked point at infinity other than $\vec z_i$.
We also remark that here $\Sigma_i$ is assumed to have no
interior or boundary node. So if $\Sigma_i$ is not a disc then it is stable and if $\Sigma_i$ is a disc, the condition
$k_i \ge 2$ is equivalent to the stability.
\end{assump}
We will define a map: \index[syindex]{Dui@$D_{u_i}\overline{\partial}$}
\begin{equation}\label{lineeq}
D_{u_i}\overline{\partial}
:
W^2_{m+1,\delta}((\Sigma_i,\partial \Sigma_i);u_i^*TX,u_i^*TL) \to
L^2_{m,\delta}(\Sigma_i;u_i^{*}TX \otimes \Lambda^{0,1}),
\end{equation}
which is the linearization of $\delbar$ at $u_i$. Here we
denote $\overline\partial = \overline \partial_J$ by omitting $J$.
The domain, that is, the function space \index[syindex]{W2m+1@$W^2_{m+1,\delta}((\Sigma_i,\partial \Sigma_i);u_i^*TX,u_i^*TL)$}
$$
W^2_{m+1,\delta}((\Sigma_i,\partial \Sigma_i);u_i^*TX,u_i^*TL),
$$
is defined below.

We assume that the constant $\epsilon_1$ given in Condition \ref{conds:smalldiam}
is at least smaller than $\iota'_X$ given in Condition \ref{cond23}.
Then to each $v \in T_{p_0}L$, we define the section $v^{\rm Pal} \in
\Gamma([-5T,5T]\times [0,1];u^*_1TX)$ by \index[syindex]{vpal@$v^{\rm Pal}$}
\be\label{eq:rho-v}
v^{\rm Pal}(\tau,t) \equiv  \Pal_{p_0}^{u_i(\tau,t)}(v).
\ee
We note that the sections \eqref{eq:rho-v} satisfy the boundary condition at $t = 0, \,1$
because we use the Levi-Civita connection of a metric with respect to which the Lagrangian
submanifold $L$ is totally geodesic. (See Lemma \ref{lem:g0}.)
\par
More generally if $x \in X$, $v \in T_xX$ and
$u : \Sigma_0 \to X$ is a map to the $\iota_X$ neighborhood of $x$ from
a 2 dimensional manifold, we define $v^{\rm Pal}
\in \Gamma(\Sigma_0,u^*TX)$ by
\begin{equation}\label{eq:vpal}
v^{\rm Pal}(z) = \Pal_{x}^{u(z)}(v).
\end{equation}

\begin{defn}\label{defn3535}(\cite[Section 7.1.3]{fooo:book1})\footnote{In \cite{fooo:book1} $L^p_1$ space
is used in stead of $L^2_{m}$ space.} Let $p_0 \in L$ be the point given in (\ref{tauinf}) or in (\ref{tauminf}).
Denote by $L^2_{m+1,loc}((\Sigma_i,\partial \Sigma_i);u_i^*TX;u_i^*TL)$ the set of the
sections $s$ of $u_i^*TX$ which are locally of $L^2_{m+1}$-class. (Namely its differential up to order $m+1$
is locally of $L^2$-class. Here $m$ is sufficiently large, say larger than $10$.)
We also assume $s(z) \in u_i^*TL$ for $z \in \partial \Sigma_i$.
We define a Hilbert space
$$
W^2_{m+1,\delta}((\Sigma_i,\partial \Sigma_i);u_i^*TX,u_i^*TL)\\
$$
as the set of all pairs
$(s,v)$
where $s \in L^2_{m+1,loc}((\Sigma_i,\partial \Sigma_i);u_i^*TX;u_i^*TL)$,
$v  \in T_{p_0}L$ such that
\begin{equation}\label{weight1}
\sum_{k=0}^{m+1} \int_{\Sigma_i \setminus K_i}  e^{2\delta\vert \tau \pm 5T\vert}
\vert \nabla^k(s - v^{\rm Pal})\vert^2 < \infty.
\end{equation}
Here $\pm = +$ for $i=1$ and $\pm=-$ for $i=2$.
We define its norm $\Vert (s,v)\Vert_{W^2_{m+1,\delta}(\Sigma_i)}$ by
\begin{equation}\label{normWW}
\Vert (s,v)\Vert_{W^2_{m+1,\delta}(\Sigma_i)}^2
= \text{(\ref{weight1})}
+
\sum_{i=1,2} \sum_{k=0}^{m+1} \int_{K_i}\vert \nabla^ks\vert^2
+ \Vert v\Vert^2.
\end{equation}
Hereafter we omit $\Sigma_i$ from the notation
$\Vert (s,v)\Vert_{W^2_{m+1,\delta}(\Sigma_i)}$
in case the domain is obvious from the context.
We denote by 
$L^2_{m+1,\delta}((\Sigma_i,\partial \Sigma_i);u_i^*TX,u_i^*TL)$ 
the subspace of $W^2_{m+1,\delta}((\Sigma_i,\partial \Sigma_i);u_i^*TX,u_i^*TL)$
consisting of elements with $v=0$.\index[syindex]{L2m+1deltaSigma@$L^2_{m+1,\delta}((\Sigma_i,\partial \Sigma_i);u_i^*TX,u_i^*TL)$}
\par
The space $L^2_{m,\delta}(\Sigma_i;u_i^{*}TX \otimes \Lambda^{0,1})$
\index[syindex]{L2mdelta@$L^2_{m,\delta}(\Sigma_i;u_i^{*}TX \otimes \Lambda^{0,1})$}
is defined similarly without boundary condition and without $v$.
(See (\ref{normformjula52}).)
\end{defn}
\par
Since the relevant $s$
for $T = \infty$ does not decay to 0 but converges to an element $v \in T_{p_0}X$, we denote
the subscript of $\|s\|^2_{W^2_{m+1,\delta}}$ to be $W^2_{m+1,\delta}$ instead of $L^2_{m+1,\delta}$.
\par
We let $D_{u_i}\overline{\partial}$ act trivially on the $v$-component and just act on
the $s$-component of $(s,v) \in  W^2_{m+1,\delta}((\Sigma_i,\partial \Sigma_i);u_i^*TX,u_i^*TL)$.
We  choose $\delta > 0$ with
\begin{equation}\label{form310310}
\delta \leq \frac{\delta_1}{10}
\end{equation}
where $\delta_1> 0$ is the constant given in
Lemma \ref{lem:C1expdecay}.
{(We refer \eqref{form627new}, (\ref{112312}) and the discussion therearound
for the reason  we make such a choice.)}
{\it We never change this constant $\delta$ in the rest of this paper.}
Lemma \ref{lem:Ckexpdecay}
(\ref{approestu})
implies
\begin{equation}\label{form310}
\left\vert (D_{u_i}\delbar) (v^{\rm Pal})\right\vert_{C^k} < C_{k,(\ref{form310})}
C_{k,(\ref{eq2930})} e^{-\delta_1|\tau|}
\Vert v\Vert.
\end{equation}
Therefore the operator (\ref{lineeq}), especially with $W^2_{m+1,\delta}((\Sigma_i,\partial \Sigma_i);u_i^*TX,u_i^*TL)$
as its domain, is defined and bounded.
It is a standard fact (see \cite{lock-mcowen}) that this operator is also Fredholm
if we choose $\delta$ sufficiently small. In the current context, $\delta$ depends only
on the geometry of $(X,g,J)$ and $L \subset X$.
\begin{rem}
We would like to note that $W^2_{m+1,\delta}((\Sigma_i,\partial \Sigma_i);u_i^*TX,u_i^*TL)$
is a completion of the set of infinitesimal variations of
$u_i: \Sigma_i \to X$ satisfying the boundary conditions
$$
u_i(\del \Sigma_i) \subset L, \quad \lim_{\tau \to +\infty} u_{ 1}(\tau,t)
 = \lim_{\tau \to -\infty} u_{ 2}(\tau,t){= p_0}  \in L
$$
in particular allowing the point $p_0$ {to vary} inside $L$. We have an exact sequence
$$
\aligned
0 & \to L^2_{m+1,\delta}((\Sigma_i,\partial \Sigma_i);u_i^*TX,u_i^*TL)
\\
&\to  W^2_{m+1,\delta}((\Sigma_i,\partial \Sigma_i);u_i^*TX,u_i^*TL) \to
T_{p_0}L \to 0
\endaligned
$$
where the first map is the inclusion $s \mapsto (s,0)$ and the second map is the
asymptotic evaluation of the section $s$ at $\pm \infty$ for $i = 1, \, 2$ respectively.
\end{rem}
We next define our (family of) obstruction spaces $\mathcal E_i(u')$.
\par
We fix a pair of pseudoholomorphic maps $u_i^{\frak{ob}} : (\Sigma_i,\partial \Sigma_i) \to (X,L)$
for $i=1,2$. We assume
\begin{equation}\label{disococ}
d(u_i(z),u^{\frak{ob}}_i(z)) \le \frac{\iota'_X}{2}
\end{equation}
where $\iota'_X$ is as in (\ref{cond23}).
\par
We take
a finite dimensional linear subspace \index[syindex]{Eiob@$\mathcal E^{\frak{ob}}_i$}
\begin{equation}\label{Eitake}
\mathcal E_i^{{\frak{ob}}} = \mathcal E_i(u^{\frak{ob}}_i) \subset \Gamma(K_i;(u^{\frak{ob}}_i)^*TX\otimes \Lambda^{0,1}), \quad i =1, \,2
\end{equation}
which consists of smooth sections of $(u^{\frak{ob}}_i)^*TX\otimes \Lambda^{0,1}$ supported in
a compact subset
$K_i^{0}$ of ${\rm Int}\,K_i$.
\begin{rem}
When we use Theorems \ref{gluethm1} and  \ref{exdecayT} for the construction of the
Kuranishi structure on a moduli space of bordered pseudoholomorphic
curves,
we take several $u^{\frak{ob}}_i$'s and {associate $\mathcal E_i^{\frak{ob}}$ to} each of them. Then
in place of $\mathcal E_1(u') \oplus  \mathcal E_2(u')$ we take a sum of
finitely many of them.
See \cite[page 1003]{FOn} and \cite[Definition 18.12]{fooo:techI}.
For the simplicity of exposition, we focus on the case
when we have one set of $u_i^{\frak{ob}}$, $\mathcal E_i^{\frak{ob}}$
in this paper.
The case of several such  $u_i^{\frak{ob}}$'s and $\mathcal E_i^{\frak{ob}}$'s is
treated in the same way as that of Theorems \ref{gluethm1} and  \ref{exdecayT}.
\end{rem}
\subsection{{Analytic framework for gluing}}
\par
We state the first half of the main result as Theorem \ref{gluethm1} below.
The second half is Theorem \ref{exdecayT}.
Actually Theorem \ref{gluethm1} itself is {a classical result.
We provide its statement here because we write its proof in such a way that is suitable
for the proof of derivative estimates given in Theorem \ref{exdecayT}.}
Let
\begin{equation}\label{uprime}
u' : (\Sigma_T,\partial\Sigma_T) \to (X,L)
\end{equation}
be a smooth map.

We consider the following conditions for given constant $\epsilon >0$.
\begin{conds}\label{nearbyuprime}
$ $
\begin{enumerate}
\item
$u'\vert_{K_i}$ is $\epsilon$-close to $u_i\vert_{K_i}$ in {the} $C^1$ sense.
\footnote{We use a Riemannian metric on $K_i$ and $X$ to define {the} $C^1$ norm we use.}
\item
The diameter of
$u'(X \setminus (K_1 \cup K_2))$
is smaller than $\epsilon$.
\end{enumerate}
\end{conds}
\par
We define the map \index[syindex]{pal01@$(\Pal_{u_i^{\frak{ob}}}^{u'})^{(0,1)}$}
\begin{equation}\label{defpali}
(\Pal_{u_i^{\frak{ob}}}^{u'})^{(0,1)}: \Gamma(K_i;(u_i^{\frak{ob}})^*TX \otimes \Lambda^{0,1}) \to \Gamma(\Sigma_T;(u')^*TX \otimes \Lambda^{0,1})
\end{equation}
by composing the map
$
(\Pal_{u_i^{\frak{ob}}}^{u'})^{(0,1)} :
\Gamma(K_i;(u_i^{\frak{ob}})^*TX \otimes \Lambda^{0,1}) \to \Gamma(K_i;(u')^*TX \otimes \Lambda^{0,1})
$
in (\ref{paluv01})
with the map induced by the inclusion $K_i \subset \Sigma_T$.
\par
If we take $\epsilon> 0$ sufficiently small, then (\ref{disococ}) and Condition \ref{nearbyuprime} (1) imply that
this map (\ref{defpali}) is defined for such $u'$.
\par
We then define the map \index[syindex]{Iuprimei@$I_{u',i}$}
\begin{equation}\label{mapIuu}
I_{u',i} : \mathcal E_i^{\frak{ob}} \to \Gamma(\Sigma_T;(u')^*TX \otimes \Lambda^{0,1})
\end{equation}
to be the restriction of $(\Pal_{u_i^{\frak{ob}}}^{u'})^{(0,1)}$ to $\mathcal E_i^{\frak{ob}}$ and put
\begin{equation}\label{Eiuiprime}
\mathcal E_i(u') =  I_{u',i}(\mathcal E_i^{\frak{ob}}).
\end{equation}
Recall that elements of $\mathcal E_i^{\frak{ob}}$ are supported on $K_i^{0}
\subset {\rm Int}\,K_i$.
Therefore for this definition, we first take the parallel transport\footnote{more precisely 
its complex linear part. We abuse the notation sometimes in this way.} of
$\eta \in \mathcal E_i^{\frak{ob}}$ along
the shortest geodesic between $u^{\frak{ob}}_i(z)$ and $u'(z)$ for $z \in
K_i^{0} \subset \Sigma_T$ and then extend the resulting section by zero to the
rest of $\Sigma_T$.
\par
The equation we study in the rest of the paper is of the form
\begin{equation}\label{mainequation}
\overline\partial u' \equiv 0 , \quad \mod \mathcal E_1(u') \oplus  \mathcal E_2(u').
\end{equation}
\begin{defn}\label{defn310}
We denote by $\mathcal M^{\mathcal E_1\oplus \mathcal E_2}((\Sigma_T,\vec z);u_1,u_2)_{\epsilon}$\index
[syindex]{ME1E2@$\mathcal M^{\mathcal E_1\oplus \mathcal E_2}((\Sigma_T,\vec z);u_1,u_2)_{\epsilon}$}
the set of solutions of (\ref{mainequation}) satisfying the Condition \ref{nearbyuprime}.
\end{defn}

Theorem \ref{gluethm1} will describe all the solutions of (\ref{mainequation})
`sufficiently close to the (pre-)glued map $u_1\# u_2$.'
To make this statement precise, we need to prepare more notations.
\par
Let $u_i': (\Sigma_i,\partial\Sigma_i) \to (X,L)$ be any smooth map, not necessarily
pseudoholomorphic. We put the following conditions on $u'_i$.
\begin{conds}\label{uiconds}
$ $
\begin{enumerate}
\item
$u_i'\vert_{K_i}$ is $\epsilon$-close to $u_i\vert_{K_i}$ in $C^1$ sense.
\item
The diameter of $u_1'(\Sigma_1 \setminus K_1)$, (resp.  $u_2'(\Sigma_2 \setminus K_2)$)
is smaller than $\epsilon$.
\end{enumerate}
\end{conds}

Then we define \index[syindex]{Iuprimei@$I_{u'_i}$}
\begin{equation}\label{newform320}
I_{u'_i} : \mathcal E_i^{\frak{ob}} \to \Gamma(\Sigma_i;(u_i')^*TX \otimes \Lambda^{0,1})
\end{equation}
in the same way as $I_{u'}$
i.e.,
$$
I_{u'_i}(\eta) := (\Pal_{u_i^{\frak{ob}}}^{u'_i})^{(0,1)}(\eta), \quad \eta \in
\mathcal E_i^{\frak{ob}}.
$$
(This makes sense if $u'_i$ satisfies Condition \ref{uiconds} for $\epsilon < \iota'_X/2$.
In fact, then we have $d(u'_i(x),u_i^{\frak{ob}}(x)) < \iota'_X$.)
We put \index[syindex]{Ei@$\mathcal E_i(u'_i)$}
\begin{equation}\label{eq3212}
\mathcal E_i(u'_i) = I_{u'_i}(\mathcal E_i^{\frak{ob}}).
\end{equation}
So we can define the equation
\begin{equation}\label{mainequationui}
\overline\partial u_i' \equiv 0 , \quad \mod \mathcal E_i(u'_i).
\end{equation}
We write $\mathcal E_i$ in place of $\mathcal E_i(u'_i)$ sometimes.

\begin{rem}\label{remark3838}
Note that in the current situation where 
all the irreducible components of $\Sigma_i$ intersect 
with boundary and $\Sigma_i$ has nontrivial
boundary with at least one boundary marked point (that is, either $\tau =\infty$ or $-\infty$ depending on
$i=1$ or $i = 2$),
Assumption \ref{assm34} implies that
$\Sigma_i$ carries no nontrivial automorphism.
In the case when there is an irreducible component which
does not intersect boundary, the automorphism group can be nontrivial.
We can use the standard trick to add marked points and then use codimension 2 submanifold
to kill the extra freedom of moving added marked points.
(See \cite[Appendix]{FOn} or \cite[Part IV]{fooo:techI}, for example.)
\end{rem}
\begin{defn}
We denote by $\mathcal M^{\mathcal E_i}((\Sigma_i,\vec z_i); u_i)_{\epsilon}$ 
\index[syindex]{MEiuiepsilon
@$\mathcal M^{\mathcal E_i}((\Sigma_i,\vec z_i); u_i)_{\epsilon}$} 
the set of
all maps $u'_i
: (\Sigma_i,\partial\Sigma_i) \to (X,L)$
whose domain is the marked
bordered Riemann surface
$
(\Sigma_i,\vec{z_i})
$ {given}
as in (\ref{eq:Sigma12}) for $i = 1, \,2$
such that
\begin{enumerate}
\item $u'_i$ satisfies
\eqref{mainequationui}.
\item Condition \ref{uiconds} is satisfied.
\end{enumerate}
\end{defn}
When $\beta_i = [u_i] \in H_2(X,L)$ we write 
$\mathcal M^{\mathcal E_i}((\Sigma_i,\vec z_i); \beta_i)_{\epsilon}$
sometimes in place of $\mathcal M^{\mathcal E_i}((\Sigma_i,\vec z_i); u_i)_{\epsilon}$. \index[syindex]{MEiuiepsilon
beta@$\mathcal M^{\mathcal E_i}((\Sigma_i,\vec z_i); \beta_i)_{\epsilon}$}

We work under the following assumption.
This assumption is put on the pair $(u_1,u_2)$ which we glue.
\par

\begin{assump}\label{DuimodEi}
Let $u_1(\infty) = u_2(-\infty) = p_0 \in L$ and
let
$
(D_{u_i}\overline{\partial})^{-1}(\mathcal E_i(u_i))
$
be the kernel of (\ref{DuimodEi0}). We assume:
\begin{enumerate}
\item[{1.}]{(Mapping transversality)}\index{Mapping transversality}
\begin{equation}\label{DuimodEi0}
D_{u_i}\overline{\partial} :
W^2_{m+1,\delta}((\Sigma_i,\partial \Sigma_i);u_i^*TX,u_i^*TL)
\to L^2_{m,\delta}(\Sigma_i;u_i^{*}TX \otimes \Lambda^{0,1})/\mathcal E_i(u_i)
\end{equation}
is surjective.
\item[{2.}]{(Evaluation transversality)} \index{Evaluation transversality} For each $i = 1, \, 2$, define the linearized evaluation map
\begin{equation}\label{Duiev}
D{\rm ev}_{i,\infty} : W^2_{m+1,\delta}((\Sigma_i,\partial \Sigma_i);u_i^*TX,u_i^*TL)
\to  T_{p_0}L
\end{equation}
by
$$
D{\rm ev}_{i,\infty}(s,v) = v.
$$
Then
\begin{equation}\label{Duievsurj}
D{\rm ev}_{1,\infty} - D{\rm ev}_{2,\infty} : (D_{u_1}\overline{\partial})^{-1}(\mathcal E_1(u_1))
\oplus (D_{u_2}\overline{\partial})^{-1}(\mathcal E_2(u_2))
\to T_{p_0}L
\end{equation}
is surjective.
\end{enumerate}
\end{assump}

The surjectivity of (\ref{Duiev}), (\ref{Duievsurj}) and the implicit function theorem
imply that, if $\epsilon$ is {small enough}, there exists a finite dimensional vector space $\widetilde V_i$
and {a} neighborhood $V_i(\epsilon)$ of $0$ {therein}, which provides a local parametrization of
$\mathcal M^{\mathcal E_i}((\Sigma_i,\vec z_i);u_i)$,
\begin{equation}\label{eq:MMEiVi}
\mathcal M^{\mathcal E_i}((\Sigma_i,\vec z_i); u_i)_{\epsilon} \cong V_i(\epsilon)
\end{equation}
near the given solution $u_i$.
For any $\rho_i \in V_i(\epsilon)$, \index[syindex]{Viepsilon@$V_i(\epsilon)$} we denote by $u_i^{\rho_i} : (\Sigma_i,\partial\Sigma_i) \to (X,L)$
the corresponding solution of (\ref{mainequationui}).
\par
We have an asymptotic evaluation map\index[syindex]{eviinfty@$\text{\rm ev}_{i,\infty}$}
$$
\text{\rm ev}_{i,\infty} : \mathcal M^{\mathcal E_i}((\Sigma_i,\vec z_i);u_i)_{\epsilon} \to L.
$$
Namely, we define
$$
\text{\rm ev}_{i,\infty} (u'_i)  = \lim_{\tau\to \pm\infty}u'_i(\tau,t).
$$
(Here $\pm = +$ for $i=1$ and $-$ for $i=2$.).

We consider the fiber product:
\begin{equation}\label{fpmoduli}
V_1(\epsilon) \times_L V_2(\epsilon) =
\mathcal M^{\mathcal E_1}((\Sigma_1,\vec z_1);u_1)_{\epsilon} \times_L
\mathcal M^{\mathcal E_2}((\Sigma_2,\vec z_2);u_2)_{\epsilon}.
\end{equation}
The surjectivity of (\ref{Duievsurj}) implies that
the fiber product (\ref{fpmoduli})
is transversal and
so
$
V_1(\epsilon) \times_L V_2(\epsilon)
$
is a smooth manifold.
We write elements of $V_1(\epsilon) \times_L V_2(\epsilon)$  as $\rho = (\rho_1,\rho_2)$.\index[syindex]{GlueT@$\text{\rm Glue}_T$}
\begin{thm}\label{gluethm1}
For each  $\epsilon(1) > 0$, there exist
$\epsilon(2) > 0$, $T_{1,\epsilon(1)}>0$ and a map
$$
\text{\rm Glue}_T :
\mathcal M^{\mathcal E_1}((\Sigma_1,\vec z_1);u_1)_{\epsilon(2)} \times_L \mathcal M^{\mathcal E_2}
((\Sigma_2,\vec z_2);u_2)_{\epsilon(2)}
\to
\mathcal M^{\mathcal E_1\oplus \mathcal E_2}((\Sigma_T,\vec z);u_1,u_2)_{\epsilon(1)}
$$
for $T > T_{1,\epsilon(1)}${that satisfies the following. Here
$\vec z = \vec z_1 \cup \vec z_2$.
\begin{enumerate}
\item
The map $\text{\rm Glue}_T$
is a diffeomorphism onto its image.
\item
There exist $\epsilon_{\epsilon(1),\epsilon(2),\langle\!\langle
{\rm \ref{gluethm1}}\rangle\!\rangle} > 0$
and
$T_{2,\epsilon(1),\epsilon(2)} > 0$
depending on $\epsilon(1)$, $\epsilon(2)$
such that the image of $\text{\rm Glue}_T$
contains $\mathcal M^{\mathcal E_1\oplus \mathcal E_2}((\Sigma_T,\vec z);u_1,u_2)_{\epsilon_{\epsilon(1),\epsilon(2),\langle\!\langle
{\rm \ref{gluethm1}}\rangle\!\rangle}}$
if $T > T_{2,\epsilon(1),\epsilon(2)}$.
\end{enumerate}
}
\end{thm}
The main result about exponential decay estimate of this map will be given in Chapter \ref{subsecdecayT}.
(Theorem \ref{exdecayT}.)

\section[Cut-off functions and
weighted Sobolev norm]{Proof of the gluing theorem  : I - Cut-off functions and
weighted Sobolev norm}
\label{subsec12}

The proof of Theorem \ref{gluethm1}
was given in \cite[Section 7.1.3]{fooo:book1}.
Here we follow the scheme of the proof of \cite[Section A1.4]{fooo:book1} and
provide more of its details. As mentioned there, the scheme of the proof is
originated from Donaldson's paper \cite{Don86I},
and its Morse-Bott version in \cite{Fuk96II}.
\par
We first introduce certain cut-off functions.
First let $\mathcal A_T \subset \Sigma_T$
and $\mathcal B_T \subset \Sigma_T$ be the domains
defined by \index[syindex]{AT@$\mathcal A_T$}, \index[syindex]{BT@$\mathcal B_T$}, 
$$
\mathcal A_T = [-T-1,-T+1] \times [0,1],
\qquad
\mathcal B_T = [T-1,T+1] \times [0,1].
$$
We may regard $\mathcal A_T,\mathcal B_T$ as subsets of $\Sigma_i$.
The third domain is 
$$
\mathcal X = [-1,1] \times [0,1] \subset \Sigma_T.
$$
\begin{figure}
\centering
\includegraphics{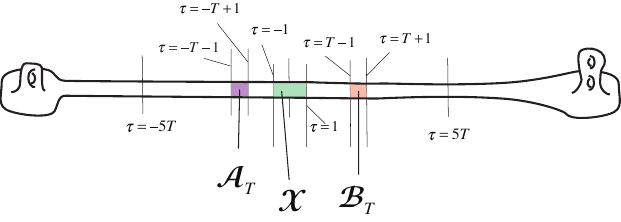}
\caption{$\mathcal A_T,\mathcal B_T$ and $\mathcal X$.}
\label{Figure2-5}
\end{figure}
See Figure \ref{Figure2-5}.
We may also regard $\mathcal X \subset \Sigma_i$. We recall readers the definitions of
$\Sigma_i$ from \eqref{eq:Sigma12}.
\par
Let $\chi_{\mathcal A}^{\leftarrow}$,
$\chi_{\mathcal A}^{\rightarrow}$ \index[syindex]{chiA@$\chi_{\mathcal A}^{\leftarrow}$} be the smooth functions on $[-5T,5T]\times [0,1]$ defined by
\index{cutoff function}
\begin{equation}\label{eq:chiA}
\chi_{\mathcal A}^{\leftarrow}(\tau,t)
= \begin{cases}
1   & \tau < -T-1 \\
0  & \tau > -T+1.
\end{cases}
\end{equation}
$$
\chi_{\mathcal A}^{\rightarrow}  = 1 - \chi_{\mathcal A}^{\leftarrow}.
$$
See Figure \ref{Figure3}.
We also define \index[syindex]{chiB@$\chi_{\mathcal B}^{\leftarrow}$}
\begin{equation}\label{eq:chiB}
\chi_{\mathcal B}^{\leftarrow}(\tau,t)
= \begin{cases}
1   & \tau <  T-1 \\
0  & \tau >  T+1.
\end{cases}
\end{equation}
$$
\chi_{\mathcal B}^{\rightarrow}  = 1 - \chi_{\mathcal B}^{\leftarrow}
$$
and \index[syindex]{chiX@$\chi_{\mathcal X}^{\leftarrow}$}
\begin{equation}\label{eq:chiX}
\chi_{\mathcal X}^{\leftarrow}(\tau,t)
= \begin{cases}
1   & \tau <  -1 \\
0  & \tau >  1.
\end{cases}
\end{equation}
$$
\chi_{\mathcal X}^{\rightarrow}  = 1 - \chi_{\mathcal X}^{\leftarrow}.
$$
We extend these functions to $\Sigma_T$ and $\Sigma_i$ ($i=1,2$) so that
they are locally constant outside $[-5T,5T]\times [0,1]$.
We denote them by the same symbol.
\begin{figure}
\centering
\includegraphics{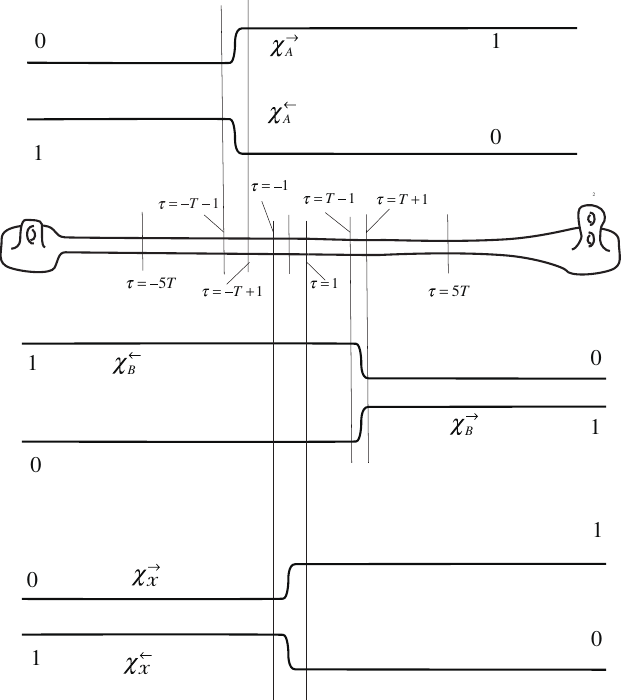}
\caption{Cut off functions.}
\label{Figure3}
\end{figure}
\par\medskip
We now introduce weighted Sobolev norms \index{weighted Sobolev norm} and their local versions for sections
on $\Sigma_T$ or $\Sigma_i$.
We first define weight functions $e_{i,\delta} : \Sigma_i \to [1,\infty)$, $i = 1,\, 2$ of $C^{\infty}$-class
as follows:\index[syindex]{e1delta@$e_{1,\delta}$}\index{weight function}
\begin{equation}\label{e1delta}
e_{1,\delta}(\tau,t)
\begin{cases}
=e^{\delta (\tau + 5T)} &\text{if $\tau > 1 -  5T$ }\\
=1   &\text{on  $K_1$} \\
\in [1,10] &\text{if $\tau < 1 -  5T$}
\end{cases}
\end{equation}
\begin{equation}\label{e2delta}
e_{2,\delta}(\tau,t)
\begin{cases}
=e^{\delta(-\tau + 5T)} &\text{if $\tau <  5T-1$ }\\
=1   &\text{on  $K_2$} \\
\in [1,10] &\text{if $\tau >  5T-1$}.
\end{cases}
\end{equation}
\index[syindex]{e2delta@$e_{2,\delta}$}

(We choose $\delta > 0$ so small that it satisfies $1 < e^\delta < 10$ above.)
We also define $e_{T,\delta} : \Sigma_T \to [1,\infty)$ as follows.
(See Figure \ref{Figure4-2}.) \index[syindex]{eTdelta@$e_{T,\delta}$}
\begin{equation}\label{e2delta2}
e_{T,\delta}(\tau,t)
\begin{cases}
=e^{\delta(-\tau + 5T)} &\text{if $1<\tau <  5T-1$ }\\
= e^{\delta(\tau + 5T)} &\text{if $-1>\tau > 1-5T$ }\\
=1   &\text{on  $K_1\cup K_2$} \\
\in [1,10] &\text{if $\vert\tau - 5T\vert < 1$ or $\vert\tau + 5T\vert < 1$} \\
\in [e^{5T\delta}/10,e^{5T\delta}] &\text{if $\vert\tau\vert < 1$}.
\end{cases}
\end{equation}
\begin{figure}
\centering
\includegraphics{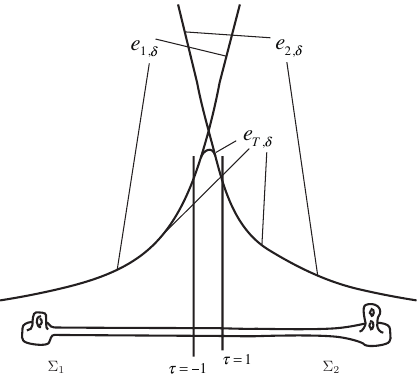}
\caption{Weight function.}
\label{Figure4-2}
\end{figure}
The weighted Sobolev norm we use for
$L^2_{m,\delta}(\Sigma_i;u_i^*TX\otimes \Lambda^{0,1})$
is given by\index[syindex]{L2mnorm@$\Vert\Vert_{L^2_{m,\delta}}$}
\begin{equation}
\| s\|^2_{L^2_{m,\delta}} = \sum_{k=0}^m \int_{\Sigma_i}
e^2_{i,\delta} \vert \nabla^k s\vert^2 \text{\rm vol}_{\Sigma_i}.
\end{equation}
For
$(s,v) \in W^2_{m+1,\delta}((\Sigma_i,\partial \Sigma_i);u_i^*TX,u_i^*TL)$
we defined in (\ref{normWW}) 
\begin{equation}\label{normformula}
\aligned
\| (s,v)\|^2_{W^2_{m+1,\delta}} = &\sum_{k=0}^{m+1} \int_{K_i}
\vert \nabla^k s\vert^2 \text{\rm vol}_{\Sigma_i}\\
&+
\sum_{k=0}^{m+1} \int_{\Sigma_i \setminus K_i}  e^2_{i,\delta}\vert \nabla^k(s - v^{\rm Pal})\vert^2 \text{\rm vol}_{\Sigma_i}
+ \| v\|^2,
\endaligned
\end{equation}
where $v^{\rm Pal}$ is defined in (\ref{eq:vpal}).\index[syindex]{W2mnorm@$\Vert\Vert_{W^2_{m+1,\delta}}$}
\par
We next define a weighted Sobolev norm for the sections on
$\Sigma_T$.
Let $s$ be a section of $u^*TX$ with $s|_{\del \Sigma_T} \in \Gamma(u^*TL)$ such that
$$
s \in L^2_{m+1}((\Sigma_T,\partial \Sigma_T);u^*TX,u^*TL).
$$
Since we take $m$ large, $s$ is continuous.
So the value $s(0,1/2) \in T_{u(0,1/2)}X$ at $(0,1/2)$ is well defined.

We define the norm of $s$ by
\begin{equation}\label{normformjula5}
\aligned
\| s\|^2_{W^2_{m+1,\delta}} = &\sum_{k=0}^{m+1} \int_{K_1}
\vert \nabla^k s\vert^2 \text{\rm vol}_{\Sigma_1} + \sum_{k=0}^{m+1} \int_{K_2}
\vert \nabla^k s\vert^2 \text{\rm vol}_{\Sigma_2}\\
&+
\sum_{k=0}^{m+1} \int_{[-5T,5T]\times [0,1]}  e^2_{T,\delta}\vert \nabla^k(s - s(0,1/2)^{\rm Pal})\vert^2 \text{\rm vol}_{\Sigma_i}
\\&+ \| s(0,1/2)\|^2.
\endaligned
\end{equation}
For
$
\eta \in L^2_{m}((\Sigma_T,\partial \Sigma_T);u^*TX\otimes \Lambda^{0,1}),
$
we define
\begin{equation}\label{normformjula52}
\| \eta \|^2_{L^2_{m,\delta}} = \sum_{k=0}^{m} \int_{\Sigma_T}
e^2_{T,\delta}\vert  \nabla^k \eta \vert^2 \text{\rm vol}_{\Sigma_1}.
\end{equation}
These norms were used in \cite[Section 7.1.3]{fooo:book1}.
\par

\begin{rem} We remark that the norm $\| s\|_{W^2_{m+1,\delta}}$ is equivalent to
the standard $L^2_{m+1}$-norm on $L^2_{m+1}((\Sigma_T,\partial \Sigma_T);u^*TX,u^*TL)$,
when $T$ is fixed. However to study its relationship with the norm \eqref{normformula} as $T \to \infty$,
making this choice of $T$-dependent norm in this way is necessary.
We remark that the ratio between the norm $\| s\|_{W^2_{m+1,\delta}}$ and the
standard $L^2_{m+1}$-norm diverges as $T$ goes to infinity.
\end{rem}
For a subset $W$ of $\Sigma_i$ or $\Sigma_T$ we define
$
\| s \|^2_{W^2_{m,\delta}(W\subset \Sigma_i)}
$,
$
\| \eta\|^2_{L^2_{m,\delta}(W\subset \Sigma_T)}
$
by restricting the domain of the integration
(\ref{normformjula5}) or (\ref{normformjula52}) to $W$.
We use this notation in order to specify the weight etc. we use.
In case there is no possibility of confusion of the domain and weight etc. we use,
we omit the domain $\Sigma_i$ etc. from the notation
$
\| s \|^2_{W^2_{m,\delta}(\Sigma_i)}
$ etc..
\par
We also define the inner product on $W^2_{m+1,\delta}((\Sigma_i,\partial \Sigma_i);u_i^*TX,u_i^*TL)$
by setting
\begin{equation}\label{innerprod1}
\aligned
\langle\!\langle (s_1,v_1),(s_2,v_2)\rangle\!\rangle_{L^2}
=
& \int_{\Sigma_i\setminus K_i}  (s_1-v_1^{\rm Pal},s_2-v^{\rm Pal}_2)
\\
&+\int_{K_i} (s_1,s_2)
+
( v_1,v_2)
\endaligned
\end{equation}
for $(s_j,v_j) \in W^2_{m+1,\delta}((\Sigma_i,\partial \Sigma_i);u_i^*TX,u_i^*TL)$
for $j=1,2$.
\section{Proof of the gluing theorem : II - Gluing by alternating method}
\label{alternatingmethod}

We motivate the gluing construction we employ in this section
by comparing it with the well-known Newton's iteration \index{Newton's iteration} scheme of solving the equation
$f(x) = 0$ for a real-valued function, starting from an approximate solution $x_0$ that is sufficiently close
to a genuine solution nearby. The iteration scheme is the inductive scheme via the
recurrence relation
\be\label{eq:Newton's}
x_{n+1} = x_n - \frac{f(x_n)}{f'(x_n)}.
\ee
Therefore
\be\label{eq:difference}
|x_{n+1} - x_n| = \left|\frac{f(x_n)}{f'(x_n)}\right|
\ee
and the expected solution is given by the infinite sum
$$
x_\infty = x_0 + \sum_{n=1}^\infty (x_{n+1} - x_n)
$$
provided the series converges. A simple way of ensuring the required convergence is to
establish existence of constant $C > 0, \, 0 \leq \mu < 1$ independent of $n$ such that
\be\label{eq:geometric}
\left|\frac{f(x_n)}{f'(x_n)}\right| \leq C \mu^n
\ee
for all $n$.
\par
\medskip
With this iteration scheme in our mind, we proceed with performing the scheme.

\subsection{{The first trial approximate solution}}
\label{subsec:1st-trial}
\smallskip
In our construction we obtain an approximate solution of the linearized equation
({whose error} corresponds to $f(x_0)/f'(x_0)$ in  (\ref{eq:Newton's}))
by an alternating method.
We start with a pair of maps
$$
u^{\rho} = (u_1^{\rho_1},u_2^{\rho_2})
\in \mathcal M^{\mathcal E_1}((\Sigma_1,\vec z_1);u_1)_{\epsilon_2}
\times_L \mathcal M^{\mathcal E_2}((\Sigma_2,\vec z_2);u_2)_{\epsilon_2}
$$
for each sufficiently small $\epsilon_2> 0$.
Here $\rho_i \in V_i$ and $u^{\rho_i}_i$ is the corresponding map $(\Sigma_i,\partial \Sigma_i) \to (X,L)$
given by \eqref{eq:MMEiVi}.
Let $\rho = (\rho_1,\rho_2)$.\index[syindex]{rho@$\rho$}
We put
$$
p^{\rho}_0 = \lim_{\tau\to \infty} u_1^{\rho_1}(\tau,t) = \lim_{\tau\to -\infty} u_2^{\rho_2}(\tau,t).
$$
\par\medskip
\noindent{\bf Pregluing}: This step corresponds to {picking} the initial trial approximate solution $x_0$ in
Newton's scheme above.

\par\medskip
\noindent{\bf Step 0-i} {\bf(Approximate solutions)}:

\begin{defn}\label{defn:uT0rho}
We define \index[syindex]{uT0@$u_{T,(0)}^{\rho}$}
\begin{equation}\label{form5656}
u_{T,(0)}^{\rho} =
\begin{cases} \Exp\Big(p_0^\rho,
\chi_{\mathcal B}^{\leftarrow} {\rm E}(p_0^{\rho},u_1^{\rho_1}) +
\chi_{\mathcal A}^{\rightarrow} {\rm E}(p_0^\rho, u_2^{\rho_2})\Big)
& \text{on $[-5T,5T] \times [0,1]$} \\
u_1^{\rho_1} & \text{on $K_1$} \\
u_2^{\rho_2} & \text{on $K_2$}.
\end{cases}
\end{equation}
\end{defn}

This pre-glued map $u_{T,(0)}^{\rho}$ plays the role of $x_0$ in the Newton scheme for
solving the equation
\be\label{eq:mainequation}
\Pi_{\mathcal E_T}^\perp \delbar u = 0
\ee
on
$
L^2_{m+1,\delta}(\Sigma_T,\del \Sigma_T; (u_{T,(0)}^{\rho})^*TX,(u_{T,(0)}^{\rho})^*TL)
$
over $\Sigma_T$, where $\Pi_{\mathcal E_T}^\perp$ is the projection to a subspace
$
\mathcal E_T^\perp.
$
We will need to take the derivative of the map $u \mapsto \Pi_{\mathcal E_T}^\perp\delbar u$
the analog to $f'(x_n)$ in the Newton's scheme. We will actually use an
approximate derivative thereof whose explanation is now in order.
\par
\par\medskip
\noindent{{\bf Step 0-ii (Approximate linearization)}}:
Note that the subspace $\mathcal E_i(u'_i)$ varies on $u'_i$. (See \eqref{Eiuiprime} for the definition.)
So to obtain a linearized equation of (\ref{eq:mainequation})
at the given map $u'_i$, we need to take this dependence into account.

Let $\Pi_{\mathcal E_i(u_i')}$ {denote} the $L^2$-orthogonal projection of $L^2_{m,\delta}(\Sigma_i;(u_i')^{*}TX \otimes \Lambda^{0,1})$
to $\mathcal E_i(u_i')$ or its associated idempotent as an endomorphism
$$
\Pi_{\mathcal E_i(u_i')}:L^2_{m,\delta}(\Sigma_i;(u_i')^{*}TX \otimes \Lambda^{0,1})
\to L^2_{m,\delta}(\Sigma_i;(u_i')^{*}TX \otimes \Lambda^{0,1}).
$$
More explicitly, we represent the projection as follows:
we choose a
basis $\{\mathbf e_{i,a}\}$ of $\mathcal E_i$ and set \index[syindex]{eprimeia@$\mathbf e'_{i,a}$}
$$
\mathbf e'_{i,a}(u'_i): = I_{u_i'}(\mathbf e_{i,a}), \quad a = 1, \dots, \dim \mathcal E_i(u_i').
$$
We apply the standard Gram-Schmidt process to ${\bf e}'_{i,a}(u'_i)$ and denote
by ${\bf e}_{i,a}(u'_i)$ the resulting orthonormal basis of $\mathcal  E_i(u_i')$.
Then,
for $A\in L^2_{m,\delta}(\Sigma_i;(u_i')^{*}TX \otimes \Lambda^{0,1})$ we put \index
[syindex]{PiEi@$\Pi_{\mathcal E_i(u_i')}$}
\begin{equation}\label{form118}
\Pi_{\mathcal E_i(u_i')}(A) = \sum_{a=1}^{\dim E_i}
\langle\!\langle A,\mathbf e_{i,a}(u'_i) \rangle\!\rangle_{L^2(K_1)}\mathbf e_{i,a}(u'_i),
\end{equation}
where $\mathbf e_{i,a}(u_i')$ are supported in $K_i$.
\par
For given $V \in \Gamma((\Sigma_i,\partial \Sigma_i),(u'_i)^*TX,(u'_i)^*TL)$, we put
\begin{equation}\label{DEidef}
(D_{u'_i}\mathcal E_i)(A,V) := \frac{d}{ds}\Big\vert_{s=0}(\Pal_{i}^{(0,1)})^{-1}
\Pi_{\mathcal E_i(\Exp (u_i',sV))}(\Pal_{i}^{(0,1)}(A)).
\end{equation}
Here $\Pal_i^{(0,1)}$ is the closure of the map \index[syindex]{DuiEi@$(D_{u'_i}\mathcal E_i)$}
$(\Pal_{u'_i}^{{\rm Exp}(u'_i,sV)})^{(0,1)}$ defined in (\ref{paluv01}).
The right hand side is well defined in view of (\ref{form118}).
We then put
$$
\Pi_{\mathcal E_i(u_i')}^{\perp}(A) = A - \Pi_{\mathcal E_i(u_i')}(A).
$$
\par
The equation (\ref{mainequationui}) is equivalent to 
$\Pi_{\mathcal E_i(u_i')}^{\perp}(\overline\partial) = 0$. By evaluating the derivative
$$
\left.\frac{\partial}{\partial s}\right\vert_{s=0} (\Pal_i^{(0,1)})^{-1}
\left(\Pi_{\mathcal E_i(\Exp(u_i',s V))}^{\perp} \overline\partial \Exp(u_i',s V)\right)
$$
we obtain the linearization operator of \eqref{mainequationui}.

{
As an operator into $L^2_{m,\delta}(\Sigma_i;(u_i')^{*}TX \otimes \Lambda^{0,1})$ it
becomes
$$
V \mapsto (\Pi^{\perp}_{\mathcal E_i(u'_i)} \circ D_{u'_i}\overline\partial)(V)
- (D_{u'_i}\mathcal E_i)(\overline\partial u'_i,V).
$$
We define:
\begin{equation}\label{linearized2221}
(D_{u'_i}\overline \partial)^{\perp}(V) =  (D_{u'_i}\overline\partial) (V) - (D_{u'_i}\mathcal E_i)(\overline\partial u'_i,V).
\end{equation}}

{
We use a variant of this operator as an approximate derivative of the map $u \mapsto \Pi_{\mathcal E_i(u)}^\perp\delbar u$.
See \eqref{144op}.
}

\par\medskip
\noindent{\bf Step 0-3 (Error estimates)}:
We recall {that} $\supp(\mathcal E_i^{\frak{ob}}) \subset \operatorname{Int} K_i$
by {the} definition of $\mathcal E_i^{\frak{ob}}$.
\begin{lem}\label{lem18}
There exists $\frak e^{\rho} _{i,T,(0)} \in \mathcal E^{{\frak{ob}}}_i$
such that
\begin{equation}\label{startingestimate}
\|\overline\partial u^{\rho} _{T,(0)} - \frak e^{\rho} _{1,T,(0)}
- \frak e^{\rho} _{2,T,(0)}\|_{L_{m,\delta}^2} < C_{1,m}e^{-\delta_1 T}.
\end{equation}
Here $C_{1,m}$ is a constant depending on $X, L, u_i$ and $m$
and $u^{\rho} _{T,(0)}$ is as in Definition \ref{defn:uT0rho}.
\par
Moreover
for any $\epsilon(4) > 0$ there exists $\epsilon_{\epsilon(4),(\ref{frakeissmall})}$
such that
\begin{equation}\label{frakeissmall}
\| \frak e^{\rho} _{i,T,(0)}\|_{L^2(K_i)} < \epsilon(4)
\end{equation}
if $\rho = (\rho_1,\rho_2) \in V_1(\epsilon) \times_{L} V_2(\epsilon)$ with
$\epsilon < \epsilon_{\epsilon(4).(\ref{frakeissmall})}$.
\par
Furthermore $\frak e^{\rho} _{i,T,(0)}$ is supported in $K_i$ and is independent of $T$
as an element of $L^2_m(K_i)$.
\end{lem}
{
In (\ref{startingestimate}) we regard  $\frak e^{\rho} _{i,T,(0)} \in {\mathcal E_i^{\frak{ob}}}$
as an element of $\mathcal E_i(u^{\rho}_{T,(0)})$ by (\ref{Eiuiprime}) with 
$u' = u^{\rho} _{T,(0)}$.}
\begin{rem}
Note throughout the discussion of Chapters \ref{sec:statement}-\ref{surjinj}
we fix $X,L,u_i$. So we do not mention dependence of constants on them hereafter.
\end{rem}
\begin{proof}
By definition, $\overline\partial u^{\rho}_i \in \mathcal E_i(u_i^\rho)$.
We put
\begin{equation}\label{eq:erhoiT(0)}
\frak e_{i,T,(0)}^\rho : = \overline\partial u^{\rho}_i \in \mathcal E_i(u^{\rho}_i)
\cong \mathcal E^{\frak ob}_i
\end{equation}
extended to be zero to $\Sigma_T \setminus K_i$.
Then, by definition, the support of
$\overline\partial u^{\rho} _{T,(0)} - \frak e^{\rho} _{1,T,(0)} - \frak e^{\rho} _{2,T,(0)}$
is contained in $[-2T,2T]\times [0,1]$.
In fact on $K_1 \cup [-5T,-2T] \times [0,1]$, $u^{\rho} _{T,(0)} = u^{\rho_1}_1$
and on $K_2 \cup [2T,5T] \times [0,1]$, $u^{\rho} _{T,(0)} = u^{\rho_2}_2$,
{so}
$\overline\partial u^{\rho} _{T,(0)} = \frak e^{\rho} _{1,T,(0)} + \frak e^{\rho} _{2,T,(0)}$,
{thereon}.
Moreover by (\ref{approestu}) and (\ref{form5656}) the $C^k$ norm of  the map
$(\tau,t) \mapsto {\rm E}(p^{\rho}_0,u^{\rho} _{T,(0)}(\tau,t)) : [-2T,2T]\times [0,1]
\to T_{p^{\rho}_0}X$ at $(\tau,t)$ is smaller than
$C_ke^{- \delta_1 d(\partial [-5T,5T],\tau)}
\le C_ke^{- 3\delta_1 T}$, for any $k$.
Therefore we obtain the estimate (\ref{startingestimate}).
\end{proof}

\begin{rem}\label{remark54}
Note (\ref{eq:erhoiT(0)}) (resp. (\ref{558atoato}), (\ref{formD9atoato}))
specifies the choice of $\frak e^{\rho} _{i,T,(0)}$
(resp. $\frak e^{\rho} _{i,T,(1)}$, $\frak e^{\rho} _{i,T,(\kappa)}$).
This is the choice taken in this proof.
\end{rem}

\par\smallskip
\noindent{\bf Step 0-4 (Separating error terms into two parts)}:
{
We denote the initial error by
\be\label{Err(0)}
{\rm Err}^{\rho}_{T,(0)}
= \overline\partial u^{\rho} _{T,(0)} - \frak e^{\rho} _{1,T,(0)} - \frak e^{\rho} _{2,T,(0)}.
\ee
}

\begin{defn}\label{deferfirst}
We put
$$
\aligned
{\rm Err}^{\rho}_{1,T,(0)}
&= 
{\chi_{\mathcal X}^{\leftarrow}({\rm Err}^{\rho}_{T,(0)})}
=
\chi_{\mathcal X}^{\leftarrow} (\overline\partial u^{\rho} _{T,(0)} - \frak e^{\rho} _{1,T,(0)}), \\
{\rm Err}^{\rho}_{2,T,(0)}
&= 
{\chi_{\mathcal X}^{\rightarrow}({\rm Err}^{\rho}_{T,(0)})}
=
\chi_{\mathcal X}^{\rightarrow} (\overline\partial u^{\rho} _{T,(0)} -  \frak e^{\rho} _{2,T,(0)}).
\endaligned$$
We regard them as elements of the weighted Sobolev spaces
$L^2_{m,\delta}(\Sigma_1;(u_1^{\rho})^*TX\otimes \Lambda^{0,1})$
and
$L^2_{m,\delta}(\Sigma_2;(u_2^{\rho})^*TX\otimes \Lambda^{0,1})$
respectively.
(We extend them by $0$ outside $\supp \chi_{\mathcal X}^\leftarrow$, 
$\supp \chi_{\mathcal X}^\rightarrow$, respectively.)
\end{defn}
\par\medskip
\subsection{{The first iteration}}
\label{bubsec:firstit}
{We now start the first step of iteration.}
\par
\smallskip

\noindent{\bf Step 1-1 (Approximate solution for linearization)}:
 This step corresponds to solving the linearized equation
\begin{equation}\label{eq:Newtonf'}
f'(x_0)(v_0) = f(x_0),
\qquad
\text{namely
$v_0 = \frac{f(x_0)}{f'(x_0)}$}
\end{equation}
in Newton's iteration scheme.

We first cut-off $u^{\rho}_{T,(0)}$ and extend it to obtain maps $\hat u^{\rho}_{i,T,(0)} :
(\Sigma_i,\partial\Sigma_i) \to (X,L)$ $(i=1,2)$ as follows.
(This map is used to set the linearized operator (\ref{lineeqstep0}).)

\begin{equation}\label{form514514}
\aligned
&\hat u^{\rho}_{1,T,(0)}(z) \\
&=
\begin{cases} \Exp\left(p_0^\rho,\chi_{\mathcal B}^{\leftarrow}(\tau-T,t)
{\rm E}(p_0^\rho,u^{\rho}_{T,(0)}(\tau,t))\right)
&\text{if $z = (\tau,t) \in [-5T,5T] \times [0,1]$} \\
 u^{\rho}_{T,(0)}(z)
&\text{if $z \in K_1$} \\
p_0^{\rho}
&\text{if $z \in [5T,\infty)\times [0,1]$}.
\end{cases} \\
&\hat u^{\rho}_{2,T,(0)}(z) \\
&=
\begin{cases} \Exp\left(p_0^\rho, \chi_{\mathcal A}^{\rightarrow}(\tau+ T,t)
{\rm E}(p_0^\rho,u^{\rho}_{T,(0)}(\tau,t))\right)
&\text{if $z = (\tau,t) \in [-5T,5T] \times [0,1]$} \\
 u^{\rho}_{T,(0)}(z)
&\text{if $z \in K_2$} \\
p_0^{\rho}
&\text{if $z \in (-\infty,-5T]\times [0,1]$}.
\end{cases}
\endaligned
\end{equation}
Now we let
\begin{equation}\label{lineeqstep0}
\aligned
D_{\hat u^{\rho}_{i,T,(0)}}\overline{\partial}
:
W^2_{m+1,\delta}\Big((\Sigma_i,\partial \Sigma_i);&(\hat u^{\rho}_{i,T,(0)})^*TX,(\hat u^{\rho}_{i,T,(0)})^*TL\Big)
\\
&\to
L^2_{m,\delta}\left(\Sigma_i;(\hat u^{\rho}_{i,T,(0)})^{*}TX \otimes \Lambda^{0,1}\right)
\endaligned
\end{equation}
be the (covariant) linearization of the Cauchy-Riemann equation at $\hat u_{i,T,(0)}^\rho$, i.e.,
\be\label{Dudelbar}
(D_{\hat u^{\rho}_{i,T,(0)}}\overline{\partial})(V) =
\left.\frac{\partial}{\partial s}\right\vert_{s=0} \left((\Pal_i^{(0,1)})^{-1}
(\overline \partial \Exp(\hat u^{\rho}_{i,T,(0)},s V))\right).
\ee
Here $\Pal_i^{(0,1)}$ is the $(0,1)$ part of the parallel translation along the
map $r \mapsto \Exp (\hat u^{\rho}_{i,T,(0)},rV)$, $r \in [0,s]$ for some $s \in [0,1]$.
\begin{lem}\label{surjlineastep11}
There exist $\epsilon_{2},  T_{(\ref{surj1step10})}$ with the following properties.
Let $\mathcal E_i^{\frak{ob}}$ be the subspace chosen in \eqref{Eitake}.
Define
\be\label{eq:Eionurho}
\mathcal E_i(\hat u^{\rho}_{i,T,(0)}): = I_{\hat u^{\rho}_{i,T,(0)}}(\mathcal E_i^{\frak{ob}})
\ee
as the subspace of $L^2_{m,\delta}\left(\Sigma_i;(\hat u^{\rho}_{i,T,(0)})^{*}TX \otimes \Lambda^{0,1}\right)$. Then
the following holds if $T > T_{(\ref{surj1step10})}$ and
$\rho \in V_1(\epsilon) \times_L V_2(\epsilon)$ with $\epsilon < \epsilon_{2}$:
\begin{equation}\label{surj1step10}
\text{\rm Im}(D_{\hat u^{\rho}_{i,T,(0)}}\overline{\partial}) + \mathcal E_i(\hat u^{\rho}_{i,T,(0)})
= L^2_{m,\delta}\left(\Sigma_i;(\hat u^{\rho}_{i,T,(0)})^{*}TX \otimes \Lambda^{0,1}\right).
\end{equation}
Moreover
\begin{equation}\label{Duievsurjstep1}
\aligned
&D{\rm ev}_{1,\infty} - D{\rm ev}_{2,\infty}v\\
& : (D_{\hat u^{\rho}_{1,T,(0)}}
\overline{\partial})^{-1}(\mathcal E_1(\hat u^{\rho}_{1,T,(0)}))
 \oplus (D_{\hat u^{\rho}_{2,T,(0)}}\overline{\partial})^{-1}(\mathcal E_2
 (\hat u^{\rho}_{2,T,(0)}))
\to T_{p^{\rho}}L
\endaligned
\end{equation}
is surjective.
\end{lem}
\begin{proof}
Since $\hat u^{\rho}_{i,T,(0)}$ is close to $u_i$ in {an} exponential order, this is a consequence of
Assumption \ref{DuimodEi}.
\end{proof}

We note that
$
\overline \partial \hat u^{\rho}_{i,T,(0)} - \frak e^{\rho} _{i,T,(0)}
$
is exponentially small. Therefore the map
$$
V \mapsto (D_{\hat u^{\rho}_{i,T,(0)}}\overline{\partial})(V)
- (D_{\hat u^{\rho}_{i,T,(0)}}\mathcal E_i)(\frak e^{\rho} _{i,T,(0)}, V),
$$
which is obtained by replacing $\overline \partial \hat u^{\rho}_{i,T,(0)}$ by $\frak e^{\rho} _{i,T,(0)}$,
is exponentially close to the linearization operator $(D_{_{\hat u^{\rho}_{i,T,(0)}}}\delbar)^{\perp}$,
defined in
\eqref{linearized2221}.
We will consider this approximate linearization instead and denote it by
\index[syindex]{DurhoiTapp@$D_{\hat u^{\rho}_{i,T,(0)}}^{{\rm app},(0)}$}
$$
D_{\hat u^{\rho}_{i,T,(0)}}^{{\rm app},(0)}:
W^2_{m+1,\delta}(\Sigma_i;(\hat u^{\rho}_{i,T,(0)})^{*}TX) \to L^2_{m,\delta}(\Sigma_i;(\hat u^{\rho}_{i,T,(0)})^{*}TX \otimes \Lambda^{0,1})
$$
with
\begin{equation}\label{144op}
D_{\hat u^{\rho}_{i,T,(0)}}^{\text{\rm app},(0)}(V): = (D_{\hat u^{\rho}_{i,T,(0)}}\overline{\partial})(V)
- (D_{\hat u^{\rho}_{i,T,(0)}}\mathcal E_i)(\frak e^{\rho} _{i,T,(0)}, V).
\end{equation}

\begin{lem}\label{lem112}
There exist $T_{(\ref{surj1step1modififed})}>0$ {and $\epsilon_{2}>0$} with the following properties.
If $T > T_{(\ref{surj1step1modififed})}$ and
$\rho \in V_1(\epsilon) \times_L V_2(\epsilon)$ with $\epsilon < \epsilon_{2}$
{then}
\begin{equation}\label{surj1step1modififed}
\text{\rm Im } D_{\hat u^{\rho}_{i,T,(0)}}^{\text{\rm app},(0)}
+ \mathcal E_i(\hat u^{\rho}_{i,T,(0)})
= L^2_{m,\delta}(\Sigma_i;(\hat u^{\rho}_{i,T,(0)})^{*}TX \otimes \Lambda^{0,1}).
\end{equation}
Moreover
\begin{equation}\label{Duievsurjstep12}
D{\rm ev}_{1,\infty} - D{\rm ev}_{2,\infty} :
(D_{\hat u^{\rho}_{1,T,(0)}}^{\text{\rm app},(0)})^{-1}(\mathcal E_1) \oplus
(D_{\hat u^{\rho}_{2,T,(0)}}^{\text{\rm app},(0)})^{-1}(\mathcal E_2)
\to T_{p^{\rho}}L
\end{equation}
is surjective. Here we put $\mathcal E_i = \mathcal E_i(\hat u^{\rho}_{i,T,(0)})$.
\end{lem}
\begin{proof} The inequality
(\ref{frakeissmall}) implies that $ (D_{\hat u^{\rho}_{1,T,(0)}}\mathcal E_1)(\frak e^{\rho} _{1,T,(0)}, \cdot)$ is
small in the operator norm. The lemma then follows from Lemma \ref{surjlineastep11}.
\end{proof}
\begin{rem}\label{independetmm}
{Recall that
$$
V_i (\epsilon)\subset \mathcal M^{\mathcal E_i}((\Sigma_i,\vec z_i);\beta_i)_{\epsilon_2}
$$
consists of smooth maps.}
Note that (\ref{frakeissmall})
and conclusions of Lemmata \ref{surjlineastep11} and \ref{lem112}
 are proved by taking  a small neighborhood $V_i(\epsilon)$ of $0$ (in $\tilde V_i$)
with respect to the $C^m$ norm.

However we can take $V_i(\epsilon)$ {so that it becomes independent of} $m$ and the conclusion of Lemma \ref{lem112} holds for
all $m$.
In fact the elliptic regularity implies that if the conclusion of Lemma \ref{lem112} holds for some $m$ then
it holds for all $m'>m$.
The inequality (\ref{frakeissmall}) holds for that particular $m$ only.
Only this inequality
for $m$ is used to show Lemma \ref{lem112}.
In other words, $\epsilon_2$ is independent of $m$.
\end{rem}
{
Based on this remark, we fix}  $V_1 = V_1(\epsilon_{2}), V_2 = V_2(\epsilon_{2})$ now and will never change 
them in Chapters \ref{alternatingmethod}, \ref{subsecdecayT}.
In other words, the constant $\epsilon_{2}$ is fixed at this stage.

We consider the finite dimensional subspace
\begin{equation}\label{gibker}
\text{\rm Ker}(D{\rm ev}_{1,\infty} - D{\rm ev}_{2,\infty})
\cap \left(
((D_{u_1}\delbar)^{\perp})^{-1}(\mathcal E_1) \oplus
((D_{u_2}\delbar)^{\perp})^{-1}(\mathcal E_2)\right)
\end{equation}
of
\begin{equation}\label{cocsissobolev}
\text{\rm Ker}(D{\rm ev}_{1,\infty} - D{\rm ev}_{2,\infty}) \cap
\bigoplus_{i=1}^2 W^2_{m+1,\delta}((\Sigma_i,\partial \Sigma_i);u_i^{*}TX,(u_i^{*}TL))
\end{equation}
which consists of smooth sections.  Here we recall $\mathcal E_i = \mathcal E_i(u_i)$.
Note $\overline\partial u_i =0$. Therefore
$
(D_{u_i}\overline\partial)^{\perp} = D_{u_i}\overline\partial
$
by (\ref{linearized2221}).

\begin{defn}\label{defnfrakH}
We define a linear subspace  $\frak H(\mathcal E_1,\mathcal E_2)$ of
(\ref{cocsissobolev}) by \index[syindex]{HE1E2@$\frak H(\mathcal E_1,\mathcal E_2)$}
\begin{equation}
\frak H(\mathcal E_1,\mathcal E_2) \perp {\rm (\ref{gibker})},
\qquad
\frak H(\mathcal E_1,\mathcal E_2) \oplus  {\rm (\ref{gibker})}
= {\rm (\ref{cocsissobolev})}.
\end{equation}
In other words, it is the
$L^2$ orthogonal complement of (\ref{gibker}) in (\ref{cocsissobolev}).
Here the $L^2$ inner product is defined by (\ref{innerprod1}).
\par
Note the supports of elements of $\mathcal E_i(u_i)$ are contained in $K_i$.
Therefore the projection of $C^k$ (resp. $L^2_k$) section to $\frak H(\mathcal E_1,\mathcal E_2)$
is  $C^k$ (resp. $L^2_k$).
\end{defn}
We use the parallel transport to define an isomorphism from
\begin{equation}\label{domain527}
W^2_{m+1,\delta}((\Sigma_i,\partial \Sigma_i);u_{i}^{*}TX,u_{i}^{*}TL)
\end{equation}
to
\begin{equation}\label{domain528}
W^2_{m+1,\delta}((\Sigma_i,\partial \Sigma_i);(\hat u^{\rho}_{i,T,(0)})^{*}TX,(\hat u^{\rho}_{i,T,(0)})^{*}TL),
\end{equation}
as follows.
\begin{defn}
We define the isomorphism $\Phi_{i;(0)}(\rho,T) : (\ref{domain527}) \to (\ref{domain528})$ 
{to be the closure of the map ${\rm Pal}^{\hat u^{\rho}_{i,T,(0)}}_{u_i}$
given in (\ref{paluv}). With a slight abuse of notation, we still denote}
\begin{equation}\label{form5304}
\Phi_{i;(0)}(\rho,T) = {\rm Pal}^{\hat u^{\rho}_{i,T,(0)}}_{u_i}.
\end{equation}
Here the right hand side is the closure of the map ${\rm Pal}^{\hat u^{\rho}_{i,T,(0)}}_{u_i}$ in (\ref{paluv}).
\end{defn}

\begin{lem}\label{lemma51212}
$ $
\begin{enumerate}
\item $\Phi_{i;(0)}(\rho,T)$ {is an isomorphism between (\ref{domain527}) and (\ref{domain528}).}
.
\item  The linear map
$(\Phi_{1;(0)}(\rho,T),\Phi_{2;(0)}(\rho,T))$ sends the subspace
(\ref{gibker}) into  the subspace
\begin{equation}\label{cocsissobolev2}
\aligned
& \text{\rm Ker}(D{\rm ev}_{1,\infty} - D{\rm ev}_{2,\infty}) \\
& \cap
\bigoplus_{i=1}^2 W^2_{m+1,\delta}((\Sigma_i,\partial \Sigma_i);
(\hat u^{\rho}_{i,T,(0)})^{*}TX,(\hat u^{\rho}_{i,T,(0)})^{*}TL).
\endaligned
\end{equation}
\end{enumerate}
\end{lem}
With Lemma \ref{lem112} the proof is easy and is omitted.
\begin{defn}
We denote by $\frak H_{(0)}(\mathcal E_1,\mathcal E_2;\rho,T)$
the image of
the restriction of the linear map $(\Phi_{1,(0)}(\rho,T),\Phi_{2,(0)}(\rho,T))$
 to  $\frak H(\mathcal E_1,\mathcal E_2)$. It
 is a subspace of (\ref{cocsissobolev2}).
\end{defn}

\begin{defn}\label{def105}
Let ${\rm Err}^{\rho}_{i,T,(0)}$ 
\index[syindex]{Erriroh@${\rm Err}^{\rho}_{i,T,(0)}$}be the functions defined in Definition \ref{deferfirst}.
We define the triple $(V^{\rho}_{T,1,(1)},V^{\rho}_{T,2,(1)},\Delta p^{\rho}_{T,(1)})$ to be the unique solution
satisfying the requirements \index[syindex]{VrhoT1(1)@$(V^{\rho}_{T,1,(1)},V^{\rho}_{T,2,(1)},\Delta p^{\rho}_{T,(1)})$}
\begin{equation}\label{144ffff}
D_{\hat u^{\rho}_{i,T,(0)}}^{\text{app},(0)}(V^{\rho}_{T,i,(1)}) + {\rm Err}^{\rho}_{i,T,(0)}
\in \mathcal E_i(\hat u^{\rho}_{i,T,(0)}),
\end{equation}
and
\begin{equation}\label{pairinfrakH}
((V^{\rho}_{T,1,(1)},\Delta p^{\rho}_{T,(1)}),(V^{\rho}_{T,2,(1)},\Delta p^{\rho}_{T,(1)})) \in
\frak H_{(0)}(\mathcal E_1,\mathcal E_2;\rho,T).
\end{equation}
\end{defn}
Lemma \ref{lemma51212} implies that such $(V^{\rho}_{T,1,(1)},V^{\rho}_{T,2,(1)},\Delta p^{\rho}_{T,(1)})$
exists and is unique
if $T > T_{(\ref{form536536})}$ with the sufficiently large
$T_{(\ref{form536536})}$ and $\rho$ is in $V_1 \times_L V_2$
chosen at Remark \ref{independetmm}.
In fact,  when we consider the subspace
\begin{equation}\label{gibker2}
\text{\rm Ker}(D{\rm ev}_{1,\infty} - D{\rm ev}_{2,\infty})
\cap
\bigoplus_{i=1}^2 (D_{\hat u^{\rho}_{i,T,(0)}}^{\text{\rm app},(0)})^{-1}(\mathcal E_i)
\end{equation}
of (\ref{cocsissobolev2}) we still have a direct sum decomposition
\begin{equation}\label{form536536}
(\ref{gibker2}) \oplus \frak H_{(0)}(\mathcal E_1,\mathcal E_2;\rho,T)
= (\ref{cocsissobolev2}).
\end{equation}
Since (\ref{form536536}) holds for $T = \infty$ by definition, so it holds
for $T > T_{(\ref{form536536})}$.
\par
The unique solution $(V^{\rho}_{T,1,(1)},V^{\rho}_{T,2,(1)},\Delta p^{\rho}_{T,(1)})$
corresponds to the solution $v_0$
to \eqref{eq:Newtonf'} in Newton's scheme. Then the following
estimate corresponds to estimating $|x_1 - x_0| = |v_0|$ for the solution $v_0$ to the linearized
equation \eqref{eq:Newtonf'} in Newton's scheme.
\begin{rem}
Instead of $\frak H_{(0)}(\mathcal E_1,\mathcal E_2;\rho,T)$ we can use an 
$L^2$ orthogonal complement of (\ref{gibker2}) in (\ref{cocsissobolev2}).
This choice looks more natural and actually works for the proof of
Theorems \ref{gluethm1} and \ref{exdecayT}. We take the current choice since
when we study $T$ and $\rho$ derivative in Section \ref{subsec62},
it makes calculation slightly shorter.
\end{rem}

\begin{lem}\label{lem516}
There exist $C_{2,m}>0, T_{m,(\ref{142form})}>0$ such that
if $T > T_{m,(\ref{142form})}$, then
\begin{equation}\label{142form}
\| (V^{\rho}_{T,i,(1)},\Delta p^{\rho}_{T,(1)})\|_{W^2_{m+1,\delta}(\Sigma_i)} \le C_{2,m}e^{-\delta_1 T}
.
\end{equation}
\end{lem}
This is immediate from construction and  uniform boundedness of the right inverse of
the operator
$$
\Pi_{\mathcal E_i}^{\perp} \circ
D_{\hat u^{\rho}_{i,T,(0)}}^{\text{\rm app},(0)}=
\Pi_{\mathcal E_i}^{\perp} \circ(D_{\hat u^{\rho}_{i,T,(0)}}\overline{\partial} -
(D_{\hat u^{\rho}_{i,T,(0)}}\mathcal E_i)(\frak e^{\rho} _{i,T,(0)}, \cdot))
$$
and the exponential decay estimates of the error term.
(Lemma \ref{lem18}.)
\par\medskip
\noindent{\bf Step 1-2 (Gluing solutions)}:
We use $(V^{\rho}_{T,1,(1)},V^{\rho}_{T,2,(1)},\Delta p^{\rho}_{T,(1)})$ to find an
approximate solution $u^{\rho}_{T,(1)}$ of the next level. This corresponds to writing down
$$
x_1 = x_0 - v_0 = x_0 -\frac{f(x_0)}{f'(x_0)}
$$
in  Newton's scheme.

\begin{defn}\label{defn518}
We define $u^{\rho}_{T,(1)}(z)$ \index[syindex]{urhoT1@$u^{\rho}_{T,(1)}(z)$} as follows.
\begin{enumerate}
\item If $z \in K_1$, we put
\begin{equation}
u^{\rho}_{T,(1)}(z)
=
\Exp (\hat u^{\rho}_{1,T,(0)}(z),V^{\rho}_{T,1,(1)}(z)).
\end{equation}
\item If $z \in K_2$, we put
\begin{equation}
u^{\rho}_{T,(1)}(z)
=
\Exp (\hat u^{\rho}_{2,T,(0)}(z),V^{\rho}_{T,2,(1)}(z)).
\end{equation}
\item
If $z  = (\tau,t) \in [-5T,5T]\times [0,1]$,
we put
\begin{equation}\label{urhoT(1)}
\aligned
u^{\rho}_{T,(1)}(\tau,t) = & \Exp \left(u_{T,(0)}^\rho(\tau,t),
\chi_{\mathcal B}^{\leftarrow}(\tau,t) (V^{\rho}_{T,1,(1)}(\tau,t) - (\Delta p^{\rho}_{T,(1)})^{\rm Pal}) \right.\\
&\qquad + \left. \chi_{\mathcal A}^{\rightarrow}(\tau,t)(V^{\rho}_{T,2,(1)}(\tau,t)-
(\Delta p^{\rho}_{T,(1)})^{\rm Pal})
+  (\Delta p^{\rho}_{T,(1)})^{\rm Pal}\right).
\endaligned
\end{equation}
\end{enumerate}
\end{defn}
We recall that $\hat u^{\rho}_{1,T,(0)}(z) = u^{\rho}_{T,(0)}(z)$ on $K_1$
and $\hat u^{\rho}_{2,T,(0)}(z) = u^{\rho}_{T,(0)}(z)$ on $K_2$.
\par\medskip
\noindent{\bf Step 1-3 (Error estimates):}
Let a constant $0 < \mu < 1$ be given, which we fix throughout the rest of the proof.
For example we can take $\mu = 1/2$.
\par
This step corresponds to establishing the inequality \eqref{eq:geometric} for $n = 1$
in Newton's iteration scheme.

\begin{prop}\label{mainestimatestep13}
There exists $C_{3,m}$ such that for any $\epsilon(4) > 0$, {we can choose $T_{m,
\epsilon(4), (\ref{frakeissmall20})}> 0$ with the following properties: there exist
 $\frak e^{\rho} _{i,T,(1)} \in \mathcal E_i$ that satisfies
\begin{equation}\label{frakeissmall20}
\|\text{\rm Err}_{T,(1)}^\rho\|_{L_{m,\delta}^2} < C_{1,m}\epsilon(4)\mu e^{-\delta_1 T}
\end{equation}
and
\begin{equation}\label{frakeissmall2}
\| \frak e^{\rho} _{i,T,(1)}\|_{L_{m}^2(K_i)} < C_{3,m}e^{-\delta_1 T}.
\end{equation}
for all $T > T_{m,(\ref{frakeissmall20})}$.
(Here $C_{1,m}$ is the constant given in Lemma \ref{lem18}.)}
\par
{In (\ref{frakeissmall20}) we denote \index[syindex]{ErrT1rho@$\text{\rm Err}_{T,(1)}^\rho$}
\begin{equation}\label{Err(1)}
\text{\rm Err}_{T,(1)}^\rho = \overline\partial u^{\rho}_{T,(1)} 
- \sum_{a=0}^{1}\frak e^{\rho}_{1,T,(a)} -\sum_{a=0}^{1}\frak e^{\rho}_{2,T,(a)}.
\end{equation}
}
\end{prop}
Note Remark \ref{remark54} applies here.
\begin{proof}
The existence of $\frak e^{\rho} _{i,T,(1)}$ satisfying
$$
\aligned
&\|\overline\partial u^{\rho} _{T,(1)} - (\frak e^{\rho} _{1,T,(0)}
+ \frak e^{\rho} _{1,T,(1)}) - (\frak e^{\rho} _{2,T,(0)}
+ \frak e^{\rho} _{2,T,(1)})\|_{L_{m,\delta}^2(K_1\cup K_2\subset \Sigma_T)} 
\\
&< C_{1,m}\epsilon(4)
\mu e^{-\delta_1 T}/10
\endaligned
$$
is a consequence of the fact that (\ref{linearized2221}) is the linearized equation of (\ref{mainequation}) and
the estimate (\ref{142form}). More explicitly, we can prove this by a standard
quadratic estimate whose details are now
in order.

We do the necessary estimates on  the regions  $K_1, \, K_2$,
$([-5T,5T] \setminus [-T-1, T+1]) \times [0,1]$ and region $[-T-1, T+1] \times [0,1]$
separately.
\par \medskip

\noindent{{\bf Step 1-3-1: Estimate on $K_i$ and Definition of $\frak e_{i,T,(1)}^\rho$}}
\smallskip

We first estimate on $K_1$.
We put $K_1^+ = K_1 \cup [-5T,-5T+1] \times [0,1]$.
\begin{rem}
The estimate on $K_1$ is rather lengthy. However it consists only of standard and straightforward
calculation. This estimate is {\it not} a part of the main idea of the proof.
The main idea of the proof of this paper {lies in} the points mentioned in
Remarks \ref{rem:dropofweight} and \ref{rem611}.
\end{rem}
During this estimate on $K_1$ and on $[-5T,-T-1] \times [0,1]$
we use the following simplified notation.
\begin{equation}\label{simplenote}
\aligned
u &= \hat u^{\rho}_{1,T,(0)}, \qquad &V=V^{\rho}_{T,1,(1)},
 \\
\mathcal P &= (\Pal_1^{(0,1)})^{-1},  \qquad &\frak e =\frak e^{\rho} _{1,T,(0)}.
\endaligned
\end{equation}
Here $\Pal_1^{(0,1)}$ is the $(0,1)$ part of the parallel translation along the
curve $s \mapsto \Exp (u,sV)$, $s \in [0,s_0]$ for  $s_0 \in [0,1]$.

By Fundamental Theorem of Calculus

\begin{equation}\label{FTE}
g(1) = g(0) + \int_0^1 g'(s)\, ds = g(0) + g'(0) +  \int_0^1 ds \left(\int_0^s g''(r)\, dr \right)
\end{equation}
applied to the function $g$ which is valued in $L_{m,\delta}^2(K_1^+;(\hat u_{1,T,(0)}^\rho)^*TX
\otimes \Lambda^{0,1})$ and is given by
$
g(s) = \mathcal P \overline\partial \left(\Exp (u,sV)\right),
$
we derive
\begin{equation}\label{151ffff}
\aligned
&\mathcal P\overline \partial (\Exp (u,V))  \\
& =
\overline\partial (\Exp (u,0)) +\int_0^1 \frac{\del}{\del s}\left(\mathcal P\overline\partial (\Exp (u,sV))\right) ds
 \\
& =  \overline\partial u
+ (D_{u}\overline\partial)(V)
 +\int_0^1 ds \int_0^s
\left(\frac{\del}{\del r}\right)^2 \left(\mathcal P \overline\partial (\Exp (u,rV)\right)dr
\endaligned
\end{equation}
on $L_{m,\delta}^2(\Sigma_T;u^*TX\otimes \Lambda^{0,1})$.

On the other hand, we have the following estimate.

\begin{lem}\label{lem:mgeq2} Let $m \geq 2$. Then
\begin{equation}\label{152ff}
\aligned
\left\|\int_0^1 ds \int_0^s
\left(\frac{\del}{\del r}\right)^2 \left(\mathcal P \overline\partial (\Exp (u,rV)\right)\, dr
\right\|_{L^2_{m}(K_1)}
&\le C_{m,(\ref{152ff})}\| V\|_{L^2_{m+1,\delta}(K_1^+)}^2 \\
&\le C'_{m,(\ref{152ff})}e^{-2 \delta_1 T}.
\endaligned
\end{equation}
\end{lem}
\begin{proof} {The second inequality follows from \eqref{142form}.
On the other hand, the first one} is rather an immediate consequence of the fact that
the covariant derivative $\nabla_r^2 \mathcal P = \left(\frac{\del}{\del r}\right)^2 \mathcal P$
do not differentiate $V$. Especially this inequality is obvious for the case of
flat metric and the standard complex structure on $\C^n$. For the readers' convenience,
we provide a full explanation on why the presence of exponential maps and the parallel transports
do not hinder obtaining the required estimate in  Appendix \ref{appendixA}.
\end{proof}

We also have
\begin{equation}\label{155ff}
\aligned
& \mathcal P \circ \Pi_{\mathcal{E}_1(\Exp (u,V))}^{\perp} \circ \mathcal P^{-1}\\
& =
\Pi_{\mathcal E_1(u)}^{\perp} +
\int_0^1 \frac{d}{ds}\left(\mathcal P\circ \Pi_{\mathcal E_1(\Exp (u,sV)}^{\perp}
\circ \mathcal P^{-1}\right)\, ds
\\
& =
\Pi_{\mathcal E_1(u)}^{\perp} -
(D_{u}\mathcal E_1)(\cdot,V) +
\int_0^1 ds \int_0^s \frac{d^2}{dr^2}\left(\mathcal P\circ
\Pi_{\mathcal E_1(\Exp (u,rV))}^{\perp}\circ \mathcal P^{-1}
\right)\, dr.
\endaligned
\end{equation}
{(See \eqref{DEidef} for the definition of $D_u\mathcal E_1$.)}

We can estimate the third term of the right hand side of (\ref{155ff}) in the same way as in (\ref{152ff}) and also get
\be\label{eq:intintPiE1}
\left\|\int_0^1 ds \int_0^s \frac{d^2}{dr^2}
\left(\mathcal P  \circ \Pi_{\mathcal E_1(\Exp (u,rV))}^{\perp}\circ \mathcal P^{-1}\right)\, dr
\right\|_{L^2_{m}(K_1)} \leq C_{m,(\ref{eq:intintPiE1})}e^{-2\delta_1 T}.
\ee
Here the left hand side is the operator norm as an endomorphism of  $L^2_m(K_1)$.
\par
On the other hand, (\ref{startingestimate}), (\ref{142form}) and (\ref{151ffff}) imply 
\begin{equation}\label{156ff}
\left\|\overline\partial (\Exp (u,V))
- \mathcal P^{-1}\frak e\right\|_{L^2_{m}(K_1)}
\le C_{m,(\ref{156ff})}e^{-\delta_1 T}.
\end{equation}
(Note we regard $\frak e =\frak e^{\rho} _{1,T,(0)} \in L^2(K_1,u^*TX \otimes \Lambda^{0,1})$.)
We put
$$
Q = \overline\partial (\Exp (u,V))
- \mathcal P^{-1}\frak e.
$$
By (\ref{155ff}), (\ref{eq:intintPiE1}) and (\ref{156ff}), we have
\begin{equation}\label{156ff2}
\aligned
&\left\|\Pi_{\mathcal E_1(\Exp (u,V))}^{\perp}(Q)
- \mathcal P ^{-1}\Pi_{\mathcal E_1(u)}^{\perp}(\mathcal P Q) \right.
\left. +  \mathcal P ^{-1}(D_{u}\mathcal E_1)(\mathcal P Q,V)\right\|_{L^2_{m}(K_1)}\\
&=  \left\|\int_0^1 ds \int_0^s \frac{d^2}{dr^2}
\left(\mathcal P  \circ \Pi_{\mathcal E_1(\Exp (u,rV))}^{\perp}(Q)\right)\, dr
\right\|_{L^2_{m}(K_1)} \\
&\le C_{m,(\ref{156ff2})} e^{-3\delta_1 T}.
\endaligned
\end{equation}
By (\ref{142form}) we have $\left\|V\right\|_{L^2_{m+1}(\Sigma_i)} \le C_{2,m} e^{-\delta_1 T}$.
Therefore using (\ref{156ff}), (\ref{156ff2}) and
\begin{equation}\label{form5505550}
\left\Vert\mathcal P ^{-1}(D_{u}\mathcal E_1)(\mathcal P Q,V)\right\Vert_{L^2_{m}(K_1)}
\le
C_{m,(\ref{form5505550})}\left\|V\right\|_{L^2_{m}(K_1)}\left\|Q\right\|_{L^2_{m}(K_1)}
\end{equation}
we have
\begin{equation}\label{156ff23}
\left\|\Pi_{\mathcal E_1(\Exp (u,V))}^{\perp}(Q)
- \mathcal P ^{-1}\Pi_{\mathcal E_1(u)}^{\perp}(\mathcal P Q)\right\|_{L^2_{m}(K_1)}
\le C_{m,(\ref{156ff23})} e^{-2\delta_1 T}.
\end{equation}
Substituting $Q = \overline\partial (\Exp (u,V))
- \mathcal P^{-1}\frak e$ {into} (\ref{156ff23}) back, we obtain
\begin{equation}\label{form5500}
\aligned
&\bigg\|\Pi_{\mathcal E_1(\Exp (u,V))}^{\perp}\overline\partial (\Exp (u,V))
-
\mathcal P^{-1} \Pi_{\mathcal E_1(u)}^{\perp}(\mathcal P
\overline\partial (\Exp (u,V))\\
&\quad-
\Pi_{\mathcal E_1(\Exp (u,V))}^{\perp}\left(\mathcal P^{-1}
\frak e\right) +
\mathcal P^{-1} \Pi_{\mathcal E_1(u)}^{\perp} (\frak e)
\bigg\|_{L^2_{m}(K_1)} \le C_{m,(\ref{form5500})}e^{-2\delta_1 T}.
\endaligned\end{equation}
Therefore applying (\ref{155ff}) and \eqref{eq:intintPiE1}
to the 3rd and 4th terms of 
{the left hand side of} (\ref{form5500}), we obtain
\begin{equation}\label{160ff}
\aligned
\Big\|\Pi_{\mathcal E_1(\Exp (u,V))}^{\perp}&\overline\partial (\Exp (u,V))
-\mathcal P^{-1}
\Pi_{\mathcal E_1(u)}^{\perp}(\mathcal P \overline\partial (\Exp (u,V))\\
&\qquad\qquad+
\mathcal P^{-1}(D_{u}\mathcal E_1)\left(\frak e,V\right)\Big\|_{L^2_{m}(K_1)}
\le C_{m,(\ref{160ff})}e^{-2\delta_1 T}.
\endaligned\end{equation}
We recall  $u = \hat u^{\rho}_{1,T,(0)} = u_1^{\rho_1}$ on $K_1^+$. Therefore
we derive
\begin{equation}\label{161ff}
\overline\partial u
+ (D_{u}\overline\partial)(V)- (D_{u}\mathcal E_1)(\frak e,V)
\in  \mathcal E_1(u)
\end{equation}
on $K_1^+$ from (\ref{144op}), (\ref{144ffff}) and Definition \ref{deferfirst}.
\par
Then, (\ref{160ff}) and (\ref{161ff}) imply
\begin{equation}\label{162ff}
\aligned
&\Big\|\Pi_{\mathcal E_1(\Exp (u,V))}^{\perp}\overline\partial (\Exp (u,V))
-\mathcal P^{-1}
\Pi_{\mathcal E_1(u)}^{\perp}\mathcal P\overline\partial (\Exp (u,V))\\
&\quad+
\mathcal P^{-1}\Pi_{\mathcal E_1(u)}^{\perp}\overline\partial u+\mathcal P^{-1}
\Pi_{\mathcal E_1(u)}^{\perp}(D_{u}\overline\partial)(V)\Big\|_{L^2_{m}(K_1)} \\
&\qquad\qquad\qquad\qquad\qquad\le C_{m,(\ref{162ff})}e^{-2\delta_1 T}.
\endaligned\end{equation}
Combined with (\ref{151ffff}) and (\ref{152ff}), this implies
\begin{equation}\label{eq164}
\aligned
\left\|\Pi_{\mathcal E_1(\Exp (u,V))}^{\perp}\left(\overline\partial
(\Exp (u,V))\right)\right\|_{L^2_{m}(K_1)}&\le C_{m,(\ref{eq164})}e^{-2\delta_1 T} \\
&\le C_{1,m}e^{-\delta_1 T}\epsilon(4)\mu/10,
\endaligned
\end{equation}
for $T>T_{m,\epsilon(4),(\ref{eq164})}$ if we choose $T_{m,\epsilon(4),(\ref{eq164})}$  so that  it satisfies
\be\label{eq:mu}
e^{-\delta_1 T_{m,\epsilon(4),(\ref{eq164})}} < \frac{C_{1,m}}{10C_{m,(\ref{eq164})}}\epsilon(4)\mu.
\ee
This gain of decay rate by the order of $\mu>0$ independent of $T$ and $m$ is one of
the crucial elements in the iteration scheme.
(We refer readers to \eqref{eq181} to see how this gain
is used.)
We use (\ref{156ff}) and (\ref{eq164}) to show:
\begin{lem}\label{triangleholonomy}
There exists $C_{m,(\ref{triangleholonomy2})}$ such that:
\begin{equation}\label{triangleholonomy2}
\left\|\Pi_{\mathcal E_1(\Exp (u,V)))}(\overline\partial (\Exp (u,V))
- \frak e\right\|_{L^2_{m}(K_1)}
\le C_{m,(\ref{triangleholonomy2})}e^{-\delta_1 T}.
\end{equation}
Here we regard  $\frak e = \frak e^{\rho}_{1,T,(0)}$ as an element of
$L^2_m(K_1;({\rm Exp}_1(u,V)))^*TX \otimes \Lambda^{0,1})$.
\end{lem}
\begin{proof}
This is a consequence of (\ref{156ff}), $\frak e \in \mathcal E_1(\Exp (u,V))$
and certain estimate of parallel transport.
We postpone the proof of this lemma
till Appendix \ref{appendixA2}.
\end{proof}
Then if we set \index[syindex]{erho1T1@$\frak e^{\rho} _{1,T,(1)}$}
\begin{equation}\label{558atoato}
\frak e^{\rho} _{1,T,(1)}
= \Pi_{\mathcal E_1(\Exp (u,V))}\left( \overline\partial
\left (\Exp (u,V)\right) - \frak e^{\rho} _{1,T,(0)}\right)
\end{equation}
\eqref{frakeissmall2} follows.
(Recall $\frak e = \frak e^{\rho}_{1,T,(0)}$.)

The estimate of the left hand side of
 (\ref{frakeissmall20}) on $K_1$ follows from  (\ref{eq164}).
We have thus finished the estimate on $K_1$.
The estimate on $K_2$ is the same.
\par\medskip
\noindent{{\bf Step 1-3-2: Estimate on the neck region}}
\smallskip

We next consider the domain $[-5T,-T-1] \times [0,1]$.
Let $S \in [-5T,-T-1]$ and
$$
\Sigma(S) = [S,S+1] \times [0,1].
$$
We  also  put $v = \Delta p_{T,(1)}^{\rho}$. \index[syindex]{SigmaS@$\Sigma(S)$}
\par
In the same way as the proof of Lemma \ref{lem:mgeq2}
we can use  (\ref{142form}) to prove
\begin{equation}\label{152ffg}
\aligned
&\left\|\int_0^1 ds \int_0^s
\left(\frac{\del}{\del r}\right)^2 \left(\mathcal P \overline\partial (\Exp (u,rV)\right)\, dr
\right\|_{L^2_{m}(\Sigma(S))} \\
&
\le C_{m,(\ref{152ffg})} \Vert V \Vert_{L^2_{m+1}(\Sigma(S))}^2
\\
&\le C_{m,(\ref{152ffg})}
\left( \| V - v^{\rm{Pal}}\|_{L^2_{m+1}(\Sigma(S))}^2
+ \Vert v^{\rm{Pal}}\Vert_{L^2_{m+1}(\Sigma(S))}^2\right)
\\
&\le C'_{m,(\ref{152ffg})}e^{-2\delta_1 T}.
\endaligned
\end{equation}
(See Appendix \ref{appendixA}.)
Here to prove the inequality of the last line we also use the boundedness of the domain
$\Sigma(S)$ and an equality
$
\Vert v^{\rm{Pal}}\Vert_{L^2_{m+1}(\Sigma(S))}
\le C \Vert v\Vert
$,
which follows from the definition \eqref{eq:rho-v} of $v^{\text{Pal}}$.

We remark that the domain $\Sigma(S)$ is disjoint from the supports of elements of
$\mathcal E_i$. Therefore (\ref{151ffff}), (\ref{152ffg}) together with
(\ref{144ffff}) imply
\begin{equation}\label{form5600}
\Vert  \overline\partial (\Exp (u,V)) \Vert_{L^2_m(\Sigma(S))}
\le C_{m,(\ref{form5600})} e^{-2\delta_1 T}.
\end{equation}
Note $e_{T,\delta} \le 10 e^{5T\delta} \le 10 e^{T\delta_1/2}$ by
(\ref{form310310}).
Hence (\ref{form5600}) implies
\begin{equation}\label{form56156}
\aligned
&\Vert \overline\partial (\Exp (u,V))\Vert^2_{L^2_{m,\delta}([-5T,-T-1]\times [0,1])} \\
&\le \sum_{S \in ([-5T-1,-T-2] \cap \Z) \cup \{-T-2\}}
\Vert \overline\partial (\Exp (u,V))\Vert^2_{L^2_{m,\delta}(\Sigma(S))} \\
&
\le
 C_{m,(\ref{form56156})} Te^{-3\delta_1T}
 \le \left(C_{1,m}e^{-\delta_1 T}\epsilon(4)\mu/10\right)^2
\endaligned
\end{equation}
if $T > T_{m,\epsilon(4),(\ref{form56156})}$.
\par
The inequality (\ref{form56156}) is the estimate of the left hand side of
 (\ref{frakeissmall20})  on $[-5T,-T-1] \times [0,1]$.
\par
The estimate on $[T+1,5T]\times [0,1]$ is the same.
(We use  the notations given in  (\ref{simplenote}) up to here.)
\par\medskip
Now we do estimate $\overline\partial u^{\rho} _{T,(1)}$ on $[-T+1,T-1]\times [0,1]$.
The inequality
$$
\left\|\overline\partial u^{\rho} _{T,(1)}\right \|_{L_{m,\delta}^2([-T+1,T-1]\times [0,1]\subset \Sigma_T)}
< C_{1,m}\mu\epsilon(4) e^{-\delta_1 T}/10
$$
is also a consequence of the fact that (\ref{linearized2221}) is the linearized equation of (\ref{mainequation}) and
of the estimate (\ref{142form}).
In fact, since the bump functions $\chi_{\mathcal B}^{\leftarrow}$
and $\chi_{\mathcal A}^{\rightarrow}$ are $\equiv 1$ there
the proof is the same as the proof of (\ref{form56156}).
\par
On $\mathcal A_T$,  by definition of $u_{T,(1)}^\rho$, we have
\begin{equation}\label{estimateatA1}
\overline\partial u^{\rho} _{T,(1)}
= \overline\partial\left(\Exp\left(u^{\rho}_{T,(0)}, \chi_{\mathcal A}^{\rightarrow}(V^{\rho}_{T,2,(1)}
-(\Delta p^{\rho}_{T,(1)})^{\rm Pal})+V^{\rho}_{T,1,(1)}\right)\right).
\end{equation}
Note
\begin{equation}\label{estimate6161}
\aligned
&\left\| \chi_{\mathcal A}^{\rightarrow}(V^{\rho}_{T,2,(1)}-
(\Delta p^{\rho}_{T,(1)})^{\rm Pal})\right\|_{L^2_{m+1}(\mathcal A_T)} \\
&\le
C_{m,(\ref{estimate6161})}e^{-6T\delta} \left \| V^{\rho}_{T,2,(1)}-(\Delta p^{\rho}_{T,(1)})^{\rm Pal}\right\|_{L^2_{m+1,\delta}(\mathcal A_T \subset \Sigma_{2})}
\\
&\le C'_{m,(\ref{estimate6161})} e^{-6\delta T-\delta_1T}.
\endaligned
\end{equation}

The first inequality follows from the fact that the weight function $e_{2,\delta}$
is around  $e^{6T\delta}$ on $\mathcal A_T$.
The second inequality follows from (\ref{142form}).
On the other hand the weight function $e_{T,\delta}$ is around $e^{4\delta T}$ at $\mathcal A_T$.
See Figure \ref{Figure5}.
Therefore

\begin{equation}\label{2ff160}
\aligned
&
\left\vert\Vert \overline\partial u^{\rho} _{T,(1)}\Vert_{L^2_{m,\delta}(\mathcal A_T\subset \Sigma_T)} -
\Vert\overline\partial(\Exp(u^{\rho}_{T,(0)},V^{\rho}_{T,1,(1)})
\Vert_{L^2_{m,\delta}(\mathcal A_T\subset \Sigma_T)}
\right\vert\\
&\le
C_{m,(\ref{2ff160})}
\| \chi_{\mathcal A}^{\rightarrow}(V^{\rho}_{T,2,(1)}-(\Delta p^{\rho}_{T,(1)})^{\rm Pal})
\|_{L^2_{m+1,\delta}(\mathcal A_T\subset \Sigma_T)} \\
&
\le
C'_{m,(\ref{2ff160})} e^{-2\delta T-\delta_1T}.
\endaligned
\end{equation}
See Appendix \ref{appendixA}
for the proof of the first inequality.

\begin{figure}
\centering
\includegraphics{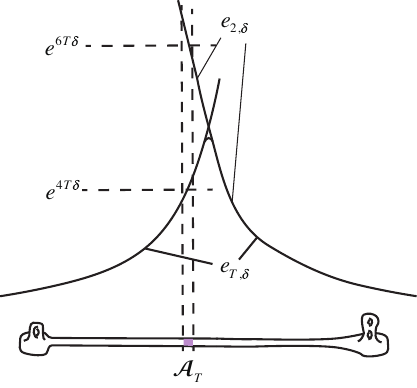}
\caption{Drop of weight.}
\label{Figure5}
\end{figure}
\begin{rem}\label{rem:dropofweight}
This drop of the weight \index{drop of the weight} is the main part of the idea. It was used in \cite[page 414]{fooo:book1}
(and also in \cite[(8.7.2)]{Fuk96II}).
\end{rem}
By Definition \ref{deferfirst}, $
{\rm Err}^{\rho}_{2,T,(0)}   = 0$
on $\mathcal A_T$. Using this in the same way as we did to obtain \eqref{form5600} we derive
\begin{equation}\label{2ff161}
\|\overline\partial(\Exp(u^{\rho}_{T,(0)},V^{\rho}_{T,1,(1)})\|_{L^2_{m,\delta}(\mathcal A_T \subset \Sigma_T)}
\le C_{1,m}e^{-\delta_1 T}\epsilon(4)\mu/20
\end{equation}
for $T > T_{m,\epsilon(4),(\ref{2ff162})}$.
\par
Therefore combining \eqref{estimateatA1}-\eqref{2ff161}, we can find a constant
$T_{m,\epsilon(4),(\ref{2ff162})}$
such that
\begin{equation}\label{2ff162}
\|\overline\partial u^{\rho} _{T,(1)} \|_{L_{m,\delta}^2(\mathcal A_T\subset \Sigma_T)} < C_{1,m}
\mu\epsilon(4) e^{-\delta_1 T}/10
\end{equation}
for all $T > T_{m,\epsilon(4),(\ref{2ff162})}$.
\par
The estimate on $\mathcal B_T$ is similar and so omitted.
The proof of Proposition \ref{mainestimatestep13} is now complete.
\end{proof}
\par\medskip

\noindent{\bf Step 1-4 (Separating error terms into two parts)}: 
{We split the error (\ref{Err(1)}) into the two parts
as follows.}
\begin{defn}\label{defn523}
We put \index[syindex]{Erriroh2@${\rm Err}^{\rho}_{i,T,(1)}$}
$$
\aligned
{\rm Err}^{\rho}_{1,T,(1)}
&  {= \chi_{\mathcal X}^{\leftarrow}\text{\rm Err}_{T,(1)}^\rho}
= \chi_{\mathcal X}^{\leftarrow} (\overline\partial u^{\rho} _{T,(1)} - (\frak e^{\rho} _{1,T,(0)}
+ \frak e^{\rho} _{1,T,(1)})) , \\
{\rm Err}^{\rho}_{2,T,(1)}
&{= \chi_{\mathcal X}^{\rightarrow}\text{\rm Err}_{T,(1)}^\rho}
= \chi_{\mathcal X}^{\rightarrow} (\overline\partial u^{\rho} _{T,(1)}  - (\frak e^{\rho} _{2,T,(0)} + \frak e^{\rho} _{2,T,(1)})).
\endaligned
$$
We regard them as elements of the weighted Sobolev spaces
$L^2_{m,\delta}(\Sigma_1;(\hat u_{1,T,(1)}^{\rho})^*TX\otimes \Lambda^{0,1})$
and
$L^2_{m,\delta}(\Sigma_2;(\hat u_{2,T,(1)}^{\rho})^*TX\otimes \Lambda^{0,1})$
respectively, by extending them to be $0$ outside the support of $\chi$.
\end{defn}
\par\medskip
We put \index[syindex]{pT1rho@$p^{\rho}_{T,(1)}$}
\begin{equation}\label{eq:prho(1)}
p^{\rho}_{T,(1)} = \Exp(p^\rho, \Delta p^\rho_{T,(1)}).
\ee
\par\medskip
\subsection{{The general $\kappa$-th iteration with $\kappa \geq 1$}}
\smallskip
We then come back to the  Step $\kappa$-1 for $\kappa=1$ to establish the following and continue
our inductive steps for $\kappa \geq 1$.
\index[syindex]{erriTkapparho@$\frak e^{\rho} _{i,T,(\kappa)}$}
\begin{eqnarray}
\left\|  (V^{\rho}_{T,i,(\kappa)},\Delta p^{\rho}_{T,(\kappa)})\right\|_{W^2_{m+1,\delta}(\Sigma_i)}
&<& C_{2,m}\mu^{\kappa-1}e^{-\delta_1 T}, \label{form0182}
\\
\left\| {\rm E}( u^{\rho}_{T,(0)}, u^{\rho}_{T,(\kappa)})  \right\|_{W^2_{m+1,\delta}(\Sigma_{T})}
&<& C_{4,m}(2-\mu^{\kappa}) e^{-\delta_1 T}, \label{form0184a}
\\
\left\| {\rm Err}^{\rho}_{i,T,(\kappa)} \right\|_{L^2_{m,\delta}(\Sigma_i)}
&<& C_{1,m}\epsilon(5)\mu^{\kappa}e^{-\delta_1 T}, \label{form0185}
\\
\left\| \frak e^{\rho} _{i,T,(\kappa)}\right\|_{L^2_{m}(K_i)}
&<& C_{3,m}\mu^{\kappa-1}e^{-\delta_1 T},
\quad \text{for $\kappa \ge 1$}. \label{form0186}
\end{eqnarray}
\par
Note we also have
\begin{equation}\label{newform570}
\overline\partial u^{\rho} _{T,(\kappa)} -
{\rm Err}^{\rho}_{1,T,(\kappa)} - {\rm Err}^{\rho}_{2,T,(\kappa)}
=
\sum_{i=1}^2\sum_{j=0}^{\kappa}\frak e^{\rho} _{i,T,(j)}.
\end{equation}
See Definitions \ref{defn523} and \ref{defn:Errrho12T(kappa)}.
\par
{Formulas (5.68)-(5.71) are proved by 
induction on $\kappa$. Since various constants appear 
during the inductive proof of (5.68)-(5.71) 
it is important to clarify the dependence between 
them.
We state the inductive scheme of the 
proof of Formulas (5.68)-(5.71) 
clarifying the dependence of the various constants
as Proposition 5.22 below.}
\par
Note we have already chosen the constants $C_{1,m}, C_{2,m}$, and $C_{3,m}$
in (\ref{startingestimate}), 
(\ref{142form}), and 
(\ref{frakeissmall2}), respectively.

\begin{prop}\label{prop:kappatokappa+1}
$ $
\begin{enumerate}
\item
For any $\epsilon(5)>0$ and $C_{4,m}$ there exists $T_{3,m,\epsilon(5)}>0$
such that
if $T > T_{3,m,\epsilon(5)}$ then
(\ref{form0182}),  (\ref{form0184a})
for $\kappa\le\kappa_0$
and (\ref{form0185}) and (\ref{form0186}) for
$\kappa \le \kappa_0-1$
imply (\ref{form0185}) and (\ref{form0186}) for $\kappa = \kappa_0$.
\item
For any $C_{4,m}$ 
there exists $\epsilon_{3,m}>0$ independent of $T,\kappa_0$ such that
if $\epsilon(5) < \epsilon_{3,m}$, then (\ref{form0185}) and (\ref{form0186})
for $\kappa\le\kappa_0-1$ imply (\ref{form0182})
for $\kappa = \kappa_0$.
\item
We can choose $C_{4,m}$ such that the following holds.
There exists $T_{4,m}>0$ such that 
if (\ref{form0182}) holds for $\kappa \le \kappa_0$, 
(\ref{form0184a}) holds for $\kappa \le \kappa_0-1$
and $T > T_{4,m}$
then (\ref{form0184a}) holds for $\kappa = \kappa_0$.
\end{enumerate}
\end{prop}
See Figure \ref{Prpp522logic}.
\begin{rem}
$ $
\begin{enumerate}
\item
The constant $T_{3,m,\epsilon(5)}$ in Proposition \ref{prop:kappatokappa+1} (1) may depend on $C_{1,m}$, $C_{2,m}$, $C_{3,m}$,
$C_{4,m}$.
\item
The constant $\epsilon_{3,m}$ 
in Proposition \ref{prop:kappatokappa+1} (2) may depend on $C_{1,m}$, $C_{2,m}$, $C_{3,m}$,
$C_{4,m}$.
\item
The constant $T_{4,m}$ 
in  Proposition \ref{prop:kappatokappa+1} (3) may depend on 
$C_{1,m}$, $C_{2,m}$, $C_{3,m}$,
$C_{4,m}$.
\item
In Proposition \ref{prop:kappatokappa+1} (3) the map 
$u^{\rho}_{T,(\kappa)}$ 
$u^{\rho}_{T,(\kappa)}$ 
for $\kappa \le\kappa_0$, 
is defined from
$u^{\rho}_{T,(\kappa-1)}$ and $(V^{\rho}_{T,i,(\kappa)},\Delta p^{\rho}_{T,(\kappa)})$ as in Definition \ref{defn:urhoT(kappa)}
and (\ref{form0184a}).
\item
It is important that the constants appearing in Proposition 
\ref{prop:kappatokappa+1} are independent of $\kappa_0$.
\end{enumerate}
\end{rem}

\begin{figure}
\centering
\includegraphics{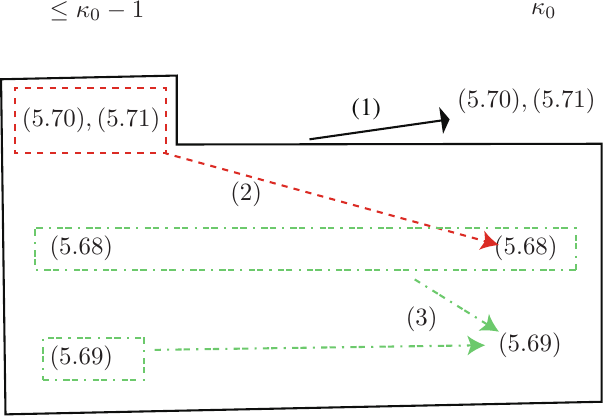}
\caption{Logical scheme in Proposition \ref{prop:kappatokappa+1}}
\label{Prpp522logic}
\end{figure}

{
We prove Formulas (5.68)-(5.71) using Proposition \ref{prop:kappatokappa+1}
as follows.
We first fix $C_{4,m}$ so that 
Proposition \ref{prop:kappatokappa+1} (3) holds.
Then $T_{3,m,\epsilon(5)}$ (depending on $\epsilon(5)$) is 
determined by Proposition \ref{prop:kappatokappa+1} (1).
$\epsilon_{3,m}$ is determined by 
Proposition \ref{prop:kappatokappa+1} (2).
$T_{4,m}$ is determined by Proposition \ref{prop:kappatokappa+1} (3).
Formulas (5.68)-(5.71) hold
for 
$
T \ge \max\{ T_{3,m,\epsilon(5)},  T_{4,m}\}
$ 
by Proposition \ref{prop:kappatokappa+1}. 
We may take $\epsilon(5) = \epsilon_{3,m}/2$ and so 
$T_{3,m,\epsilon(5)}= T_{3,m,\epsilon_{3,m}/2}$ is determined.}
\begin{proof}[Proof of Proposition \ref{prop:kappatokappa+1}] The rest of this chapter will be occupied with the proof of this proposition by
describing each of Steps $\kappa$-1,\dots,$\kappa$-4.
{For given $\kappa_0$ the Step $\kappa_0$-1 corresponds to Proposition \ref{prop:kappatokappa+1} (2).
Steps $\kappa_0$-2,$\kappa_0$-3,$\kappa_0$-4 correspond to
Proposition \ref{prop:kappatokappa+1} (1).}
The proof of
Proposition \ref{prop:kappatokappa+1} (3)
is straightforward. We provide its detail in Appendix \ref{appendixB-}
together with the versions involving $T$ (and $\rho$) derivatives.
See Proposition \ref{prop:inequalitieskappa} (3).
\par
Since Steps $\kappa$-1,\dots,$\kappa$-4 are largely repetitions of
the same kind of tedious but straightforward estimates as in $\kappa = 1$, we will only state the main definitions and
statements required to perform the inductive schemes that will be needed for the
exponential estimate of the $T$-derivatives in the next section, and leave the details of
the relevant estimates to Appendix \ref{appendixB}.
\begin{rem}
Various constants $C_{m,(*.*)}$ and $T_{m,\epsilon(5),(*.*)}$ etc. will appear in the course of the proof of
Proposition \ref{prop:kappatokappa+1} (including the appendices quoted there.)
Those constants depend on $m$ as well as the numbers explicitly appearing in the subscript and
the previously given constants $C_{1,m}$-$C_{4,m}$.
(Especially they depend on the $W_{m+1,\delta}$ norm of
$u^{\rho}_{T,(\kappa)}$ which is estimated by (\ref{form0184a}).)
The important point is that they are independent of $\kappa$ and $T$.
\par
Therefore, even though we have infinitely many steps $\kappa = 0,1,2,\dots$
to work out, we need to make choices of those constants only finitely many times.
\end{rem}
\par\medskip
\noindent{\bf Step $\kappa$-1}:
In this step we prove Proposition \ref{prop:kappatokappa+1} (2).
We first cut-off $u^{\rho}_{T,(\kappa-1)}$ and extend it by a constant 
to obtain the maps $\hat u^{\rho}_{i,T,(\kappa-1)} :
(\Sigma_i,\partial\Sigma_i) \to (X,L)$ $(i=1,2)$ as follows: We first define
$p^{\rho}_{T,(\kappa-1)}$ inductively by \index[syindex]{pTkappa-1rho@$p^{\rho}_{T,(\kappa-1)}$}
\begin{equation}\label{form572572}
p^{\rho}_{T,(\kappa-1)} = {\rm Exp}(p^{\rho}_{T,(\kappa-2)},\Delta p^{\rho}_{T,(\kappa-1)}),
\end{equation}
starting from \eqref{eq:prho(1)}. Then we put
\begin{equation}\label{form573new}
\aligned
&\hat u^{\rho}_{1,T,(\kappa-1)}(z) \\
&=
\begin{cases} \Exp\Big(p^{\rho}_{T,(\kappa-1)}, \chi_{\mathcal B}^{\leftarrow}(\tau-T,t)
&\!\!\!\!\!\!E(p^{\rho}_{T,(\kappa-1)},u^{\rho}_{T,(\kappa-1)}(\tau,t))\Big)
 \\
&{}
\text{if $z = (\tau,t) \in [-5T,5T] \times [0,1]$} \\
 u^{\rho}_{T,(\kappa-1)}(z)
&\text{if $z \in K_1$} \\
p^{\rho}_{T,(\kappa-1)}
&\text{if $z \in [5T,\infty)\times [0,1]$}.
\end{cases} \\
&\hat u^{\rho}_{2,T,(\kappa-1)}(z) \\
&=
\begin{cases}  \Exp\Big(p^{\rho}_{T,(\kappa-1)}, \chi_{\mathcal A}^{\rightarrow}(\tau+T,t)
&  \!\!\!\! \!\!E(p^{\rho}_{T,(\kappa-1)},  u^{\rho}_{T,(\kappa-1)}(\tau,t))\Big) \\
{}
&\text{if $z = (\tau,t) \in [-5T,5T] \times [0,1]$} \\
 u^{\rho}_{T,(\kappa-1)}(z)
&\text{if $z \in K_2$} \\
p^{\rho}_{T,(\kappa-1)}
&\text{if $z \in (-\infty,-5T]\times [0,1]$}.
\end{cases}
\endaligned
\end{equation}
We have the following estimates of  $\hat u^{\rho}_{i,T,(\kappa-1)}$ on the neck region.
\begin{lem}
\begin{equation}\label{form57575}
\aligned
&\left\Vert {\rm E}(p^{\rho}_0,\hat u^{\rho}_{1,T,(\kappa-1)})
\right\Vert_{L^2_{m+1}([T,9T) \times [0,1])}
\le C_{m,(\ref{form57575})} e^{-\delta_1 T},
\\
&\left\Vert {\rm E}(p^{\rho}_0,\hat u^{\rho}_{2,T,(\kappa-1)})
\right\Vert_{L^2_{m+1}((-9T,-T] \times [0,1])}
\le C_{m,(\ref{form57575})} e^{-\delta_1 T}.
\endaligned
\end{equation}
\end{lem}
We remark that in the left hand side of (\ref{form57575}) we take the 
$L^2_{m+1}$ norm {\it without} weight.
\begin{proof}
For $\kappa - 1 = 0$ this is a consequence of
Condition \ref{conds:smalldiam} and Lemma \ref{lem:Ckexpdecay}.
Then using (\ref{form0182}) we can prove the lemma by induction on $\kappa$.
Its version with $T$ and $\rho$ derivatives included, is
(\ref{form185}) proven in Appendix \ref{appendixB-}.
\end{proof}
Now we consider the linearization of (\ref{mainequation})
at $\hat u^{\rho}_{i,T,(\kappa-1)}$, which is

\begin{equation}\label{lineeqstepkappa}
\aligned
D_{\hat u^{\rho}_{i,T,(\kappa-1)}}\overline{\partial}
:
W^2_{m+1,\delta}((\Sigma_i,\partial \Sigma_i);&(\hat u^{\rho}_{i,T,(\kappa-1)})^*TX,(\hat u^{\rho}_{i,T,(\kappa-1)})^*TL)
\\
&\to
L^2_{m,\delta}(\Sigma_i;(\hat u^{\rho}_{i,T,(\kappa-1)})^{*}TX \otimes \Lambda^{0,1}).
\endaligned
\end{equation}
We denote
\begin{equation}
(\frak {se})^{\rho} _{i,T,(\kappa-1)} = \sum_{a=0}^{\kappa-1}\frak e^{\rho} _{i,T,(a)}.
\end{equation}

Similarly as for the operator $D_{\hat u^{\rho}_{i,T,(0)}}^{\text{\rm app},(0)}$ in \eqref{144op}, we define the
approximate linearization of (\ref{mainequation})
\begin{equation}\label{144opkappa}
D_{\hat u^{\rho}_{i,T,(\kappa-1)}}^{\text{\rm app},(\kappa -1 )}: = D_{\hat u^{\rho}_{i,T,(\kappa-1)}}\overline{\partial} -
(D_{\hat u^{\rho}_{i,T,(\kappa-1)}}\mathcal E_i)((\frak {se})^{\rho} _{i,T,(\kappa-1)}, \cdot)
\end{equation}
replacing $\delbar \hat u^{\rho}_{i,T,(\kappa-1)}$ by $(\frak {se})^{\rho} _{i,T,(\kappa-1)}$ in
the expression of the linearization operator
$$
(D_{\hat u^{\rho}_{i,T,(\kappa-1)}}\overline{\partial})^{\perp}
= D_{\hat u^{\rho}_{i,T,(\kappa-1)}}\overline{\partial} -
(D_{\hat u^{\rho}_{i,T,(\kappa-1)}}\mathcal E_i)(\delbar \hat u^{\rho}_{i,T,(\kappa-1)}, \cdot).
$$
\begin{lem}\label{lem112kappa}
Denote $\mathcal E_i = \mathcal E_i(\hat u^{\rho}_{i,T,(\kappa-1)})$.
There exists $T_{m,(\ref{surj1step1modififedkappa})}$ such that if $T > T_{m,(\ref{surj1step1modififedkappa})}$
then we have
\begin{equation}\label{surj1step1modififedkappa}
\text{\rm Im}(D_{\hat u^{\rho}_{i,T,(\kappa-1)}}^{\text{\rm app},(\kappa-1)})  + \mathcal E_i
= L^2_{m,\delta}(\Sigma_i;(\hat u^{\rho}_{i,T,(\kappa-1)})^{*}TX \otimes \Lambda^{0,1}).
\end{equation}
Moreover the map
\begin{equation}\label{Duievsurjstep1kappa2}
D{\rm ev}_{1,\infty} - D{\rm ev}_{2,\infty}
: (D_{\hat u^{\rho}_{1,T,(\kappa-1)}}^{\text{\rm app},(\kappa-1)})^{-1}(\mathcal E_1)
\oplus (D_{\hat u^{\rho}_{2,T,(\kappa-1)}}^{\text{\rm app},(\kappa-1)})^{-1}(\mathcal E_2)
\to T_{p^{\rho}_{T,(\kappa-1)}}L
\end{equation}
is surjective.
\end{lem}

\begin{proof}
Using the inequality
\eqref{form0186}, we estimate
\begin{equation}\label{eq181}
\left\| \frak e^{\rho} _{i,T,(0)} -
\sum_{a=0}^{\kappa-1}\frak e^{\rho} _{i,T,(a)}\right\|_{L_{m}^2(K_i)}
\leq \sum_{a = 1}^{\kappa -1} \|\frak e^{\rho} _{i,T,(a)}\|_{L_{m}^2(K_i)}
< C_{3,m}\frac{e^{-\delta_1 T}}{1-\mu}.
\end{equation}
This in particular implies that
\begin{equation}\label{eq18122}
\aligned
Y\mapsto
&(D_{\hat u^{\rho}_{1,T,(\kappa-1)}}\mathcal E_1)( (\frak {se})^{\rho} _{i,T,(\kappa-1)},
{\rm Pal}_{u_1}^{\hat u^{\rho}_{1,T,(\kappa-1)}}(Y)) \\
&-
{\rm Pal}^{\hat u^{\rho}_{1,T,(\kappa-1)}}_{\hat u^{\rho}_{1,T,(0)}}((D_{\hat u^{\rho}_{1,T,(0)}}\mathcal E_1)(\frak e^{\rho} _{1,T,(0)},
{\rm Pal}_{u_1}^{\hat u^{\rho}_{1,T,(0)}}(Y)))
\endaligned
\end{equation}
is small in the operator norm. We prove this fact in Appendix \ref{appendixA2bis}.
Then this smallness implies that  the operator norm of the difference operator
$$
D_{\hat u^{\rho}_{1,T,(\kappa-1)}}^{\text{\rm app},(\kappa-1)}
\circ {\rm Pal}_{u^{\rho}_1}^{\hat u^{\rho}_{1,T,(\kappa-1)}}-
{\rm Pal}^{\hat u^{\rho}_{1,T,(\kappa-1)}}_{\hat u^{\rho}_{1,T,(0)}} \circ
D_{\hat u^{\rho}_{1,T,(0)}}^{\text{\rm app},(0)}
\circ {\rm Pal}_{u^{\rho}_1}^{\hat u^{\rho}_{1,T,(0)}}
$$
is small.
This can be easily seen
by rewriting the value of the difference operator as
$$
\aligned
& D_{\hat u^{\rho}_{1,T,(\kappa-1)}}^{\text{\rm app},(\kappa-1)}(
{\rm Pal}_{u^{\rho}_1}^{\hat u^{\rho}_{1,T,(\kappa-1)}}(Y))-
{\rm Pal}^{\hat u^{\rho}_{1,T,(\kappa-1)}}_{\hat u^{\rho}_{1,T,(0)}}
(D_{\hat u_{1,T,(0)}}^{\text{\rm app},(0)}({\rm Pal}_{u^{\rho}_1}^{\hat u^{\rho}_{1,T,(0)}}(Y)))\\
 = & \left((D_{\hat u^{\rho}_{1,T,(\kappa-1)}}\delbar) ({\rm Pal}_{u^{\rho}_1}^{\hat u^{\rho}_{1,T,(\kappa-1)}}(Y)) -
{\rm Pal}^{\hat u^{\rho}_{1,T,(\kappa-1)}}_{\hat u^{\rho}_{1,T,(0)}} (
(D_{\hat u^{\rho}_{i,T,(0)}}\delbar)( {\rm Pal}_{u^{\rho}_1}^{\hat u^{\rho}_{1,T,(0)}}(Y)))\right)\\
& \quad  - \Big((D_{\hat u^{\rho}_{1,T,(\kappa-1)}}\mathcal E_1)( (\frak {se})^{\rho}_{i,T,(\kappa-1)},
{\rm Pal}_{u^{\rho}_1}^{\hat u^{\rho}_{1,T,(\kappa-1)}}Y) \\
&\qquad\qquad -{\rm Pal}^{\hat u^{\rho}_{1,T,(\kappa-1)}}_{\hat u^{\rho}_{1,T,(0)}} (D_{\hat u^{\rho}_{1,T,(0)}}\mathcal E_1)(\frak e^{\rho}_{1,T,(0)},
{\rm Pal}_{u^{\rho}_1}^{\hat u^{\rho}_{1,T,(0)}}(Y))\Big).
\endaligned
$$
The first summand of the right hand side is clearly small by the smallness of $L^2_{m+1,\delta}$-norm
of ${\rm E}(\hat u^{\rho}_{i,T,(\kappa-1)},\hat u^{\rho}_{i,T,(0)})$ arising from
the induction hypothesis \eqref{form0184a}.
Then the lemma follows by combining Lemma \ref{lem112} and the above mentioned smallness.
\end{proof}
Note that Remark \ref{independetmm} still applies to Lemma \ref{lem112kappa}.
\begin{defn}\label{def528}
We define a linear map \index[syindex]{Phiikappa-1@$\Phi_{i;(\kappa-1)}$}
$$
\aligned
\Phi_{i;(\kappa-1)}(\rho,T) :
&W^2_{m+1,\delta}((\Sigma_i,\partial \Sigma_i);u_{i}^{*}TX,u_{i}^{*}TL) \\
&\to
W^2_{m+1,\delta}((\Sigma_i,\partial \Sigma_i);(\hat u^{\rho}_{i,T,(\kappa-1)})^{*}TX,(\hat u^{\rho}_{i,T,(\kappa-1)})^{*}TL),
\endaligned
$$
as the closure of the map ${\rm Pal}^{\hat u^{\rho}_{i,T,(\kappa-1)}}_{u_i}$
{given} in (\ref{paluv01}).
\end{defn}
\begin{lem}
$ $
\begin{enumerate}
\item $\Phi_{i;(\kappa-1)}(\rho,T)$ sends (\ref{domain527})
to
\begin{equation}\label{domain528k}
W^2_{m+1,\delta}((\Sigma_i,\partial \Sigma_i);(\hat u^{\rho}_{i,T,(\kappa-1)})^{*}TX,(\hat u^{\rho}_{i,T,(\kappa-1)})^{*}TL).
\end{equation}
\item
The map $(\Phi_{1;(\kappa-1)}(\rho,T),\Phi_{2;(\kappa-1)}(\rho,T))$ sends the subspace
(\ref{cocsissobolev}) to the subspace
\begin{equation}\label{cocsissobolev2k}
\aligned
&\text{\rm Ker}(D{\rm ev}_{1,\infty} - D{\rm ev}_{2,\infty}) \\
&\cap \bigoplus_{i=1}^2 W^2_{m+1,\delta}((\Sigma_i,\partial \Sigma_i);(\hat u^{\rho}_{i,T,(\kappa-1)})^{*}TX,(\hat u^{\rho}_{i,T,(\kappa-1)})^{*}TL).
\endaligned
\end{equation}
\end{enumerate}
\end{lem}
The proof is easy and is omitted.
 
\begin{defn}\label{defn530}
We denote by $\frak H_{(\kappa-1)}(\mathcal E_1,\mathcal E_2;\rho,T)$
the image of the subspace \index[syindex]{Hkappa-1@$\frak H_{(\kappa-1)}(\mathcal E_1,\mathcal E_2;\rho,T)$}
$\frak H(\mathcal E_1,\mathcal E_2)$
(See Definition \ref{defnfrakH}.)  {under} the map  $(\Phi_{1;(\kappa-1)}(\rho,T),\Phi_{2;(\kappa-1)}(\rho,T))$.
It is a subspace of (\ref{cocsissobolev2k}).
\end{defn}
 \index[syindex]{VTikappa@$(V^{\rho}_{T,i,(\kappa)},\Delta p^{\rho}_{T,(\kappa)})$}
\begin{defn}\label{defn532532}
We define $(V^{\rho}_{T,1,(\kappa)},V^{\rho}_{T,2,(\kappa)},\Delta p^{\rho}_{T,(\kappa)})$ by
\begin{equation}\label{formula158}
D_{\hat u^{\rho}_{i,T,(\kappa-1)}}^{\text{\rm app},(\kappa-1)}(V^{\rho}_{T,i,(\kappa)})
+ {\rm Err}^{\rho}_{i,T,(\kappa-1)}
\in \mathcal E_i(\hat u^{\rho}_{i,T,(\kappa-1)})
\end{equation}
and
\begin{equation}\label{inHhanru}
((V^{\rho}_{T,1,(\kappa)},\Delta p^{\rho}_{T,(\kappa)}),(V^{\rho}_{T,2,(\kappa)},\Delta p^{\rho}_{T,(\kappa)}))
\in
\frak H_{(\kappa-1)}(\mathcal E_1,\mathcal E_2;\rho,T).
\end{equation}
\end{defn}
Lemma \ref{lem112kappa} implies that such $(V^{\rho}_{T,1,(\kappa)},V^{\rho}_{T,2,(\kappa)},\Delta p^{\rho}_{T,(\kappa)})$
exists and is unique if $T > T_{m,(\ref{surj1step1modififedkappa})}$.

\begin{lem}\label{estimageVkappa}
There exist $T_{m,(\ref{142form23})}$ and $\epsilon_{3,m}>0$ such that
if $\epsilon(5) < \epsilon_{3,m}$, then (\ref{form0185})
for $\kappa-1$ implies that the next inequality for  $T > T_{m,(\ref{142form23})}$.
\begin{equation}\label{142form23}
\| (V_{T,i,(\kappa)},\Delta p^{\rho}_{T,(\kappa)})\|_{W^2_{m+1,\delta}(\Sigma_i)}
\le C_{2,m}\mu^{\kappa-1}e^{-\delta_1 T}.
\end{equation}
\end{lem}
\begin{proof}
We take $T_{m,(\ref{142form23})}$ {so} large  that
the inverse of
$$
\Pi^{\perp}_{\mathcal E_i(\hat u^{\rho}_{i,T,(\kappa-1)})}
\circ D^{{\rm app},(\kappa-1)}_{\hat u^{\rho}_{i,T,(\kappa-1)}}
$$
is uniformly
bounded and choose $\epsilon_{3,m}$ {so} small  that $\epsilon_{3,m}^{-1}$
is much larger than the norm of the above mentioned inverse.
Then (\ref{142form23}) follows from uniform boundedness of the
above mentioned inverse together with
(\ref{form0185}) for $\kappa-1$.
\end{proof}
This lemma implies the inequality (\ref{form0182}).
We have finished the proof of Proposition \ref{prop:kappatokappa+1} (2).
\par\medskip
\noindent{\bf Step $\kappa$-2}:
We start the proof of Proposition \ref{prop:kappatokappa+1} (1).
\par
We use $(V^{\rho}_{T,1,(\kappa)},V^{\rho}_{T,2,(\kappa)},\Delta p^{\rho}_{T,(\kappa)})$ to find an approximate solution $u^{\rho}_{T,(\kappa)}$ of the next level.
We remark that $(V^{\rho}_{T,1,(\kappa)},V^{\rho}_{T,2,(\kappa)},\Delta p^{\rho}_{T,(\kappa)})$ is the counterpart 
of $v_\kappa$
given by $v_\kappa = \frac{f(x_\kappa)}{f'(x_\kappa)}$ and the next definition corresponds to
the next iteration $x_{\kappa + 1} = x_\kappa + v_\kappa$ in Newton's scheme.
\index[syindex]{uTkapparho@$u^{\rho}_{T,(\kappa)}(z)$}

\begin{defn}\label{defn:urhoT(kappa)}
We define $u^{\rho}_{T,(\kappa)}(z)$ as follows.
\begin{enumerate}
\item If $z \in K_1$, we put
\begin{equation}
u^{\rho}_{T,(\kappa)}(z)
=
\Exp (\hat u^{\rho}_{1,T,(\kappa-1)}(z),V^{\rho}_{T,1,(\kappa)}(z)).
\end{equation}
\item If $z \in K_2$, we put
\begin{equation}
u^{\rho}_{T,(\kappa)}(z)
=
\Exp (\hat u^{\rho}_{2,T,(\kappa-1)}(z),V^{\rho}_{T,2,(\kappa)}(z)).
\end{equation}
\item
If $z  = (\tau,t) \in [-5T,5T]\times [0,1]$,
we put
\begin{equation}\label{newform588}
\aligned
u^{\rho}_{T,(\kappa)}(\tau,t) =
&\Exp\Big(u^{\rho}_{T,(\kappa-1)}(\tau,t), \chi_{\mathcal B}^{\leftarrow}(\tau,t) (V^{\rho}_{T,1,(\kappa)}(\tau,t) - (\Delta p^{\rho}_{T,(\kappa)})
^{\rm Pal})\\
&+\chi_{\mathcal A}^{\rightarrow}(\tau,t)(V^{\rho}_{T,2,(\kappa)}(\tau,t)-(\Delta p^{\rho}_{T,(\kappa)})^{\rm Pal})
+(\Delta p^{\rho}_{T,(\kappa)})^{\Pal}\Big).
\endaligned
\end{equation}
\end{enumerate}
\end{defn}
We note that $\hat u^{\rho}_{1,T,(\kappa-1)}(z) = u^{\rho}_{T,(\kappa-1)}(z)$ on $K_1$
and $\hat u^{\rho}_{2,T,(\kappa-1)}(z) = u^{\rho}_{T,(\kappa-1)}(z)$
on $K_2$.
\par\medskip
\noindent{\bf Step $\kappa$-3}:
The proof of the following proposition is largely a duplication of
that of Proposition \ref{mainestimatestep13}
plus (\ref{form56156}) and so its details will be postponed till Appendix \ref{appendixB}.

\begin{prop}\label{mainestimatestep13kappa}
For any $\epsilon(5) > 0$, there exists $T_{m,
\epsilon(5), (\ref{frakeissmall2kappa0})}> 0$ with the following properties.
If $T > T_{m,
\epsilon(5), (\ref{frakeissmall2kappa0})}$ we can define $\frak e^{\rho} _{i,T,(\kappa)} \in \mathcal E_i$ that satisfies
\begin{equation}\label{frakeissmall2kappa0}
\left\|\overline\partial u^{\rho} _{T,(\kappa)} - \sum_{a=0}^{\kappa}\frak e^{\rho} _{1,T,(a)} -\sum_{a=0}^{\kappa}\frak e^{\rho} _{2,T,(a)}
\right\|_{L_{m,\delta}^2(\Sigma_T)} < C_{1,m}\mu^{\kappa}\epsilon(5) e^{-\delta_1 T},
\end{equation}
and
\begin{equation}\label{frakeissmall2kappaa}
\| \frak e^{\rho} _{i,T,(\kappa)}\|_{L_{m}^2(K_i)} <  C_{3,m}\mu^{\kappa-1}e^{-\delta_1 T}.
\end{equation}
\end{prop}
Note Remark \ref{remark54} applies here.
\par\medskip

\noindent{\bf Step $\kappa$-4}: Similarly as before, we introduce the following definition{s}.
\index[syindex]{ErrTkapparho@${\rm Err}^{\rho}_{T,(\kappa)}$}
{
\be\label{Err(kappa)}
{\rm Err}^{\rho}_{T,(\kappa)} = \overline\partial u^{\rho} _{T,(\kappa)}
 - \sum_{a=0}^{\kappa}\frak e^{\rho} _{1,T,(a)}- \sum_{a=0}^{\kappa}\frak e^{\rho} _{2,T,(a)}.
 \ee
}
\begin{defn}\label{defn:Errrho12T(kappa)}
We put\index[syindex]{ErriTkapparho@${\rm Err}^{\rho}_{i,T,(\kappa)}$}
$$
\aligned
{\rm Err}^{\rho}_{1,T,(\kappa)}
= {\chi_{\mathcal X}^{\leftarrow}({\rm Err}^{\rho}_{T,(\kappa)})}
&= \chi_{\mathcal X}^{\leftarrow} \left(\overline\partial u^{\rho} _{T,(\kappa)}
 - \sum_{a=0}^{\kappa}\frak e^{\rho} _{1,T,(a)}  \right), \\
{\rm Err}^{\rho}_{2,T,(\kappa)}
= {\chi_{\mathcal X}^{\rightarrow}({\rm Err}^{\rho}_{T,(\kappa)})}
&= \chi_{\mathcal X}^{\rightarrow} \left(\overline\partial u^{\rho} _{T,(\kappa)}
- \sum_{a=0}^{\kappa}\frak e^{\rho} _{2,T,(a)}\right).
\endaligned$$
We regard them as elements of the weighted Sobolev spaces
$$
L^2_{m,\delta}(\Sigma_i;(\hat u_{i,T,(\kappa)}^{\rho})^*TX\otimes \Lambda^{0,1})
$$
for $i=1, \, 2$ respectively.
(We extend them by $0$ outside a compact set.)
\end{defn}

We put $p^{\rho}_{T,(\kappa)} = \Exp(p^\rho_{T,(\kappa-1)}, \Delta p^\rho_{T,(\kappa)})$.
See (\ref{form572572}).  Then
Proposition \ref{mainestimatestep13kappa} implies (\ref{form0185}) and (\ref{form0186}).
\par
\medskip
We have thus finished the proof of Proposition \ref{prop:kappatokappa+1}
(1).
\par
The proof of Proposition \ref{prop:kappatokappa+1} is complete.
\end{proof}

Now we are ready to produce the solution to the gluing problem.
For each fixed $m$ there exists $T_m$ such that if $T > T_m$ then
$$
\lim_{\kappa \to \infty} u^{\rho} _{T,(\kappa)}
$$
converges in $L_{m+1,\delta}^2$ sense to a solution of (\ref{mainequation}) by
Proposition \ref{mainestimatestep13kappa}.
The limit map is automatically of $C^{\infty}$-class by elliptic regularity.
(\ref{form0185}) and (\ref{newform570}) imply that the limit solves the equation
(\ref{mainequation}).
We have thus constructed the gluing map used in Theorem \ref{gluethm1}.
\section{Exponential decay of $T$ derivatives}
\label{subsecdecayT}

We first state the result of this chapter which is also the main result of this paper.
We recall that for $T$ sufficiently large and $\rho = (\rho_1,\rho_2) \in V_1\times_L V_2$
we have defined $u^{\rho} _{T,(\kappa)}$ for each $\kappa = 0, \, 1, \, 2, \ldots$. We denote its limit by
\begin{equation}\label{limitkappainf}
u^{\rho} _{T} = \lim_{\kappa \to \infty} u^{\rho} _{T,(\kappa)} : (\Sigma_T,\partial\Sigma_T) \to (X,L).
\end{equation}
The main result of this chapter
is an estimate of $T$ and $\rho$ derivatives of this map.
We prepare some notations to state the result.
\par
\subsection{{Statement of the main result}}
For  $T > 0$, we defined the curve $\Sigma_T$ as the union\index[syindex]{SigmaT@$\Sigma_T$}
$$
\Sigma_T = K_1 \cup ([-5T,5T]_{\tau} \times [0,1]) \cup K_2
$$
by identifying \index[syindex]{dK1@$\del K_1$}$\del K_1 \cong \{-5T\} \times [0,1]$ and $\del K_2 \cong \{5T\} \times [0,1]$
where we denote by $(\tau,t)$ the coordinates on $[-5T,5T]_{\tau} \times [0,1]$.
{See \eqref{eq:Sigma12}.}
Therefore we have the natural inclusion
$
K_i \subset \Sigma_T.
$
\par
We introduce new coordinate{s}  $(\tau',t)$ and $(\tau'',t)$ such that
the  relationship between
the three coordinates $(\tau,t)$, $(\tau',t)$ and $(\tau'',t)$ are given by\index[syindex]{tauprime@$\tau'$}
\begin{equation}
\tau' = \tau + 5T
\end{equation}
and
\begin{equation}
\tau'' =\tau - 5T.
\end{equation}
\begin{figure}
\centering
\includegraphics{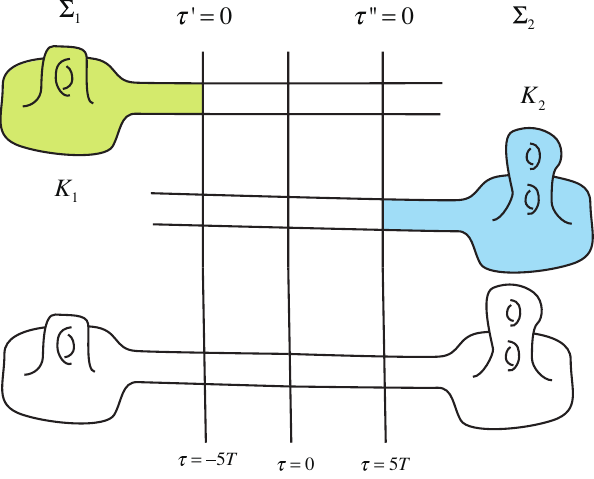}
\caption{$\tau'$ and $\tau''$.}
\label{Figure8}
\end{figure}
See Figure \ref{Figure8}.\index[syindex]{tauprimeprime@$\tau''$}
We may use either $(\tau',t)$ or $(\tau'',t)$ as the coordinate{s}  of
$
\Sigma_T \setminus (K_1\cup K_2)
$.
Namely
$$
\Sigma_T = K_1 \cup ([0,10T]_{\tau'} \times [0,1]) \cup K_2
=
K_1 \cup ([-10T,0]_{\tau''} \times [0,1]) \cup K_2
$$
\begin{rem}\label{rem6161}
Here and hereafter we write $[0,10T]_{\tau'}$, $[-10T,0]_{\tau''}$, $[-5T,5T]_{\tau}$
etc. to clarify the coordinate ($\tau'$, $\tau''$, or $\tau$) we use.
\end{rem}
We can also use {the} $(\tau',t)$ coordinates on $\Sigma_1 \setminus K_1$
and the $(\tau'',t)$ coordinates on $\Sigma_2 \setminus K_2$ so that
$$
\aligned
&\Sigma_1 = K_1 \cup ([-5T,\infty)_{\tau} \times [0,1])
=
K_1 \cup ([0,\infty)_{\tau'} \times [0,1],
\\
&\Sigma_2 = K_2 \cup (-\infty,5T]_{\tau} \times [0,1])
=
K_2 \cup (-\infty,0]_{\tau''} \times [0,1].
\endaligned$$
\begin{rem}
We remark that {the} $(\tau',t)$ coordinates of $\Sigma_1$ and
{the} $(\tau'',t)$ coordinates of $\Sigma_2$ are independent of $T$
but {the}  $(\tau,t)$ coordinates of them are $T$ dependent.
\end{rem}
We define \index[syindex]{K1S@$K_1^S$}
\begin{equation}\label{1104}
\aligned
K_1^S & = K_1 \cup ([0,S]_{\tau'}\times [0,1]), \\
K_2^S & = ([-S,0]_{\tau''}\times [0,1]) \cup K_2
\endaligned
\end{equation}
for each positive constant $S> 0$.
We have a natural embedding $K_1^S \hookrightarrow \Sigma_1$
(resp. $K_2^S \hookrightarrow \Sigma_2$) via the coordinates
$\tau'$ (resp. via the coordinates $\tau''$).
\par
If $S > 0$, $T > \frac{S}{10}$, we have embeddings of $K_1^S \to \Sigma_T$
by setting $\tau = \tau' - 5T$ and $K_2^S \to \Sigma_T$ by $\tau = \tau'' + 5T$.
\par
We then obtain a map\index[syindex]{GlueiS@$\text{\rm Glures}_{i,S}$}
$$
\text{\rm Glures}_{i,S} : [T_m,\infty) \times V_1 \times_L V_2 \to \text{\rm Map}_{L^2_{m+1}}((K_i^S,K_i^S\cap\partial \Sigma_i),(X,L))
$$
by
\begin{equation}
\begin{cases}
\text{\rm Glures}_{1,S}(T,\rho)(x)  &=  u^{\rho} _{T}(x) \qquad x \in K_1\\
\text{\rm Glures}_{1,S}(T,\rho)(\tau',t) &= u^{\rho} _{T}(\tau',t)
\,\,\, (= u^{\rho}_T(\tau+5T,t))
\end{cases}
\end{equation}
\begin{equation}
\begin{cases}
\text{\rm Glures}_{2,S}(T,\rho)(x)  &=  u^{\rho} _{T}(x) \,\,\, x \in K_2\\
\text{\rm Glures}_{2,S}(T,\rho)(\tau'',t) &= u^{\rho} _{T}(\tau'',t)
\quad (= u^{\rho}_T(\tau-5T,t)).
\end{cases}
\end{equation}
Here the map `$\text{Glures}$' stands for the phrase `gluing followed by restriction',
and $\text{\rm Map}_{L^2_{m+1}}((K_i^S,K_i^S\cap\partial \Sigma_i),(X,L))$ is the space of maps of $L^2_{m+1}$ class
($m$ is sufficiently large, say $m>10$.) It has the structure of Hilbert manifold in an obvious way.
This Hilbert manifold is independent of $T$. So we can define the $T$ derivative of a family of elements
of $\text{\rm Map}_{L^2_{m+1}}((K_i^S,K_i^S\cap\partial \Sigma_i),(X,L))$ parameterized by $T$.
\footnote{We can also use $C^k((K_i^S,K_i^S\cap\partial \Sigma_i),(X,L))$, the Banach manifold
of $C^k$ maps, in place of $\text{\rm Map}_{L^2_{m+1}}((K_i^S,K_i^S\cap\partial \Sigma_i),(X,L))$.}
\begin{rem}\label{differentm}
The domain and the {codomain} of the map $\text{\rm Glures}_{i,S}$ depend on $m$.
However its image actually lies in the set of smooth maps. Also
none of the constructions of $u^{\rho} _{T}$  depend on $m$.
(The proof of the convergence of (\ref{limitkappainf}) depends on $m$. So the numbers
$T_{3,m,\epsilon(5)}$, $T_{4,m}$ in Proposition \ref{prop:kappatokappa+1}
(and $T_{5,m,\epsilon(6)}$, $T_{6,m}$ in Proposition \ref{prop:inequalitieskappa})
 depend on $m$.)
Therefore the map $\text{\rm Glures}_{i,S}$ is {\it independent} of $m$ on the intersection of the domains.
Namely the map $\text{\rm Glures}_{i,S}$ constructed by using {the} $L^2_{m_1}$ norm
coincides with  the map $\text{\rm Glures}_{i,S}$ constructed by using the $L^2_{m_2}$ norm
{on $[T', \infty) \times V_1\times_L V_2$ where
$$
T' = \max\{T_{3,m_1,\epsilon(5)},T_{4,m_1},T_{5,m_1,\epsilon(6)},
T_{6,m_1}
T_{3,m_2,\epsilon(5)},T_{4,m_2},T_{5,m_2,\epsilon(6)},
T_{6,m_2}\}.
$$}
\end{rem}

\begin{thm}\label{exdecayT}
For each given $m$ and $S$,  there exist $T_{m,S,(\ref{form67})} > S$, $C_{m,S,(\ref{form67})}> 0$
such that
\begin{equation}\label{form67}
\left\| \nabla_{\rho}^n \frac{d^{\ell}}{dT^{\ell}} \text{\rm Glures}_{i,S}\right\|_{L^2_{m+1-\ell}}
<C_{m,S,(\ref{form67})}e^{-\delta_1 T}
\end{equation}
holds for all $T>T_{m,S,(\ref{form67})}$ and for $(n,\ell)$ with $m-2 \ge n, \ell \ge 0$ and $\ell > 0$.
Here $\nabla_{\rho}^n$ is the $n$-th derivative in the $\rho$-direction.
\end{thm}
\begin{rem}
Theorem \ref{exdecayT} is  equivalent to \cite[Lemma A1.58]{fooo:book1}.
The proof below uses the same inductive scheme as the one in \cite[page 776]{fooo:book1}.
We use {the} $L^2_m$ norm in place of {the} $L^p_1$ norm and
add thorough detail.
\end{rem}
\begin{rem}\label{rem64}
We remark that the map
$$
\mathcal M^{\mathcal E_1\oplus \mathcal E_2}((\Sigma_T,\vec z);u_1,u_2)_{\epsilon_1}
\to
\prod_{i=1}^2 \,\,\text{\rm Map}_{L^2_{m+1}}((K_i^S,K_i^S\cap\partial \Sigma_i),(X,L)),
$$
which is obtained by restricting the domain, is an embedding for each fixed $T$.
This is a consequence of unique continuation.  We note that the {codomain} of the map
does not depend on $T$.
Therefore we can use Theorem \ref{exdecayT}
to study $T$-dependence of the moduli space
$\mathcal M^{\mathcal E_1 \oplus  \mathcal E_2}((\Sigma_T,\vec z);\beta)_{\epsilon}$.
We used this fact and Theorem \ref{exdecayT}  to show
smoothness of the coordinate change of the Kuranishi structure
on the moduli space of bordered pseudoholomorphic curves in
\cite[II, Appendix A, page 764-773]{fooo:book1}.
See Chapter \ref{sec:smoothness of coordinate change} for more details.
\end{rem}

The remaining chapter will be occupied by the proof of this theorem.

\par
The construction of $u_{T,(\kappa)}^{\rho}$ was given by induction on $\kappa$.
We divide the inductive step of the construction of  $u_{T,(\kappa)}^{\rho}$
from  $u_{T,(\kappa-1)}^{\rho}$ into two:
\begin{enumerate}
\item[(Part A)]
Start from  $(V^{\rho}_{T,1,(\kappa)},V^{\rho}_{T,2,(\kappa)},\Delta p^{\rho}_{T,(\kappa)})$ and end with
${\rm Err}^{\rho}_{1,T,(\kappa)}$ and ${\rm Err}^{\rho}_{2,T,(\kappa)}$.
This is the steps $\kappa$-2,$\kappa$-3 and $\kappa$-4. This is the step of the error estimates
for the $\kappa$-th iteration map $u_{T,(\kappa )}^\rho$, which is essentially a computational step.
\item[(Part B)]
Start from
${\rm Err}^{\rho}_{1,T,(\kappa-1)}$ and ${\rm Err}^{\rho}_{2,T,(\kappa-1)}$ and end with
$(V^{\rho}_{T,1,(\kappa)},V^{\rho}_{T,2,(\kappa)},\linebreak \Delta p^{\rho}_{T,(\kappa)})$.
This is the step $\kappa$-1.  This step involves inverting the approximate right inverse of
the linearization of the equation (\ref{mainequation}) at $u_{T,(\kappa-1)}^\rho$.
\end{enumerate}
\par\medskip
We denote by \index[syindex]{Pirhoi@$\mathscr P_{i,\rho_i}$}
\begin{equation}\label{mathscrPform}
\mathscr P_{i,\rho_i} : (u_i^{\rho_i})^*TX \to u_i^*TX, 
\quad
(\text{resp. }{\mathcal P}_{p^{\rho}_0}^{p_0} : T_{p^{\rho}_0}X \to T_{p_0}X)
\end{equation}
the bundle map (resp. the map) induced by the parallel transportation along the minimal geodesic.
{The main task of this chapter is to  prove the following inequalities inductively over $\kappa \geq 0$:}
\begin{equation}
\left\| \nabla_{\rho}^n \frac{\partial^{\ell}}{\partial T^{\ell}}
(V^{\rho}_{T,i,(\kappa)},\Delta p^{\rho}_{T,(\kappa)})\right\|_{W^2_{m+1-\ell,\delta}(\Sigma_i)}
< C_{5,m}\mu^{\kappa-1}e^{-\delta_1 T}, \label{form182}
\end{equation}
\begin{equation}
\aligned
&\left\| \nabla_{\rho}^n \frac{\partial^{\ell}}{\partial T^{\ell}}
({\mathscr P_{i,\rho_i}({\rm E}(u^{\rho_i}_i,u^{\rho}_{T,(\kappa)})),{\mathcal P}_{p^{\rho}_0}^{p_0}({\rm E}(p^{\rho}_0,p^{\rho}_{T,(\kappa)})}))
\right\|_{W^2_{m+1-\ell,\delta}(K_i^{5T+1}
\subset \Sigma_i)} \\
&\hskip7cm < C_{6,m}
(2 - \mu^{\kappa})
{e^{-\delta_1 T}},
\label{form184}
\endaligned
\end{equation}
\begin{eqnarray}
\left\Vert
\nabla_{\rho}^n \frac{\partial^{\ell}}{\partial T^{\ell}}
{\rm E}(p^{\rho}_0,{u_{T,(\kappa)}^{\rho}})
\right\Vert_{L^2_{m+1-\ell}(K^{9T}_i \setminus K^T_i)}
&<&
C_{7,m}
(2 - \mu^{\kappa})
e^{-\delta_1 T}, \label{form185}
\\
\!\!\left\| \nabla_{\rho}^n \frac{\partial^{\ell}}{\partial T^{\ell}}{\rm Err}^{\rho}_{i,T,(\kappa)} \right\|_{L^2_{m-\ell,\delta}(\Sigma_i)}
&<& C_{8,m}\epsilon(6)\mu^{\kappa}e^{-\delta_1 T},
\label{form183}\\
\left\| \nabla_{\rho}^n \frac{\partial^{\ell}}{\partial T^{\ell}} \frak e^{\rho} _{i,T,(\kappa)}\right\|_{L^2_{m-\ell}(K_i^{\frak{ob}})}
&<& C_{9,m}
\mu^{\kappa-1}e^{-\delta_1 T}. \label{form186}
\end{eqnarray}
Here $0 \le \ell,n \le m-2$.

\begin{rem}
$ $
\begin{enumerate}
\item
Note we use the $T$-independent coordinates $(\tau',t)$ on $K_1^{5T+1} \setminus K_1$,
$K_1^{9T} \setminus K^T_1$
and $(\tau'',t)$ on $K_2^{5T+1} \setminus K_2$, $K_2^{9T} \setminus K^T_2$.
\item
In {(\ref{form185})}, (\ref{form183}) we use {the} $\tau'$ coordinate for $i=1$ and
{the}  $\tau''$ coordinate for $i=2$.
\item
We also remark that $\Sigma_T = K_1^{5T+1} \cup K_2^{5T+1}$.
(See the definition of $\Sigma_T$ at the beginning of this chapter
and (\ref{1104}).)
\item
{
Note (\ref{form185}) implies the same inequality with 
$u_{T,(\kappa)}^{\rho}$  replaced by $\hat u_{i,T,(\kappa)}^{\rho}$,
which is defined in (\ref{form573new}).}
\end{enumerate}
\end{rem}
{Some explanation of the derivatives appearing in \eqref{form182}-\eqref{form186}
is also needed.}
Recall from  Definition \ref{defn532532} that the pair $(V^{\rho}_{T,i,(\kappa)},\Delta p^{\rho}_{T,(\kappa)})$ appearing in (\ref{form182})
is an element of the weighted Sobolev space
$$
W^2_{m+1,\delta}((\Sigma_i,\partial \Sigma_i);(\hat u^{\rho}_{i,T,(\kappa-1)})^*TX,(\hat u^{\rho}_{i,T,(\kappa-1)})^*TL)
$$
which depends on $T$ and $\rho$.
We use the inverse of {parallel transport map}
$\Phi_{i,(\kappa-1)}(\rho,T)$ (Definition \ref{def528}), to
{make the identification}
$$
\aligned
&W^2_{m+1,\delta}((\Sigma_i,\partial \Sigma_i);(\hat u^{\rho}_{i,T,(\kappa-1)})^*TX,(\hat u^{\rho}_{i,T,(\kappa-1)})^*TL)\\
&\cong
W^2_{m+1,\delta}((\Sigma_i,\partial \Sigma_i);u_{i}^*TX,u_{i}^*TL).
\endaligned
$$
{where the latter space is $(T,\rho)$-independent.} Namely we put
\begin{equation}\label{newform612}
\Psi_{i;(\kappa-1)}(\rho,T;(s,v)) = \left((\Phi_{i;(\kappa-1)}(\rho, T)^{-1}(s),
\big({\rm Pal}^{p^{\rho}_{(\kappa-1)}}_{p_0}\big)^{-1}(v)\right).
\end{equation}
{Through this identification, the precise meaning of the formula \eqref{form182}
is the inequality:}
\begin{equation}\label{form1822}
\aligned
\left\| \nabla_{\rho}^n \frac{\partial^{\ell}}{\partial T^{\ell}}
(\Psi_{i;(\kappa-1)}(\rho, T;(V^{\rho}_{T,i,(\kappa)},\Delta p^{\rho}_{T,(\kappa)})))\right\|_{W^2_{m+1-\ell,\delta}(\Sigma_i)}& \\
&\!\!\!\!\!\!\!\!\!\!\!\!\!\!\!\!\!\!\!\!\!\!\!\!\!\!\!\!\!\!\!\!\!\!\!\!\!\
\le C_{5,m}\mu^{\kappa-1}e^{-\delta_1 T}.
\endaligned
\end{equation}
We use {the} $\tau'$ coordinate in case $i=1$.

We can make sense of (\ref{form183}) in the same way as
(\ref{form1822}).
We use the isomorphism
\begin{equation}\label{form615}
L^2_{m,\delta}(\Sigma_i;(\hat u^{\rho}_{i,T,(\kappa-1)})^*TX \otimes \Lambda^{0,1})
\cong
L^2_{m,\delta}(\Sigma_i;u_{i}^*TX\otimes \Lambda^{0,1}),
\end{equation}
{to make sense out of (\ref{form186}). The isomorphism is nothing
but the closure of}
$$
\Big(({\rm Pal}_{u_i}^{\hat u^{\rho}_{i,T,(\kappa-1)}})^{(0,1)}\Big)^{-1},
$$
where $({\rm Pal}_{u_i}^{\hat u^{\rho}_{i,T,(\kappa-1)}})^{(0,1)}$ is as in
(\ref{paluv01}),
to formulate (\ref{form186}).
\par
A similar remark applies to (\ref{form185}).
(\ref{form185}) means
\begin{equation}\label{form185real}
\left\Vert
\nabla_{\rho}^n \frac{\partial^{\ell}}{\partial T^{\ell}}
{\rm Pal}^{p_0}_{p_0^{\rho}}({\rm E}(p^{\rho}_0,{u_{T,(\kappa)}^{\rho}}))
\right\Vert_{L^2_{m+1-\ell}(K^{9T}_i \setminus K^T_i)}
\le
C_{7,m}
(2 - \mu^{\kappa})
e^{-\delta_1 T}
\end{equation}
and in  (\ref{form186}) we regard
$
\frak e^{\rho} _{i,T,(\kappa)}
\in \mathcal E_i^{\frak{ob}}.
$
\par
In this way we can safely work with \eqref{form182}-\eqref{form186}.

\begin{rem}
\par
Similar remarks also apply to
the case $i=2$ using the $\tau''$ coordinate.
\end{rem}
The inductive proof of (\ref{form182})-(\ref{form186})
is written as the proof of the next proposition.

\begin{prop}\label{prop:inequalitieskappa}
{
We can choose  $C_{5,m}$, $C_{8,m}$, $C_{9,m}$ independent of $\kappa \geq 1$
so that the following holds: 
\begin{enumerate}
\item For any $\epsilon(6) > 0$ and
$C_{6,m}$, $C_{7,m}$, there exists $T_{5,m,\epsilon(6)}>0$ such that
(\ref{form182})-(\ref{form185})
for $\kappa \le \kappa_0$ imply (\ref{form183}) and (\ref{form186})
for $\kappa \le \kappa_0$ for all
$T > T_{5,m,\epsilon(6)}$.
\item For any given $C_{6,m}$, $C_{7,m}$, we can choose $\epsilon_{4,m}$
so that if $\epsilon(6) < \epsilon_{4,m}$ then (\ref{form182})-(\ref{form186}) for $\kappa \le\kappa_0-1$ imply (\ref{form182})
for $\kappa_0$. 
\item
We can choose $C_{6,m}$, $C_{7,m}$ with the following properties:
There exists $T_{6,m}$ such that
the inequalities (\ref{form184}), (\ref{form185})
for $\kappa_0$
follow from (\ref{form182}) for $\kappa \le \kappa_0$ and (\ref{form184}), (\ref{form185})
for $\kappa\le\kappa_0-1$
if $T > T_{6,m}$.
\end{enumerate}
}
\end{prop}
\begin{rem}
The constants
$T_{5,m,\epsilon(6)}$, $\epsilon_{4,m}$, $T_{6,m}$ may depend on $C_{5,m}$, $C_{6,m}$, $C_{7,m}$, $C_{8,m}$, $C_{9,m}$.
\end{rem}

The rest of this chapter will be occupied by the proof of Proposition \ref{prop:inequalitieskappa}.
(3) is elementary.
We provide its proof in Appendix \ref{appendixB-} for completeness' sake.
We will prove (1) and (2) in this chapter.
We divide our proof into two parts,
the {proofs} of (1) and of (2).
\begin{rem}
We choose the constants $C_{5,m}$, $C_{8,m}$, $C_{9,m}$ so that
(\ref{form182}), (\ref{form183}) and (\ref{form186}) hold
for $\kappa =0,1$. We do not need to change them in {the} later steps.
The constants $C_{6,m}$, $C_{7,m}$ are chosen during the proof of
 Proposition \ref{prop:inequalitieskappa} (3) given in Appendix \ref{appendixB-}.
\end{rem}
{
\begin{lem}\label{Lemma6.9}
{Let $\kappa_0 \geq 1$ be given.}
The inequalities
(\ref{form182}), (\ref{form184}), (\ref{form185}),
(\ref{form183}) and (\ref{form186})   for {$ \kappa \le \kappa_0$} imply
\begin{equation}\label{form616nwe}
\left\| \nabla_{\rho}^n\frac{\partial^{\ell}}{\partial T^{\ell}}
{\rm E}(u_i,u^{\rho}_{T,({\kappa_0})})\right\|_{W^2_{m+1-\ell,\delta}(K_i^{5T+1}
\subset \Sigma_i)}
\begin{cases} \le C_{m,(\ref{form616nwe})}
\quad &\text{$\ell = 0$} \\
\le C'_{m,(\ref{form616nwe})}e^{-T\delta_1}
\quad &\text{$\ell > 0$}
\end{cases}
\end{equation}
for $0 \le n \le m-2$, $0 \le \ell \le m-2$.
\end{lem}
\begin{proof}
The case $\ell > 0$ follows from (\ref{form184}) by using the fact $u^{\rho_i}_i$ is independent of $T$.
Using the fact that the $C^{\infty}$ norm between $u_i$ and $u^{\rho_i}_i$ is uniformly 
bounded and small as $\rho_i$ moves, we can use (\ref{form184}) to prove the case $\ell=0$ of (\ref{form616nwe}) 
by using Lemma \ref{expest2}.
\end{proof}
}
\subsection{Part A: error estimates}
\label{subsec61}

In this section we prove Proposition \ref{prop:inequalitieskappa} (1).
This section corresponds to the discussion in \cite[page 776 paragraph (A) and (B)]{fooo:book1}.

Suppose that the triple
$(V^{\rho}_{T,1,(\kappa)},V^{\rho}_{T,2,(\kappa)},\Delta p^{\rho}_{T,(\kappa)})$
satisfies (\ref{form182}).
{
We denote

\be
\text{\rm Err}_{T,(\kappa)}^\rho = \overline\partial u^{\rho} _{T,(\kappa)} - \sum_{a=0}^{\kappa}\frak e^{\rho} _{1,T,(a)} -\sum_{a=0}^{\kappa}\frak e^{\rho} _{2,T,(a)}
\ee
and
$$
\text{\rm Err}_{1,T,(\kappa)}^\rho = \chi_{{\mathcal X}}^\leftarrow \text{\rm Err}_{T,(\kappa)}^\rho, \qquad
\text{\rm Err}_{2,T,(\kappa)}^\rho = \chi_{{\mathcal X}}^\rightarrow \text{\rm Err}_{T,(\kappa)}^\rho.
$$
}
Then  noting that $\supp \frak{e}_{i,T,(\kappa)} \subset \Int K_i$, we find that
\begin{enumerate}
\item
\begin{equation}\label{1113formu}
{\rm Err}^{\rho}_{1,T,(\kappa)}(z)
=
\Pi_{\mathcal E_1(\hat u^{\rho}_{1,T,(\kappa-1)})}^{\perp}
\overline \partial\left({\rm Exp}\left(\hat u^{\rho}_{1,T,(\kappa-1)}(z),V^{\rho}_{T,1,(\kappa)}(z)\right)\right)
\end{equation}
for $z \in K_1$.
\item
\begin{equation}\label{form1106}
\aligned
&{\rm Err}^{\rho}_{1,T,(\kappa)}(\tau',t)
\\
&= (1-\chi(\tau'-5T))\times \\
&\quad \overline\partial\Big(\Exp\Big(u^{\rho}_{T,(\kappa-1)}(\tau',t),
\chi(\tau'-4T)\left(V^{\rho}_{T,2,(\kappa)}(\tau'-10T,t) - (\Delta p^{\rho}_{T,(\kappa)})^{\rm Pal}\right)
\\
&\qquad\qquad\qquad\qquad\qquad\qquad+V^{\rho}_{T,1,(\kappa)}(\tau',t)\Big)\Big)
\endaligned
\end{equation}
for $(\tau',t) \in [0,\infty)_{\tau'}\times [0,1]$. (Recall $\tau = \tau' - 5T$, $\tau' = \tau''{-}10T$
and $V^{\rho}_{T,2,(\kappa)}$ {was} defined in terms of the variable of $(\tau'',t)$.)
See Figure \ref{Figurenewnew}.
\end{enumerate}
Here $\chi : \R \to [0,1]$ is a smooth function such that \index[syindex]{chitau@$\chi(\tau)$}
\begin{eqnarray}\label{chichi}
\chi(\tau) & = &
\begin{cases}
0  & \tau < -1 \\
1  & \tau > 1
\end{cases} \nonumber \\
\chi' (\tau)&  > 0 &\, \text{ for } \, \tau \in (-1,1).
\end{eqnarray}
\begin{figure}
\centering
\includegraphics{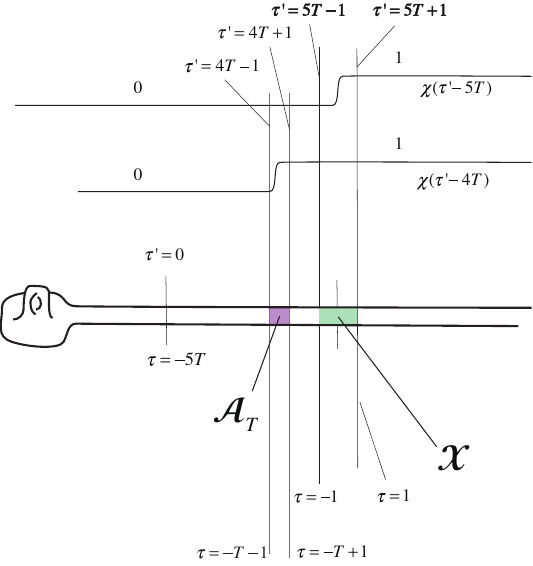}
\caption{Cut off function.}
\label{Figurenewnew}
\end{figure}
The same kind of statements also hold for $i=2$. Since the latter case
can be dealt exactly in the same way, we omit its details.
\begin{rem}\label{rem611}
Note that in Formulae (\ref{form182})-(\ref{form186}) the Sobolev norms
in the left hand side are of $W^2_{m+1-\ell,\delta}(\Sigma_i)$
etc. and are not of $W^2_{m+1,\delta}(\Sigma_i)$ etc.
The origin of this loss of differentiability (in the sense of Sobolev space)\index{loss of differentiability}
comes from the term  $V^{\rho}_{T,2,(\kappa)}(\tau'-10T)$.
In fact, we have
$$
\frac{\partial }{\partial T}V^{\rho}_{T_1,2,(\kappa)}(\tau'-10T)
= -10 \frac{\partial V^{\rho}_{T_1,2,(\kappa)}}{\partial \tau''}(\tau'-10T)
$$
for a fixed $T_1$.
Hence $\partial/\partial T$ is continuous as $L^2_{m+1} \to L^2_m$.
We remark in (\ref{form182}) for $i=2$ we use the coordinates $(\tau'',t)$
on $(-\infty,0]\times [0,1]$ to define {the} $T$ derivative of $V_{T,2,(\kappa)}^{\rho}$.
\end{rem}
\par
Taking this remark  into account, we continue with the proof.
The proof is divided into three parts.\index[syindex]{K1up4T-1@$K_1^{4T-1}$}
{
\begin{enumerate}
\item On $K_1^{4T-1}\cup K_2^{4T-1}$,
where we denote\footnote{The notation $[0,4T-1]_{\tau'}$ is introduced in Remark \ref{rem6161}.}
\begin{equation}\label{formnew623}
\aligned
K_1^{4T-1} & = K_1 \cup [0,4T-1]_{\tau'} \times [0,1], \\
K_2^{4T-1} & = K_2 \cup [-4T+1,0]_{\tau''} \times [0,1].
\endaligned
\end{equation}
\item On the central neck region containing $\mathcal X = [-1,1]_\tau \times [0,1]$.
\item On the transition region
$$
([4T-1,4T+1]_{\tau'}\times [0,1]) \cup ([-4T-1,-4T+1]_{\tau''} \times [0,1]).
$$
\end{enumerate}
}
\par\smallskip
\noindent {\bf(Estimate 1)}:
{Let $K_1^{4T-1},K_2^{4T-1}$ be as in (\ref{formnew623}).}
We remark that on $K_1^{4T-1} \setminus K_1$
the formula (\ref{form1106}) is reduced to
\begin{equation}\label{form11062200}
{\rm Err}^{\rho}_{1,T,(\kappa)}(\tau',t) = \overline\partial\Big(\Exp\Big(u^{\rho}_{T,(\kappa-1)}(\tau',t),
V^{\rho}_{T,1,(\kappa)}(\tau',t)\Big)\Big)
\end{equation}
since $1- \chi(\tau' - 5T) \equiv 1$ as $\tau' - 5T \leq -T +1 < -1$
and  $\chi(\tau' - 4T) \equiv 0$ as $\tau' - 4T \leq  -1$.
Using the fact that
$V^{\rho}_{T,2,(\kappa)}$  does not appear in (\ref{form11062200}),
we can estimate $T$ and $\rho$ derivative{s}  of ${\rm Err}^{\rho}_{1,T,(\kappa)}$
on $K^{4T-1}_1$, $K^{4T-1}_2$ by the same way as the corresponding part (that is, Formula (\ref{formB8})) of
Proposition \ref{mainestimatestep13kappa}
given in Appendix \ref{appendixB} as follows.
\begin{lem}\label{mainestimatestep13kappaT}
For any $\epsilon(7) > 0$, there exists $T_{m,
\epsilon(7), (\ref{form619form})}> 0$ with the following properties.
If $T > T_{m,
\epsilon(7), (\ref{form619form})},$
then
the element $\frak e^{\rho} _{i,T,(\kappa)} \in \mathcal E_i$
in Proposition \ref{mainestimatestep13kappa}
 satisfies the following
for $0 \le \ell,n\le m-2$.
\begin{equation}\label{form619form}
\aligned
\!\!\!\!\!
\left\|
\nabla_{\rho}^n\frac{d^{\ell}}{dT^{\ell}}
\Big(({\rm Pal}_{u_i}^{u^{\rho} _{T,(\kappa)}})^{(0,1)}\Big)^{-1}
\!\left({{\rm Err}_{T,(\kappa)}^\rho}\right)
\right\|_{L_{m-\ell,\delta}^2(K^{4T-1}_i)} &
\le C_{5,m}\mu^{\kappa}\epsilon(7) e^{-\delta_1 T},
\endaligned
\end{equation}
and
\begin{equation}\label{frakeissmall2kappa}
\left\|
\nabla_{\rho}^n\frac{d^{\ell}}{dT^{\ell}}
\left(\frak e^{\rho} _{i,T,(\kappa)}
\right)\right\|_{L_{m-\ell}^2(K^{4T-1}_i)} <\frac{C_{8,m}}{10}\mu^{\kappa-1}e^{-\delta_1 T}.
\end{equation}
\end{lem}
We provide its proof in  Appendix \ref{appendixC} for completeness.
\par\medskip
We next study the neck region. The point explained in Remark \ref{rem611} appears here.
We will do the corresponding estimates for ${\rm Err}^{\rho}_{1,T,(\kappa)}$ given in (\ref{form1106}) by considering them
for $\tau' \in [4T+1,\infty)_{\tau'}$ and for $\tau' \in [4T-1,4T +1]_{\tau'}$, separately.
\par\smallskip
\noindent {\bf(Estimate 2)}:
We first consider the domain $(\tau',t) \in [4T+1,\infty)_{\tau'} \times [0,1]_t$.
(Note this domain contain{s}  $\frak X$.)
We put:
\begin{equation}\label{eq:x,y0}
\aligned
h^0(\tau',t) &= u^{\rho}_{T,(\kappa-1)}(\tau',t)\\
h^1(\tau',t) & = \Exp\Big(u^{\rho}_{T,(\kappa-1)}(\tau',t),
\\
&\qquad\qquad V^{\rho}_{T,2,(\kappa)}(\tau'-10T,t) - (\Delta p^{\rho}_{T,(\kappa)})^{\rm Pal}
+V^{\rho}_{T,1,(\kappa)}(\tau',t)\Big).\\
\endaligned
\end{equation}
Then for $\tau' \in [4T+1,\infty)_{\tau'}$ the formula (\ref{form1106})
is reduced to
\bea\label{form11010}
\aligned
& {\rm Err}^{\rho}_{1,T,(\kappa)}(\tau',t)  = (1-\chi(\tau'-5T)) \overline\partial h^1,
\endaligned
\eea
because $\chi(\tau' - 4T) = 1$ on $[4T +1, \infty)$ as $\tau' - 4T \geq 1$.
\par
We remark that
$ {\rm Err}^{\rho}_{1,T,(\kappa)}(\tau',t) = 0$ on  $\tau' > 5T+1$ since $1-\chi(\tau'-5T) = 0$ there.
So
we need to study only on $\tau' \in [4T+1,5T+1]_{\tau'}$.
{
To estimate this error term,
we consider a family of maps $h^r$ such that 
$r \mapsto h^r(\tau',t)$ is the minimal geodesic joining 
$h^0(\tau',t)$ and $h^1(\tau',t)$.
Namely
$$
h^r(\tau',t) = \Exp(h^0(\tau',t),r{\rm E}(h^0(\tau',t),h^1(\tau',t))).
$$
More explicitly:
\begin{equation}\label{eq:x,y}
\aligned
h^r(\tau',t) & = \Exp\Big(u^{\rho}_{T,(\kappa-1)}(\tau',t),
\\
&\qquad\qquad r(V^{\rho}_{T,2,(\kappa)}(\tau'-10T,t) - (\Delta p^{\rho}_{T,(\kappa)})^{\rm Pal}
+V^{\rho}_{T,1,(\kappa)}(\tau',t))\Big).\\
\endaligned
\end{equation}
We again apply \eqref{FTE} pointwise to the function
\beastar
g(s) = \mathcal P^{-1}(\overline\partial h^s)
\eeastar
where $\mathcal P$ is induced from the parallel transport along the curve
$
r \mapsto h^r(\tau',t)
$, 
$r \in [0,s]$ at each given point $(\tau',t)$. We also note
$$
g(1)(\tau',t) = \frac{1}{(1-\chi(\tau'-5T))} \mathcal P^{-1}({\rm Err}^{\rho}_{1,T,(\kappa)}(\tau',t)).
$$
}
We remark that
\begin{equation}\label{form630}
\aligned
g(0)(\tau',t) & = \overline\partial u^{\rho}_{T,(\kappa-1)}(\tau',t), \\
g'(0)(\tau',t) & =  \left((D_{h^0}\overline\partial)({\rm E}(h^0,h^1)\right)(\tau',t).
\endaligned
\end{equation}
 
We denote the parallel transport along the shortest path by
$$
\frak P = \left({\rm Pal}^{h^0}_{u_1}\right)^{(0,1)}
$$
on the domain $[0,5T]_{\tau'} \times [0,1]$. \index[syindex]{Pfrak@$\frak P$}
\par
We first examine $g(0)$. By (\ref{newform570}) we have
$$
\overline\partial h^0(\tau',t)
=
\overline\partial u^{\rho}_{T,(\kappa-1)}(\tau',t)
=
{\rm Err}^{\rho}_{1,T,(\kappa-1)}(\tau',t) + {\rm Err}^{\rho}_{2,T,(\kappa-1)}(\tau',t)
$$
outside the supports of $\mathcal E_1$, $\mathcal E_2$,
which contain the subset $[4T+1,5T+1]_{\tau'} \times [0,1]$ we are studying.
\par
Next we examine the term $g'(0)$. The second formula of (\ref{form630}) implies:
\begin{equation}\label{form631631}
\aligned
g'(0)(\tau',t) = (D_{u^{\rho}_{T,(\kappa-1)}}\overline\partial)\big(&V^{\rho}_{T,2,(\kappa)}(\tau'-10T,t)
\\
&+V^{\rho}_{T,1,(\kappa)}(\tau',t) - (\Delta p^{\rho}_{T,(\kappa)})^{\rm Pal}\big).
\endaligned
\end{equation}
Using {(\ref{form616nwe})} and (\ref{form183}) we obtain
\begin{equation}\label{newform6240}
\aligned
&\left\Vert \nabla_{\rho}^n\frac{d^{\ell}}{dT^{\ell}} \frak P^{-1}
(D_{u^{\rho}_{T,(\kappa-1)}}\overline\partial)\big((\Delta p^{\rho}_{T,(\kappa)})^{\rm Pal}\big)
\right\Vert_{L^2_{m-\ell}(
[T,9T]_{\tau'}\times [0,1])}
\\
&\le
\sum_{\ell' \le \ell, n'\le n}
C_{m,(\ref{newform6240})} T e^{-\delta_1T} \left\Vert
 \nabla_{\rho}^{n'}\frac{d^{\ell'}}{dT^{\ell'}} \Delta
p^{\rho}_{T,(\kappa)}\right\Vert_{L^2_{m-\ell'}
([T,9T]_{\tau'} \times [0,1])}
\\
&
\le
C'_{m,(\ref{newform6240})} T \mu^{\kappa-1} e^{-2\delta_1T}.
\endaligned
\end{equation}
To show this inequality for $\ell' = \ell$ we use
the fact that $(\Delta p_{T,(\kappa)}^{\rho})^{\rm Pal}$
is almost a `constant'  and its first derivative is small.
See the last part of Appendix \ref{appendixC}.
\par
Moreover by (\ref{formula158}) in Definition \ref{defn532532} and
Definition \ref{defn:Errrho12T(kappa)},
we have
\begin{equation}\label{newform632632}
-(D_{\hat u^{\rho}_{i,T,(\kappa-1)}}\overline\partial)\big(V^{\rho}_{T,i,(\kappa)})
=
{\rm Err}^{\rho}_{i,T,(\kappa-1)}(\tau',t)
\end{equation}
for $i = 1,2$, on $[4T+1,5T+1]_{\tau'} \times [0,1]$, which lies outside the supports of $\mathcal E_1$, $\mathcal E_2$.
Note
$\hat u^{\rho}_{i,T,(\kappa-1)} = u^{\rho}_{T,(\kappa-1)}$ on the domain
$[4T+1,5T+1]_{\tau'} \times [0,1]$,
since $\chi_{\mathcal A}$ and $\chi_{\mathcal B}$ are $1$ therein. (See (\ref{form573new}).)
\par
Finally we establish the following estimate.

\begin{lem}\label{newlem6110} Let $g$ be as above.
For any $\epsilon(8) > 0$ there exists $T_{m,
\epsilon(8), (\ref{ineq623})}$ such that
\begin{equation}\label{ineq623}
\aligned
\left\Vert
\nabla_{\rho}^n\frac{d^{\ell}}{dT^{\ell}}
\left(\int_0^1\left(\int_0^s g''(r)dr\right)ds\right) \right\Vert_{L^2_{m-\ell,\delta}
([4T+1,5T+1] \times [0,1])}
\le  \mu^{\kappa}\epsilon(8) e^{-\delta_1 T}
\endaligned
\end{equation}
holds for all
$T > T_{m,\epsilon(8), (\ref{ineq623})}$.
\end{lem}
\begin{proof}
We put
$$
V(\tau',t)  =  (\frak P')^{-1}\left((V^{\rho}_{T,2,(\kappa)}(\tau'-10T,t) - (\Delta p^{\rho}_{T,(\kappa)})^{\rm Pal})
+V^{\rho}_{T,1,(\kappa)}(\tau',t)\right),
$$
where \index[syindex]{Pprimefrak@$\frak P'$}
$$
\frak P' = {\rm Pal}_{u_1}^{h^0}.
$$
(Note  on the domain $[4T+1,5T+1]_{\tau'} \times [0,1]$, we have
$\hat u^{\rho}_{2,T,(\kappa-1)} = u^{\rho}_{T,(\kappa-1)}$.
So we may regard $V^{\rho}_{T,2,(\kappa)}$ as a section of
$(u^{\rho}_{T,(\kappa-1)})^*TX$.)

We take {the} $L^2_{m-\ell}$ norm in place of {the} $L^2_{m-\ell,\delta}$ norm of
the left hand side of (\ref{ineq623}). Then in  the same way as
in the proof of Lemma \ref{mainestimatestep13kappaT}
given in Appendix \ref{appendixC} we can estimate the norm by
\begin{equation}\label{totenew624}
\aligned
C_{m,(\ref{totenew624})}
\sum_{\ell_1+\ell_2 \le\ell
\atop n_1 + n_2 \le n}
&\left\Vert \nabla_{\rho}^{n_1}\frac{d^{\ell_1}}{dT^{\ell_1}}
V \right\Vert_{L^2_{m+1-\ell_1}([4T+1,5T+1] \times [0,1])}
\\
& \times \left\Vert \nabla_{\rho}^{n_2}\frac{d^{\ell_2}}{dT^{\ell_2}}
V \right\Vert_{L^2_{m+1-\ell_2}([4T+1,5T+1] \times [0,1])}.
\endaligned
\end{equation}
Then applying the formula
\begin{equation}\label{1106}
\aligned
&\frac{\partial^{\ell}}{\partial T^{\ell}}
(\mathfrak P')^{-1} \left(\left.
V^{\rho}_{T,2,(\kappa)}(\tau' - 10 T)
\right)\right\vert_{T=T_2} \\
& =
\sum_{\ell_1+\ell_2 = \ell} (-10)^{\ell_2}
\frac{\partial^{\ell_1}}{\partial T^{\ell_1}}\frac{\partial^{\ell_2}\big((\mathfrak P')^{-1}(V^{\rho}_{T,2,(\kappa)})\big)}{\partial \tau^{\prime\prime \ell_2}}
(\tau' - 10 T_2),
\endaligned
\end{equation}
we obtain:
\begin{equation}\label{form626new}
\aligned
&\left\Vert \nabla_{\rho}^{n_i}\frac{d^{\ell_i}}{dT^{\ell_i}} V
\right\Vert_{L^2_{m+1-\ell_i}([4T+1,5T+1]_{\tau'} \times [0,1])} \\
&\le
C_{m,(\ref{form626new})}
\left\Vert \nabla_{\rho}^{n_i}\frac{d^{\ell_i}}{dT^{\ell_i}}
(\frak P')^{-1}(V^{\rho}_{T,1,(\kappa)})
\right\Vert_{L^2_{m+1-\ell_i}([4T+1,5T+1]_{\tau'} \times [0,1])}
\\
&\quad +
C_{m,(\ref{form626new})}
\left\Vert \nabla_{\rho}^{n_i}\frac{d^{\ell_i}}{dT^{\ell_i}}
(\frak P')^{-1}(V^{\rho}_{T,2,(\kappa)})
\right\Vert_{L^2_{m+1}([-6T_1+1,-5T+1]_{\tau''} \times [0,1])} \\
&\quad +
C_{m,(\ref{form626new})}
\left\Vert \nabla_{\rho}^{n_i}\frac{d^{\ell_i}}{dT^{\ell_i}}
(\frak P')^{-1}((\Delta p^{\rho}_{T,(\kappa)})^{\rm Pal})
\right\Vert_{L^2_{m+1-\ell_i}([4T+1,5T+1]_{\tau'} \times [0,1])}
\\
&\le C'_{m,(\ref{form626new})} \mu^{\kappa-1} e^{-\delta_1 T}.
\endaligned
\end{equation}
Note the weight function on our domain $[4T+1,5T+1] \times [0,1]$ is not
greater than $10 e^{5T \delta}$.
Therefore substituting \eqref{form626new} into \eqref{totenew624}, we obtain
\begin{equation}\label{form627new}
\text{\rm LHS of (\ref{ineq623})}
\le C_{m,(\ref{form627new})} e^{5T\delta} e^{-2T\delta_1}\mu^{2\kappa-2}
\le \mu^{\kappa}\epsilon(8) e^{-\delta_1 T}
\end{equation}
by taking $T_{m,\epsilon(8), (\ref{ineq623})}$ so that
 $C_{m,(\ref{form627new})} e^{-5T_{m,\epsilon(8), (\ref{ineq623})}\delta} \mu^{\kappa-2}
\le \epsilon(8)$.
Here we also used (\ref{form310310}).
\end{proof}
\par
{
We use Lemma \ref{newlem6110} and (\ref{form631631}), (\ref{newform6240}), (\ref{newform632632}) to show
\begin{equation}\label{newform624}
\aligned
&\left\Vert \nabla_{\rho}^n\frac{d^{\ell}}{dT^{\ell}}
\frak P^{-1}{\rm Err}^{\rho}_{1,T,(\kappa)}\right\Vert_{L^2_{m-\ell,\delta}([4T+1,5T+1]_{\tau'} \times [0,1])} \\
&=
\left\Vert \nabla_{\rho}^n\frac{d^{\ell}}{dT^{\ell}}
g(1)\right\Vert_{L^2_{m-\ell,\delta}([4T+1,5T+1]_{\tau'} \times [0,1])}
\\
& \le
\left\Vert \nabla_{\rho}^n\frac{d^{\ell}}{dT^{\ell}}
(g(0) + g'(0))\right\Vert_{L^2_{m-\ell,\delta}([4T+1,5T+1]_{\tau'} \times [0,1])} + \text{\rm LHS of (\ref{ineq623})}
\\
&\le
C'_{m,(\ref{newform6240})} T \mu^{\kappa-1} e^{-2\delta_1T} + \mu^{\kappa}\epsilon(8) e^{-\delta_1 T}
\\
&\le  2\mu^{\kappa}\epsilon(8) e^{-\delta_1 T},
\endaligned
\end{equation}
for $T > T_{m,\epsilon(8), (\ref{ineq623})}$, by taking 
$C'_{m,(\ref{newform6240})} T \mu^{\kappa-1} e^{-\delta_1T} \le \mu^{\kappa}\epsilon(8)$.
}
\par\smallskip
\noindent {\bf(Estimate 3)}:
We next consider $\tau' \in [4T-1,4T+1]_{\tau'}$.
In other words we study the estimate on the domain
$\mathcal A_T$.
There the formula (\ref{form1106}) for the error term is reduced to
\begin{equation}\label{form110622}
\aligned
&{\rm Err}^{\rho}_{1,T,(\kappa)}(\tau',t) \\
= & \overline\partial\Big(\Exp\Big(u^{\rho}_{T,(\kappa-1)}(\tau',t),
V^{\rho}_{T,1,(\kappa)}(\tau',t)+
\\
&\qquad \qquad\qquad
\chi(\tau'-4T)\Big(V^{\rho}_{T,2,(\kappa)}(\tau'-10T,t) -  (\Delta p^{\rho}_{T,(\kappa)})^{\rm Pal}\Big)
\Big)\Big)
\endaligned
\end{equation}
since $1- \chi(\tau' - 5T) \equiv 1$ as $\tau' - 5T \leq -T +1 < -1$.
\par
{
We consider the {path of maps} $h^s(\tau',t)$ given by
\begin{equation}\label{eq:h}
\aligned
h^s(\tau',t) &= \Exp\Big(u^{\rho}_{T,(\kappa-1)}(\tau',t),
V^{\rho}_{T,1,(\kappa)}(\tau',t)+
\\
&\qquad \qquad\qquad
s \chi(\tau'-4T)\Big(V^{\rho}_{T,2,(\kappa)}(\tau'-10T,t) -  (\Delta p^{\rho}_{T,(\kappa)})^{\rm Pal}\Big)
\Big)
\endaligned
\end{equation}
which satisfies
\beastar
\overline\partial h^0 & = & \overline\partial \Exp\Big(u^{\rho}_{T,(\kappa-1)}(\tau',t),
V^{\rho}_{T,1,(\kappa)}(\tau',t)\Big)\\
\overline\partial h^1 & = & {\rm Err}^{\rho}_{1,T,(\kappa)}.
\eeastar

The next lemma claims that the term containing
$V^{\rho}_{T,2,(\kappa)}(\tau'-10T,t)$ in (\ref{form110622})
is small.
\par
Let $(\mathcal P')^{-1}$ denote the parallel transport along the path
$s \mapsto h^s(\tau',t)$ ($s \in [0,r]$) and $\mathcal P^{-1}$ its complex linear part.
{We denote by $\frak P'$ the
parallel transport along the minimal geodesic which maps a section of
$u_1^* TX$ to a section of $(u')^*TX$ for various $u'$ which is $C^0$ close to $u_1$}.
Let $\frak P$ be its complex linear part.
\begin{lem}\label{newlem611} Let $h^s$ be as above.
For any positive number $\epsilon(9)$, there exists $T_{m,\epsilon(9),(\ref{newform626})}$
such that the next inequality holds for $T > T_{m,\epsilon(9),(\ref{newform626})}$.
{
\begin{equation}\label{newform626}
\aligned
&\Big\Vert \nabla_{\rho}^n\frac{d^{\ell}}{dT^{\ell}}
\Big(\frak P^{-1}\left(\overline\partial h^0 - \mathcal P\left(\overline\partial h^1\right)\right)\Big)
\Big\Vert_{L^2_{m-\ell,\delta}([4T-1,4T+1]_{\tau'} \times [0,1]
\subset \Sigma_{T})} \\
&\le \mu^{\kappa}\epsilon(9) e^{-\delta_1 T}.
\endaligned
\end{equation}
}
\end{lem}
\begin{proof}
This is a consequence of `drop of the weight' we mentioned in Remark \ref{rem:dropofweight}.
\par
The left hand side of (\ref{newform626}) is the $L^2_{m-\ell,\delta}$
norm of the next formula
\begin{equation}\label{form633new}
\nabla_{\rho}^n\frac{d^{\ell}}{dT^{\ell}} \Big(\frak P^{-1}\left(\overline\partial h^0 - \mathcal P\left(\overline\partial h^1\right)\right)\Big)
= -\int_0^1 \nabla_{\rho}^n\frac{d^{\ell}}{dT^{\ell}} \frak P^{-1}
 \frac{\partial}{\partial r} \mathcal P (\overline\partial h^r) \, dr
\end{equation}
Recalling the definition \eqref{eq:h},
we compute
\beastar
\frac{\partial}{\partial r} \mathcal P \left(\overline\partial h^r\right)
& = & \mathcal P D_{h^r} \overline \partial \left(\chi(\tau'-4T)\mathcal P^{-1}\Big(V^{\rho}_{T,2,(\kappa)}(\tau'-10T,t) -  (\Delta p^{\rho}_{T,(\kappa)})^{\rm Pal}\Big)\right).
\eeastar
\par
Therefore using the product rule (\ref{1106}) as before,
we can estimate the integrand of (\ref{form633new})
\begin{equation}\label{form634newnew}
\aligned
&\left\Vert \nabla_{\rho}^n\frac{d^{\ell}}{dT^{\ell}}
\frak P^{-1} \frac{\partial}{\partial r} \mathcal P (\overline \partial h^r) 
\right\Vert_{L^2_{m-\ell}}
\\
\le&
C_{m,(\ref{form634newnew})} \sum_{\ell' \le \ell}\sum_{n' \le n} \\
&  \left\Vert
\nabla_{\rho}^{n'}\frac{\partial^{\ell'}}{\partial T^{\ell'}}
(\frak P')^{-1} \Big(
V^{\rho}_{T,2,(\kappa)}(\tau'',t)
\right.
\\
&\qquad\qquad\qquad\qquad\left.-  (\Delta p^{\rho}_{T,(\kappa)})^{\rm Pal}\Big)
\right\Vert_{L^2_{m+1-\ell'}([-6T-1,-6T+1]_{\tau''} \times [0,1])}.
\endaligned
\end{equation}
in the way similar to the proof of (\ref{2ff160}) given at the end of Appendix \ref{appendixA}.
Note the norm in (\ref{form634newnew}) is {the} $L^2_{m+1-\ell'}$ norm
without weight.

\par
By the induction hypothesis the
$L^2_{m+1-\ell'}$ norm of
$$
\nabla_{\rho}^{n'} \frac{\partial^{\ell'}}{\partial T^{\ell'}}(\frak P')^{-1}
\big(V^{\rho}_{T,2,(\kappa)}(\tau'',t) -  (\Delta p^{\rho}_{T,(\kappa)})^{\rm Pal}
\big)
$$
with weight $e_{2,\delta}$
is estimated by $C_{2,m}\mu^{\kappa-1}e^{-\delta_1 T}.$
\par
Over our domain $[4T-1,4T+1]_{\tau'} \times [0,1]$,
the  weight $e_{2,\delta}$ is around $e^{6T\delta}$.
See Figure \ref{Figure5}.
Therefore
\begin{equation}\label{form635newho}
\left\Vert \nabla_{\rho}^n\frac{d^{\ell}}{dT^{\ell}}
\frak P^{-1}\Big(\int_0^1 \frac{\partial}{\partial r} \mathcal P (\overline \partial h^r) \, dr\Big)
\right\Vert_{L^2_{m-\ell}}
\le C_{m,(\ref{form635newho})} e^{-6T\delta}\mu^{\kappa-1}e^{-\delta_1 T}.
\end{equation}
On the other hand the weight $e_{1,\delta}$,
which we use as the weight of the $L^2_{m - \ell,\delta}$
norm in the left hand side of (\ref{newform626})  is around $e^{4T\delta}$.
\par
Therefore
\begin{equation}\label{supernew626}
\text{LHS of (\ref{newform626})}
\le C_{m,(\ref{supernew626})} e^{-2T\delta}\mu^{\kappa-1}e^{-\delta_1 T}.
\end{equation}
\par
We take $T_{m,\epsilon(9),(\ref{newform626})}>0$ such that
$C_{m,(\ref{supernew626})} e^{-2\delta T_{m,\epsilon(9),(\ref{newform626})}} \le \epsilon(9) \mu$.
The lemma follows.
\end{proof}
Again applying \eqref{FTE} to the map
$s \mapsto  \overline\partial\left(\Exp(u^\rho_{T,(\kappa-1)},s V^{\rho}_{T,1,(\kappa)}(\tau',t))\right)$
as in the proof of Proposition \ref{mainestimatestep13kappa}, we calculate
\begin{equation}\label{form627}
\aligned
& \overline\partial\left(\Exp(u^\rho_{T,(\kappa-1)}, V^{\rho}_{T,1,(\kappa)}(\tau',t))\right)\\
=& \overline\partial u^\rho_{T,(\kappa-1)}(\tau',t) + (D_{u^\rho_{T,(\kappa-1)}}
\overline\partial)(V^{\rho}_{T,1,(\kappa)})(\tau',t) \\
& + \int_0^1ds \int_0^s
\left(\frac{\partial^2}{\partial r^2}\right) \mathcal P
\left(\overline\partial \left(\Exp(u^{\rho}_{T,(\kappa-1)},
rV^{\rho}_{T,1,(\kappa)}(\tau',t)\right)\right)\, dr.
\endaligned
\end{equation}
Here $\mathcal P^{-1}$ is the complex linear part of the parallel transport along the path
$s \mapsto
\Exp(u^{\rho}_{T,(\kappa-1)},sV^{\rho}_{T,1,(\kappa)}(\tau',t))
$ 
($s \in [0,r]$).
We can estimate the third term of the right hand side of (\ref{form627})
in the same way as the proof of the inequality (\ref{estimateE3})
given in Appendix \ref{appendixC}
and obtain
\begin{equation}\label{form638super}
\aligned
&\!\!\!\!\left\Vert
\nabla_{\rho}^n\frac{d^{\ell}}{dT^{\ell}}
\frak P^{-1}\text{(3rd term of (\ref{form627}))}
\right\Vert_{L^2_{m-\ell}([4T-1,4T+1]_{\tau'} \times [0,1])} \\
&\!\!\!\!\le
C_{m,(\ref{form638super})}
\sum_{\ell'\le \ell, n'\le n}
\left\Vert
\nabla_{\rho}^{n'}\frac{\partial^{\ell'}}{\partial T^{\ell'}}
(\frak P')^{-1} V^{\rho}_{T,1,(\kappa)}
\right\Vert_{L^2_{m+1-\ell'}([4T-1,4T+1]_{\tau'} \times [0,1])}^2.
\endaligned
\end{equation}
Since
\begin{equation}\label{newnew6.4040}
\aligned
&\left\Vert
\nabla_{\rho}^{n'}\frac{\partial^{\ell'}}{\partial T^{\ell'}}
(\frak P')^{-1} V^{\rho}_{T,1,(\kappa)}
\right\Vert_{L^2_{m+1-\ell'}([4T-1,4T+1]_{\tau'} \times [0,1])}
\\
&\le
\left\Vert
\nabla_{\rho}^{n'}\frac{\partial^{\ell'}}{\partial T^{\ell'}}
(\frak P')^{-1} (V^{\rho}_{T,1,(\kappa)} - (\Delta p^{\rho}_{T,(\kappa)})^{\rm Pal})
\right\Vert_{L^2_{m+1-\ell'}([4T-1,4T+1]_{\tau'} \times [0,1])} \\
&\quad +
\left\Vert
\nabla_{\rho}^{n'}\frac{\partial^{\ell'}}{\partial T^{\ell'}}
(\frak P')^{-1}(\Delta p^{\rho}_{T,(\kappa)})
\right\Vert
\\
&\le
C_{m,(\ref{newnew6.4040})} \mu^{\kappa-1}e^{-\delta_1 T}
\endaligned
\end{equation}
by (\ref{form182}) we have
\begin{equation}\label{shortform641}
\text{(\ref{form638super})} \le C_{m,(\ref{shortform641})}  \mu^{2(\kappa-1)} e^{-2\delta_1 T}.
\end{equation}
\par
Finally we observe that
\begin{equation}\label{newform634}
\overline\partial u^\rho_{T,(\kappa-1)}
+ (D_{u^\rho_{T,(\kappa-1)}}
\overline\partial)(V^{\rho}_{T,1,(\kappa)}(\tau',t)) = 0,
\end{equation}
on $[4T-1,4T+1]_{\tau'} \times [0,1]$.
This follows from (\ref{newform570}) and
 Definition \ref{defn532532} (\ref{formula158}) together with
 ${\rm Err}^{\rho}_{2,T,(\kappa)}(\tau',t) = 0$ on $[4T-1,4T+1]_{\tau'} \times [0,1]$.
\par
In sum, by Lemma \ref{newlem611} and
(\ref{form627}),
(\ref{shortform641}) and (\ref{newform634}), we obtain
\begin{equation}\label{newform63535}
\aligned
&\left\Vert \nabla_{\rho}^n\frac{d^{\ell}}{dT^{\ell}} \frak P^{-1}\mathcal P
({\rm Err}^{\rho}_{1,T,(\kappa)})\right\Vert_{L^2_{m-\ell,\delta}([4T-1,4T+1]_{\tau'} \times [0,1])} \\
&\le C_{m,(\ref{newform63535})}\mu^{\kappa}\epsilon(9) e^{-\delta_1 T},
\endaligned
\end{equation}
for $T > T_{m,\epsilon(9),(\ref{newform63535})}$.
\par
It implies 
\begin{equation}\label{newform635350}
\aligned
&\left\Vert \nabla_{\rho}^n\frac{d^{\ell}}{dT^{\ell}} \frak P^{-1}
({\rm Err}^{\rho}_{1,T,(\kappa)})\right\Vert_{L^2_{m-\ell,\delta}([4T-1,4T+1]_{\tau'} \times [0,1])} \\
&\le C_{m,(\ref{newform635350})}\mu^{\kappa}\epsilon(9) e^{-\delta_1 T}.
\endaligned
\end{equation}
In fact $\frak P^{-1}\mathcal P$ in (\ref{newform63535}) is induced by 
${\rm Pal}_{h^0}^{u_1} \circ {\rm Pal}_{h^1}^{h^0}$
and $\frak P^{-1}$ in (\ref{newform635350}) is induced by
${\rm Pal}_{h^1}^{u_1}$.
Therefore we can show that (\ref{newform63535}) implies 
(\ref{newform635350}) using (\ref{form182})-(\ref{form185}) 
and Lemma \ref{Lemma6.9} for $\kappa$, $\kappa-1$.
Namely we can use them to estimate $T$ and $\rho$ derivatives of 
$$
{\rm Pal}_{h^1}^{u_1} \circ {\rm Pal}^{h^1}_{h^0} \circ {\rm Pal}^{h^0}_{u_1}
$$
in the same way as in Appendix \ref{appendixA2}. 
\par
We can now complete the proof of Proposition \ref{prop:inequalitieskappa} (1).
We take $\epsilon(7)$, $\epsilon(8)$, $\epsilon(9)$  such that
$\epsilon(7) < C_{8,m}\epsilon(6)/10$,
$C_{m,(\ref{newform624})}\epsilon(8) < C_{8,m}\epsilon(6)/10$
and $\epsilon(9) < C_{8,m}\epsilon(6)/10$.
Then if $T > \max\{
T_{m,
\epsilon(7), (\ref{form619form})}, T_{m,
\epsilon(8), (\ref{ineq623})},T_{m,\epsilon(9),(\ref{newform626})},
T_{m,\epsilon(9),(\ref{newform63535})}\}$
Lemmata \ref{mainestimatestep13kappaT}, \ref{newlem611}
and Formulae (\ref{newform624}), (\ref{newform635350}) imply (\ref{form183}).
(\ref{form186}) then follows from Lemma \ref{mainestimatestep13kappaT} (\ref{frakeissmall2kappa}).
\qed
\par\medskip

\par\medskip
\begin{rem}\label{Abremark}
In \cite{Abexotic} Abouzaid used {the} $L^p_1$ norm for the maps $u$. He then proved that
the gluing map is continuous with respect to $T$ (that is $S$ in the notation of \cite{Abexotic})
but does not prove its differentiability with respect to $T$.
(Instead he used the technique to remove the part of the
moduli space with $T>T_0$.
This technique certainly works for the purpose of \cite{Abexotic}.)
In fact if we use {the} $L^p_1$ norm instead of {the} $L^2_m$ norm
then the left hand side of (\ref{form185}) becomes {the} $L^p_{-1}$ norm
which is hard to use.
\par
Abouzaid mentioned in  \cite[Remark 5.1]{Abexotic} that this point is related
to the fact that quotients of Sobolev spaces by the diffeomorphisms in the
source are not naturally equipped with the structure of smooth Banach manifold.
Indeed in the situation where there is an automorphism on $\Sigma_2$,
for example when $\Sigma_2$ is the disk with one boundary marked point at $\infty$,
then the $T$ parameter is killed by a part of the automorphism.
So the shift of $V^{\rho}_{T,2,(\kappa)}$ by $T$ that appears in the
second term of (\ref{form1106}) will be equivalent to the action of
the automorphism group of $\Sigma_2$ in such a situation.
The shift of $T$ causes the loss of differentiability in the sense of Sobolev space
in the formulas (\ref{form182}) -(\ref{form186}).
However at the end of the day we can still get the differentiability of $C^{\infty}$ order and its
exponential decay by using various {\it weighted} Sobolev spaces with various $m$ simultaneously
using the fact that $C^\infty$ topology is a Frech\^et topology,
as we show during the proof of Lemma \ref{lem82}.
(See Remark \ref{differentm} also.)
\end{rem}
\begin{rem}\label{rem12}
In \cite[Subsection 7.2.3]{DKbook} Donaldson-Kronheimer mentioned that there is
exactly one place where their construction of the basic package on the moduli space
of ASD-connections uses  {the}  $L^p_m$ space for $p \ne 2$ which is not conformally invariant.
It is exactly the place of gluing construction, the place similar
to what we are studying in this paper.
There Donaldson-Kronheimer used the $L^p_1$ space.
The reason  they do need {the} $L^p_m$  norm for $p\ne 2$ is that
in \cite{DKbook} Donaldson-Kronheimer do not use {\it weighted} Sobolev or Banach norm.
\par
In the  framework  of \cite{DKbook}  (which is adapted to the pseudoholomorphic curve by
\cite{McSa94}) the `neck region' of 4 manifold is regarded as $D^4(1) \setminus D^4(1/R^2)$
with the standard Riemannian metric
on $D^4(1) \setminus D^4(1/R)$
and the metric induced from {the} standard metric by $x \mapsto Rx/\vert x\vert^2$ on
$D^4(1/R) \setminus D^4(1/R^2)$.
(In \cite{McSa94} the `neck region'  of the source curve is $D^2(1) \setminus D^2(1/R^2)$
with a similar metric.)
So their metric is different from the cylindrical metric on $[-5T,5T] \times S^1$,
which is one we use in this paper.
\par
If we change the variables from $D^2(1) \setminus D^2(1/R^2)$ to $[-5T,5T] \times S^1$
(with $R = e^{5\pi T}$)
but still use  the above mentioned Riemannian metric of  $D^2(1) \setminus D^2(1/R^2)$
then it is equivalent to using the cylindrical metric of $[-5T,5T] \times S^1$
together with some weight.
Note this weight function is $1$ if and only if the Sobolev space involved is conformally
invariant. So in the situation  such as the one appearing in \cite[Subsection 7.2.3]{DKbook} this weight function
is nontrivial.
Actually one can observe that the weight function appears in that way is similar to
the weight
$e_{\delta,T}$, which
we use in this paper.
In other words, if we use  an  appropriate $L^p_1$ norm,
using the cylindrical metric with weight is not  very different from using the
the above mentioned metric on $D^2(1) \setminus D^2(1/R^2)$.
\par
However if we consider {the} $L^2_m$  norm with $m$ large enough then our weighted
norm (after changing the variables to  $D^2(1) \setminus D^2(1/R^2)$)
does not coincide with the $L^2_m$ norm with respect to the above mentioned
metric.
Thus it seems important to use  a weighted norm for the study of higher derivative{s}
with respect to the gluing parameter $T$.
\par
We remark that in  \cite{freedUhlen} Freed-Uhlenbeck worked out the gluing analysis
of ASD connections in the frame work of $L^2$ theory (that is, without using $L^p_k$ spaces
but using only $L^2_k$ spaces).
Freed-Uhlenbeck used the cylindrical metric on $\R \times S^3$.
So the method of \cite{freedUhlen} is closer to ours.
It seems that Freed-Uhlenbeck do not need to use weighted Sobolev norm since
in their case they can use the fact that their 3 manifold is $S^3$ and show exponential
decay without using  weighted Sobolev norm. Actually, in their situation, {the} Chern-Simons functional
on $S^3$ is not only a Morse-Bott function but also a Morse function.
In our situation, the non-linear Cauchy-Riemann equation on $\R \times S^1$ or
on $\R \times [0,1]$ with Lagrangian boundary condition is degenerate at infinity.
In other words we are in Morse-Bott situation. By this reason, it seems inevitable to
use weighted Sobolev norm when we work with the cylindrical metric.
\end{rem}
\begin{rem}
Another difference between our construction and the construction of
\cite{DKbook},\cite{McSa94} is the choice of cut off function.
In the formula at the end of \cite[page 172]{McSa94}
they used the cut off function $\beta$ appearing in
\cite[Lemma 7.2.10]{DKbook},
\cite[Lemma A.A.1]{McSa94}  to obtain the right inverse.
If we rewrite their formula in terms of the cylindrical coordinates
the cut off function appearing there has
mostly of constant slope $\vert 1/\log\delta\vert$
and the support of its first derivative has
length $\sim \vert \log\delta\vert$, here $\delta$
is a small number.
So the size of the error term caused by the derivative of
the cut off function is (pointwise)
$\sim \vert 1/\log\delta\vert$, which is small.
\par
Our choice of cut off function is,
for example, $\chi_{\mathcal A}^{\leftarrow}$.  In the cylindrical coordinates the size of its derivative is
$\sim 1$ and the length of the support of its first derivative is also $\sim 1$.
So the size of the error term caused by the derivative of
the cut off function is (pointwise)
$\sim 1$, which is {\it not} small.
We use the `drop of the weight argument' mentioned in Remark
\ref{rem:dropofweight} to show that this error term is small
in our {\it weighted} Sobolev norm.
See Figure \ref{newcutofffig}.
\par
Since we use the {\it weighted} Sobolev space, the estimate is
easier to carry out in case the support of the derivative
of the cut off function has bounded length.
(This is because then the ratio between maximum and minimum of the
{weight} function on the above mentioned support is bounded.)
\end{rem}
\begin{figure}
\centering
\includegraphics{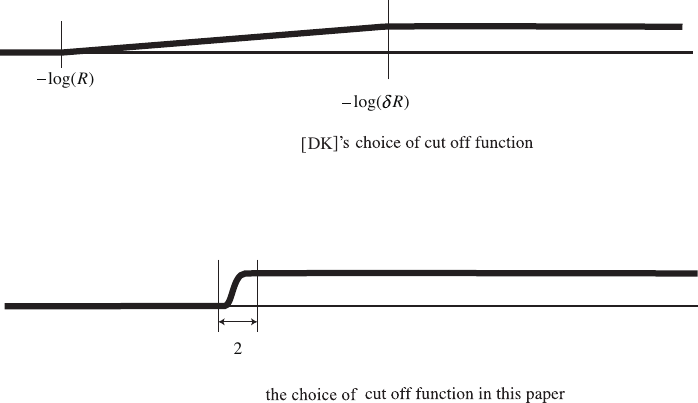}
\caption{The choices of cut off functions.}
\label{newcutofffig}
\end{figure}
\par\medskip
\subsection{Part B: estimates for the approximate inverse}
\label{subsec62}
In this section we prove Proposition \ref{prop:inequalitieskappa} (2).
This section corresponds to the discussion given in
\cite[page 776 the paragraph next to (B)]{fooo:book1}.

We assume  (\ref{form182})-(\ref{form186}) for $\kappa$ and will prove
(\ref{form182}) for $\kappa+1$.
(So we are doing Step $\kappa+1$.)
\par
This part is nontrivial only because the construction here is global.
(Solving linear equation.)
So we first review the set-up of the function space that is independent of $T,\rho,\kappa$.
\par
In Definition \ref{defn530} we defined a function space $\frak H_{(\kappa)}(\mathcal E_1,\mathcal E_2;\rho,T)$,
which is a subspace of (\ref{cocsissobolev2k}).
We solved the linearized equation on it. (See (\ref{inHhanru}).)
The space  (\ref{cocsissobolev2k}) is  $T,\rho,\kappa$-dependent. However
$\frak H_{(\kappa)}(\mathcal E_1,\mathcal E_2;\rho,T)$ is the image of
$\frak H(\mathcal E_1,\mathcal E_2)$, which is independent of $T,\rho,\kappa$,
by the map $(\Phi_{1;(\kappa)}(\rho,T), \Phi_{2;(\kappa)}(\rho,T))$. We recall
$\frak H(\mathcal E_1,\mathcal E_2)$ is defined in Definition \ref{defnfrakH}.
We put\index[syindex]{I0kapparho@$I^0_{(\kappa),\rho,T}$}
$$
I^0_{(\kappa),\rho,T} =
(\Phi_{1;(\kappa)}(\rho,T), \Phi_{2;(\kappa)}(\rho,T))
: \frak H(\mathcal E_1,\mathcal E_2) \to \frak H_{(\kappa)}(\mathcal E_1,\mathcal E_2;\rho,T).
$$
By composing it we can take the domain independent of $T,\rho,\kappa$.
\par
We next consider the {codomain of the above approximate linearization}.
Using the map in (\ref{paluv01}) we define\index[syindex]{I1kappairho@$I^1_{i,(\kappa),\rho,T}$}
$$
I^1_{i,(\kappa),\rho,T}: L^2_{m,\delta}(\Sigma_i;u_{i}^*TX \otimes \Lambda^{0,1}) \to
L^2_{m,\delta}(\Sigma_i;(\hat u^{\rho}_{i,T,(\kappa)})^{*}TX \otimes \Lambda^{0,1})
$$
as the closure of
$({\rm Pal}^{\hat u^{\rho}_{i,T,(\kappa)}}_{u_i})^{(0,1)}.
$
\par
Then we define\index[syindex]{I1kapparho@$I^1_{(\kappa),\rho,T}$}
$$
\aligned
I^1_{(\kappa),\rho,T} & :  L^2_{m,\delta}(\Sigma_1;u_{1}^*TX \otimes \Lambda^{0,1})
\oplus  L^2_{m,\delta}(\Sigma_2;u_{2}^*TX \otimes \Lambda^{0,1}) \\
& \longrightarrow
L^2_{m,\delta}(\Sigma_1;(\hat u^{\rho}_{1,T,(\kappa)})^{*}TX \otimes \Lambda^{0,1}) \oplus
L^2_{m,\delta}(\Sigma_2;(\hat u^{\rho}_{2,T,(\kappa)})^{*}TX \otimes \Lambda^{0,1})
\endaligned
$$
as the direct sum $I^1_{(\kappa),\rho,T} = I^1_{1;(\kappa),\rho,T}\oplus I^1_{2;(\kappa),T}$.
Thus the composition
$$
(I^1_{(\kappa),\rho,T})^{-1}\circ \left(D_{\hat u^{\rho}_{1,T,(\kappa)}}\overline{\partial}\oplus
D_{\hat u^{\rho}_{2,T,(\kappa)}}\overline{\partial}\right)\circ I^0_{(\kappa),\rho,T}
$$
defines an operator, which we denote by
\begin{equation}\label{DkapparhoT}
D_{(\kappa),\rho,T} :\frak H(\mathcal E_1,\mathcal E_2)
\to L^2_{m,\delta}(\Sigma_1;u_{1}^*TX \otimes \Lambda^{0,1})\oplus L^2_{m,\delta}(\Sigma_2;u_{2}^*TX \otimes \Lambda^{0,1}).
\end{equation}
{Both domain and codomain of this operator now are independent of $\kappa,\rho,T$.}
\par
We need to invert the operator $D^{{\rm app},(\kappa)}_{\hat u^{\rho}_{1,T,(\kappa)}}\overline{\partial}\oplus
D^{{\rm app},(\kappa)}_{\hat u^{\rho}_{2,T,(\kappa)}}\overline{\partial}$
 modulo  $\mathcal E_1(\hat u^{\rho}_{1,T,(\kappa)}) \oplus \mathcal E_2(\hat u^{\rho}_{2,T,(\kappa)})$.
 (See (\ref{144ffff}) and (\ref{formula158}).)
We remark the subspace
\begin{equation}\label{form625}
\mathcal E_{i,(\kappa),\rho,T} : =
(I^1_{(\kappa),\rho,T})^{-1}(\mathcal E_1(\hat u^{\rho}_{i,T,(\kappa)})
\oplus \mathcal E_2(\hat u^{\rho}_{i,T,(\kappa)}))
\end{equation}
{is different from $\mathcal E_1(u_1) \oplus \mathcal E_2(u_2)$} and is $(\kappa,\rho,T)$-dependent {in general:}
In fact,
by definition $\mathcal E_i(\hat u^{\rho}_{i,T,(\kappa)})$ is the
image of $\mathcal E^{\frak{ob}}_i$ {under} the parallel transport
$({\rm Pal}_{u_{i}^{\frak{ob}}}^{\hat u^{\rho}_{i,T,(\kappa)}})^{(0,1)}$.
(See (\ref{Eiuiprime}).)
Therefore
$$
\aligned
&(I^1_{(\kappa),\rho,T})^{-1}(\mathcal E_1(\hat u^{\rho}_{1,T,(\kappa)})
\oplus \mathcal E_2(\hat u^{\rho}_{2,T,(\kappa)}))
\\
&=
\bigoplus_{i=1}^2
\left(\left(({\rm Pal}^{\hat u^{\rho}_{i,T,(\kappa)}}_{u_i})^{(0,1)}\right)^{-1}
\circ ({\rm Pal}^{\hat u^{\rho}_{i,T,(\kappa)}}_{u_i^{\frak{ob}}})^{(0,1)}\right)
(\mathcal E^{\frak{ob}}_i)
\endaligned
$$
which is different from $\mathcal E_1(u_1) \oplus \mathcal E_2(u_2)$
since
\begin{equation}\label{form2626}
({\rm Pal}^{u_i}_{u_i^{\frak{ob}}})^{(0,1)}(\mathcal E_i^{\frak{ob}})
\ne
\left(\left(({\rm Pal}^{\hat u^{\rho}_{i,T,(\kappa)}}_{u_i})^{(0,1)}\right)^{-1}
\circ ({\rm Pal}^{\hat u^{\rho}_{i,T,(\kappa)}}_{u_i^{\frak{ob}}})^{(0,1)}\right)
(\mathcal E_i^{\frak{ob}})\end{equation}
in general.
\begin{rem}
In our situation where we consider only one $u_i^{\frak{ob}}$
we can trivialize the bundle $u' \mapsto L^2_m(\Sigma_T,(u')^*TX \otimes \Lambda^{0,1})$
{over $\{u'\}$ near $\hat u^{\rho}_{T,(\kappa)}$} by sending them to
\begin{equation}\label{differentparallel}
\bigoplus_{i=1}^2L^2_m(\Sigma_i,(u^{\frak{ob}}_i)^*TX \otimes \Lambda^{0,1})
\end{equation}
using parallel transport. Then
the image of $\mathcal E_i(u')$ in (\ref{differentparallel})
will not vary.
\par
However for our application, we need to consider
several different $u^{\frak{ob}}_i$'s, in which case there is no choice of the trivialization of
$u' \mapsto \mathcal E_i(u') \subset L^2_m(\Sigma_T,(u')^*TX \otimes \Lambda^{0,1})$
so that the direct sum of the obstruction spaces do not vary.
\par
However the way to estimate this discrepancy  in the case when
we have several $u^{\frak{ob}}_i$'s is the same as we do in this section.
\end{rem}
We can estimate the $(T,\rho)$ and $\kappa$ dependence of (\ref{form625})
and use certain elementary functional analysis to go around the problem of
this dependence as follows.
\par
We first observe the next lemma.
Let $\{{\bf e}_{i,a} \mid a =1,\dots,\dim \mathcal E^{\frak{ob}}_i\}$
be a basis of $\mathcal E^{\frak{ob}}_i$. We put\index[syindex]{eprimeiakappa@${\bf e}'_{i,a;(\kappa)}(\rho,T)$}
\begin{equation}
\aligned
{\bf e}'_{i,a;(\kappa)}(\rho,T)
&=
\left(\Big(({\rm Pal}^{\hat u^{\rho}_{i,T,(\kappa)}}_{u_i})^{(0,1)}\Big)^{-1} \circ
({\rm Pal}^{\hat u^{\rho}_{i,T,(\kappa)}}_{u^{\frak{ob}}_i})^{(0,1)}\right)({\bf e}_{i,a}) \\
&\in L^2_{m,\delta}(\Sigma_i;u_{i}^*TX \otimes \Lambda^{0,1}).
\endaligned
\end{equation}
Note $\{{\bf e}'_{i,a;(\kappa)}(\rho,T)\mid a =1,\dots,\dim \mathcal E^{\frak{ob}}_i\}$  is a basis of
$\mathcal E_{i,(\kappa),\rho,T}$.
\par
We denote by ${\bf e}_{i,a;(\kappa)}(\rho,T)$ ($a=1,\dots,\dim \mathcal E^{\frak{ob}}_i$) the basis obtained
from ${\bf e}'_{i,a;(\kappa)}(\rho,T)$ by applying the Gram-Schmidt orthogonalization to
${\bf e}'_{i,a;(\kappa)}(\rho,T)$.
(We use the $L^2$ inner product (\ref{innerprod1}) on
$L^2_{m,\delta}(\Sigma_i;u_{i}^*TX \otimes \Lambda^{0,1})$.)

\begin{lem}\label{lemma615}
There exists $C_{m,(\ref{form6321})}$ such that
\begin{equation}\label{form6321}
\left\Vert \nabla_{\rho}^n\frac{\partial^{\ell}}{\partial T^{\ell}}{\bf e}_{i,a;(\kappa)}(\rho,T)
\right\Vert_{L^2_{m+1-\ell}}
\le
C_{m,(\ref{form6321})} e^{-\delta_1 T},
\end{equation}
for {$m-2\ge n\ge 0$, $m-2\ge \ell > 0$}.
\end{lem}
We remark that since the support of ${\bf e}_{i,a;(\kappa)}(\rho,T)$
is in $K_i$ we do not need to use a weighted norm.
\begin{proof}
We prove this lemma in Appendix \ref{appendixA2bisbis}.
\end{proof}
Let $\mathcal E_{i,(\kappa),\rho,T}^{\perp}$ be the $L^2$ orthonormal complement
of $\mathcal E_{i,(\kappa),\rho,T}$ in
$L^2_{m,\delta}(\Sigma_i;u_{i}^*TX \otimes \Lambda^{0,1})$.
In other words, we define\index[syindex]{EikapparhoT@$\mathcal E_{i,(\kappa),\rho,T}^{\perp}$}
$$
\aligned
\mathcal E_{i,(\kappa),\rho,T}^{\perp}
=
\{ W &\in L^2_{m,\delta}(\Sigma_i;u_{i}^*TX \otimes \Lambda^{0,1})
\mid \\
&\langle\!
\langle W,{\bf e}_{i,a;(\kappa)}(\rho,T)
\rangle\!\rangle_{L^2} =0,\,\, a=1,\dots,\dim\mathcal E_i^{\frak{ob}}
\}.
\endaligned
$$
Here we use the $L^2$ inner product defined in
(\ref{innerprod1}).
Since elements  ${\bf e}_{i,a;(\kappa)}(\rho,T)$ are supported in $K_i$
and are smooth, we can
safely use the $L^2$ inner product, without weight.
We also define  the projections
$$
\Pi_{\mathcal E_{i,(\kappa),\rho,T}} : L^2_{m,\delta}(\Sigma_i;u_{i}^*TX \otimes \Lambda^{0,1})
\to \mathcal E_{i,(\kappa),\rho,T},
$$
by
\begin{equation}\label{formnew649}
\Pi_{\mathcal E_{i,(\kappa),\rho,T}}(W) =
\sum_{a=1}^{\dim\mathcal E_i^{\frak{ob}}}
\langle\!
\langle
W,{\bf e}_{i,a,(\kappa)}(\rho,T)
\rangle\!\rangle_{L^2}
{\bf e}_{i,a,(\kappa)}(\rho,T).
\end{equation}
We put
$$
\mathcal E_{(\kappa),\rho,T} = \bigoplus_{i=1}^2\mathcal E_{i,(\kappa),\rho,T},
\quad
\mathcal E_{(\kappa),\rho,T}^{\perp} = \bigoplus_{i=1}^2\mathcal E_{i,(\kappa),\rho,T}^{\perp}
$$
and
$
\Pi_{\mathcal E_{(\kappa),\rho,T}} =
\Pi_{\mathcal E_{1,(\kappa),\rho,T}} \oplus \Pi_{\mathcal E_{2,(\kappa),\rho,T}}.
$
The operator we need to invert is
$$
({\rm id} - \Pi_{\mathcal E_{(\kappa),\rho,T}}) \circ D_{(\kappa),\rho,T}
: \frak H(\mathcal E_1,\mathcal E_2)
\to \mathcal E^{\perp}_{(\kappa),\rho,T}.
$$
\par
Let $\mathcal E_{i}^{\perp}$ be the $L^2$ orthogonal complement of $\mathcal E_{i}
= \mathcal E_i(u_i)$ in
$L^2_{m,\delta}(\Sigma_i;u_i^{*}TX \otimes \Lambda^{0,1})$,
$\Pi_{\mathcal E_{i}}^{\perp} : L^2_{m,\delta}(\Sigma_i;u_{i}^*TX \otimes \Lambda^{0,1})
\to \mathcal E_{i}^{\perp}$ the associated $L^2$ projection.
We put  $\Pi = \Pi_{\mathcal E_{1}}^{\perp} \oplus \Pi_{\mathcal E_{2}}^{\perp}$.
By Lemma \ref{lemma615} the restriction of $\Pi$ induces an isomorphism
\begin{equation}
\Pi : \mathcal E^{\perp}_{(\kappa),\rho,T} \to \mathcal E_{1}^{\perp}
\oplus \mathcal E_{2}^{\perp}.
\end{equation}
\par
The map $D_{(\kappa),\rho,T}$ induces a map\index[syindex]{DkapparhoT@$\overline D_{(\kappa),\rho,T}$}
\begin{equation}\label{formnew653}
\overline D_{(\kappa),\rho,T} =
\Pi \circ
({\rm id} - \Pi_{\mathcal E_{(\kappa),\rho,T}}) \circ D_{(\kappa),\rho,T} :\frak H(\mathcal E_1,\mathcal E_2)
\to \mathcal E_{1}^{\perp}\oplus \mathcal E_{2}^{\perp}.
\end{equation}
Both the {domain and the codomain} of $\overline D_{(\kappa),\rho,T}$ are independent of
$\kappa, T, \rho$. We will invert this operator 
{and examine its $T, \, \rho$ dependence} by using the next lemma.
\par
Recall from Definition \ref{defnfrakH} that $ \frak H(\mathcal E_1,\mathcal E_2)$ is a subspace of
$$
\bigoplus_{i=1}^2W^2_{m+1,\delta}((\Sigma_i,\partial \Sigma_i);u_i^*TX,u_i^*TL),
$$
{consisting of} the pairs $(V_1,V_2)$ where $V_i = (s_i,v_i)$ is a pair of a section $s_i$ and its asymptotic value $v_i$.
(Note $v_1 = v_2$ for an element of $ \frak H(\mathcal E_1,\mathcal E_2)$.)

\begin{lem}\label{lem616}
For $V \in \frak H(\mathcal E_1,\mathcal E_2)$ the following holds:
\begin{enumerate}
\item
There exist $C_{m,(\ref{form626})}, C'_{m,(\ref{form626})} >0$ such that
\begin{equation}\label{form626}
C_{m,(\ref{form626})} \| V\|_{W^2_{m+1,\delta}} \le \|\overline D_{(\kappa),\rho,T}(V)\|_{L^2_{m,\delta}} \le
C'_{m,(\ref{form626})}  \| V\|_{W^2_{m+1,\delta}}.
\end{equation}
\item
There exist $C_{m,(\ref{estimateDover})}, C_{m,(\ref{11231})} >0$ such that
\begin{equation}\label{estimateDover}
\left\|{\nabla_{\rho}^n}(\overline D_{(\kappa),\rho,T}(V) -  \overline D_{(0),\rho,T}(V)) \right\|_{L^2_{m,\delta}}
\le C_{m,(\ref{estimateDover})}  e^{-\delta_1 T/10}\| V\|_{W^2_{m+1,\delta}},
\end{equation}
for $m\ge n$ and
\begin{equation}\label{11231}
\left\|\nabla_{\rho}^n \frac{\partial^{\ell}}{\partial T^{\ell}}  \overline D_{(\kappa),\rho,T}(V)
\right\|_{L^2_{m-\ell,\delta}} \le C_{m,(\ref{11231})}  e^{-\delta_1T/10}
\| V\|_{W^2_{m+1,\delta}}\end{equation}
for $m\ge n \ge 0,{m\ge \ell > 0}$.
\end{enumerate}
\end{lem}
\begin{proof} The inequality
(\ref{form626}) follows from (\ref{estimateDover}) and the invertibility of 
the operator $ \overline D_{(0),\rho,T}$.
\par
The proofs of (\ref{estimateDover}), \eqref{11231} will occupy the rest of the proof.
\par
First, we remark that, by Lemma \ref{lemma615} and (\ref{formnew649}) we have:
\begin{equation}\label{estimateforBBB}
\left\Vert \nabla_{\rho}^n\frac{\partial^{\ell}}{\partial T^{\ell}}
(\Pi \circ
({\rm id} -  \Pi_{\mathcal E_{(\kappa),\rho,T}}))
\right\Vert_{L^2_{m+1-\ell}}
\le
C_{m,(\ref{estimateforBBB})}e^{-\delta_1 T}.
\end{equation}
\par\smallskip
We next study $D_{(\kappa),\rho,T}$.
By {(\ref{form184}) and (\ref{form185})
(See also Lemma \ref{Lemma6.9}.)},
we obtain the next inequalities by induction hypothesis:
\begin{eqnarray}
\displaystyle
{\left\| \nabla_{\rho}^n
\mathscr P_{i,\rho_i}{\rm E}(u^{\rho}_i,\hat u^{\rho}_{i,T,(\kappa)})\right\|_{L^2_{m+1-\ell}(K^T_i)}}
&\le& C_{m,(\ref{1124-1})}{e^{-\delta_1 T}},\label{1124-1}\\
\displaystyle\left\| \nabla_{\rho}^n
{\rm Pal}_{p_0^{\rho}}^{p_0}({\rm E}(p_0^{\rho},\hat u^{\rho}_{i,T,(\kappa)}))\right\|_{L^2_{m+1-\ell}
(K_i^{9T} \setminus K_i^T)}
&\le& C_{m,(\ref{1125-1})}e^{-\delta_1 T},\label{1125-1}
\end{eqnarray}
and
\begin{eqnarray}
\displaystyle\left\| \nabla_{\rho}^n \frac{\partial^{\ell}}{\partial T^{\ell}}
{\rm E}(u_i,\hat u^{\rho}_{i,T,(\kappa)})\right\|_{L^2_{m+1-\ell}(K^T_i)}
&\le& C_{m,(\ref{1124})}e^{-\delta_1 T},\label{1124}\\
\displaystyle\left\| \nabla_{\rho}^n \frac{\partial^{\ell}}{\partial T^{\ell}}
{\rm Pal}_{p_0^{\rho}}^{p_0}
({\rm E}(p_0^{\rho},\hat u^{\rho}_{i,T,(\kappa)}))\right\|_{L^2_{m+1-\ell}(K^{9T}_i \setminus K^T_i)}
&\le& C_{m,(\ref{1125})}e^{-\delta_1 T}.\label{1125}
\end{eqnarray}
Here $m-2\ge n$ and ${m-2\ge \ell > 0}$.
Note we use $\tau'$ as the coordinate of $ [0,\infty)_{\tau'}$
and $\tau''$ as the coordinate of $(-\infty,0]_{\tau''}$ in the left hand sides of
(\ref{1125-1}) and (\ref{1125}), respectively.
Hereafter we assume $n$, $\ell$ satisfy
 $m-2\ge n$ and ${m-2\ge \ell {\ge} 0}$ until the end of the proof of Lemma \ref{lem616}.
In fact (\ref{1125-1}), (\ref{1125}) follow from (\ref{form185}).
The inequality (\ref{1124-1}) follows from (\ref{form184}).
The inequality (\ref{1124}) follows from (\ref{form184}).
\par
We also remark that
$$
\hat u^{\rho}_{i,T,(\kappa-1)} \equiv p^{\rho}_{T,(\kappa-1)}
$$
on $\Sigma_i \setminus K_i^{7T+1}$.
\par
We will derive  (\ref{estimateDover}), (\ref{11231})
with $\overline D_{(\kappa),\rho,T}$ replaced by $D_{(\kappa),\rho,T}$
from these four inequalities as follows.
\par
We put $V = (V_1,V_2)$ and $V_i = (s_i,v_i)$.
\par
We consider the case of $\Sigma_1$ only since the case of $\Sigma_2$ is the same.
We divide the domain $\Sigma_1$ into the pieces $K_{1}$ and $[k-1,k+1]_{\tau'} \times [0,1]$,
($k\in [0,\infty) \cap \Z$).
\par
We denote each piece of these domains by $\Sigma(a)$.
For each such piece $\Sigma(a)$, we define $\Sigma(a,+)$ as follows.
(We use the $\tau'$ coordinate.)
\begin{enumerate}
\item[(a)]
If $\Sigma(a) = K_1$ then $\Sigma(a,+) = K_{1} \cup [0,1]_{\tau'}\times [0,1]$.
\item[(b)]
If $\Sigma(a) = [k-1,k+1]_{\tau'} \times [0,1]$ then $\Sigma(a,+) = [k-2,k+2] _{\tau'}\times [0,1]$.
\end{enumerate}
{We recall from the definition (\ref{formnew653}) and \eqref{DkapparhoT}
that $D_{(0),\rho,T}$, $\overline D_{(0),\rho,T}$ is independent of $T$ for any $\rho
\in \overline{V_1(\epsilon_2) \times_L V_2(\epsilon_2)}$. So we denote the common operators by 
$D_{\rho}$, $\overline D_{\rho}$.}

We first consider the case when $a$ is as in (a) above.
Then from (\ref{1124-1}),(\ref{1124}) we derive
\begin{equation}\label{1123120-1}
\aligned
&\left\|\nabla_{\rho}^n \frac{\partial^{\ell}}{\partial T^{\ell}}  {(D_{(\kappa),\rho,T}- D_{\rho})}(V)  \right\|_{L^2_{m-\ell}
(\Sigma(a))} \\
&\le C_{m,(\ref{1123120-1})}  e^{-\delta_1 T}
\|  s_1\|_{L^2_{m+1}(\Sigma(a,+))}.
\endaligned
\end{equation}
(Here we use the fact that 
$V$ is  independent of $T,\rho$.)
\par
We next consider the case when $a$ is as in (b),
that is, $\Sigma(a) = [k-1,k+1]_{\tau'} \times [0,1]$.
Then
\begin{eqnarray}
&&\left\|\nabla_{\rho}^n \frac{\partial^{\ell}}{\partial T^{\ell}}  {(D_{(\kappa),\rho,T}- D_{\rho})}(V)  \right\|_{L^2_{m-\ell}
(\Sigma(a))} \nonumber\\
&\le&
\left\| \nabla_{\rho}^n \frac{\partial^{\ell}}{\partial T^{\ell}} {(D_{(\kappa),\rho,T}- D_{\rho})}(s_1 - v_1^{\rm pal})
\right\|_{L^2_{m-\ell}(\Sigma({a}))}  \nonumber\\
&&\quad+
\left\| \nabla_{\rho}^n \frac{\partial^{\ell}}{\partial T^{\ell}} {(D_{(\kappa),\rho,T}- D_{\rho})}(v_1^{\rm pal})
\right\|_{L^2_{m-\ell}(\Sigma({a}))}
\label{11231-+}\\
&
\le& C_{m,(\ref{1123120-+})}e^{-\delta_1 T}\Vert s_1 -  v_1^{\rm pal} \Vert_{L^2_{m+1}(\Sigma(a,+))}
+
C_{m,(\ref{1123120-+})}e^{-\delta_1 T} \Vert v_1 \Vert.
\label{1123120-+}\end{eqnarray}
Here we use (\ref{form182})-(\ref{form186})
 to {derive the  inequality (\ref{1123120-+})}.
\par
Moreover for those $\tau'$ with $\Sigma(a,+) \cap [0,7T+1]_{\tau'} \times [0,1] = \emptyset$ we may
improve
(\ref{1123120-+}) to be
$\le C_{m,(\ref{1123120-+})}e^{-\delta_1}\Vert s_1 -  v_1^{\rm pal} \Vert_{L^2_{m}(\Sigma(a,+))}$.
(Namely the second term of (\ref{1123120-+}) drops out.)
This is because $\hat u^{\rho}_{1,T,(\kappa)}$ is then constant on $\Sigma(a,+)$ and
therefore $D_{(\kappa),\rho,T}(v_1^{\rm pal}) = 0$ thereon.
\par
The norm
$$
\sum_a
e_{1,\delta}(p(a))^2
\left\Vert  Y \right\Vert_{L^2_{m+1}(\Sigma(a))}^2,
$$
is equivalent to the $L^2_{m+1,\delta}$ norm $\| Y\|_{L^2_{m+1,\delta}(\Sigma_1)}^2$.
(Here $e_{1,\delta}$ is the weight function in (\ref{e1delta}) and
$p(a) \in \Sigma(a)$ is any chosen point for each $a$.)
\par
Moreover the norm
$$
\sum_{a: \text{Case (a)}} \left\Vert  s_1 \right\Vert_{L^2_{m+1}(\Sigma(a))}^2
+
\sum _{a: \text{Case (b)}}e_{1,\delta}(p(a))^2  \left\Vert  s_1 - v_1^{\rm Pal} \right\Vert_{L^2_{m+1}(\Sigma(a))}^2
+ \Vert v_1\Vert^2,
$$
is equivalent to the $W^2_{m+1,\delta}$ norm $\Vert (s_1,v_1)\Vert^2_{W^2_{m+1,\delta}(\Sigma_1)}$.
\par
Therefore
taking the weighted sum
of the square of  the inequalities (\ref{1123120-1}), (\ref{1123120-+}) with weight $e_{T,\delta}(p(a))^2$ and using the above mentioned
equivalence of norms,
we derive the estimate
\begin{equation}\label{112312}
\aligned
& \left\|\nabla_{\rho}^n \frac{\partial^{\ell}}{\partial T^{\ell}}
{(D_{(\kappa),\rho,T}- D_{\rho})}(V)  \right\|^2_{L^2_{m-\ell,\delta}
(\Sigma_1)}\\
&\le C_{m,(\ref{112312})}  e^{-2\delta_1 T}
\sum_{a: \text{Case (a)}}\|  s_1\|^2_{L^2_{m+1}(\Sigma(a,+))} \\
&\quad
+C'_{m,(\ref{112312})}\sum _{a: \text{Case (b)}}e^{-2\delta_1 T}e_{1,\delta}(p(a))^2  \left\Vert  s_1 - v_1^{\rm Pal} \right\Vert_{L^2_{m+1}(\Sigma(a,+))}^2
\\
&\quad
+ T e^{-2\delta_1T} e^{14 \delta T}C''_{m,(\ref{112312})}  \Vert v_1\Vert^2.
\endaligned
\end{equation}
Here we use the fact that if $\Sigma(a,+) \cap [0,7T+1]_{\tau'} \times [0,1] \ne \emptyset$ then
$e_{1,\delta}(p(a)) \le C e^{7 \delta T}$.
Using $\delta < \delta_1/10$ from (\ref{form310310}), (\ref{112312}) implies
\begin{equation}\label{new651}
\aligned
&\left\|\nabla_{\rho}^n \frac{\partial^{\ell}}{\partial T^{\ell}}  {(D_{(\kappa),\rho,T}- D_{\rho})}(V)  \right\|_{L^2_{m-\ell,\delta}
(\Sigma_1)}\\
&\le
C_{m,(\ref{new651})}
e^{-\delta_1 T/10}\Vert (s_1,v_1)\Vert_{W^2_{m+1,\delta}(\Sigma_1)}.
\endaligned
\end{equation}
This inequality and the same inequality for $\Sigma_2$ together with
(\ref{formnew653}) and (\ref{estimateforBBB}) imply  (\ref{11231}).
\end{proof}

Since the operator $\overline D_{(\kappa),\rho,T}$  is invertible, we can expand the operator
$$
\overline D_{(\kappa),\rho,T}^{-1} :  \mathcal E_{1}^{\perp}\oplus \mathcal E_{2}^{\perp}
\to \frak H(\mathcal E_1,\mathcal E_2)
$$
into
\begin{equation}\label{formnew666}
\aligned
\overline D_{(\kappa),\rho,T}^{-1}& =
\left( ( 1 + (\overline D_{(\kappa),\rho,T} -  \overline D_{{\rho}})\overline D_{{\rho}}^{-1})\overline D_{{\rho}}\right)^{-1} \\
& = \overline D_{{\rho}}^{-1}\sum_{k=0}^{\infty} (-1)^k ((\overline D_{(\kappa),\rho,T} -  \overline D_{{\rho}})\overline D_{{\rho}}^{-1})^k.
\endaligned
\end{equation}
{It follows from \eqref{form626}, \eqref{estimateDover} that} the right hand side converges
{as an operator from $(\mathcal E_{1}^{\perp}\oplus \mathcal E_{2}^{\perp})
\cap L^2_{m,\delta}$ to $W^2_{m+1,\delta}$}  as far as $\rho$ is in a sufficiently small neighborhood
$V(\rho'_0)$ of $\rho'_0$ and $T >T_{m,(\ref{formnew666})}$
\par
Therefore by differentiating this by $T$ using the Leibnitz rule, we derive
\beastar
\frac{\partial}{\partial T} \overline D_{(\kappa),\rho,T}^{-1}
& = & \overline D_{{\rho}}^{-1}\sum_{k_1=0}^{\infty}\sum_{k_2=0}^{\infty} (-1)^{k_1+k_2+1}
\left((\overline D_{(\kappa),\rho,T} - \overline D_{{\rho}})\overline D_{{\rho}}^{-1}\right)^{k_1} \\
&{}& \quad \times
\left(\frac{\partial}{\partial T} \overline D_{(\kappa),\rho,T}\overline D_{{\rho}}^{-1}\right)
\left((\overline D_{(\kappa),\rho,T} -  \overline D_{{\rho}})\overline D_{{\rho}}^{-1}\right)^{k_2}.
\eeastar
Along the way, we lose {one degree of differentiability}.
Here we note that the operator $(\overline D_{(\kappa),\rho,T} -  \overline D_{{\rho}})\overline D_{{\rho}}^{-1}$
is uniformly bounded as $\kappa,\rho,T$ vary with $\rho \in V(\rho_0')$,
$T >T_{m,(\ref{formnew666})}$ . In the same way, we can differentiate
with respect to $\rho$ without loss of derivative.

In a similar way we  can derive
\begin{equation}\label{derivativeDest}
\left\| \nabla_{\rho}^n \frac{\partial^{\ell}}{\partial T^{\ell}}
\overline D_{(\kappa),\rho,T}^{-1}(W)\right\|_{W^2_{m+1 - \ell,\delta}}
\leq C_{m,(\ref{derivativeDest})}e^{-\delta_1 T/10} \|W\|_{L^2_{m,\delta}}
\end{equation}
for $\ell > 0$ and $0\le \ell, n \le m-2$. (Here we assume $W$ is $T,\rho$ independent.)
Note (\ref{derivativeDest}) holds at $\rho \in V(\rho'_0)$.
Covering the compact set $\overline{V_1(\epsilon_2) \times_L V_2(\epsilon_2)}$
by finitely many such open sets $V(\rho'_0)$,
the inequality (\ref{derivativeDest}) holds at any $\rho \in V_1(\epsilon_2) \times_L V_2(\epsilon_2)$, 
by increasing the constant $C_{m,(\ref{derivativeDest})}$ if necessary.
\par
By definition we have
$$
\aligned
&(V^{\rho}_{T,1,(\kappa+1)},V^{\rho}_{T,2,(\kappa+1)},\Delta p^{\rho}_{T,(\kappa+1)})
\\
&=
(I^0_{(\kappa),\rho,T}\circ \overline D_{(\kappa),\rho,T}^{-1}
\circ (I^1_{(\kappa),\rho,T})^{-1})
({\rm Err}^{\rho}_{1,T,(\kappa)},{\rm Err}^{\rho}_{2,T,(\kappa)}).
\endaligned
$$
We remark that
$(I^1_{(\kappa),\rho,T})^{-1}
({\rm Err}^{\rho}_{1,T,(\kappa)},{\rm Err}^{\rho}_{2,T,(\kappa)})$
is estimated by (\ref{form184})
since the isomorphism  (\ref{form615}) is nothing but $(I^1_{i,(\kappa),\rho,T})^{-1}$.
\par
Since $\Psi_{i;(\kappa-1)}(\rho,T)$ in (\ref{newform612}) is nothing but
$(I^0_{i,(\kappa),\rho,T})^{-1}$,
the estimate (\ref{form182}) is one for
$(I^0_{(\kappa),\rho,T})^{-1}
(V^{\rho}_{T,1,(\kappa+1)},V^{\rho}_{T,2,(\kappa+1)},\Delta p^{\rho}_{T,(\kappa+1)})$.
\par
Therefore,
using Lemma \ref{lem616}, we derive (\ref{form182})  for $\kappa +1$,
by choosing $T_{2,m,\epsilon(6)}$ such that
$C_{m,(\ref{derivativeDest})}e^{-\delta_1 T_{2,m,\epsilon(6)}/10}
\le C_{5,m}\mu$.
\par
The proof of Proposition \ref{prop:inequalitieskappa} (2) is complete.
Thus the proof of Theorem \ref{exdecayT} is now complete.
\qed
\section{Surjectivity and injectivity of the gluing map}
\label{surjinj}

In this chapter we prove surjectivity and injectivity of the map $\text{\rm Glue}_T$ in Theorem \ref{gluethm1} and
complete the proof of Theorem \ref{gluethm1}.\footnote{
Here surjectivity means the second half of the statement of Theorem \ref{gluethm1}, that is
`The image contains $\mathcal M^{\mathcal E_1\oplus \mathcal E_2}((\Sigma_T,\vec z);u_1,u_2)_{\epsilon_{\epsilon(1),\epsilon(2),\langle\!\langle 
{\rm \ref{gluethm1}}\rangle\!\rangle}}$.'}
The proof goes along the line of \cite{Don83}. (See also \cite{freedUhlen}.)
The surjectivity proof along the line of this chapter is written in  \cite[Section 14]{FOn} and
injectivity is proved in the same way.
(\cite[Section 14]{FOn} studies the case of pseudoholomorphic curve without boundary.
It however can be easily adapted to the bordered case as we mentioned in
\cite[page 417 lines 21-26]{fooo:book1}.)
Here we explain the argument in our situation in more detail.
\par
We begin with the following a priori estimate.
\begin{prop}\label{neckaprioridecay}
There exist $\epsilon_5, C_{m,(\ref{form71})}$ such that
if $u : (\Sigma_T,\partial\Sigma_T) \to (X,L)$ is an element of
$\mathcal M^{\mathcal E_1\oplus \mathcal E_2}((\Sigma_T,\vec z);\beta)_{\epsilon}$
for $0 <\epsilon<\epsilon_5$ then we have
\begin{equation}\label{form71}
\left\| \frac{\partial u}{\partial \tau}\right\|_{C^m([\tau-1,\tau+1] \times [0,1])}
\le C_{m,(\ref{form71})} e^{-\delta_1 (5T - \vert \tau\vert)}.
\end{equation}
\end{prop}
{\begin{proof} We first recall that elements in $\mathcal E_i$ are supported in $K_i$
respectively. Therefore $u \in \mathcal M^{\mathcal E_1\oplus \mathcal E_2}((\Sigma_T,\vec z);\beta)_{\epsilon}$
satisifies $\overline\partial u = 0$ on the neck region $[-5T,5T] \times [0,1]$.
Then \eqref{neckaprioridecay} follows from Lemma \ref{lem:Ckexpdecay}.
\end{proof}
}
We also have the following:
\begin{lem}\label{gluedissmoothindex}
$\mathcal M^{\mathcal E_1\oplus \mathcal E_2}((\Sigma_T,\vec z);\beta)_{\epsilon}$
is a smooth manifold of dimension $\dim V_1 + \dim V_2 - \dim L$.
\end{lem}
\begin{proof}
Assumption \ref{DuimodEi}  and the Mayer-Vietoris principle (\cite{Mr}) imply that the linearized operator:
$$
W^2_{m+1,\delta}((\Sigma_T,\partial \Sigma_T);u^*TX,u^*TL)
\to L^2_{m,\delta}(\Sigma_T;u^*TX\otimes \Lambda^{0,1})/(\mathcal E_1(u)\oplus \mathcal E_2(u))
$$
of the equation (\ref{mainequation}), which is
defined by
$$
V \mapsto (D_u\overline\partial)(V) - (D_u\mathcal E_1)(\frak e_1,V) - (D_u\mathcal E_2)
(\frak e_2,V)
\mod \mathcal E_1(u)\oplus \mathcal E_2(u)
$$
is surjective. (See \cite[Proposition 7.1.27]{fooo:book1}.)
Here $\overline{\partial} u = (\frak e_1,\frak e_2) \in \mathcal E_1(u) \oplus
\mathcal E_2(u)$.
The lemma then follows from the implicit function theorem to solve the equation
$\delbar u'  \in \mathcal E_1(u) \oplus \mathcal E_2(u)$, for $u'$.
\end{proof}
\par
\medskip
\begin{proof}[Proof of surjectivity]
Let $u :  (\Sigma_T,\partial\Sigma_T) \to (X,L)$ be an element of the moduli space
$\mathcal M^{\mathcal E_1\oplus \mathcal E_2}((\Sigma_T,\vec z);u_1,u_2)_{\epsilon}$.
The purpose here is to show that $u$ {lies} in the image of $\text{\rm Glue}_T$.
{For this purpose, we first decompose $u$ into two pieces}
by defining {two maps} $u_i' : (\Sigma_i,\partial\Sigma_i) \to (X,L)$ as follows.
We put $p^u_0 = u(0,0) \in L$.
\begin{equation}
\aligned
&u'_{1}(z) \\
&=
\begin{cases} \Exp\left(p_0^u, \chi_{\mathcal B}^{\leftarrow}(\tau-T,t) E(p^u_0, u(\tau,t))\right)
&\text{if $z = (\tau,t) \in [-5T,5T] \times [0,1]$} \\
 u(z)
&\text{if $z \in K_1$} \\
p_0^u
&\text{if $z \in [5T,\infty)\times [0,1]$}.
\end{cases} \\
&u'_2(z) \\
&=
\begin{cases} \Exp\left(p_0^u, \chi_{\mathcal A}^{\rightarrow}(\tau+ T,t) E(p^u_0, u(\tau,t))\right)
&\text{if $z = (\tau,t) \in [-5T,5T] \times [0,1]$} \\
 u(z)
&\text{if $z \in K_2$} \\
p_0^u
&\text{if $z \in (-\infty,-5T]\times [0,1]$}.
\end{cases}
\endaligned
\end{equation}
Proposition \ref{neckaprioridecay} implies
\begin{equation}\label{form73}
\|
\Pi^{\perp}_{\mathcal E_i(u'_i)}\overline\partial u'_{i}
\|_{L^2_{m,\delta}(\Sigma_i)}
\le C_{m,(\ref{form73})} e^{-\delta_1 T}.
\end{equation}
On the other hand, by assumption and elliptic regularity we have
\begin{equation}\label{form74}
\|
u'_i - u_i
\|_{W^2_{m+1,\delta}(\Sigma_i)}
\le C_{m,(\ref{form74})} \epsilon.
\end{equation}
{Here we abuse the notation} $u'_i - u_i = {\rm E}(u_i,u'_i)$ as
we mentioned in Remark \ref{rem2222}.
\begin{lem}
{
There exist $\epsilon_{m,(\ref{1121})}$
and $C_{m,(\ref{1121})}$ with the following properties:
If $u \in \mathcal M^{\mathcal E_1\oplus \mathcal E_2}((\Sigma_T,\vec z);u_1,u_2)_{\epsilon}$
with $\epsilon < \epsilon_{m,(\ref{1121})}$
then there exist $\rho_i \in V_i$ ($i=1,2$) such that}
{
\begin{equation}\label{1121}
\|
{\rm E}(u^{\rho_i}_i,u'_i)\|_{W^2_{m+1,\delta}(\Sigma_i)}
\le C_{m,(\ref{1121})} e^{-\delta_1 T},
\end{equation}
where $u_i^{\rho_i}$ is the map associated to $\rho_i$ in $W^2_{m+1,\delta}((\Sigma_i,\del \Sigma_i),(X,L))$.
}
\end{lem}
\begin{proof}
{We consider the set of triples $(v_1,v_2,q)$ where
$v_i : (\Sigma_i,\partial\Sigma_i) \to (X,L)$ 
are maps of $L^2_{m+1,loc}$ class and 
$q \in L$ such that the section
${\rm E}(q,v_i)$ of $T_qX$ has finite $L^2_{m+1,\delta}$ norm 
on $[0,\infty)_{\tau'} \times [0,1]$ ($i=1$),
$(-\infty,0]_{\tau''} \times [0,1]$ ($i=2$).
}
{
The space of such $(v_1,v_2,q)$ becomes a Hilbert manifold, 
which we denote by $\mathscr L$.
}

{
Suppose $d(v_i,u_i) \le \iota'_X$ pointwise.
We define
$$
G_i(v_1,v_2,q)
= 
\mathcal P (\overline\partial v_i) 
\in 
L^2_{m,\delta}(\Sigma_i,u_i^*TX \otimes \Lambda^{0,1}),
$$
where 
$$
\mathcal P : 
L^2_{m,\delta}(\Sigma_i,v_i^*TX \otimes \Lambda^{0,1})
\to
L^2_{m,\delta}(\Sigma_i,u_i^*TX \otimes \Lambda^{0,1})
$$
is defined by using parallel transport along the minimal geodesic. Then
$$
(G_1,G_2) : \mathscr U \to 
\bigoplus_{i=1}^2 L^2_{m,\delta}(\Sigma_i,u_i^*TX \otimes \Lambda^{0,1})
$$
is a smooth map. Here $\mathscr U$ is a small neighborhood of $(u_1,u_2,p_0)$
in $\mathscr L$.}
\par
{
We also define a smooth map \index[syindex]{Eiscr@$\mathscr E_i$}
$$
\mathscr E_i : \mathscr U \times \mathcal E_i^{\frak{ob}} 
\to  L^2_{m,\delta}(\Sigma_i,u_i^*TX \otimes \Lambda^{0,1})
$$
by composing the two maps induced by the parallel transports
$$
\aligned
&L^2_{m,\delta}(\Sigma_i,(u_i^{\frak{ob}})^*TX \otimes \Lambda^{0,1})
\to 
L^2_{m,\delta}(\Sigma_i,v_i^*TX \otimes \Lambda^{0,1}),
\\
&
L^2_{m,\delta}(\Sigma_i,v_i^*TX \otimes \Lambda^{0,1})
\to 
L^2_{m,\delta}(\Sigma_i,u_i^*TX \otimes \Lambda^{0,1}).
\endaligned
$$
(The map $\mathscr E_i$ is linear in the second factor.)
By Assumption \ref{DuimodEi}
$$
(G_1,G_2) + (\mathscr E_1,\mathscr E_2)
:\mathscr U  \times \mathcal E_1^{\frak{ob}} \times \mathcal E_2^{\frak{ob}}
\to 
\bigoplus_{i=1}^2 L^2_{m,\delta}(\Sigma_i,u_i^*TX \otimes \Lambda^{0,1})
$$ 
is transversal to zero.
Moreover restriction of the projection 
$\mathscr U  \times \mathcal E_1^{\frak{ob}} \times \mathcal E_2^{\frak{ob}}
\to \mathscr U$  
induces an isomorphism of
$$
((G_1,G_2) + (\mathscr E_1,\mathscr E_2))^{-1}(0)
$$
onto  an open neighborhood of the origin in $V_1
\times_L  V_2$. (Here the origin is the point corresponding to
$(u_1,u_2,p_0)$.)}
{
Finally we note that (\ref{form74}) implies $(u'_1,u'_2,p_0^u) \in \mathscr U$
and (\ref{form73}) implies
$$
\Vert (\Pi^{\mathcal E_1^\perp\oplus \mathcal E_2^\perp}) (G(u'_1,u'_2,p_0^u))\Vert_{L^2_{m,\delta}}
\le C e^{-\delta_1 T}.
$$
Therefore there exists $\frak e'_i \in \mathcal E_i^{\frak{ob}}$ such that
$$
\Vert (G_1,G_2)(u'_1,u'_2,p_0^u)
+ (\mathscr E_1,\mathscr E_2)(\frak e'_1,\frak e'_2) \Vert
\le C e^{-\delta_1 T}.
$$
Now the lemma immediately follows from the implicit function theorem
applied to the equation 
$$ (G_1,G_2)(v_1,v_2,q)
+ (\mathscr E_1,\mathscr E_2)(\frak e_1,\frak e_2) = 0.
$$
}
\end{proof}
We put
$\rho = (\rho_1,\rho_2) \in V_1 \times_L V_2$.
Then $u \in \mathcal M^{\mathcal E_1\oplus \mathcal E_2}((\Sigma_T,\vec z);u_1,u_2)_{\epsilon}$
also implies
\begin{equation}\label{form7676}
\vert\rho_i\vert \le {C_{m,(\ref{form7676})}}\epsilon.
\end{equation}
\par
By (\ref{1121}) we have
\begin{equation}\label{1123}
\|
{{\rm E}(u^\rho_T,u)}
\|_{W^2_{m+1,\delta}(\Sigma_T)}
\le C_{m,(\ref{1123})} e^{-\delta_1 T}.
\end{equation}
Here $u^{\rho}_T =\text{\rm Glue}_T(\rho)$.
\par
We put $V =
{\rm E}(u^{\rho}_T,u) \in \Gamma((\Sigma_T,\partial \Sigma_T);(u^{\rho}_T)^*TX;(u^{\rho}_T)^*TL)$.
Then
$$
u(z) =\Exp (u^{\rho}_T(z),V(z)).
$$
For $s \in [0,1]$
we define $u^s : (\Sigma_T,\partial \Sigma_T) \to (X,L)$ by
\begin{equation}\label{form7878}
u^s(z) =\Exp (u^{\rho}_T(z),sV(z)).
\end{equation}
(\ref{1123}) implies
\begin{equation}\label{form79}
\|
\Pi^{\perp}_{(\mathcal E_1\oplus \mathcal E_2)(u^s)}\overline\partial u^s
\|_{L^2_{m,\delta}(\Sigma_T)}
\le C_{m,(\ref{form79})} e^{-\delta_1 T}
\end{equation}
and
\begin{equation}\label{form792}
\left\|
\frac{\partial}{\partial s} u^s
\right\|_{W^2_{m+1,\delta}(K_i^S)}
\le C_{m,(\ref{form792})} e^{-\delta_1 T}
\end{equation}
for each $s \in [0,1]$.
\begin{lem}\label{lem7474}
If $T$ is sufficiently large, then there exists $\tilde u^s :  (\Sigma_T,\partial \Sigma_T) \to (X,L)$ $(s \in [0,1])$
with the following properties.
\begin{enumerate}
\item
$$
\overline\partial \tilde u^s \equiv 0 \mod (\mathcal E_1\oplus \mathcal E_2)(\tilde u^s).
$$
\item
\begin{equation}\label{sderovatove}
\left\|
\frac{\partial}{\partial s} \tilde u^s
\right\|_{W^2_{m+1,\delta}(K_i^S)}
\le C_{m,S,(\ref{sderovatove})} e^{-\delta_1 T}.
\end{equation}
\item $\tilde u^s = u^s$ for $s=0,1$.
\end{enumerate}
\end{lem}
\begin{proof}
Run the alternating method described in Chapter \ref{alternatingmethod}
in the one-parameter family version.
Since $u^s$ is already a solution for $s=0,1$, {
the Newton's iteration in Chapter \ref{alternatingmethod} 
does not change it}.
More precisely, we regard $u^s$ in
(\ref{form7878}) as
$u^{\rho}_{T,(0)} = u^{s}_{T,(0)}$ and start
our inductive construction at Step 0-3 (Lemma \ref{lem18}).
\par

Then for $s = 0,1$,
$$
{\rm Err}^{s}_{1,T,(0)}
= \chi_{\mathcal X}^{\leftarrow} (\overline\partial u^{s}_{T,(0)} - \frak e^{s}_{1,T,(0)})
= 0.
$$
We can show ${\rm Err}^{s}_{2,T,(0)} = 0$ in the same way.
(Here ${\rm Err}^{s}_{i,T,(0)}$ is as in Definition \ref{deferfirst}.)
Hence $V^{s}_{T,1,(1)} = V^{s}_{T,2,(1)} = 0$, $\Delta p^{\rho}_{T,(1)} = 0$ for $s=0,1$ by
(\ref{144ffff}).
Therefore
$u^{s}_{T,(1)} = u^{s}_{T,(0)}$ by Definition \ref{defn518}.
Thus $u^{s}_{T,(\kappa)} = u^{s}_{T,(0)} = u^s$ for all $\kappa$ and $s=0,1$.
\par
Then $\tilde u^s = \lim_{\kappa}u^{s}_{T,(\kappa)} = u^s$ for $s=0,1$ follows.
\end{proof}
\begin{lem}\label{immersion}
The map
$\text{\rm Glue}_T : V_1(\epsilon) \times_L V_2(\epsilon) \to \mathcal M^{\mathcal E_1\oplus\mathcal E_2}((\Sigma_T,\vec z);
u_1,u_2)_{\epsilon}$
is an immersion if  $\epsilon < \epsilon_2$ and  $T$ is sufficiently large.
\end{lem}
\begin{proof}
We consider the composition of $\text{\rm Glue}_T$ with
$$
\mathcal M^{\mathcal E_1\oplus\mathcal E_2}((\Sigma_T,\vec z);u_1,u_2)_{\epsilon}
\to  {L^2_{m+1}}((K_i^S,K_i^S\cap\partial \Sigma_i),(X,L))
$$
defined by restriction.
In the case $T = \infty$ this composition is obtained by restriction of maps.
By the unique continuation, this is certainly an immersion for $T=\infty$.
Then the exponential decay property stated in Theorem \ref{exdecayT} implies
that it is an immersion for sufficiently large $T$.
\end{proof}
{We are now ready to prove $u$ lies}  in the image of ${\rm Glue}_T$ by showing 
that the following set
$$
A = \{s \in [0,1] \mid \tilde u^{r} \in {\text{image of $\text{\rm Glue}_T$ 
for $r \le s$}}\}
$$
is nonempty, open and closed in $[0,1]$.
Obviously $0 \in A$ since $\widetilde u^0 = u_T^\rho$ {by Lemma \ref{lem7474} (3).} 
Lemma \ref{gluedissmoothindex} implies that $\mathcal M^{\mathcal E_1\oplus\mathcal E_2}((\Sigma_T,\vec z);
u_1,u_2)_{\epsilon}$
is a smooth manifold and has the same dimension as $V_1 \times_L V_2$.
Therefore Lemma \ref{immersion} implies that $A$ is open.
The closedness of $A$ follows from 
{Lemma \ref{immersion}. In fact
we may shrink the domain of ${\rm Glue}_T$ a bit 
if necessary so that it is not only an immersion but also 
the first derivative of its local inverse is uniformly bounded.
Then if $s$ is in the closure of $A$, the path $r \mapsto u_r^\rho$
$(r \in [0,s))$ has a lift  to the domain of ${\rm Glue}_T$ 
with length $\le C e^{-\delta_1 T}$. 
(We use Lemma \ref{lem7474} (2) here.) Therefore the closure of the lift 
is contained in the domain of ${\rm Glue}_T$,
when we take the constant $\epsilon_{\epsilon(1),\epsilon(2),\langle\!\langle
{\rm \ref{gluethm1}}\rangle\!\rangle}$ in Theorem \ref{gluethm1} (2)
sufficiently small.
It implies $s \in A$}.
\par
Therefore $1\in A$. Namely $u$ {lies} in the image of $\text{\rm Glue}_T$ as required.
\end{proof}
\begin{proof}[Proof of injectivity]
Let $\rho^j = (\rho_1^j,\rho_2^j) \in V_1\times_L V_2$ for $j=0,1$. We assume
\begin{equation}
\text{\rm Glue}_T(\rho^0) = \text{\rm Glue}_T(\rho^1)
\end{equation}
and
\begin{equation}
\| \rho_i^j\| < \epsilon.
\end{equation}
We will prove that $\rho^0 = \rho^1$ if $T$ is sufficiently large and $\epsilon$ is
sufficiently small.
We may assume that $V_1\times_L V_2$ is connected and simply connected, {by
taking $\epsilon$  sufficiently small.}

Then, we {can take} a path $s \mapsto \rho^s =(\rho^s_1,\rho^s_2) \in V_1 \times_L V_2$
such that
\begin{enumerate}
\item $\rho^s = \rho^j$ for $s=j$, $j=0,1$.
\item
$$
\left\| \frac{\partial}{\partial s} \rho^s \right\| \le \Phi_1(\epsilon)
+ \Psi_1(T)
$$
where $\lim_{\epsilon \to 0}\Phi_1(\epsilon) = 0$,
$\lim_{T \to \infty}\Psi_1(T) = 0$.
\end{enumerate}
We put $V(s)
= {\rm E}(u^{\rho^0}_T,u^{\rho^s}_T) \in  \Gamma((\Sigma_T,\partial \Sigma_T);(u^{\rho^0}_T)^*TX;(u^{\rho^0}_T)^*TL)$.
Then
$$
u^{\rho^s}_T(z) =\Exp (u^{\rho^0}_T(z),V(s)(z)).
$$
(By (2) $u^{\rho^s}_T(z)$ is $C^0$-close to $u^{\rho^0}_T(z)$, as $\epsilon \to 0$. Therefore
$V(s)$ is well defined if $\epsilon$ is small.)
Note $V(1) = V(0)$ since $u^{\rho^1} = u^{\rho^0}$.
Then for $w \in D^2 = \{w \in \C \mid \vert w\vert \le 1\}$, there exists $V(w)$ such that:
\begin{enumerate}
\item
$V(s) = V(w)$ if $w = e^{2\pi\sqrt{-1}s}$.
\item
We put $w = x + \sqrt{-1}y$.
\begin{equation}
\left\|
\frac{\partial}{\partial x} V(w)
\right\| _{W^2_{m+1,\delta}(\Sigma_T)}
+
\left\|
\frac{\partial}{\partial y} V(w)
\right\| _{W^2_{m+1,\delta}(\Sigma_T)}
\le \Phi_2(\epsilon) + \Psi_2(T)
\end{equation}
where $\lim_{\epsilon \to 0}\Phi_2(\epsilon) = 0$,
$\lim_{T \to \infty}\Psi_2(T) = 0$.
\end{enumerate}
We put $u^w(z) = \Exp (u^{\rho^0}_T(z),V(w)(z))$.
\begin{lem}
If $T$ is sufficiently large and $\epsilon$ is sufficiently small
then there exists $\tilde u^w :  (\Sigma_T,\partial \Sigma_T) \to (X,L)$ $(s \in [0,1])$
with the following properties.
\begin{enumerate}
\item
$$
\overline\partial \tilde u^w \equiv 0 \mod (\mathcal E_1\oplus\mathcal E_2)(\tilde u^w).
$$
\item
\begin{equation}\label{sderovatove131}
\left\|
\frac{\partial}{\partial x} \tilde u^w
\right\|_{W^2_{m+1,\delta}(K_i^S)}
+
\left\|
\frac{\partial}{\partial y} \tilde u^w
\right\|_{W^2_{m+1,\delta}(K_i^S)}
\le \Phi_3(\epsilon) + \Psi_3(T)
\end{equation}
with $\lim_{\epsilon \to 0}\Phi_3(\epsilon) = 0$,
$\lim_{T \to \infty}\Psi_3(T) = 0$.
\item $\tilde u^w = u^w$ for $w\in \partial D^2$.
\end{enumerate}
\end{lem}
\begin{proof}
Run the alternating method described in Chapter \ref{alternatingmethod}
in the two-parameter family version.
(3) is proved in the same way as Lemma \ref{lem7474}.
\end{proof}
\begin{lem}\label{liftinglemma}
If $T$ is sufficiently large and $\epsilon$ is sufficiently small, there exists
a smooth map $F : D^2 \to V_1\times_L V_2$ such that
\begin{enumerate}
\item
$\text{\rm Glue}_T(F(w)) = \tilde u^{w}$.
\item If $s \in [0,1]$ then we have:
$$
F(e^{2\pi\sqrt{-1}s}) =\rho^s.
$$
\end{enumerate}
\end{lem}
\begin{proof}
Note that $\rho \mapsto \text{\rm Glue}_T(\rho)$ is a local diffeomorphism.
So we can apply the proof of homotopy lifting property as follows.
Let $D^2(r)
= \{z \in \C \mid \vert z - (r-1) \vert \le r\}$.
We put
$$
A = \{r \in [0,1] \mid \text{$\exists$ $F : D^2(r) \to V_1\times_L V_2$
satisfying (1) above and $F(-1) = \rho^{1/2}$}\}.
$$
Since $\text{\rm Glue}_T(\rho)$ is a local diffeomorphism,
$A$ is open. We can use (\ref{sderovatove131}) to show
closedness of $A$.
Then, since $0\in A$, it follows that $1\in A$. The proof of
Lemma \ref{liftinglemma} is complete.
\end{proof}
Lemma \ref{liftinglemma} implies {$
\rho^0 = F(e^{2\pi\sqrt{-1}\cdot 0}) = F(e^{2\pi\sqrt{-1}\cdot 1}) = \rho^1
$}.
The proof of Theorem \ref{gluethm1} is now complete.
\end{proof}
\section[Smoothness of coordinate change]{Exponential decay estimate implies smoothness of coordinate change}
\label{sec:smoothness of coordinate change}

In this chapter we demonstrate the way  we use Theorems \ref{gluethm1} and \ref{exdecayT}
to prove smoothness of coordinate change of the Kuranishi structure
of the moduli space of bordered stable maps.
Here we provide an argument in the case when we glue two
source curves which are non-singular and stable.
The proof for the general case is given in
\cite[Part IV]{fooo:techI}.
(See especially its Section 21.)
\par

\subsection{Including deformation of source curve}
\label{subsec81}
We first generalize Theorems \ref{gluethm1} and \ref{exdecayT} and include 
deformations of complex structures of the source curves
$(\Sigma_i,\vec z_i)$.
We consider the situation of Theorem \ref{gluethm1}.
Note we assumed that each $(\Sigma_i,\vec z_i)$ is stable.
We also remark that there is no automorphism
of  $(\Sigma_i,\vec z_i)$. (See Remark \ref{remark3838}.)
\par
Let $g_i$ be the genus of $\Sigma_i$.
For simplicity of the notation we assume that
the boundary of $\Sigma_i$ is connected.
Let $k_i + 1$ be the number of boundary
marked points. (For the sake of simplicity of notations we consider the case
when we have only boundary marked points.
The case when we also have interior marked points can be studied in the same way.)
We denote by \index[syindex]{Mgki@$\mathcal M_{g_i;k_i+1}$}
\begin{equation}
\mathcal M_{g_i;k_i+1}
\end{equation}
the moduli space of bordered and marked
Riemann surfaces with the same topological type
as $(\Sigma_i,\vec z_i)$.
Let
$\mathcal{CM}_{g_i;k_i+1}$ be its\index[syindex]{CMgki@$\mathcal{CM}_{g_i;k_i+1}$}
compactification which consists of stable marked bordered curves.
(See \cite[Subsection 2.1.2]{fooo:book1} for its definition.)
\par
Let  $\mathcal V'_i$ ($i=1,2$) be a neighborhood of $[\Sigma_i,\vec z_i]$
in $\mathcal M_{g_i;k_i+1}$ \index[syindex]{Vprimeical@$\mathcal V'_i$} 
\index[syindex]{Ciprimecal@$\mathcal C'_i$} 
and
\begin{equation}\label{unifamiC}
\pi_i : \mathcal C'_i \to \mathcal V'_i,  \qquad  \frak{z}_j : \mathcal V'_i \to \mathcal C'_i,
\quad (j=0,\dots,k_i)
\end{equation}
be the universal family.
Namely $\mathcal C'_i$ has a fiberwise
complex structure
such that
$(\pi_i^{-1}(\sigma),(\frak{z}_0(\sigma),\dots,\frak{z}_{k_i}(\sigma)))$
is a representative of $\sigma \in \mathcal V'_i$.
\par
We may choose $\mathcal V'_i$ so small that the bundle (\ref{unifamiC}) is
topologically trivial.
We fix a trivialization
$
\mathcal C'_i \cong \vert\Sigma_i \vert\times \mathcal V'_i
$
such that $\pi_i$ is the  projection to the second factor and that
$$
\frak{z}_j :  \mathcal V'_i \to \mathcal C'_i \to \vert\Sigma_i \vert \times \mathcal V'_i \to \vert\Sigma_i \vert
$$
are constant maps.
(Here $\vert\Sigma_i \vert$ is the bordered curve $\Sigma_i$ with
its complex structure forgotten.
Hereafter we write $\Sigma_i$ in place of $\vert\Sigma_i \vert$ by a slight abuse of notation.)
We denote\index[syindex]{Sigmaisigma@$\Sigma_i(\sigma)$}
\index[syindex]{zveci@$\vec z_i(\sigma)$}\index[syindex]{zifrak@$\frak z_i(\sigma)$}
$$
\Sigma_i(\sigma)
 =\pi_i^{-1}(\sigma),
 \quad
 \vec z_i(\sigma)
 = (
 \frak z_0(\sigma),\dots, \frak z_{k_i}(\sigma)).
 $$
\par
We will next define $\Sigma_1(\sigma_1) \#_T \Sigma_2(\sigma_2)
= \Sigma_T^{\sigma}$, \index[syindex]{SigmaTsigma@$\Sigma_T^{\sigma}$} which is obtained from
$\Sigma_1(\sigma_1)$ and $\Sigma_2(\sigma_2)$ by gluing.
To specify the way to glue we need to fix families of coordinates at the $0$-th marked
points $z_{i,0} \in \Sigma_i$.
\par
For our purpose it is useful to take an analytic family of coordinates, which we define below.
(Definition \ref{def8585}.)
To define it we start with the analogue for a closed Riemann surface.
Let $\mathcal M^{\rm cl}_{g,\ell+1}$ \index[syindex]{Mclgell@$\mathcal M^{\rm cl}_{g,\ell+1}$} 
be the moduli space of closed Riemann surface of
genus $g$ with $\ell+1$ marked points and $\mathcal{CM}^{\rm cl}_{g,\ell+1}$ 
its compactification consisting of stable curves.
(See \cite{Delignemum}.)
Let \index[syindex]{CMclgell+1@$\mathcal{CM}^{\rm cl}_{g,\ell+1}$}
\begin{equation}
\pi : \mathcal C^{\rm cl}_{g,\ell+1} \to \mathcal{CM}^{\rm cl}_{g,\ell+1}
\end{equation}
be the universal family.\index[syindex]{Cclgell@$\mathcal C^{\rm cl}_{g,\ell+1}$}
Let $[\Sigma,\vec z] \in \mathcal{CM}^{\rm cl}_{g,\ell+1}$
and $\Gamma$ be the automorphism group of $[\Sigma,\vec z]$.
Then a neighborhood of   $[\Sigma,\vec z]$ in
$\mathcal{CM}^{\rm cl}_{g,\ell+1}$ may be regarded as $\mathcal V/\Gamma$, where
$\mathcal V$ is a complex manifold.
$\pi^{-1}(\mathcal V/\Gamma)$ is identified with $\mathcal C^{\rm cl}(\mathcal V)/\Gamma$
where $\mathcal C^{\rm cl}(\mathcal V)$ is a complex manifold and
$\pi_{\mathcal V} : \mathcal C^{\rm cl}(\mathcal V) \to \mathcal V$ is a $\Gamma$ equivariant holomorphic map.
It comes with
$\Gamma$ equivariant
sections $\frak{z}_j : \mathcal V \to \mathcal C^{\rm cl}(\mathcal V)$,
$j=0,1,\dots,\ell$ such that
$$
(\Sigma(\sigma),\vec z(\sigma)) =
(\pi^{-1}(\sigma),(\frak{z}_0(\sigma),\dots,\frak{z}_{\ell}(\sigma)))
$$
is a representative of $\sigma$.
We put $z_j(\sigma) = \frak{z}_{j}(\sigma)$.
\begin{defn}\label{defn81}
A {\em complex analytic family of coordinates} \index{complex analytic family of coordinates} at 
the $j$-th marked point
on $\mathcal V$ is
a map
$\widetilde{\varphi} : D^2 \times \mathcal V \to \mathcal C^{\rm cl}(\mathcal V)$
with the following properties.
\begin{enumerate}
\item
$\widetilde{\varphi}$ is a biholomorphic map onto its image, which is an open subset.
\item
$\pi \circ \widetilde{\varphi} :  D^2 \times \mathcal V \to \mathcal V$ coincides with the projection to the
second factor.
\item
$\widetilde{\varphi}(0,\sigma) = z_j(\sigma)$.
\item
We consider the $\Gamma$ action on $T_{z_j}\Sigma \cong\C$
and use it to define a $\Gamma$ action on $D^2$.
Then $\widetilde{\varphi}$ is $\Gamma$ invariant.
\end{enumerate}
\end{defn}
Definition \ref{defn81} implies that the restriction of $\tilde\varphi$ to $D^2 \times \{\sigma\}$
becomes a complex coordinate of $\Sigma(\sigma)$ at $
\frak z_j(\sigma)$.
\par
{Existence} of a complex analytic family of coordinates is a consequence of 
{the}
$\Gamma$-equivariant version of implicit function theorem in complex analytic category
and is {well-known.} (See Appendix \ref{appendixF}.)
\par
Let $\mathcal V_i$  \index[syindex]{Vical@$\mathcal V_i$} be neighborhoods of $[\Sigma^{\rm cl}_i,\vec z_i] \in \mathcal{M}^{\rm cl}_{g_i,\ell_i+1}$, 
for $i=1,2$.
We assume that $[\Sigma^{\rm cl}_1,\vec z_1] \ne [\Sigma^{\rm cl}_2,\vec z_2]$.\footnote{In case $[\Sigma^{\rm cl}_1,\vec z_1] = [\Sigma^{\rm cl}_2,\vec z_2]$
we may have extra $\Z_2$ symmetry. Other than that the argument is the same as 
{in} the case
$[\Sigma^{\rm cl}_1,\vec z_1] \ne [\Sigma^{\rm cl}_2,\vec z_2]$.}
We put $\Gamma_i = {\rm Aut}(\Sigma^{\rm cl}_i,\vec z_i)$.
Let
$\widetilde{\varphi}_i : D^2_i \times \mathcal V_i \to \mathcal C^{\rm cl}_i(\mathcal V_i)$
be two complex analytic families of coordinates at {the} $0$-th marked points on neighborhoods {$\mathcal V_i$} of
$[\Sigma^{\rm cl}_i,\vec z_i]$.
\par
We identify $z_{1,0} \in \Sigma^{\rm cl}_1$ with
$z_{2,0} \in \Sigma^{\rm cl}_2$. Then we obtain an element
$$
[\Sigma^{\rm cl}_{\infty},\vec z_{\infty}] \in \mathcal{CM}^{\rm cl}_{g_1+g_2,\ell_1+\ell_2}.
$$
Using our complex analytic families of coordinates we define a map
\begin{equation}
{\rm Glusoc} : \mathcal V_1 \times \mathcal V_2 \times D^2(\epsilon)
\to \mathcal{CM}^{\rm cl}_{g_1+g_2,\ell_1+\ell_2}
\end{equation}
as follows.
(Here ${\rm Glusoc}$ stands for ``gluing source''.)

Let $\sigma_i \in \mathcal V_i$ and $(\Sigma_i(\sigma_i),\vec z_i(\sigma_i))$ be a
marked Riemann surface representing it.
\par
We define {a map}
$
\varphi_{i,\sigma_i} : D^2 \to \Sigma_i(\sigma_i)
$
by
\begin{equation}
\varphi_{i,\sigma_i}(z) =  \widetilde \varphi_i(z,\sigma_i).
\end{equation}
{For each given} $\frak r \in D^2(\epsilon)$.
We consider the disjoint union
$$
(\Sigma_1(\sigma_1) \setminus \varphi_{1,\sigma_1}(D^2(
\vert \frak r\vert)))
\sqcup (\Sigma_2(\sigma_2) \setminus \varphi_{2,\sigma_2}(D^2(
\vert \frak r\vert))).
$$
(Here and hereafter $D^2(r) = \{  z\in \C \mid \vert z\vert < r\}$.)
 We identify
$\varphi_{1,\sigma_1}(z) \in \Sigma_1(\sigma_1) \setminus \varphi_{1,\sigma_1}(D^2(
\vert \frak r\vert))$
with $\varphi_{2,\sigma_2}(w) \in \Sigma_2(\sigma_2) \setminus
\varphi_{2,\sigma_2}(D^2(
\vert \frak r\vert))
$
if
$$
zw = \frak r.
$$
See Figure \ref{Fig81}.
\par
By this identification we obtain a Riemann surface, which we denote
by the symbol
$\Sigma_1(\sigma_1)\#_{\frak r} \Sigma_2(\sigma_2)$.
We put\index[syindex]{Glusoc@${\rm Glusoc}$}
\begin{equation}
{\rm Glusoc}(\sigma_1,\sigma_2,\frak r)
= (\Sigma_1(\sigma_1)\#_{\frak r} \Sigma_2(\sigma_2),\vec z'_1(\sigma_1) \cup \vec z'_2(\sigma_2)),
\end{equation}
where $\vec z'_i(\sigma_i) = \vec z_i(\sigma_i) \setminus \{z_{i,0}(\sigma_i)\}$.
\begin{figure}
\centering
\includegraphics{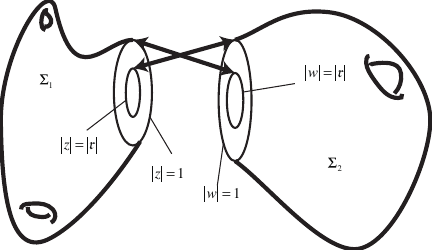}
\caption{Source curve gluing.}
\label{Fig81}
\end{figure}
\par
We define the $\Gamma_i$ action on $D^2(\epsilon)$ by identifying $D^2(\epsilon)$ with the ball of radius $\epsilon$
centered at origin in the
tangent space $T_{z_{i,0}}\Sigma_i$.
Then   ${\rm Glusoc}$ is $\Gamma_1 \times \Gamma_2$ equivariant.
\begin{lem}\label{lem820}
The map ${\rm Glusoc}$ induces a biholomorphic map
$$
\frac{\mathcal V_1 \times \mathcal V_2 \times D^2(\epsilon)}{\Gamma_1 \times \Gamma_2}
\to
\mathcal{CM}^{\rm cl}_{g_1+g_2,\ell_1+\ell_2}
$$
onto an open neighborhood of $[\Sigma^{\rm cl},\vec z]$.
\end{lem}
\begin{proof}
We define
\begin{equation}
\Omega_1 =
\mathcal V_2 \times \coprod_{\frak r \in D^2(\epsilon), \sigma_1 \in \mathcal V_1}
(\Sigma_1(\sigma_1) \setminus \varphi_{1,{\sigma_1}}(D^2(
\vert \frak r\vert))) \times \{\frak r\},
\end{equation}
\begin{equation}
\Omega_2 =
\mathcal V_1 \times \coprod_{\frak r \in D^2(\epsilon), \sigma_2 \in \mathcal V_2}
(\Sigma_2(\sigma_2) \setminus \varphi_{2,{\sigma_2}}(D^2(
\vert \frak r\vert))) \times \{\frak r\}.
\end{equation}
We regard $\Omega_1$ (resp. $\Omega_2$)
as an open subset of $\mathcal V_2 \times \mathcal C^{\rm cl}(\mathcal V_1)$
(resp. $\mathcal V_1 \times \mathcal C^{\rm cl}(\mathcal V_2)$).
So they are complex manifolds.
We also put
\begin{equation}
\Omega_3 = \{(z,w) \mid  \vert z\vert, \vert w\vert < 1,
\,\,\,  \vert zw\vert< \epsilon\}
\times \mathcal V_1 \times \mathcal V_2.
\end{equation}
We identify $(z,w,\sigma_1,\sigma_2) \in \Omega_3$
with $(\sigma_2,\varphi_1^{\sigma}(z), zw)
\in \Omega_1$
and with $(\sigma_1,\varphi_2^{\sigma}(w),zw)
\in \Omega_2$.
We obtain a complex manifold by this identification, which we denote by
$\mathcal C(\mathcal V_1,\mathcal V_2,\epsilon)$.
We define $\pi : \mathcal C(\mathcal V_1,\mathcal V_2,\epsilon)
\to \mathcal V_1 \times \mathcal V_2 \times D^2(\epsilon)$
by
\begin{equation}
\aligned
\pi(\sigma_2,\hat z,\frak r) &= (\pi_1(\hat z),\sigma_2,\frak r),
\qquad &\text{on $\Omega_1$},\\
\pi(\sigma_1,\hat w,\frak r) &= (\sigma_1,\pi_2(\hat w),\frak r),
\qquad &\text{on $\Omega_2$},
\\
\pi((z,w),\sigma_1,\sigma_2) &= (\sigma_1,\sigma_2,zw),
\qquad &\text{on $\Omega_3$}.
\endaligned
\end{equation}
Here $\hat z \in \mathcal C(\mathcal V_1)$,
$\hat w \in \mathcal C(\mathcal V_2)$, $\frak r \in D^2(\epsilon)$
and $\pi_i : \mathcal C(\mathcal V_i) \to \mathcal V_i$ is the projection.
$\pi$ is a well defined holomorphic map.
\par
Sections $\frak z_{i,j} : \mathcal V_i \to \mathcal C(\mathcal V_i)$,
$j=0,\dots,k_i$ (which give $j$-th marked point of the fiber)
induce the sections
$\frak z_j :  \mathcal V_1 \times \mathcal V_2 \times D^2(\epsilon)
\to \mathcal C(\mathcal V_1,\mathcal V_2,\epsilon)$
for $j=1,\dots,k_1+k_2$ in an obvious way.
\par
For $(\sigma_1,\sigma_2,\frak r) \in \mathcal V_1 \times \mathcal V_2 \times D^2(\epsilon)$
it is easy to see from the definition that
$$
(\pi^{-1}(\sigma_1,\sigma_2,\frak r),(\frak z_j(\sigma_1,\sigma_2,\frak r))_{j=1,\dots,k_1+k_2})
$$
is a representative of ${\rm Glusoc}(\sigma_1,\sigma_2,\frak r)$.
\par
We furthermore observe that $\pi : \mathcal C(\mathcal V_1,\mathcal V_2,\epsilon)
\to \mathcal V_1 \times \mathcal V_2 \times D^2(\epsilon)$
is $\Gamma_1 \times \Gamma_2$ equivariant.
\par
Therefore the biholomorphicity of the map ${\rm Glusoc}$ is a consequence of the
definition of the universal family.
\end{proof}
Now we go back to the case of bordered curves.
Let $[\Sigma,\vec z] \in \mathcal M_{g;k+1}$.
We take its {\it double} as in \cite[page 44]{fooo:book1} to obtain a closed Riemann surface,
$[\Sigma^{\rm cl},\vec z^{\rm cl}] \in \mathcal M^{\rm cl}_{2g;k+1}$.
There exits an anti-holomorphic involution
$
\tau : \Sigma^{\rm cl}\to \Sigma^{\rm cl}
$
such that :
\begin{enumerate}
\item[(a)]
The fixed point set $\Sigma^{\rm cl}_{\R}$ of $\tau$ is $S^1$,
which contains all the marked points.
\item[(b)]
The complement $\Sigma^{\rm cl} \setminus \Sigma^{\rm cl}_{\R}$ consists of
two connected components. The closure of one of them with
marked points is biholomorphic to $(\Sigma,\vec z)$.
\end{enumerate}
See Figure \ref{FIgurerealdouble}.
\begin{figure}
\centering
\includegraphics{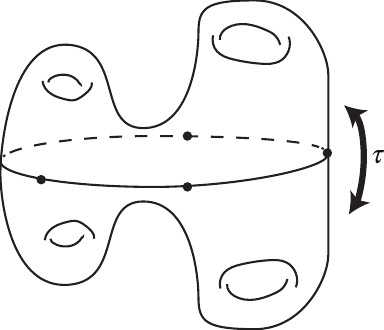}
\caption{Doubling bordered curve.}
\label{FIgurerealdouble}
\end{figure}
Let $\pi : \mathcal C(\mathcal V) \to \mathcal V$ be the universal family
in a neighborhood of $[\Sigma^{\rm cl},\vec z^{\rm cl}]$.
\begin{lem}\label{lem8383}
There exist anti-holomorphic involutions
$\tau : \mathcal C(\mathcal V) \to \mathcal C(\mathcal V)$,
$\tau : \mathcal V \to \mathcal V$ with the following properties.
\begin{enumerate}
\item $\pi \circ \tau = \tau \circ \pi$.
\item
The real dimension of the fixed point set $\mathcal V^{\R}$ of $\tau : \mathcal V \to \mathcal V$
is equal to the complex dimension of $\mathcal V$.
$\mathcal V^{\R}$ is identified with an open neighborhood of $[\Sigma,\vec z]$ in
$\mathcal M_{g;k+1}$.
\item $\sigma_0 =[\Sigma^{\rm cl},\vec z^{\rm cl}] \in \mathcal V$
is a fixed point of $\tau$. \index[syindex]{VcalR@$\mathcal V^{\R}$}
\item
The restriction of $\tau : \mathcal V \to \mathcal V$ to the fiber of $\sigma_0$ coincides with the
map $\tau : \Sigma^{\rm cl}\to \Sigma^{\rm cl}$.
\item
We restrict the universal family to the fixed point set $\mathcal V^{\R}$ of
$\tau : \mathcal V \to \mathcal V$.
We obtain a family of bordered marked Riemann surfaces in the same way as (b)
applied to each of the fiber of $\sigma \in \mathcal V^{\R}$.
This family is the universal family on
$\mathcal V^{\R} \subset \mathcal M_{g;k+1}$ of
bordered Riemann surfaces.
\end{enumerate}
\end{lem}
\begin{proof}
By \cite{Delignemum}, $\mathcal C^{\rm cl} \to \mathcal M^{\rm cl}_{g;k+1}$
is a morphism in the category of the stacks defined over $\R$.
The marked curve $[\Sigma^{\rm cl},\vec z^{\rm cl}]$ together with $\tau :
\Sigma^{\rm cl} \to \Sigma^{\rm cl}$
defines an $\R$-valued point of $\mathcal M^{\rm cl}_{g;k+1}$.
The lemma is a consequence of this fact.
\end{proof}
\begin{lem}\label{lem8484}
In the situation of Lemma \ref{lem8383} there exists a
complex analytic family of coordinates at the $j$-th marked point,
$\varphi : D^2 \times \mathcal V \to  C^{\rm cl}(\mathcal V)$,
in the sense of Definition \ref{defn81} such that
\begin{equation}\label{compcoordR}
\varphi(\overline z,\tau(\sigma))
=
\tau(\varphi(z,\sigma))
\end{equation}
in addition.
\end{lem}
\begin{proof}
In place of $\Gamma = {\rm Aut}(\Sigma^{\rm cl},\vec z^{\rm cl})$
we consider its $\Z_2$ extension $\Gamma_+ =
\Gamma\cup\{\tau\gamma \mid \gamma \in \Gamma\}$.
Then the proof is the same as the proof of existence of a complex analytic family of
coordinates, which uses an equivariant version of the implicit function theorem.
For completeness' sake, we provide the detail of the proof in Appendix \ref{appendixF}.
\end{proof}
We take a complex analytic family of coordinates
$\varphi : D^2 \times \mathcal V \to  C^{\rm cl}(\mathcal V)$
as in Lemma \ref{lem8484}.
Note $\Sigma = \Sigma(\sigma_0)$ is the closure of one of the
connected components of $\Sigma^{\rm cl} \setminus \Sigma^{\rm cl}_{\R}$.
Replacing $\varphi$ by the map 
$(z,\sigma) \mapsto \varphi(-z,\sigma)$ if necessary
we may assume
$$
\varphi (D^2_{\le 0}(1) \times \{\sigma_0\}) \subset \Sigma.
$$
(Here and hereafter we write $D^2_{\le 0}(r) = \{z \in D^2(r) \mid {\rm Im} z \le 0\}$.)
Then for any $\sigma \in \mathcal V^{\R}$ \index[syindex]{D2le0@ $D^2_{\le 0}(r)$}
we have
$$
\varphi (D^2_{\le 0}(1) \times \{\sigma\}) \subset \Sigma(\sigma).
$$
We identify the bordered Riemann surface $\Sigma(\sigma)$
as a connected component of
$\Sigma^{\rm cl}(\sigma) \setminus \Sigma^{\rm cl}_{\R}(\sigma)$
by Lemma \ref{lem8383} (5).
\par
We thus obtain
\begin{equation}\label{form8686}
\varphi^{\R} : D^2_{\le 0}(1) \times \mathcal V^{\R}
\to \mathcal C(\mathcal V^{\R}).
\end{equation}
Here $\pi : \mathcal C(\mathcal V^{\R}) \to \mathcal V^{\R}$
is the universal family of bordered Riemann surfaces.
For each $\sigma \in \mathcal V^{\R}$
the map $\varphi^{\R}$ determines a (complex) coordinate at the $j$-th marked point
of the bordered curve $\Sigma(\sigma)$.
\begin{defn}\label{def8585}
We define an {\em analytic family of coordinates} \index{analytic family of coordinates}
at the $j$-th marked point on $\mathcal V^{\R}$ to be a map $\varphi^{\R}$ as in (\ref{form8686}) obtained
from the complex analytic family of coordinates satisfying the conclusion of
Lemma \ref{lem8484}.
\end{defn}
Let $[\Sigma_i,\vec z_i] \in \mathcal M_{g_i;k_i+1}$ for $i=1,2$.
We take their doubles  $[\Sigma^{\rm cl}_i,\vec z^{\rm cl}_i] \in \mathcal M^{\rm cl}_{g_i;k_i+1}$.
Let $\mathcal C(\mathcal V_i) \to \mathcal V_i$ be the universal family on
a neighborhood $\mathcal V_i$ of $[\Sigma^{\rm cl}_i,\vec z^{\rm cl}_i]$
which satisfies the conclusion of  Lemma \ref{lem8383}. We take
$\varphi_i : D^2(1) \times \mathcal V_i \to \mathcal C(\mathcal V_i)$ which is a
complex analytic family of coordinates at the $0$-th marked points.
We assume that $\varphi_i $ satisfies the conclusion
(\ref{compcoordR}) of Lemma \ref{lem8484}.
We then obtain analytic families of coordinates at $0$-th marked points
on $\mathcal V_i^{\R}$,
\begin{equation}\label{form86862}
\varphi_i^{\R} : D^2_{\le 0}(1) \times \mathcal V_i^{\R}
\to \mathcal C(\mathcal V_i^{\R}).
\end{equation}
We define
$
\varphi_{i,\sigma_i}^{\R} : D^2_{\le 0} \to \Sigma_i(\sigma_i)
$
by
\begin{equation}
\varphi_{i,\sigma_i}^{\R}(z) = \varphi_i^{\R}(z,\sigma_i).
\end{equation}
Let $r \in [0,\epsilon)$, $(\sigma_1,\sigma_2) \in
\mathcal V_1^{\R} \times \mathcal V_2^{\R}$. We define
$\Sigma_1(\sigma_1) \#_r \Sigma_2(\sigma_2)$ as follows.
Let $z_{i,0}(\sigma_i) \in \Sigma_i(\sigma_i)$ be the $0$-th marked point.
We consider the disjoint union
$$
\left(\Sigma_1(\sigma_1) \setminus \varphi_{1,\sigma_1}^{\R}(D^2_{\le 0}(r))\right)
\sqcup
\left(\Sigma_2(\sigma_2) \setminus \varphi_{2,\sigma_2}^{\R}(D^2_{\le 0}(r))\right).
$$
We define an equivalence relation on this set such that
$
\varphi_{1,\sigma_1}^{\R}(z) \in \Sigma_1(\sigma_1) \setminus \varphi_{1,\sigma_1}^{\R}(D^2_{\le 0}(r)),
$
is equivalent to
$
\varphi_{2,\sigma_2}^{\R}(w) \in \Sigma_2(\sigma_2) \setminus \varphi_{2,\sigma_2}^{\R}(D^2_{\le 0}(r))
$
if and only if
\begin{equation}\label{form814-}
zw = -r.
\end{equation}
Let $\Sigma_1(\sigma_1) \#_r \Sigma_2(\sigma_2)$ be the set of
the equivalence classes of the above equivalence relation.
It becomes a bordered curve.
Thus we have defined\index[syindex]{GlusocR@${\rm Glusoc}^{\R}$}
\begin{equation}\label{form8140}
{\rm Glusoc}^{\R} : \mathcal V^{\R}_1 \times \mathcal V^{\R}_2 \times [0,\epsilon)
\to \mathcal{CM}_{g_1+g_2,\ell_1+\ell_2}.
\end{equation}
The next diagram commutes:
\begin{equation}\label{diag862}
\begin{CD}
\mathcal V^{\R}_1 \times \mathcal V^{\R}_2 \times [0,\epsilon)  @ >{{\rm Glusoc}^{\R}}>>
 \mathcal{CM}_{g_1+g_2,\ell_1+\ell_2}  \\
@ VV{}V @ VVV\\
\mathcal V_1 \times \mathcal V_2 \times D^2(\epsilon) @ >
{{\rm Glusoc}} >>
\mathcal{CM}^{\rm cl}_{g_1+g_2,\ell_1+\ell_2}
\end{CD}
\end{equation}
Here $[0,\epsilon) \to D^2(\epsilon)$ in the first vertical arrow is $r \mapsto -r$.
The other parts of the vertical arrows are obvious inclusions.
\par
The source curve gluing map ${\rm Glusoc}^{\R}$
can be identified with the one we described at the beginning of  Chapter \ref{sec:statement} as follows.
We put\index[syindex]{Kisigma@$K_i(\sigma_i)$}
\begin{equation}\label{form812}
\exp(-10\pi T) = r, \qquad K_i(\sigma_i) = \Sigma_i(\sigma) \setminus \varphi_{i,\sigma_i}^{\R}
(D_{\le 0}^2).
\end{equation}
We  define
$$
\Sigma_1(\sigma_1) \setminus K_1(\sigma_1) \cong [0,\infty)_{\tau'} \times [0,1],
\quad
\Sigma_2(\sigma_2) \setminus K_2(\sigma_2) \cong (-\infty,0]_{\tau''}\times [0,1],
$$
by
$$
\varphi_{1,\sigma_1}^{\R}(e^{\pi(x+\sqrt{-1}y)}) \mapsto (-x,-y)
=(\tau',t),
\,\, \varphi_{2,\sigma_2}^{\R}(e^{\pi(x+\sqrt{-1}(y+1))}) \mapsto (x,y)
=(\tau'',t).
$$
Then $\Sigma_1(\sigma_1) \#_T \Sigma_2(\sigma_2)$ as in (\ref{form32}) is isomorphic to
$\Sigma_1(\sigma_1) \#_r \Sigma_2(\sigma_2)$.
\par
In fact the variables $z,w$ appearing in (\ref{form814-}) are related to the
coordinates $\tau', \tau'', t$ of the neck region we used in Chapter \ref{subsecdecayT} by
$$
z = e^{-\pi(\tau' + \sqrt{-1}t)},
\quad
w = -e ^{\pi(\tau'' + \sqrt{-1}t)}.
$$
Note $zw = -r = -e^{-10\pi T}$ is equivalent  to $\tau'' = \tau' -10T$.
See Figure \ref{realcoorinf},
\begin{figure}
\centering
\includegraphics{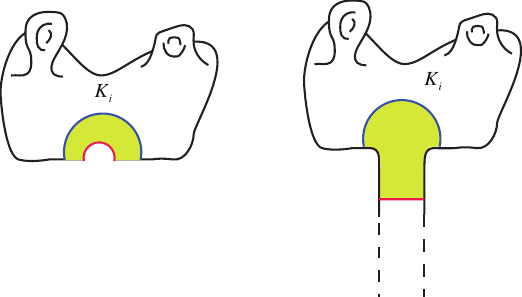}
\caption{Coordinate at infinity of bordered curve.}
\label{realcoorinf}
\end{figure}
\begin{defn}\label{data86}
For $[\Sigma_i,\vec z_i]$ in $\mathcal M_{g_i;k_i+1}$
we consider the following set $\Xi_i$ of data:\index[syindex]{Xii@$\Xi_i$}
\begin{enumerate}
\item
A neighborhood $\mathcal V^{\R}_i$ of  $[\Sigma_i,\vec z_i]$ in $\mathcal M_{g_i;k_i+1}$.
\item
An analytic family of coordinates $\varphi^{\R}_i$ at the 0-th marked points
 as in (\ref{form86862}).
\item
A trivialization of the universal family bundle $\mathcal C(\mathcal V^{\R}_i) \to \mathcal V^{\R}_i$
in the $C^{\infty}$ category.
It is a diffeomorphism $\mathcal C(\mathcal V^{\R}_i) \cong \Sigma_i \times \mathcal V^{\R}_i$
that commutes with the projections.
Moreover we require it to commute with $\varphi_i^{\R}$. Namely the composition
$$
D^2_{\le 0}(1) \times \mathcal V_i^{\R}
\to \mathcal C(\mathcal V_i^{\R})
\to \Sigma_i \times \mathcal V^{\R}_i
$$
of $\varphi_i^{\R}$ and the trivialization
is of the form $(z,\sigma) \mapsto (\phi(z),\sigma)$ where $\phi : D^2_{\le 0}(1) \to \Sigma_i$ is independent of $\sigma$.
\end{enumerate}
A {\it gluing datum} \index{gluing datum} centered at $([\Sigma_1,\vec z_1],[\Sigma_2,\vec z_2])$  is
by definition a pair $\Xi = (\Xi_1,\Xi_2)$ such that $\Xi_i$ are given as above.
\end{defn}
\begin{defnlem}\label{deflem87}
Let $\Xi = (\Xi_1,\Xi_2)$ be a gluing datum centered at $([\Sigma_1,\vec z_1],[\Sigma_2,\vec z_2])$.
\begin{enumerate}
\item
It induces the map (\ref{form8140}). We call the latter the
{\it source gluing map associated to $\Xi$} \index{source gluing map} and write ${\rm Glusoc}^{\R}_{\Xi}$.
\item
For any $\sigma=(\sigma_1,\sigma_2) \in {\mathcal V_1^{\R} \times \mathcal V_2^{\R}}$, it
induces a holomorphic embedding\index[syindex]{IXiisigma@$\tilde{\frak I}_{\Xi,i}^{\sigma}$}
$$
\tilde{\frak I}_{\Xi,i}^{\sigma} : K_i(\sigma_i) \to \Sigma_1(\sigma_1) \#_r \Sigma_2(\sigma_2).
$$
We call it the {\it canonical holomorphic embedding associated to $\Xi$}.\index{canonical holomorphic embedding}
Here the bordered curve $\Sigma_1(\sigma_1) \#_r \Sigma_2(\sigma_2)$ together with marked points
represents the element  ${\rm Glusoc}^{\R}_{\Xi}(\sigma_1,\sigma_2,r)$.
\item
For any $\sigma = (\sigma_1,\sigma_2) \in {\mathcal V}_1^{\R} \times  {\mathcal V}_2^{\R}$, it
also induces a smooth embedding
$$
{\frak I}_{\Xi,i}^{\sigma} : K_i \to \Sigma_1(\sigma_1) \#_r \Sigma_2(\sigma_2).
$$
We call it the {\it canonical  embedding associated to $\Xi$}.
Recall $K_i \subset \Sigma_i$.
\end{enumerate}
\end{defnlem}
\begin{proof}
We discussed (1) already.
Statement (2) is a consequence of (\ref{form812}).
Statement (3) follows from (2) and the trivialization of the universal family bundle
given in Definition \ref{data86} (3).
\end{proof}
\par\medskip
\par\medskip
\subsection{{Obstruction bundle and family gluing map}}

{In this section, we will formulate a {family version}
of the gluing process of pseudoholomorphic maps over the moduli space of stable
curves.}

We {first} recall the definitions of stable maps and of their moduli space.
\begin{defn}\label{isomorphicmaps} Two marked pseudoholomorphic maps $((\Sigma,\vec z),u), \,
((\Sigma',\vec z'),u')$ are said to be isomorphic, if there exists an
isomorphism $\varphi: (\Sigma,z) \to (\Sigma',\vec z')$ such that $u' = u
\circ \varphi^{-1}$. A self-isomorphism $\varphi: (\Sigma,\vec z) \to (\Sigma,\vec z)$
is called an automorphism of $(\Sigma,\vec z)$ if $u = u \circ \varphi^{-1}$. We denote
$$
\operatorname{Aut}((\Sigma,\vec z),u) = \{\varphi\in \operatorname{Aut}(\Sigma,\vec z) \mid u \circ \varphi = u\}.
$$
We call the pair $((\Sigma,\vec z) ,u)$ a \emph{stable map} if \index{stable map}
$\#\operatorname{Aut}((\Sigma,\vec z) ,u) $ is finite.
We define the moduli space of the
isomorphism classes of bordered stable maps $u : (\Sigma,\partial \Sigma) \to (X,L)$
of genus $g$ with $k+1$ boundary marked points
and homology class $\beta \in H_2(X,L;\Z)$ and denote it by\index[syindex]{MgkXLbeta@$\MM_{g,k+1} (X,L;\beta)$}
$$
\MM_{g,k+1} (X,L;\beta).
$$
See \cite[Definition 2.1.27]{fooo:book1}.
\end{defn}

We next define the obstruction spaces.
Let $(\Sigma_i^{\mathfrak{ob}},\vec z^{\frak{ob}}_i)
\in \mathcal M_{g_i,k_i+1}$,
and let $u^{\frak{ob}}_i : (\Sigma^{\frak{ob}}_i,\partial \Sigma^{\frak{ob}}_i) \to (X,L)$
be a pseudoholomorphic map.\footnote{Hereafter we abuse the notation and write an element 
of $\mathcal M_{g,k+1}$ as $(\Sigma,\vec z)$ etc. in place of $[\Sigma,\vec z]$ etc. sometimes.}

\begin{defn}\label{def88}
An {\it obstruction bundle datum} \index{obstruction bundle datum} centered at
$$
(((\Sigma_1^{\mathfrak{ob}},\vec z^{\frak{ob}}_1),u^{\frak{ob}}_1),
((\Sigma_2^{\mathfrak{ob}},\vec z^{\frak{ob}}_2),u^{\frak{ob}}_2))
$$
consists of the objects $(\Xi^{\frak{ob}},(\mathcal E^{\frak{ob}}_1,\mathcal E^{\frak{ob}}_2))
$
such that:
\begin{enumerate}
\item
$\Xi^{\frak{ob}} = (\Xi^{\frak{ob}}_1,\Xi^{\frak{ob}}_2)$ is a  gluing datum
centered at $((\Sigma_1^{\mathfrak{ob}},\vec z^{\frak{ob}}_1),
(\Sigma_2^{\mathfrak{ob}},\vec z^{\frak{ob}}_2))$.
\item
$\mathcal E^{\frak{ob}}_i$  is a finite dimensional subspace\index[syindex]{Eiob@$\mathcal E^{\frak{ob}}_i$}
of $\Gamma(\Sigma^{\frak{ob}}_i;(u^{\frak{ob}}_i)^*TX\otimes \Lambda^{0,1})$,
the space of smooth sections. We assume that the supports of the elements of $\mathcal E^{\frak{ob}}_i$
are contained in ${\rm Int}\,K^{\frak{ob}}_i$.
Here
$K^{\frak{ob}}_i = \Sigma_i^{\frak{ob}}
\setminus \varphi^{\R}_{i,\frak{ob}}(D^2_{\le 0}(1))$.
Note the map $\varphi^{\R}_{i,\frak{ob}}$ is the analytic family of coordinates at the $0$-th
marked point which is a part of $\Xi_{i}^{\frak{ob}}$.
\item The maps $u^{\frak{ob}}_1,\, u^{\frak{ob}}_2$ satisfy
Assumption \ref{DuimodEi} . Here we replace
$\Sigma_i$, $\vec z_i$, $u_i$, ${\rm ev}_{i,\infty}$ in Assumption \ref{DuimodEi}
by $\Sigma^{\frak{ob}}_i$, $\vec z^{\frak{ob}}_i$, $u^{\frak{ob}}_i$, ${\rm ev}_{i,0}^{\frak{ob}}$.
\end{enumerate}
We call
$\left(((\Sigma_1^{\mathfrak{ob}},\vec z^{\frak{ob}}_1),u^{\frak{ob}}_1),
((\Sigma_2^{\mathfrak{ob}},\vec z^{\frak{ob}}_2),u^{\frak{ob}}_2)\right)$
the {\it obstruction center}.\index{obstruction center}
\end{defn}
We define the obstruction bundle $\mathcal E_i(u')$
for an element $u' : (\Sigma',\partial \Sigma') \to (X,L)$
satisfying the next condition for $u'$ and $\nu$.
\begin{conds}\label{conds810}
\begin{enumerate}
\item $[\Sigma',\vec z']$ is an element of the image of the source
gluing map
$
{\rm Glusoc}^{\R}_{\Xi^{\frak{ob}}} : \mathcal V^{{\frak{ob}},\R}_1 \times \mathcal V^{{\frak{ob}},\R}_2
\times [0,\nu)
\to \mathcal{CM}_{g_1+g_2,\ell_1+\ell_2}
$
associated  with $\Xi^{\frak{ob}} {= (\Xi_1^{\frak{ob}}, \Xi_2^{\frak{ob}})}$. \index[syindex]{Xiob@$\Xi^{\frak{ob}}$}
\item
For $z \in K^{\frak{ob}}_i$ we have
$$
d(u_i^{\frak{ob}}(z),u'({\frak I}_{\Xi^{\frak{ob}},i}^{\sigma}(z))) \le \frac{\iota'_X}{2}.
$$
Here $\iota'_X$ is the constant defined as in Condition \ref{cond23}
and ${\frak I}_{\Xi^{\frak{ob}},i}^{\sigma}$
is the canonical  embedding associated to $\Xi^{\frak{ob}}$.
\end{enumerate}
\end{conds}
Now we define a map\index[syindex]{Iuprimei@$I_{u',i}$}
\begin{equation}\label{form816}
I_{u',i} :
\mathcal E^{\frak{ob}}_i \to C^{\infty}(\Sigma';(u')^*TX\otimes \Lambda^{0,1}(\Sigma'))
\end{equation}
as follows. (See (\ref{Eitake}).)
Let $z \in K^{\frak{ob}}_i$. We have the parallel  transport
\begin{equation}\label{form817}
({\rm Pal}_{u_i^{\frak{ob}}(z)}^{u'(z)})^J: T_{u_i^{\frak{ob}}(z)}X
\to T_{u'(z)}X,
\end{equation}
defined in (\ref{newform2525}).
On the other hand we have a projection
\begin{equation}\label{form818}
(\Lambda^{0,1}(\Sigma_i^{\frak{ob}}))_z \to (\Lambda^{0,1}(\Sigma'))_z.
\end{equation}
We remark that the map (\ref{form818}) is complex linear.
The linear map $I_{u',i}$ is induced by the tensor product
(over $\C$) of the two maps (\ref{form817}) and
(\ref{form818}).
\begin{defn}
{
Suppose $(\Sigma',\vec z')$ and $u'$ satisfy Condition \ref{conds810}. 
We define the obstruction space $\mathcal E_i(u')$ at $u'$ by \index[syindex]{Eiuprime@$\mathcal E_i(u')$}
$$
\mathcal E_i(u'): = I_{u',i}(\mathcal E_i^{\mathfrak{ob}})
$$
for the map $I_{u',i}$ in \eqref{form816}.}
\end{defn}
\begin{rem}
We use a gluing datum centered at
$(\Sigma_i^{\frak{ob}},\vec z_i^{\frak{ob}})$ to define
$I_{u',i}$ and $\mathcal E_i(u')$, 
{instead of one centered at}
$(\Sigma_i,\vec z_i)$.
This is an important point which enables us to define {the} coordinate change of
Kuranishi structure. See Remark \ref{rem818}.
\end{rem}

We now define the moduli spaces we study.
Consider an obstruction bundle datum which is centered at
$(((\Sigma_1^{\mathfrak{ob}},\vec z^{\frak{ob}}_1),u^{\frak{ob}}_1),
((\Sigma_2^{\mathfrak{ob}},\vec z^{\frak{ob}}_2),u^{\frak{ob}}_2))$, which we denote by
$(\Xi^{\frak{ob}},(\mathcal E^{\mathfrak{ob}}_1,\mathcal E^{\mathfrak{ob}}_2)$.
We also consider
$(((\Sigma_1,\vec z_1),u_1),((\Sigma_2,\vec z_2),u_2))$
and a gluing datum $\Xi$ centered at $((\Sigma_1,\vec z_1),
(\Sigma_2,\vec z_2))$.
\par
\begin{conds}\label{cond813}
We assume that the pair $(((\Sigma_1,\vec z_1),u_1),
((\Sigma_2,\vec z_2),u_2))$
is close to the obstruction center $(((\Sigma_1^{\mathfrak{ob}},\vec z^{\frak{ob}}_1),u^{\frak{ob}}_1),
((\Sigma_2^{\mathfrak{ob}},\vec z^{\frak{ob}}_2),u^{\frak{ob}}_2))$
in the following sense.
We also assume that the neighborhoods {$\mathcal V^{\R}_i$} of $(\Sigma_i,\vec z_i)$
in $\mathcal M_{g_i,k_i+1}$
which is a part of  data $\Xi_i$ is small in the following sense.
\begin{enumerate}
\item
$(\Sigma_i,\vec z_i)$ is contained in  the neighborhood $\mathcal V^{\frak{ob},\R}_i$ 
of $(\Sigma^{\frak{ob}}_i,\vec z^{\frak{ob}}_i)$
in $\mathcal M_{g_i,k_i+1}$
which is a part of $\Xi^{\frak{ob}}_i$.
Moreover {$\mathcal V^{\R}_i \subset \mathcal V^{\frak{ob},\R}_i$}
\item
Let $(\Sigma_i,\vec z_i) \cong (\Sigma^{\frak{ob}}_i(\sigma^0_i),\vec z^{\frak{ob}}_i(\sigma^0_i))$
with $\sigma^0_i \in {\mathcal V^{\frak{ob}.\R}_i}$.
The  gluing datum
$\Xi_{i}^{\frak{ob}}$ determines a diffeomorphism\index[syindex]{Ii0@$\frak I_i^0$}
$$
\frak I_i^0 :  \Sigma^{\mathfrak{ob}}_i  \to 
\Sigma^{\mathfrak{ob}}_i(\sigma_i^0) \cong 
\Sigma_i.
$$
(Here the first diffeomorphism is induced by $\Xi_{i}^{\frak{ob}}$ and the
second diffeomorphism  is the identification 
$[\Sigma_i,\vec z_i] = [\Sigma^{\mathfrak{ob}}_i(\sigma^0_i),\
\vec z^{\mathfrak{ob}}_i(\sigma^0_i)]$.)
We require
\begin{equation}\label{formula819819}
\sup \{d(u_i(\frak I_i^0(z)),u^{\frak{ob}}_i(z)) \mid z \in \Sigma^{\frak{ob}}_i, i=1,2\}
\le \frac{\iota'_X}{4}.
\end{equation}
\item
We also require
\begin{equation}
\frak I_i^0(K_i^{\frak{ob}}) \subset {\rm Int} \, K_i.
\end{equation}
\end{enumerate}
\end{conds}
We recall that $(\Sigma_{\infty},\vec z_{\infty})$ is a union of {$(\Sigma_1,\vec z_1)$ and $(\Sigma_2,\vec z_2)$}, which are glued to each
other at their
$0$-th marked points, which may also carry their marked points other than
the $0$-th ones.

\begin{defn}
We assume that $(((\Sigma_1^{\mathfrak{ob}},\vec z^{\frak{ob}}_1),u^{\frak{ob}}_1),
((\Sigma_2^{\mathfrak{ob}},\vec z^{\frak{ob}}_2),u^{\frak{ob}}_2))$, \linebreak
$(\Xi^{\frak{ob}},(\mathcal E^{\frak{ob}}_1,\mathcal E^{\frak{ob}}_2))$,
 $(((\Sigma_1,\vec z_1),u_1), ((\Sigma_2,\vec z_2),u_2))$ and $\Xi$
satisfy Condition \ref{cond813} 
and {$(\epsilon,\nu) \in \R_+^2$}.
\par
We define the moduli space
$\mathcal M_{+}^{\mathcal E_1\oplus \mathcal E_2}((\Sigma_{\infty},\vec z);u_1,u_2)_{\epsilon,\nu}$
as the set of all\index[syindex]{M+E1E2u1u2@$\mathcal M_{+}^{\mathcal E_1\oplus \mathcal E_2}((\Sigma_{\infty},\vec z);u_1,u_2)_{\epsilon,\nu}$}
the isomorphism classes of
$((\Sigma',\vec z'),u')$ such that the following three conditions are satisfied.
\begin{enumerate}
\item
$(\Sigma',\vec z')  = {\rm Glusoc}_{\Xi}^{\R}(\sigma_1,\sigma_2,T) \in
{\rm Im} ({\rm Glusoc}_{\Xi}^{\R})$.
$(\sigma_1,\sigma_2)$ is in the $\epsilon$ neighborhood of $\sigma_0$
in {$\mathcal V^{\R}_1 \times \mathcal V^{\R}_2$}.
({$\mathcal V^{\R}_1$} and {$\mathcal V^{\R}_2$} 
are parts of $\Xi$.) Here $\sigma_0$
corresponds to $(\Sigma_1,\vec z_1)$ and $(\Sigma_2,\vec z_2)$ and $T > 1/\nu$.
\item
Condition \ref{nearbyuprime} is satisfied. Namely:
\begin{enumerate}
\item
We assume
$u_i \vert_{K_i}$ is $\epsilon$ close to $u'\circ \frak I^{\sigma}_{\Xi,i}\vert_{K_i}$ in $C^1$ sense.
Here $\frak I^{\sigma}_{\Xi,i}$ is defined in Definition-Lemma \ref{deflem87}.
\item
$
{\rm Diam} \{u'(z) \mid z \in \Sigma' \setminus (K_1 \cup K_2)\} < \epsilon.
$
\end{enumerate}
\item
\begin{equation}
\overline{\partial} u' \equiv 0 \mod \mathcal E_1(u') \oplus \mathcal E_2(u').
\end{equation}
Note (1), (2) and (\ref{form816}), (\ref{formula819819}) imply that $\mathcal E_1(u')$, $\mathcal E_2(u')$ are  defined if $\epsilon$ is sufficiently
small.
\end{enumerate}
\end{defn}
Note that if $(\Sigma^{\frak{ob}}_i,\partial \Sigma^{\frak{ob}}_i)
\cong (\Sigma_i,\partial \Sigma_i)$ as bordered Riemann surfaces and $\Xi_{\frak{ob}}=\Xi$, the moduli space
$\mathcal M^{\mathcal E_1\oplus \mathcal E_2}((\Sigma_{T},\vec z);(u_1,u_2))_{\epsilon} $
we defined in Definition \ref{defn310} is a subset of the moduli space
$\mathcal M_{+}^{\mathcal E_1\oplus \mathcal E_2}((\Sigma_{\infty},\vec z_{\infty});u_1,u_2)_{\epsilon,\nu} $
we defined here, for $T > 1/\nu$.
\par
We also define
\begin{equation}\label{map821}
I_{u'_i} : \mathcal E_i^{\frak{ob}} \to \Gamma(\Sigma_i;(u_i')^*TX \otimes \Lambda^{0,1}
(\Sigma_i))
\end{equation}
{as in (\ref{newform320}).}
 Namely it is induced from the tensor product over $\C$ of
 the projection
 $
(\Lambda^{0,1}(\Sigma_i^{\frak{ob}}))_z \to (\Lambda^{0,1}(\Sigma_i(\sigma_i)))_z
$,
and the complex linear part of the parallel transport
$$
\left({\rm Pal}_{u_i^{\frak{ob}(z)}}^{u_i'(z)}\right)^J : T_{u^{\frak{ob}}(z)}X
\to T_{u_i'(z)}X.
$$
We put
\begin{equation}
\mathcal E_i(u'_i) =  I_{u'_i}(\mathcal E_i^{\frak{ob}}).
\end{equation}
So we can define an equation
\begin{equation}\label{mainequationui20}
\overline\partial u_i' \equiv 0 , \quad \mod \mathcal E_i(u'_i).
\end{equation}
\begin{defn}
$\mathcal M_+^{\mathcal E_i}((\Sigma_i,\vec z_i); u_i)_{\epsilon}$ is the set of
isomorphism classes of \index[syindex]{M+EiSigmai@$\mathcal M_+^{\mathcal E_i}((\Sigma_i,\vec z_i); u_i)_{\epsilon}$}
$((\Sigma'_i,\vec z'_i),u'_i)$ with the following properties.
\begin{enumerate}
\item
$(\Sigma'_i,\vec z'_i)$ represents an element $\sigma_i$ of $\mathcal V^{\R}_i$
and is in an $\epsilon$ neighborhood of $[\Sigma_i,\vec z_i]$.
\item
$u'_i$ is $\epsilon$ close to $u_i$ in $C^1$ topology.
\item
$u_i'$ satisfies equation (\ref{mainequationui20}).
\end{enumerate}
\end{defn}
We assumed Assumption \ref{DuimodEi} in Definition \ref{def88} (4), where
 ${\rm ev}_{i,\infty}$ in (\ref{Duiev}), (\ref{Duievsurj})
are ${\rm ev}_{0}$ here.
Then we can take $\epsilon$ small so that
$\mathcal M_+^{\mathcal E_i}((\Sigma_i,\vec z_i); u_i)_{\epsilon}$ is a smooth manifold.
Moreover the fiber product
\begin{equation}\label{fibproductbeforeglue}
\mathcal M^{\mathcal E_1}_+((\Sigma_1,\vec z_1); u_1)_{\epsilon}
\,{}_{{\rm ev}_0} \times _{{\rm ev}_0}
\mathcal M^{\mathcal E_2}_+((\Sigma_2,\vec z_2); u_2)_{\epsilon}
\end{equation}
is transversal.
(Here the fiber product is taken with respect to the evaluation maps
$\text{\rm ev}_0 = {\rm ev}_{i,\infty}$ $(i=1,2)$ given in Chapter \ref{sec:statement}.)
\par
Now Theorems \ref{gluethm1} and \ref{exdecayT} are generalized as follows.
\begin{thm}\label{gluethm12}
For any $\epsilon(11),  \nu(1) > 0$ there exist 
$\epsilon(12) >0$, $T_{7,m,\epsilon(11), \nu(1)}>0$
and a map \index[syindex]{Glue+@$\text{\rm Glue}_+$}
$$
\aligned
\text{\rm Glue}_+ :
\mathcal M^{\mathcal E_1}_+((\Sigma_1,\vec z_1); u_1)_{\epsilon(12)}
&\,\,\,{}_{{\rm ev}_0} \times _{{\rm ev}_0}
\mathcal M^{\mathcal E_2}_+((\Sigma_2,\vec z_2); u_2)_{\epsilon(12)}
\times (T_{7,m,\epsilon(11), \nu(1)},\infty]
\\
&
\to
\mathcal M_{+}^{\mathcal E_1\oplus \mathcal E_2}((\Sigma_{\infty},\vec z);u_1,u_2)_{\epsilon(11),\nu(1)}
\endaligned
$$
with the following properties.
\begin{enumerate}
\item
The map $\text{\rm Glue}_+$
is a homeomorphism onto its image.
The image contains $\mathcal M_+^{\mathcal E_1\oplus \mathcal E_2}((\Sigma_{\infty},\vec z);
u_1,u_2)_{\epsilon_{\epsilon(11),\nu(1),(\ref{CD825})}
,\nu_{\epsilon(11),\nu(1),(\ref{CD825})}}$,
where $\epsilon_{\epsilon(11),\nu(1),(\ref{CD825})}$, 
and $\nu_{\epsilon(11),\nu(1),(\ref{CD825})}$ are 
positive number{s} depending on $\epsilon(11),\nu(1)$.
\item
The next diagram commutes
\begin{equation}\label{CD825}
\begin{CD}
\mathcal M^{\mathcal E_1}_+((\Sigma_1,\vec z_1); u_1)_{\epsilon(12)}
\,\,{}_{{\rm ev}_0} \times _{{\rm ev}_0}
\atop  \mathcal M^{\mathcal E_2}_+((\Sigma_2,\vec z_2); u_2)_{\epsilon(12)} 
 \times (T_{7,m,\epsilon(11), \nu(1)},\infty]
  @ >>>
\mathcal V^{\R}_1 \times \mathcal V^{\R}_2 \times (T_{7,m,\epsilon(11), \nu(1)},\infty]
  \\
@ V{
\text{\rm Glue}_+}VV @ V{\rm Glusoc}_{\Xi}^{\R}VV\\
\mathcal M_{+}^{\mathcal E_1\oplus \mathcal E_2}((\Sigma_{\infty},\vec z_{\infty});u_1,u_2)_{\epsilon(11),\nu(1)}
 @ >  >>
\mathcal{CM}_{g_1+g_2,\ell_1+\ell_2}
\end{CD}
\end{equation}
where the horizontal arrows are defined by  forgetting the map part.
\item
The map $\text{\rm Glue}_+$  defines a fiberwise diffeomorphism 
{of $C^{m-2}$ class} onto its image.
Here fiber means the fiber of the horizontal arrows of (\ref{CD825}).
\end{enumerate}
\end{thm}
This is the family version of Theorem \ref{gluethm1}
incorporating the variation of complex structures on the source curve
into the gluing.
A generalization of Theorem \ref{exdecayT} is the following Theorem \ref{them816}.
We recall {the definition}
$$
K_1^S = K_1 \cup ([0,S]_{\tau'} \times [0,1]),
\quad
K_2^S = K_2 \cup ([-S,0]_{\tau''} \times [0,1])
$$
{from \eqref{1104}.} We define\index[syindex]{Gluresi+S@$\text{\rm Glures}_{i,+,S}$}
$$
\aligned
\text{\rm Glures}_{i,+,S} :
&\mathcal M^{\mathcal E_1}_+((\Sigma_1,\vec z_1); u_1)_{\epsilon(12)}
\,\,\,{}_{{\rm ev}_0} \times _{{\rm ev}_0}
\mathcal M^{\mathcal E_2}_+((\Sigma_2,\vec z_2); u_2)_{\epsilon(12)} \\
&\times (T_{7,m,\epsilon(11), \nu(1)},\infty]
\to
\text{\rm Map}_{L^2_{m+1}}((K_i^S,K_i^S\cap\partial \Sigma_i),(X,L))
\endaligned
$$
as the composition of $\text{\rm Glue}_+$ with the restriction map
$$
\mathcal M_{+}^{\mathcal E_1\oplus \mathcal E_2}((\Sigma_{\infty},\vec z);u_1,u_2)_{\epsilon(11),\nu(1)}
\to
\text{\rm Map}_{L^2_{m+1}}((K_i^S,K_i^S\cap\partial \Sigma_i),(X,L)).
$$
Here we use the map  ${\frak I}_{\Xi,i}^{\sigma}$ in Lemma-Definition \ref{deflem87} to make
the identification
$$
K^S_i \cong K^S_i(\sigma_i) \subset \Sigma_1(\sigma_1) \#_T \Sigma_2(\sigma_2).
$$
Note the diffeomorphism $K^S_i \cong K^S_i(\sigma_i)$ is induced by the
gluing datum $\Xi$ centered at $((\Sigma_1,\vec z_1),(\Sigma_2,\vec z_2))$.
\par
{The following exponential estimate is the crux entering in the proof of
smoothness of {the coordinate change of the} Kuranishi chart.}
\begin{thm}\label{them816}
There exists $\delta_2>0$ with the  following property.
For each $m$ and $S$,  there exist constants $T_{m,S,(\ref{form672})} > \frac{S}{10}$
and $C_{m,S,(\ref{form672})}> 0$  such that
\begin{equation}\label{form672}
\left\| \nabla_{\sigma,\rho}^{n} \frac{d^{\ell}}{dT^{\ell}} \text{\rm Glures}_{i,{+},S}\right\|_{L^2_{m+1-\ell}}
< C_{m,S,(\ref{form672})}e^{-\delta_2 T}
\end{equation}
for all {$m \ge n \ge 0$, $m \ge \ell > 0$}.
Here $\nabla_{\sigma,\rho}$ is a differentiation with respect to $((\sigma_1,\rho_1),(\sigma_2,\rho_2))$ in
the fiber product (\ref{fibproductbeforeglue}). \index[syindex]{sigmarho@$(\sigma,\rho)$}
\footnote{$\sigma_1$ and $\sigma_2$ parameterize the source curve and
$\rho_1$ and $\rho_2$ parameterize the map.}
\end{thm}
\begin{proof}[Proof of Theorems \ref{gluethm12} and \ref{them816}]
We denote  {by $u_i^{\rho_i}: \Sigma_i(\sigma_i) \to X$} the map corresponding to
$(\sigma_i,\rho_i) \in \mathcal M^{\mathcal E_i}_+((\Sigma_i,\vec z_i); u_i)_{\epsilon(12)}$.
\par
For each $((\sigma_1,\rho_1),(\sigma_2,\rho_2)) \in \mathcal M^{\mathcal E_1}_+((\Sigma_1,\vec z_1); u_1)_{\epsilon(12)}
\,\,\,{}_{{\rm ev}_0} \times _{{\rm ev}_0}
\mathcal M^{\mathcal E_2}_+((\Sigma_2,\vec z_2); u_2)_{\epsilon(12)}$
and $T$ we run the alternating method developed in Chapters
\ref{alternatingmethod} and \ref{subsecdecayT}.
\par
Namely we start with $((\Sigma_1(\sigma_1),\vec z_1(\sigma_1)),u_1^{\rho_1}),
(\Sigma_2(\sigma_2),\vec z_1(\sigma_2)),u_2^{\rho_2})$
in place of $((\Sigma_1,\vec z_1), u_1)$, $((\Sigma_2,\vec z_2),u_2)$ and use
the coordinates at $0$-th marked points which we determined as a part of
the gluing data $\Xi$ centered at $((\Sigma_1,\vec z_1),u_1),
((\Sigma_2,\vec z_2),u_2))$.
\par
Except the point which we will explain below,
the proof  goes in the same way as before and we obtain a map
$$
V_1(\sigma_1,\epsilon) \times_L V_2(\sigma_2,\epsilon) \times (T_0,\infty] \to
\mathcal M_{+}^{\mathcal E_1\oplus \mathcal E_2}((\Sigma_{\infty},\vec z);u_1,u_2)_{\epsilon,\nu}.
$$
Here $V_i(\sigma_i,\epsilon)$ is the fiber of the map
$
\mathcal M^{\mathcal E_i}_+((\Sigma_i,\vec z_i); u_i)_{\epsilon}
\to \mathcal V_i^{{\R}}
$ at {$\sigma_i \in \mathcal V_i^{\R}$} and $T_0 > 1/\nu$.
The union of these maps over $\sigma_1$, $\sigma_2$ will be the map
${\rm Glue}_+$.
\par
The point we need to clarify in adapting the proof of Chapters  \ref{alternatingmethod} and \ref{subsecdecayT}
to our situation is the following:
The way we define $\mathcal E_i(u')$ in this chapter is slightly different from that of
Chapters
\ref{alternatingmethod} and \ref{subsecdecayT}.
Namely, {in this chapter}, we start with $\mathcal E_i^{\frak{ob}}$ defined on $(\Sigma^{\frak{ob}}_i,u^{\frak{ob}}_i)$
 with $\Sigma^{\frak{ob}}_i \ne \Sigma_i$, while
in Chapters \ref{alternatingmethod} and \ref{subsecdecayT} $\mathcal E_i^{\frak{ob}}$ was defined on
$(\Sigma_i,u^{\frak{ob}}_i)$, i.e., $\Sigma^{\frak{ob}}_i = \Sigma_i$.
\par
We use  Proposition \ref{prop816} {below} to obtain the required estimate for 
the current $\mathcal E_i(u')$.
Then the arguments in Chapters \ref{alternatingmethod} and \ref{subsecdecayT} also go through for the proofs of
Theorems \ref{gluethm12} and \ref{them816}.
\par
In fact Theorem \ref{gluethm12} (1),(3) are proved in the same way as
in Chapters
\ref{alternatingmethod}, \ref{subsecdecayT} and \ref{surjinj}.
The estimate of $\sigma$ derivative is the same as that of $\rho$ derivative.
So the proof of Theorem \ref{them816} is the same as that in Chapter \ref{subsecdecayT}.
Theorem \ref{gluethm12} (2) follows from the fact that the alternating method we use does not change the
complex structure of the source.
\par
To {formulate} Proposition \ref{prop816}, we need {some preparation}.

Let $(\Sigma',\vec z') = {\rm Glusoc}_{\Xi}^{\R}((\sigma_1,\sigma_2),T)$ and
$u' : (\Sigma',\partial \Sigma') \to (X,L)$ a smooth map satisfying the conditions
\begin{equation}
\aligned
&d(u'(z),u_i(z)) \le \epsilon, \qquad \text{for $z \in K_i$}
\\
& {\rm Diam} (u'(\Sigma' \setminus (K_1 \cup K_2))) \le \epsilon.
\endaligned
\end{equation}

We take a basis ${\bf e}_{i,a}$ $(a=1,\dots,\dim \mathcal E_i^{\frak{ob}}$)
of $\mathcal E_i^{\frak{ob}} \subset \Gamma(K^{\frak{ob}}_i;(u^{\frak{ob}}_i)^*TX\otimes \Lambda^{0,1}(\Sigma_i^{\frak{ob}}))$
and put
\begin{equation}\label{form827plus}
{\bf e}^0_{i,a}(u') = (I^{u_i}_{u_i^{{\rho}_i}} \circ I_{u'}^{u_i^{{\rho}_i}} \circ I_{u',i})({\bf e}_{i,a})
\in
\Gamma(K_i;(u_i)^*TX\otimes \Lambda^{0,1}(\Sigma_i)).
\end{equation}
Here the map $ I_{u',i}$ is (\ref{form816}) and \index[syindex]{Iuprimeuirhoi@$I_{u'}^{u_i^{\rho_i}}$}
\index[syindex]{Iuisigmaiui@$I^{u_i}_{u_i^{{\rho}_i}}$}
the maps
$$
\aligned
& I_{u'}^{u_i^{\rho_i}} : \Gamma({\frak I}_{\Xi^{\frak{ob}},i}^{\sigma'}(K^{\frak{ob}}_i) ; (u')^*TX \otimes \Lambda^{0,1}(\Sigma'))
\to
\Gamma(\Sigma_i(\sigma_i) ; (u_i^{{\rho}_i})^*TX \otimes \Lambda^{0,1}(\Sigma_i(\sigma_i)))
\\
& I^{u_i}_{u_i^{{\rho}_i}} : \Gamma(\Sigma_i(\sigma_i) ; (u_i^{{\rho}_i})^*TX \otimes \Lambda^{0,1}(\Sigma_i(\sigma_i)))
\to \Gamma(K_i;(u_i)^*TX\otimes \Lambda^{0,1}(\Sigma_i))
\endaligned
$$
are defined as follows.
Note since
$(\Sigma',\vec z') = {\rm Glusoc}_{\Xi^{\frak{ob}}}^\R((\sigma'_1,\sigma'_2),T')$
we have a smooth embedding\index[syindex]{IXiobi@${\frak I}_{\Xi^{\frak{ob}},i}^{\sigma'}$}
\begin{equation}\label{form828}
{\frak I}_{\Xi^{\frak{ob}},i}^{\sigma'}  : K^{\frak{ob}}_i \to \Sigma'
\end{equation}
by Definition-Lemma \ref{deflem87}.
Its image appears in the domain of $I_{u'}^{u_i^{{\rho}_i}}$.
\par
 We use an embedding
\begin{equation}\label{form829}
{\frak I}_{\Xi,i}^{\sigma}: K_i \to \Sigma'
\end{equation}
which is also obtained by  Definition-Lemma \ref{deflem87}
in the definition of $I_{u'}^{u_i^{{\rho}_i}}$.
\begin{rem}
We remark that in (\ref{form828}) we use the gluing data $\Xi^{\frak{ob}}$
centered at $((\Sigma^{\frak{ob}}_1,\vec z^{\frak{ob}}_1),(\Sigma^{\frak{ob}}_2,\vec z^{\frak{ob}}_2))$.
In (\ref{form829}) we use the gluing data $\Xi = (\Xi_1,\Xi_2)$
centered at $((\Sigma_1,\vec z_1),(\Sigma_2,\vec z_2))$.
See Figure \ref{figure83}.
\end{rem}
We assume the next relation:
\begin{equation}\label{form830}
{\frak I}_{\Xi^{\frak{ob}},i}^{\sigma'}(K^{\frak{ob}}_i) \subset {\frak I}_{\Xi,i}^{\sigma}(K_i)
\end{equation}
which is needed so that the above composition is defined.
By Condition \ref{cond813} we may choose {$\mathcal V^{\R}_i$}, a neighborhood of $[\Sigma_i,\vec z_i]$
in the moduli space of marked bordered curves, so small that if $(\sigma_1,\sigma_2) \in {\mathcal V^{\R}_1 
\times \mathcal V^{\R}_2}$
then (\ref{form830}) is satisfied.
\begin{figure}
\centering
\includegraphics{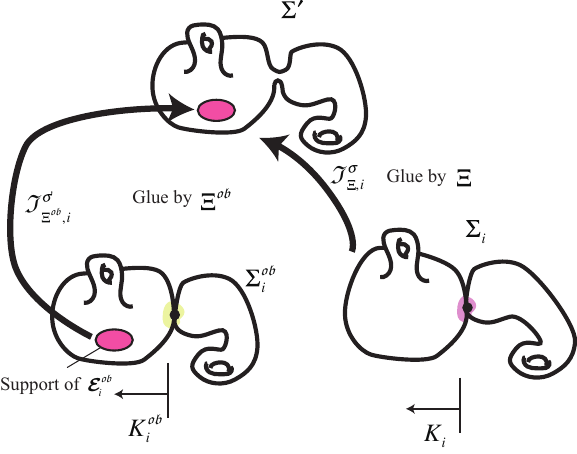}
\caption{Two gluing data give different embeddings.}
\label{figure83}
\end{figure}
\par
The smooth embedding ${\frak I}_{\Xi,i}^{\sigma}$ induces a complex linear map
\begin{equation}\label{map828}
\Lambda^{0,1}_z(\Sigma') \to \Lambda^{1}_{{\frak I}_{\Xi,i}^{\sigma}(z)}(\Sigma_i(\sigma_i))
\to  \Lambda^{0,1}_{{\frak I}_{\Xi,i}^{\sigma}(z)}(\Sigma_i(\sigma_i)).
\end{equation}
Here  the second map is the projection and the first map is
the complex linear part of the map induced by ${\frak I}_{\Xi,i}^{\sigma}$.
Namely (\ref{map828}) is given by
$$
\frac{1}{2}\left(({({\frak I}_{\Xi,i}^{\sigma})^{-1}})^*  - j_{\Sigma_i(\sigma_i)}
\circ (({\frak I}_{\Xi,i}^{\sigma})^{-1})^* \circ j_{\Sigma^{\frak{ob}}} \right).
$$
\par
We also consider the map:
\begin{equation}\label{map8290}
\left(\left({\rm Pal}^{u'(z)}_{u_i^{{\rho}_i}(z)}\right)^J\right)^{-1} :
T_{u'(z)}X\to T_{u_i^{{\rho}_i}(z)}X.
\end{equation}
Then $ I_{u'}^{u^{{\rho}_i}}$ is obtained from the tensor product of
(\ref{map828}) and (\ref{map8290}).
\par
{Similarly}
$I^{u_i}_{u_i^{{\rho}_i}}$ is induced by the tensor product of the projection
\begin{equation}\label{map835}
\Lambda^{0,1}_z(\Sigma_i(\sigma_i))
\to  \Lambda^{0,1}_{z}(\Sigma_i)
\end{equation}
and a complex linear part of the parallel transformation
\begin{equation}\label{map829}
\left(\left({\rm Pal}_{u_i(z)}^{u_i^{{\rho}_i}(z)}\right)^J\right)^{-1} :
T_{u_i^{{\rho}_i}(z)}X \to  T_{u_i(z)}X.
\end{equation}
{This finishes the explanation of the maps \eqref{form827plus}.}
\par
Let 
$$
(\sigma,\rho) = ((\sigma_1,\rho_1),(\sigma_2,\rho_2))
\in {\mathcal M^{\mathcal E_1}_+((\Sigma_1,\vec z_1); u_1)_{\epsilon} \times_L  \mathcal M^{\mathcal E_2}_+((\Sigma_2,\vec z_2); u_2)_{\epsilon}}
$$
and $(\Sigma^{{\sigma}}_T,\vec z^{{\sigma}}_T) =
{\rm Glusoc}^{\R}(\sigma_1,\sigma_2,T)$. 
\par
We denote by
$$
\hat u^{\sigma,\rho}_{i,T,(\kappa)} : \Sigma^{\sigma}_i \cong \Sigma_i
\to X
$$
the map obtained at the $\kappa$-th inductive step of our 
{inductive construction} starting from
$((\Sigma_1(\sigma_1),\vec z_1(\sigma_1)),u_1^{{\rho_1}})$, $((\Sigma_2(\sigma_2),\vec z_2(\sigma_2)),u_2^{{\rho_2}})$.
The diffeomorphism $\Sigma^{\sigma}_i \cong \Sigma_i$
is a part of datum $\Xi_i$.
\par
In other words $\hat u^{\sigma,\rho}_{i,T,(\kappa)}$
corresponds to the map $\hat u^{\rho}_{i,T,(\kappa)}$
in (\ref{form573new}).
We put
\begin{equation}\label{newname840}
{\bf e}'_{i,a;(\kappa)}(\sigma,\rho,T)
=
{\bf e}^0_{i,a}(u^{\sigma,\rho}_{T,(\kappa)})
\in
\Gamma(K_i; u_i^*TX\otimes \Lambda^{0,1}(\Sigma_i)).
\end{equation}
Here ${\bf e}^0_{i,a}(u^{\sigma,\rho}_{T,(\kappa)}) $ is {the map defined as} in (\ref{form827plus}), 
and the map $u^{\sigma,\rho}_{T,(\kappa)}$
is the $\rho$-parametrized version of
the map 
$u^{\sigma}_{T,(\kappa)}$
in Definition 5.33.
We remark that 
{the set} ${\{}{\bf e}'_{i,a;(\kappa)}(\sigma,\rho,T){\}}$ 
forms a basis of 
the subspace
$$
\Big(I^{u_i}_{u_i^{{\rho}_i}} \circ I_{\hat u^{\sigma,\rho}_{i,T,(\kappa)}}^{u_i^{{\rho}_i}}
\Big)
(\mathcal E_i(\hat u^{\sigma,\rho}_{i,T,(\kappa)})).
$$
Moreover there exists a canonical isomorphism
\begin{equation}\label{newnewform843}
\mathcal E_i(\hat u^{\sigma,\rho}_{i,T,(\kappa)})
\cong
\mathcal E_i(u^{\sigma,\rho}_{T,(\kappa)}).
\end{equation}
(\ref{newnewform843}) is a consequence of the equality
$\hat u^{\sigma,\rho}_{i,T,(\kappa)} = u^{\sigma,\rho}_{T,(\kappa)}$ on
$K_i \subset \Sigma_i(\sigma_i)$ which we regard as a subset $K_i \subset \Sigma^{\sigma}_T$
by the canonical embedding  ${\frak I}_{\Xi,i}^{\sigma}$.
This equality follows from the definition (\ref{form573new}).
\begin{prop}\label{prop816}
There exists $\delta_3>0$ such that we can estimate ${\bf e}'_{i,a;(\kappa)}({\sigma,\rho},T)$
and its $T$, $\rho$, $\sigma$ derivatives as
\begin{equation}\label{prop816sformula}
\left\Vert \nabla_{\sigma,\rho}^n\frac{\partial^{\ell}}{\partial T^{\ell}}{\bf e}'_{i,a;
(\kappa)}(\sigma,\rho,T)
\right\Vert_{L^2_{m+1-\ell}}
\le
C_{m,(\ref{prop816sformula})} e^{-\delta_3 T}
\end{equation}
for $m-2 \ge n \ge 0$, {$m-2 \ge \ell > 0$}, if
$T>T_{m,(\ref{prop816sformula})}$.
\end{prop}
This proposition is the family version of Lemma \ref{lemma615}.\footnote{Actually Lemma \ref{lemma615} gives an estimate of the
orthonormal frame obtained from ${\bf e}'_{i,a;(\kappa)}(\rho,T)$
by the Gram-Schmidt process. However the argument to study the
Gram-Schmidt process in our situation is the same as the one in the proof of Lemma \ref{lemma615}.}
The proof will be given in Section \ref{appendixG}.
\par
We can use Proposition \ref{prop816} to control the obstruction bundle $\mathcal E_i(u')$
together with its derivatives.
\par
More explicitly: we can use the proposition to show (\ref{eq:intintPiE1}): we use it to show
a version of Lemma \ref{lemE2E2},
where $\sigma$ derivative as well as $\rho$ and $T$ derivatives is included and
the definition of $\mathcal E_i(u')$ is slightly modified
as we mentioned at the beginning the proof of Theorem \ref{them816}.
\par
We can use  Proposition \ref{prop816} also in all the other places where we need to
control $\mathcal E_i(u')$.
Note we take $\delta>0$ in Definition \ref{defn3535} so that not only $\delta < \delta_1/10$ but also
$\delta < \delta_3/10$ holds.
The constant $\delta_2$ in Theorem \ref{them816} is determined by $\delta_1$ and
$\delta_3$.
\par
The proofs of Theorems \ref{gluethm12} and \ref{them816} are now complete except the
proof of Proposition \ref{prop816}.
\end{proof}

\subsection{Smoothness of coordinate change}
\label{subsec82}
We now use Theorem{s} \ref{gluethm12} {and \ref{them816}} to prove  smoothness of 
the coordinate change.
We begin with explaining the situation where we study the coordinate change.
\begin{shitu}\label{shitu820}
Let $(\Sigma_i^{\mathfrak{ob}},\vec z^{\frak{ob}}_i) \in \mathcal M_{g_i,k_i+1}$ for $i=1,2$.
We take
$\Xi^{\frak{ob}} = (\Xi^{\frak{ob}}_1,\Xi^{\frak {ob}}_2)$ a gluing datum
centered at  $((\Sigma_1^{\mathfrak{ob}},\vec z^{\frak{ob}}_1),
(\Sigma_2^{\mathfrak{ob}},\vec z^{\frak{ob}}_2))$.
We consider $u^{\frak{ob}}_i : (\Sigma_i^{\mathfrak{ob}},\partial \Sigma_i^{\mathfrak{ob}})\to
(X,L)$, a pseudoholomorphic map of homology class $\beta_i$.
For $j=1,2$, let $(\Xi^{\frak{ob}},(\mathcal E^{(j)}_1,\mathcal E^{(j)}_2))$
be an obstruction bundle datum which is centered at
$(((\Sigma_1^{\mathfrak{ob}},\vec z^{\frak{ob}}_1),u^{\frak{ob}}_1),
((\Sigma_2^{\mathfrak{ob}},\vec z^{\frak{ob}}_2),u^{\frak{ob}}_2))$.
Note we use the same gluing datum $\Xi^{\frak{ob}}$ for $j=1,2$.
\par
We assume that the next inclusion holds for $i=1,2$.
\begin{equation}\label{form844}
\mathcal E^{(1)}_i \subseteq \mathcal E^{(2)}_i.
\end{equation}
\end{shitu}
In sum we take two obstruction bundle data with the same gluing dataum 
that
satisfies (\ref{form844}).
\begin{shitu}\label{shitu821}
We next consider
$((\Sigma_i^{(j)},\vec z_i^{(j)}),u^{(j)}_{i})$
for $i,j =1,2$.
Here:
\begin{enumerate}
\item
$(\Sigma_i^{(j)},\vec z_i^{(j)}) \in \mathcal M_{g_i,k_i+1}$.
\item For each $i=1,\,2$, $j=1,\,2$,
$u^{(j)}_{i}: (\Sigma_i^{(j)},\partial \Sigma_i^{(j)}) \to (X,L)$
is a pseudoholomorphic map of homology class $\beta_i$.
\end{enumerate}
We call {such a pair} $((\Sigma_1^{(j)},\vec z_1^{(j)}),u^{(j)}_{1})),((\Sigma_2^{(j)},\vec z_2^{(j)}),u^{(j)}_{2}))$
for $j=1$ or $2$, a {\it chart center}. \index{chart center}
\par
Let $\Xi^{(j)} = (\Xi^{(j)}_1,\Xi^{(j)}_2)$ be a gluing datum
centered at $(\Sigma_i^{(j)},\vec z_i^{(j)})$.
We assume
\begin{equation}\label{form845}
\mathcal V^{(1),\R}_{i} \subset \mathcal V^{(2),\R}_{i}
\end{equation}
where $\mathcal V^{(j),\R}_{i}$ is a neighborhood of
$(\Sigma_i^{(j)},\vec z_i^{(j)})$ in $\mathcal M_{g_i,k_i+1}$
which is a part of $\Xi^{(j)}_i$.
Note (\ref{form845}) implies $(\Sigma_i^{(1)},\vec z_i^{(1)}) \in  \mathcal V^{(2),\R}_{i}$.
\par
We also assume Conditions \ref{cond822} and \ref{cond823} below.
\end{shitu}
When $u_i : (\Sigma_i,\partial \Sigma_i) \to (X,L)$ is given
and satisfies ${\rm ev}_{1,\infty}(u_1) = {\rm ev}_{1,\infty}(u_2)$,
we denote by $u_{\infty} : (\Sigma_{\infty},\partial \Sigma_{\infty}) \to (X,L)$
the map {that restricts} to $u_i$ on $\Sigma_i$.
\par
We take and fix $\epsilon_6, \nu_1 >0$.
We put $\epsilon(11) = \epsilon_6$ and $\nu(1) = \nu_1$ in Theorem \ref{gluethm12}.
(Here we will apply Theorem \ref{gluethm12} 
to $\text{\rm Glue}^{(2)}_{+}$ in (\ref{diag86}).)
We put
\begin{equation}\label{lastform845}
\epsilon_7 = \epsilon_{\epsilon(11),\nu(1),(\ref{CD825})}, 
\qquad
\nu_2 = \nu_{\epsilon(11),\nu(1),(\ref{CD825})},
\end{equation}
where the right hand sides are obtained in Theorem \ref{gluethm12}.
\begin{conds}\label{cond822}
We identify $(\Sigma_1^{(j)},\vec z_1^{(j)})$ with
$(\Sigma_2^{(j)},\vec z_2^{(j)})$ at their $0$-th marked points
to obtain $(\Sigma_{\infty}^{(j)},\vec z^{(j)})$.
Together with $u_{1}^{(j)}$ and $u_{2}^{(j)}$
it gives an element
of $\mathcal M_{g_1+g_2,k_1+k_2}(X,L;\beta)$ which we denote by
${[}(\Sigma_{\infty}^{(j)},\vec z^{(j)}),u_{\infty}^{(j)}{]}$.
We require
\begin{equation}
((\Sigma_{\infty}^{(1)},\vec z^{(1)}),u_{\infty}^{(1)})
\in
\mathcal M_+^{\mathcal E_1^{(2)}\oplus \mathcal E_2^{(2)}}((\Sigma^{(2)}_{\infty},\vec z^{(2)});
u_{1}^{(2)},u_{2}^{(2)})_{\epsilon_{7}/2,
\nu_{2}/2}.
\end{equation}
\end{conds}
Roughly speaking Condition \ref{cond822} means that $u_{i}^{(1)}$ is close to $u_{i}^{(2)}$.
\begin{conds}\label{cond823}
Let $\sigma = (\sigma_1,\sigma_2) \in \mathcal V_{1}^{(1),\R} \times \mathcal V_{2}^{(1),\R}$.
By (\ref{form845})
$\sigma \in \mathcal V_{1}^{(2),\R} \times \mathcal V_{2}^{(2),\R}$.
Using the gluing data $\Xi^{(1)}$, $\Xi^{(2)}$ taken in Situation \ref{shitu821}, 
we obtain a
$\sigma_i$ dependent family of embeddings
\begin{equation}\label{form848}
\psi_{\sigma_i} : 
K_{i}^{(2)} \subset \Sigma_{i}^{(2)} \cong \Sigma_{i}^{(2)}(\sigma_i)
\cong \Sigma_{i}^{(1)}(\sigma_i),
\end{equation}
for $i=1,2$, 
where the first inclusion is given by definition,
the first $\cong$ is the $\sigma_i$ dependent 
family of diffeomorphisms which is a part of the datum
$\Xi^{(2)}$ and the second $\cong$ is the unique biholomorphic map.
\par
We require
\begin{equation}\label{newform849}
\psi_{\sigma_i}(K_{i}^{(2)}) \subseteq
{\rm Int}\,K_{i}^{(1)}
\end{equation}
for $\sigma_i \in \mathcal V_{i}^{(2),\R}$.
\end{conds}
Note that for given $\Xi^{(1)}$, $\Xi^{(2)}$ we can always modify the analytic
family $\varphi_{i,(2)}^{\R}$ of coordinates, which is a part of $\Xi^{(2)}$
so that Condition \ref{cond823} is satisfied.
In fact we may replace $\varphi_{i,(2)}^{\R}$ by its conformal change
$z \mapsto \varphi_{i,(2)}^{\R}(\epsilon z)$ for sufficiently small $\epsilon$.
In other words we may assume Condition \ref{cond823} without loss of generality.
\par
Theorem \ref{them81}, the main result of this chapter, concerns the first vertical arrow
of the next Diagram
(\ref{diag86}).
\begin{lem}\label{lem824new}
There exist
constants
$\epsilon_{8}$,
$\epsilon_{9}$,
$\epsilon_{10,m}$,
$\nu_{3,m}$,
$\nu_{4,m}$,
and
$T_{m,(\ref{diag86})}$
such that
we have the following commutative diagram for any
positive numbers
$\epsilon \le \epsilon_{10,m}$  and
$T(0)\ge T_{m,(\ref{diag86})}$.

\begin{equation}\label{diag86}
\!\!\!\!\!\!\!\!\!\!\!\!\!\!\!\!\!\!
\begin{CD}
\mathcal M^{\mathcal E^{(1)}_1}_+((\Sigma^{(1)}_1,\vec z^{(1)}_1); u_1^{(1)})_{\epsilon}
\,\,{}_{{\rm ev}_0} \times _{{\rm ev}_0}
\atop \qquad \mathcal M^{\mathcal E^{(1)}_2}_+((\Sigma^{(1)}_2,\vec z^{(1)}_2);  u_2^{(1)})_{\epsilon} \times (T(0),\infty]  @ >{
\text{\rm Glue}^{(1)}_{+}}>>
\mathcal M_{+}^{\mathcal E^{(1)}_1\oplus \mathcal E^{(1)}_2}((\Sigma^{(1)}_{\infty},\vec z^{(1)});u^{(1)}_{1},u^{(1)}_{2})
_{\epsilon_{9},\nu_{3,m}}  \\
@ VV{\frak F}V @ VVV\\
\mathcal M^{\mathcal E_1^{(2)}}_+((\Sigma^{(2)}_1,\vec z^{(2)}_1); u_1^{(2)})
_{\epsilon_{8}}
\,\,{}_{{\rm ev}_0} \times _{{\rm ev}_0}
\atop \qquad \mathcal M^{\mathcal E^{(2)}_2}_+((\Sigma^{(2)}_2,\vec z^{(2)}_2); u_2^{(2)};
\beta_2)_{\epsilon_{8}} \times (1/\nu_{4,m},\infty] @ >
{\text{\rm Glue}^{(2)}_{+}} >>
\mathcal M_{+}^{\mathcal E^{(2)}_1\oplus \mathcal E^{(2)}_2}((\Sigma^{(2)}_{\infty},\vec z^{(2)});
u^{(2)}_{1},u^{(2)}_{2})
_{\epsilon_{6},
\nu_{1}}
\end{CD}
\end{equation}
\end{lem}
\begin{proof}
Theorem \ref{gluethm12} implies that  
we can choose $\epsilon_{8}$, $\nu_{4,m}$ so that  the lower horizontal arrow $\text{\rm Glue}^{(2)}_{+}$
exists for given $\epsilon_6$ and $\nu_1$.
Here we use the $W^2_{m+1,\delta}$ space to work out the gluing analysis.
Note the gluing map itself is independent of the Sobolev exponent but the domain of its 
convergence depends on it.
\par
Since $((\Sigma^{(1)}_{\infty},\vec z^{(1)}),u^{(1)}_{\infty})$
is an element of
$\mathcal M_{+}^{\mathcal E^{(2)}_1\oplus \mathcal E^{(2)}_2}((\Sigma^{(2)}_{\infty},\vec z^{(2)}_{\infty});
u^{(2)}_{1},u^{(2)}_{2})
_{\epsilon_{7}/2,
\nu_{2}/2}$
by  Condition \ref{cond822}, we use (\ref{form844})
and Theorem \ref{gluethm12} (1)
to show that we may take the constants
$\epsilon_{9}$,
$\nu_{3,m}$ so small that 
the natural inclusion induces the right vertical arrow of 
(\ref{diag86}) and
$$
\aligned
&{\rm Im}
\Big(
\mathcal M_{+}^{\mathcal E^{(1)}_1\oplus \mathcal E^{(1)}_2}((\Sigma^{(1)}_{\infty},\vec z^{(1)});u^{(1)}_{1},u^{(1)}_{2})
_{\epsilon_{9},\nu_{3,m}}
\Big)
\\
&
\subset
{\rm Im}
\left(
\mathcal M^{\mathcal E_1^{(2)}}_+((\Sigma^{(2)}_1,\vec z^{(2)}_1); u_1^{(2)})
_{\epsilon_{8}}
\,\,{}_{{\rm ev}_0} \times _{{\rm ev}_0}
\atop  \mathcal M^{\mathcal E^{(2)}_2}_+((\Sigma^{(2)}_2,\vec z^{(2)}_2); u_2^{(2)};
\beta_2)_{\epsilon_{8}} \times (1/\nu_{4,m},\infty]
\right)
\endaligned
$$
in 
$\mathcal M_{+}^{\mathcal E^{(2)}_1\oplus \mathcal E^{(2)}_2}((\Sigma^{(2)}_{\infty},\vec z^{(2)});
u^{(2)}_{1},u^{(2)}_{2})
_{\epsilon_{6},
\nu_{1}}
$.
\par
Then, using  Theorem \ref{gluethm12} (1)
(the surjectivity and injectivity of the gluing map),
we may take $\epsilon_{10,m}$
so small and $T_{m,(\ref{diag86})}$ so large that there exists  a unique map
$\frak F$ so that Diagram (\ref{diag86}) commutes.
\end{proof}
{We take 
$m_0 > 10$ and fix it. In the next theorem we take $\epsilon \le \epsilon_{10,m_0}$ 
and consider the map $\frak F$ as in (\ref{diag86}) for $m=m_0$.}
\begin{thm}\label{them81}
There exists $T_{{m_0},\langle\!\langle\ref{them81}\rangle\!\rangle} > 0$ with the following property.
\par
We identify $(T(1),\infty]$  
with
$[0,1/T(1))$ under the map 
$T \mapsto s = 1/T$.
\par
Then the map $\frak F$ in (\ref{diag86}) is a smooth embedding
on
$$
\mathcal M^{\mathcal E^{(1)}_1}_+((\Sigma^{(1)}_1,\vec z^{(1)}_1); u_1^{(1)})_{\epsilon}
\,\,{}_{{\rm ev}_0} \times _{{\rm ev}_0}
\mathcal M^{\mathcal E^{(1)}_2}_+((\Sigma^{(1)}_2,\vec z^{(1)}_2);  u_2^{(1)})_{\epsilon} \times (T(1),\infty] 
$$
if $T(1) > T_{{m_0},\langle\!\langle\ref{them81}\rangle\!\rangle}$.
\end{thm}
\begin{rem}\label{rem818}
We remark that the obstruction bundle data include the data of analytic family of  
coordinates at the
$0$-th marked point, which are used
to define the obstruction bundle $\mathcal E^{(j)}_i(u')$.
This way of defining the obstruction bundle is essential for the right vertical arrow of (\ref{diag86}) to exist.
{In this way the obstruction bundle $\mathcal E^{(j)}_i(u')$ can be made to depend}
only on $u'$, its source (marked bordered) curve, and the obstruction bundle data.
In particular it is independent of the
analytic family of coordinates $\varphi_i^{(j),\R}$, which we use to perform the inductive
construction of the gluing map by the  alternating method.
\par
On the other hand the analytic families of  coordinates at two  chart centers are in general
different from each other. Proposition \ref{prop819}
 below is used to estimate the discrepancy between these two {different} choices.
\end{rem}
\begin{proof}
We observe that a neighborhood of
$(\Sigma^{(j)}_{\infty},\vec z^{(j)}_{\infty}) \in \mathcal{M}_{g_1+g_2,k_1+k_2}$
is parameterized by the map (\ref{form8140}), that is,
\begin{equation}\label{form81420}
{\rm Glusoc}^{(j),\R} : \mathcal V^{(j),\R}_{1} \times \mathcal V^{(j),\R}_{2} \times (T,\infty]
\to \mathcal{CM}_{g_1+g_2,k_1+k_2}.
\end{equation}
So we obtain a map
\begin{equation}\label{form840}
\mathcal M_{+}^{\mathcal E^{(j)}_1\oplus \mathcal E^{(j)}_2}((\Sigma^{(j)}_{\infty},\vec z^{(j)}_{\infty});u^{(j)}_{1},
u^{(j)}_{2})
_{\epsilon,\nu}
\to
\mathcal V^{(j),\R}_{1} \times \mathcal V^{(j),\R}_{2} \times (1/\nu,\infty]
\end{equation}
by forgetting the map part of the stable map and compose it with the inverse of (\ref{form81420}), for $j=1,2$.
\begin{prop}\label{prop819}
There exist $\delta_3>0$,
and a (strata-wise) smooth map
$$
\Phi:
\mathcal V^{(1),\R}_{1} \times \mathcal V^{(1),\R}_{2} \times (1/\nu_{3,m},\infty]
\to
\mathcal V^{(2),\R}_{1} \times \mathcal V^{(2),\R}_{2} \times (1/\nu_{1},\infty]
$$
such that:
\begin{enumerate}
\item
The next diagram commutes.
\begin{equation}\label{diag841}
\begin{CD}
\mathcal M_{+}^{\mathcal E^{(1)}_1\oplus \mathcal E^{(1)}_2}((\Sigma^{(1)}_{\infty},\vec z^{(1)}_{\infty});
u^{(1)}_{1},u^{(1)}_{2})
_{\epsilon_{9},\nu_{3,m}}
 @ >{}>>
\mathcal V^{(1),\R}_{1} \times \mathcal V^{(1),\R}_{2} \times (1/\nu_{3,m},\infty]
  \\
@ VVV @ V{\Phi}VV\\
\mathcal M_{+}^{\mathcal E^{(2)}_1\oplus \mathcal E^{(2)}_2}((\Sigma^{(2)}_{\infty},\vec z^{(2)}_{\infty});
u^{(2)}_{1},u^{(2)}_{2})
_{\epsilon_6,
\nu_1} @ >>>
\mathcal V^{(2),\R}_{1} \times \mathcal V^{(2),\R}_{2} \times (1/\nu_{1},\infty]
\end{CD}
\end{equation}
Here the horizontal arrows are  {the} forgetful maps (\ref{form840}),
the left vertical arrow is the right vertical arrow of Diagram (\ref{diag86}).
\item
We put
$\Phi(\sigma;T) = (\sigma'(\sigma;T),T'(\sigma;T))$.
Then we have the following estimate.
\begin{equation}\label{form842}
\aligned
&\left\vert \nabla_{\sigma}^n \frac{d^{\ell}}{dT^{\ell}} \sigma'(\sigma;T) \right\vert
\le C_{m,(\ref{form842})} e^{-\delta_3 T}, \\
&\left\vert \nabla_{\sigma}^n \frac{d^{\ell}}{dT^{\ell}}
(T'(\sigma;T) - T)
\right\vert
\le C_{m,(\ref{form842})} e^{-\delta_3 T}
\endaligned
\end{equation}
for $m-2 \ge n$, $m-2 \ge \ell  >0$.
\end{enumerate}
\end{prop}
Note $\delta_3$ in this proposition is the same constant as in
Proposition \ref{prop816}.
\begin{proof}
Existence of the map $\Phi$ satisfying (1) is an immediate consequence of the fact
that during the inductive step of the construction of gluing map
in Chapters \ref{alternatingmethod} and \ref{subsecdecayT},
we never change the complex structure of the source.
The estimate (\ref{form842}) will be proved in Section \ref{appendixG}.
\end{proof}
We define\index[syindex]{Resforj@${\rm Resfor}^{(j)}$}
\begin{equation}\label{newnew855}
\aligned
{\rm Resfor}^{(j)}:\,\,
&\mathcal M_{+}^{\mathcal E^{(j)}_1\oplus \mathcal E^{(j)}_2}((\Sigma_{\infty},\vec z_{\infty})
;u_1^{(j)},u_2^{(j)})
_{\epsilon,\nu}
\\
&\to
\mathcal V^{(j),\R}_{1} \times \mathcal V^{(j),\R}_{2} \times (1/\nu,\infty] \\
&\qquad \times
\prod_{i=1}^2\text{\rm Map}_{L^2_{m+1}}((K_i^{(j),S},K_i^{(j),S}\cap\partial \Sigma_i),(X,L))
\endaligned
\end{equation}
where the
$\mathcal V^{(j),\R}_{1} \times \mathcal V^{(j),\R}_{2} \times (1/\nu,\infty]$
component of ${\rm Resfor}^{(j)}$ is the map
(\ref{form840}) and its component of the factor in the third line is the restriction map
and
\begin{equation}
K_i^{(j),S} = \Sigma^{(j)}_i \setminus \varphi_i^{(j), \R}(D^2_{\le 0}(1)).
\end{equation}
We put
$$
\aligned
&\mathcal U^{(j)}_{\epsilon^{(j)},\nu^{(j)},T(2)} \\
&=
(\text{\rm Glue}^{(j)}_{+})^{-1}
\left(\mathcal M_{+}^{\mathcal E^{(j)}_1\oplus \mathcal E^{(j)}_2}((\Sigma
^{(j)}_{\infty},\vec z^{(j)}_{\infty})
;u_1^{(j)},u_2^{(j)})_{\epsilon^{(j)},\nu^{(j)}}
\right)
\cap \pi_3^{-1}((T(2),\infty])
\\
&\subset
\mathcal M^{\mathcal E^{(j)}_1}_+((\Sigma^{(j)}_1,\vec z^{(j)}_1); u_1^{(j)})_{\epsilon^{\prime (j)}}
\,\,{}_{{\rm ev}_0} \times _{{\rm ev}_0}
\mathcal M^{\mathcal E^{(j)}_2}_+((\Sigma^{(j)}_2,\vec z^{(j)}_2);
u_2^{(j)})_{\epsilon^{\prime (j)}} \times (T(2),\infty],
\endaligned
$$
where $(\epsilon^{(1)},\nu^{(1)},\epsilon^{\prime (1)}) = (\epsilon_{9},\nu_{3,m},\epsilon)$
with $\epsilon < \epsilon_{10,m}$,
and $(\epsilon^{(2)},\nu^{(2)},\epsilon^{\prime (2)}) = (\epsilon_{6},\nu_{1},\epsilon_{8})$
and $\pi_3$ is the projection to the $(1/\nu ,\infty]$ factor.
\begin{lem}\label{lem82}
{Let $m' \ge m + n, m\ge 2$.}
We embed
$(T(2),\infty]$  to
$[0,\nu^{(j)})$  by
$T \mapsto s = 1/T$.
Then
\begin{equation}\label{newform856}
{\rm Resfor}^{(j)} \circ \text{\rm Glue}^{(j)}_{+}
: \mathcal U^{(j)}_{\epsilon^{(j)},\nu^{(j)},T(2)}
\to \text{\rm (RHS of (\ref{newnew855}))}
\end{equation}
is {an embedding of $C^{n}$ class} if $T(2) > T_{(\ref{newform856}),{m'}}$.
\end{lem}
\begin{proof}
The composition of ${\rm Resfor}^{(j)} \circ {\text{\rm Glue}^{(j)}_{+}}$ with the
projection to the factor
$\mathcal V^{(j),\R}_{1} \times \mathcal V^{(j),\R}_{2} \times (1/\nu^{(j)},\infty]$  is smooth since it coincides with the
 projection to this factor.
 In other words the next diagram commutes.
 
\begin{equation}\label{diagram88}
\begin{CD}
\mathcal U^{(j)}_{\epsilon^{(j)},\nu^{(j)},T(2)}  @ >>>
\mathcal V^{(j),\R}_{1} \times \mathcal V^{(j),\R}_{2} \times (T(2),\infty]   \\
@ VV{
{{\rm Resfor}^{(j)}} \circ {\rm Glue}^{(j)}_{+}}V @ VVV\\
\displaystyle \mathcal V^{(j),\R}_{1} \times \mathcal V^{(j),\R}_{2}
\times (T(2),\infty]  \times
\atop\displaystyle \times\prod_{i=1}^2\text{\rm Map}_{L^2_{m+1}}((K_i^{(j),S},K_i^{(j),S}\cap\partial \Sigma_i),(X,L)) @ >>>
\mathcal V^{(j),\R}_{1} \times \mathcal V^{(j),\R}_{2} \times (T(2),\infty] ,
\end{CD}
\end{equation}
\par\smallskip
\noindent
where the second vertical arrow is the identity map and
the horizontal arrows are projections.
The commutativity of Diagram (\ref{diagram88}) follows from the fact
that the construction of our map ${\rm Glue}^{(j)}_{+}$ does not
change the complex structure of the source curve.
\par
{We define  the map ${\rm Glue}^{(j)}_{+}$ 
using the $W^2_{m'+1,\delta}$ space.
Its domain is given by $\infty \ge T > T_{m'}$.}
We consider the composition of  ${\rm Resfor}^{(j)}\circ {\text{\rm Glue}^{(j)}_{+}}$
with the projection to
the factor in the third line of the right hand side of (\ref{newnew855}).
It becomes a map
\begin{equation}\label{map857}
\mathcal U^{(j)}_{\epsilon^{(j)},\nu^{(j)},T(2)}\to
\prod_{i=1}^2\text{\rm Map}_{L^2_{m+1}}((K_i^{(j),S},K_i^{(j),S}\cap\partial \Sigma_i),(X,L)).
\end{equation}
{
Then by Theorem \ref{them816} we can show that this map
${\rm Resfor}^{(j)} \circ \text{\rm Glue}^{(j)}_{+}$ is of $C^{n}$ class.
Here we use our assumption $m' \ge m+n$ on the parts $T<\infty$ and $T=\infty$.}
\par
{We next discuss differentiability at $T = \infty$.}
Here we use $s=1/T$ instead of $T$ for the coordinate of the $[0,\nu^{(j)})$ factor.
\begin{sublem}\label{sublem829}
We have
$$
\lim_{s_0 \to 0}
\left\Vert
\left.
\nabla^n_{\sigma,\rho} \frac{d^{\ell}}{ds^{\ell}} (\ref{map857})
\right\vert_{s= s_0}
\right\Vert_{L^2_{m+1}}
= 0,
$$
for {$m' \ge m+n$, $m \ge 2$, $m'-m \ge \ell > 0$}.
\end{sublem}
\begin{proof}
Note
$$
\frac{d^{\ell} }{ds^{\ell} }= \frac{1}{s^{2\ell}}\mathcal Q\left(s,\frac{d}{dT}\right)\frac{d}{dT,}
$$
where
$
\mathcal Q(x,\xi) = \sum_{i=0}^{\ell-1} Q_i(x) \xi^{i}
$
and $Q_i$ are polynomials. Therefore by the exponential decay provided 
in Theorem \ref{them816} we derive
\begin{equation}\label{shinform858}
\left\Vert \left.
\nabla^n_{\sigma,\rho} \frac{d^{\ell}}{ds^{\ell}} (\ref{map857})\right\vert_{s= s_0}
\right\Vert_{L^2_{m+1}}
\le C_{m,(\ref{shinform858})}s_0^{-2\ell} e^{-\delta_2/s_0}.
\end{equation}
The sublemma follows.
\end{proof}
\par
Thus the map (\ref{newform856}) {is of $C^{n}$ class.}
\par
To prove that it is {an} {\it embedding},
we use Diagram (\ref{diagram88}).
In view of its commutativity,
it suffices to show that ${\rm pr} \circ {\rm Resfor}^{(j)} \circ {\rm Glue}^{(j)}_{+}$ is a smooth embedding
when we restrict it to the fiber of an arbitrary point $(\sigma_1,\sigma_2,T)$ in $
 \mathcal V^{(j),\R}_{1} \times  \mathcal V^{(j),\R}_{2} \times (1/\nu^{(j)},\infty]$.
 Here ${\rm pr}$ is the projection to the factor in the third line of (\ref{newnew855}).
 \par
At $T=\infty$ (or $s=0$), the map ${\rm pr} \circ {\rm Resfor}^{(j)} \circ {\rm Glue}^{(j)}_{+}$ is actually
the restriction map, since, there, ${\rm Glue}^{(j)}_{+}$ is obtained by identifying
two stable maps at the $0$-th marked points.
Therefore it is an embedding by unique continuation of
pseudoholomorphic curve.
 \par
On the other hand, (\ref{shinform858})  implies that the restriction of the map
${\rm Resfor}^{(j)} \circ {\rm Glue}^{(j)}_{+}$
to the fiber of $(\sigma_1,\sigma_2,T)$ converges
to its restriction to $(\sigma_1,\sigma_2,\infty)$ in the $C^{1}$ sense as $T$ goes to infinity.
Therefore we obtain {the} required properties for
sufficiently large $T(2)$.
\end{proof}
Now we are in the position to complete the proof of Theorem \ref{them81}.

{
{Let $m_0 \ge 10$ and $m' \ge m_1+n$, $m_1 \ge m_0+n$.}
We consider the composition of the following three maps. The first map to be composed is
$$
\aligned
{\rm Glue}^{(1)}_{+} :~ 
\mathcal M^{\mathcal E^{(1)}_1}_+((\Sigma^{(1)}_1,\vec z^{(1)}_1); u^{(1)}_1)_{\epsilon}
\,\,& {}_{{\rm ev}_0} \times _{{\rm ev}_0} 
\mathcal M^{\mathcal E^{(1)}_2}_+((\Sigma^{(1)}_2,\vec z^{(1)}_2); u^{(1)}_2)_{\epsilon} \times
(T(3),\infty] \\
\to 
& \mathcal M_{+}^{\mathcal E^{(1)}_1\oplus \mathcal E^{(1)}_2}((\Sigma^{(1)}_{\infty},\vec z^{(1)}_{\infty});
u^{(1)}_1,u^{(1)}_2)
_{\epsilon_{9},\nu_{3,{m'}}},
\endaligned
$$
{where $\epsilon < \epsilon_{10,m'}$.}
The second map is the inclusion
$$
\mathcal M_{+}^{\mathcal E^{(1)}_1\oplus \mathcal E^{(1)}_2}((\Sigma^{(1)}_{\infty},\vec z^{(1)}_{\infty});
u^{(1)}_1,u^{(1)}_2)
_{\epsilon_{9},\nu_{3,{m'}}}
\to
\mathcal M_{+}^{\mathcal E^{(2)}_1\oplus \mathcal E^{(2)}_2}
((\Sigma^{(2)}_{\infty},\vec z^{(2)}_{\infty});
u^{(2)}_1,u^{(2)}_2)_{\epsilon_6,\nu_1}.
$$
The third map is
$$
\aligned
{\rm Pr} \circ {\rm Resfor}^{(2)}&:
\mathcal M_{+}^{\mathcal E^{(2)}_1\oplus \mathcal E^{(2)}_2}
((\Sigma^{(2)}_{\infty},\vec z^{(2)}_{\infty});
u^{(2)}_1,u^{(2)}_2)_{\epsilon_6,\nu_1} \\
&\to
\mathcal V^{(2),\R}_{1} \times \mathcal V^{(2),\R}_{2} \times (1/\nu_1,\infty] \\
&\qquad \times
\prod_{i=1}^2\text{\rm Map}_{L^2_{{m_0}+1}}((K_i^{{(2)}},K_i^{{(2)},S}\cap\partial \Sigma_i),(X,L))
\\
&\to
\prod_{i=1}^2\text{\rm Map}_{L^2_{{m_0}+1}}((K_i^{{(2)}},K_i^{{(2)},S}\cap\partial \Sigma_i),(X,L)).
\endaligned
$$
}
\begin{lem}\label{Lemma831}
{
{The composition of the above three maps is of $C^n$ class.}
Its restriction to the fiber of the map
$$
\aligned
&\mathcal M^{\mathcal E^{(1)}_1}_+((\Sigma^{(1)}_1,\vec z^{(1)}_1); u^{(1)}_1)_{\epsilon}
\,\,{}_{{\rm ev}_0} \times _{{\rm ev}_0}
\mathcal M^{\mathcal E^{(1)}_2}_+((\Sigma^{(1)}_2,\vec z^{(1)}_2); u^{(1)}_2)_{\epsilon} \times (T(3),\infty]
\\
&
\to \mathcal V^{(1),\R}_{1} \times \mathcal V^{(1),\R}_{2} \times (1/\nu_1,\infty],
\endaligned
$$
is an embedding
of $C^{n}$-class}
for all $T(3) \ge T_{{m'},\langle\!\langle \ref{Lemma831}\rangle \!\rangle}$.
\end{lem}
{
To prove the lemma we extend the map $\psi_{\sigma_i}$ in (\ref{form848})
including the $T$ parameter.
Let $(\sigma,T) = ((\sigma_1,\sigma_2),T) \in \mathcal V^{(1),\R}_{1} \times \mathcal V^{(1),\R}_{2} \times (1/\nu_1,\infty]$.
We obtain 
$
\Phi(\sigma,T) = (\sigma'(\sigma;T),T'(\sigma;T))
$
as in Proposition \ref{prop819}.
We have
$$
\Sigma_1^{(1),\sigma_1} \#_{T} \Sigma_2^{(1),\sigma_2}
\cong
\Sigma_1^{(2),\sigma'_1} \#_{T'} \Sigma_2^{(2),\sigma'_2}.
$$
Here we use the gluing datum $\Xi^{(1)}$ (resp. $\Xi^{(2)}$) to define 
the left hand side (resp. right hand side).
We consider the composition
$$
\widehat\Psi_{i,\sigma,T} : K^{(2)}_i \subset \Sigma_1^{(2),\sigma'_1} \#_{T'} \Sigma_2^{(2),\sigma'_2}
\cong
\Sigma_1^{(1),\sigma_1} \#_{T} \Sigma_2^{(1),\sigma_2}
$$
where the first map is the canonical embedding.
By definition $\widehat\Psi_{i,\sigma,\infty}$ coincides with the composition 
of 
$\psi_{\sigma_i}$ and the canonical embedding
$: \Sigma_i^{(2)} \to \Sigma_1^{(2),\sigma'_1} \#_{\infty} \Sigma_2^{(2),\sigma'_2}$.
Therefore by Condition \ref{cond823} we may assume that there exists 
a map 
$$
\Psi_{i,\sigma,T} : K^{(2)}_i \to K^{(1)}_i 
$$
such that $\widehat\Psi_{i,\sigma,T}$ is the composition of $\Psi_{i,\sigma,T}$ 
and the canonical embedding 
$K^{(1)}_i \to \Sigma_1^{(1),\sigma_1} \#_{T} \Sigma_2^{(1),\sigma_2}$.}
\par
{
This map is the special case of the map (\ref{map872}) and  the estimate (\ref{form84223})
in Proposition \ref{prof825} holds.
Namely we have
\begin{equation}\label{form8600}
\left\Vert \nabla_{\sigma}^n \frac{d^{\ell}}{dT^{\ell}} \Psi_{i,\sigma,T} \right\Vert_{C^{\ell}
(K_i^{(a)},K_i^{(b)})}
\le C_{n,\ell,(\ref{form8600})} e^{-\delta_4 T}
\end{equation}
for $\ell > 0$.}
\begin{proof}[Proof of Lemma \ref{Lemma831}]
{The composition of three maps is nothing but the composition of the following two maps
(\ref{map860}) and (\ref{map861}).
The first map is the map 
\begin{equation}\label{map860}
\aligned
& {\rm Resfor}^{(1)} \circ \text{\rm Glue}^{(1)}_{+} 
: \\
& \mathcal M^{\mathcal E^{(1)}_1}_+((\Sigma^{(1)}_1,\vec z^{(1)}_1); u^{(1)}_1)_{\epsilon} 
{}_{{\rm ev}_0} \times _{{\rm ev}_0}
\mathcal M^{\mathcal E^{(1)}_2}_+((\Sigma^{(1)}_2,\vec z^{(1)}_2); u^{(1)}_2)_{\epsilon} \times (T(3),\infty]
\\
\to &
\mathcal V^{(1),\R}_{1} \times \mathcal V^{(1),\R}_{2} \times (1/\nu_1,\infty] 
\times
\prod_{i=1}^2\text{\rm Map}_{L^2_{{m_1}+1}}((K_i^{{(1)}},K_i^{{(1)}}\cap\partial \Sigma_i),(X,L)).
\endaligned
\end{equation}

The second map is 
\begin{equation}\label{map861}
\aligned 
\mathcal V^{(1),\R}_{1} \times \mathcal V^{(1),\R}_{2} \times (1/\nu_1,\infty] 
\times&
\prod_{i=1}^2\text{\rm Map}_{L^2_{{m_1}+1}}((K_i^{{(1)}},K_i^{{(1)}}\cap\partial \Sigma_i),(X,L))
\\
\to & 
\prod_{i=1}^2\text{\rm Map}_{L^2_{{m_0}+1}}((K_i^{{(2)}},K_i^{{(2)}}\cap\partial \Sigma_i),(X,L))
\endaligned
\end{equation}
which is defined by
$$
((\sigma_1,\sigma_2,T),(v_1,v_2)) 
\mapsto
(v_1\circ\Psi_{1,\sigma,T},v_2\circ\Psi_{2,\sigma,T}).
$$
Lemma \ref{lem82} in the case $j=1$ implies that the first map (\ref{map860}) is of $C^n$ class.
We can use  (\ref{form8600}) and  $m_1 \ge m_0+n$
to show the second map (\ref{map861}) is also  of $C^n$ class.
Thus the composition of the three maps is of $C^n$ class.}
\par
The fact that it is an embedding on the fiber follows from Lemma \ref{lem82}.
\end{proof}
{Using Proposition \ref{prop819} (especially the estimate (\ref{form842})) and
Lemma \ref{lem82} applied to ${\rm Resfor}^{(2)}\circ {\text{\rm Glue}^{(2)}_{+}}$,
Lemma \ref{Lemma831} imply $\frak F$ in (\ref{diag86}) is of $C^{{n}}$ class.}
\par
{We finally prove that $\frak F$ is of $C^{\infty}$ class.
This part of the proof is a bit sketchy. We refer the reader to
\cite[Section 26]{fooo:techI} or \cite[Section 12]{foooconst1} for detail.
We fix $m_0$.
For $j=1,2$, we take a $C^{\infty}$ structure of 
\begin{equation}\label{from860}
\aligned
&\mathcal M^{\mathcal E^{(j)}_1}_+((\Sigma^{(i)}_1,\vec z^{(j)}_1); u^{(i)}_1)_{\epsilon}
\,\,{}_{{\rm ev}_0} \times _{{\rm ev}_0}
\mathcal M^{\mathcal E^{(j)}_2}_+((\Sigma^{(i)}_2,\vec z^{(j)}_2); u^{(i)}_2)_{\epsilon} \times (T(3),\infty]
\endaligned
\end{equation}
as follows. (Here $\epsilon \le \epsilon_{10,m_0}$.)
We consider the composition 
$$
\aligned
{\rm Resfor}^{(j)} \circ {\rm Glue}^{(j)}_{+} :
(\ref{from860}) 
\to &\mathcal V^{(j),\R}_{1} \times \mathcal V^{(j),\R}_{2} \times (1/\nu_1,\infty] \\
&\times
\prod_{i=1}^2\text{\rm Map}_{L^2_{m_0+1}}((K_i^{(j),S},K_i^{(j),S}\cap\partial \Sigma_i),(X,L)). 
\endaligned
$$
We claim that the image of this map is a submanifold of $C^{\infty}$ class 
of the codomain. 
This is an immediate consequence of elliptic regularity 
on the part $T \ne \infty$. }
\par
{At $T = \infty$ we use Lemma \ref{lem82} for various $m' > m = m_0$ to prove the image of 
${\rm Resfor}^{(j)} \circ {\rm Glue}^{(j)}_{+}$ is smooth there.}
\par
Now we {\it define} the $C^{\infty}$ structure of $(\ref{from860})$ so that 
${\rm Resfor}^{(j)} \circ {\rm Glue}^{(j)}_{+}$ is a 
$C^{\infty}$  embedding.
Note that the smooth structure on $(1/\nu_1, \infty]$ is given so that
$T \in (1/\nu_1, \infty] \mapsto 1/T \in [0, \nu_1)$ is a diffeomorphism.
\par
{Then using this new $C^{\infty}$ structure the composition 
in Lemma \ref{Lemma831} for various $m', m_1$ we can prove that 
$\frak F$ is an embedding of $C^{\infty}$ class as follows.
On the part $T \ne \infty$ it is a consequence of elliptic regularity.
At $T=\infty$ it suffices to show that it is of $C^{n}$ class for any $n$ at an arbitrary point with $T = \infty$.
We take $m_1 = m_0 +n$ and $m' = m_1 + n$ for various $n$, and apply 
Lemma \ref{Lemma831} to show that $\frak F$ is of $C^n$ class there.
The proof of Theorem \ref{them81} is now complete.}
\end{proof}
In a similar way we can prove the smoothness of the Kuranishi map as follows.
\par
Hereafter we write
\begin{equation}
V^{(j)} =
\mathcal M_+^{\mathcal E^{(j)}_1}((\Sigma^{(j)}_1,\vec z^{(j)}_1); u^{(j)}_1)_{\epsilon^{(j)},\nu^{(j)}}
\times_L
\mathcal M_+^{\mathcal E^{(j)}_2}((\Sigma^{(j)}_2,\vec z^{(j)}_2);  u^{(j)}_2)_{\epsilon^{(j)},\nu^{(j)}},
\end{equation}
where $(\epsilon^{(1)},\nu^{(1)},\epsilon^{\prime (1)}) = (\epsilon_{9},\nu_{3,m},\epsilon)$
with $\epsilon < \epsilon_{10}$,
and $(\epsilon^{(2)},\nu^{(2)},\epsilon^{\prime (2)}) = (\epsilon_{6},\nu_{1},\epsilon_{8})$.
\par
We consider the space
$\mathcal M_{+}^{\mathcal E^{(j)}_1\oplus \mathcal E^{(j)}_2}((\Sigma^{(1)}_{\infty},\vec z^{(j)}_{\infty});u^{(j)}_{1},u^{(j)}_{2})
_{\epsilon^{\prime (j)},\nu^{(j)}}$.
We regard the image of
$$
\text{\rm Glue}^{(j)}_{+} : V^{(j)} \times [0,\nu^{(j)}) \to \mathcal M_{+}^{\mathcal E^{(j)}_1\oplus \mathcal E^{(j)}_2}((\Sigma^{(j)}_{\infty},\vec z^{(j)}_{\infty});u^{(j)}_{1},u^{(j)}_{2})
_{\epsilon^{\prime (j)},\nu^{(j)}}
$$
as a smooth manifold so that $\text{\rm Glue}^{(j)}_{+}$is a diffeomorphism.
\par
Let $\frak x = ((\Sigma'_T, \vec z),u') \in \mathcal M_{+}^{\mathcal E^{(j)}_1\oplus \mathcal E^{(j)}_2}((\Sigma^{(j)}_{\infty},\vec z^{(j)}_{\infty});u^{(j)}_{1},u^{(j)}_{2})
_{\epsilon^{\prime (j)},\nu^{(j)}}.
$
The definition of the moduli space implies:
\begin{equation}\label{kuramap}
\overline\partial u' \equiv 0 , \quad \mod \mathcal E_1^{(j)}(u') \oplus  \mathcal E^{(j)}_2(u').
\end{equation}
We write the left hand side as
$\widetilde{\frak s}^{(j)}(\frak x)$.
By the inverse of the map $I_{u',i}$
in (\ref{form816})
we obtain
$$
{\frak s}^{(j)}(\frak x) = (I_{u',1} \oplus I_{u',2})^{-1}(\widetilde{\frak s}^{(j)}(\frak x))
\in \mathcal E^{(j)}_1 \oplus \mathcal E^{(j)}_2.
$$
Note $\mathcal E^{(j)}_1 \oplus \mathcal E^{(j)}_2$ is independent of
$\frak x$.
\par
We thus obtain a map
\begin{equation}\label{map88}
\frak s^{(j)} \circ
\text{\rm Glue}^{(j)}_{+} : V^{(j)} \times [0,\nu^{(j)})
\to \mathcal E^{(j)}_1\oplus \mathcal E^{(j)}_2.
\end{equation}
This is by definition the {\em Kuranishi map}. \index{Kuranishi map}
\begin{prop}\label{propo8484}
The map $\frak s^{(j)}$ in (\ref{map88}) is smooth.
\end{prop}
\begin{proof}
The smoothness at each of the stratum ($T=\infty$ and $T\ne \infty$)
follows from elliptic regularity.
Note the supports of elements of $\mathcal E_i^{(j)}(u')$ are in
the image of $K_i^{\frak{ob}}$.
Therefore we can use  (\ref{form672}) to show that all the derivatives
of $\frak s_i$ including at least one $s = 1/T$ derivatives vanishes at $s=0$
in the same way as in the proof of Sublemma \ref{sublem829}.
Furthermore (\ref{form672}) implies that the restriction of  $\frak s^{(j)}$ to
$s = s_0$ converges to its restriction to $s = 0$ in the $C^{{n}}$ sense for any ${n}$ as
$s_0$ goes to $0$.
The proof of Proposition \ref{propo8484} is complete.
\end{proof}
{We are finally ready to provide a construction of the Kuranishi neighborhoods,
the associated Kuranishi maps and parametrization maps, and to wrap-up the proof of smoothness of
coordinate change maps.}
We put
$\hat V^{(j)} =
V^{(j)} \times [0,\nu^{(j)})
$
and $\widehat{\mathcal E}^{(j)} = \mathcal E^{(j)}_1\oplus \mathcal E^{(j)}_2$.
Then we can find $\psi^{(j)}$ so that the {quintuple}
$
(\hat V^{(j)},\widehat{\mathcal E}^{(j)} ,\{1\},\psi^{(j)},\frak s^{(j)} )
$
becomes a Kuranishi neighborhood of the moduli space of stable bordered maps
$\mathcal M_{g_1+g_2,k_1+k_2}(X,L;\beta_1+\beta_2)$
in the sense of \cite[Definition A1.1]{fooo:book1}.
\par
In fact, the moduli space $\mathcal M_{g_1+g_2,k_1+k_2}(X,L;\beta_1+\beta_2)$
(See Definition \ref{isomorphicmaps})
is locally identified with the zero set of
$\frak s^{(j)} $.
This fact follows from the `injectivity' and `surjectivity' we proved
in Chapter \ref{surjinj}.  Therefore we obtain our parametrization map
$$
\psi^{(j)} : (\frak s^{(j)})^{-1}(0) \to \mathcal M_{g_1+g_2,k_1+k_2}(X,L;\beta_1+\beta_2).
$$
Moreover
the group of automorphisms of the objects in our neighborhood
in the moduli space $\mathcal M_{g_1+g_2,k_1+k_2}(X,L;\beta_1+\beta_2)$ is trivial.
Hence we can put $\Gamma =\{1\}$.
\par
Now we prove:
\begin{thm}\label{thm822}
There exists  a
{\it smooth} coordinate change
$$
(\phi_{21},\hat{\phi}_{21},{\rm id}) : (\hat V^{(1)},\widehat{\mathcal E}^{(1)},\{1\},\psi^{(1)},\frak s^{(1)}) \to (\hat V^{(2)},\widehat{\mathcal E}^{(2)},\{1\},\psi^{(2)},\frak s^{(2)})
$$
of Kuranishi neighborhoods in the sense of \cite[Definition A.1.3]{fooo:book1}.
\end{thm}
\begin{proof}
The map $\phi_{21} : \hat V^{(1)} \to \hat V^{(2)}$ is {nothing but} 
the embedding $\frak F$
in Diagram (\ref{diag86}).
This map is a smooth embedding by Theorem \ref{them81}.
\par
The bundle map $\hat{\phi}_{21} : \mathcal E^{(1)} \to \mathcal E^{(2)}$ is obtained
from (\ref{form844}), that is,
$ \mathcal E^{(1)}_1\oplus \mathcal E^{(1)}_2
\subset  \mathcal E^{(2)}_1\oplus \mathcal E^{(2)}_2$, as follows.
\par
If ${\frak x} \in \hat V^{(1)}$ and
$u'$ is the map part of 
$$
\text{\rm Glue}^{(1)}_{+}({\frak x}) \in \mathcal M_{+}^{\mathcal E^{(1)}_1\oplus \mathcal E^{(1)}_2}((\Sigma^{(1)}_{\infty},\vec z^{(1)}_{\infty});u^{(1)}_1,u^{(1)}_2)_{{\epsilon,\nu}}
$$
then
$$
\mathcal E_1({\frak x})
= \mathcal E_1^{(1)}(u') \oplus \mathcal E_2^{(1)}(u')
\subset
\mathcal E_1^{(2)}(u') \oplus \mathcal E_2^{(2)}(u').
$$
Note since $u'$ is also the map part of $\text{\rm Glue}^{(2)}_{+}(\phi_{21}({\frak x}))$, we have
$$
\mathcal E_2 (\text{\rm Glue}^{(2)}_{+}(\phi_{21}({\frak x}))) =
\mathcal E_1^{(2)}(u') \oplus \mathcal E_2^{(2)}(u').
$$
The fiber of $\hat{\phi}_{21}$ at ${\frak x}$ is the
obvious inclusion
$\mathcal E_i^{(1)}({\frak x})
\subset \mathcal E_i^{(2)} (\text{\rm Glue}^{(2)}_{+}(\phi_{21}({\frak x})))$.
\par
The smoothness of $\hat\phi_{21}$ is proved in the same way as
in Theorem \ref{them81} and Proposition \ref{propo8484}
using the fact that the supports of elements of $\mathcal E_i^{\frak{ob}}$ are
in $K^{\frak{ob}}_i$.
\par
The group homomorphism ${\rm id} : \{1\} \to \{1\}$ is the identity map.
Various commutativities of compositions of maps required in \cite[Definition A.1.3]{fooo:book1} are trivial
to check in our case.
\end{proof}
\begin{rem}
In this paper we assume that $(\Sigma_i,\vec z_i)$ is stable.
For the unstable case we need to add marked points
$\vec z_i^+$
and use the slices of codimension 2 submanifolds of $X$ to reduce the construction
of the Kuranishi chart to the case when the source is stable.
(We use the slices that are transversal to the map
$u_i : \Sigma_i \to X$ at $\vec z_i^+$, to cut down the moduli space
to one of correct dimension. See \cite[Appendix]{FOn} and \cite[Part IV]{fooo:techI}.)
\par
We also need to show that the change of the choices of extra marked points
and codimension 2 submanifolds, induces a {\it smooth}  coordinate change.
Once the
estimates (\ref{form672}) and (\ref{form842})
are established the rest of the
proof of this statement is rather geometric than analytic.
So we do not discuss this point in this paper, whose focus lies in the
analytic part of the story. See \cite[page 772]{fooo:book1},
\cite[Part IV]{fooo:techI}  and {\cite[Sections 10,12]{foooconst1}} for the argument on this point.
\end{rem}

\subsection[Comparison between two analytic families of coordinates]{Comparison between two choices of analytic families of coordinates}
\label{appendixG}
In this section we prove Propositions \ref{prop816} and \ref{prop819} (2).
\par
The main part of the proof is  Proposition \ref{prof825} below.
We first need to set up notations to state it.
Let $(\Sigma_i^{(j)},\vec z_i^{(j)}) \in \mathcal M_{g_i,k_i+1}$, $i =1,2$, $j =a,b$.
We choose gluing data $\Xi^{(j)} = (\Xi^{(j)}_1,\Xi^{(j)}_2)$
centered at
$([\Sigma_1^{(j)},\vec z_1^{(j)}],[\Sigma_2^{(j)},\vec z_2^{(j)}])$
for $j=a$ and $j=b$.
They induce the source gluing maps:
\begin{equation}\label{form8142}
\aligned
&{\rm Glusoc}^{\R}_{\Xi^{(a)}} : \mathcal V^{(a,+),\R}_{1} \times \mathcal V^{(a,+),\R}_{2} \times (T_0,\infty]
\to \mathcal{CM}_{g_1+g_2,k_1+k_2}, 
\\
&{\rm Glusoc}^{\R}_{\Xi^{(b)}} : \mathcal V^{(b),\R}_{1} \times \mathcal V^{(b),\R}_{2} \times (T_0,\infty]
\to \mathcal{CM}_{g_1+g_2,k_1+k_2}
\endaligned
\end{equation}
by Definition-Lemma \ref{deflem87}.
Here $\mathcal V^{(a,+),\R}_{i}$, $\mathcal V^{(b),\R}_{i}$ are open neighborhoods of
$[\Sigma_i^{(a)},\vec z_i^{(a)}]$, $[\Sigma_i^{(b)},\vec z_i^{(b)}]$ in $\mathcal M_{g_i,k_i+1}$, which are parts
of $\Xi^{(a)}_i$, $\Xi^{(b)}_i$, respectively.
\begin{rem}
In  (\ref{form8140}) the third factor of the domain is $[0,\epsilon)$.
It is related to $(T_0,\infty]$ by $T \mapsto r = e^{-10\pi T}$.
See (\ref{form812}).
\end{rem}
We put:
\begin{equation}\label{asform864}
\mathcal V^{(a),\R}_{i} = \mathcal V^{(a,+),\R}_{i} \cap \mathcal V^{(b),\R}_{i},
\end{equation}
and regard it as an open subset of $\mathcal V^{(a,+),\R}_{i}$.
Then we may choose $T_j$ and a map $\Phi$ so that the next diagram commutes.
\begin{equation}\label{diag841222}
\begin{CD}
\mathcal V^{(a),\R}_{1} \times \mathcal V^{(a),\R}_{2} \times (T_{a},\infty]
 @ >{{\rm Glusoc}_{\Xi^{(a)}}^{\R}}>>
\mathcal{CM}_{g_1+g_2,k_1+k_2}  \\
@ V{\Phi}VV @ V{\rm id}VV\\
\mathcal V^{(b),\R}_{1} \times \mathcal V^{(b),\R}_{2} \times (T_{b},\infty]
@ >{{\rm Glusoc}_{\Xi^{(b)}}^{\R}}>>
\mathcal{CM}_{g_1+g_2,k_1+k_2}
\end{CD}
\end{equation}
Let $K_i^{(j)} = \Sigma_i^{(j)} \setminus {\rm Im}(\varphi^{(j),\R}_i)$.
For $\sigma \in \mathcal V^{(a),\R}_{i}$
the canonical embedding associated to $\Xi^{(a)}$ induces
\begin{equation}
\frak I^{\sigma}_{\Xi^{(a)},i} : K_i^{(a)} \to \Sigma_1^{(a),\sigma} \#_{T} \Sigma_2^{(a),\sigma}.
\end{equation}
On the other hand, the canonical embedding associated to $\Xi^{(b)}$ induces
\begin{equation}
\frak I^{\sigma'}_{\Xi^{(b)},i} :  K_i^{(b)} \subset \Sigma_1^{(b),\sigma'} \#_{T} 
\Sigma_2^{(b),\sigma'}.
\end{equation}
Let $(\sigma',T') = \Phi(\sigma,T)$,
that is
$$
\Sigma_1^{(a),\sigma} \#_{T} \Sigma_2^{(a),\sigma}
\cong
\Sigma_1^{(b),\sigma'} \#_{T'} \Sigma_2^{(b),\sigma'}.
$$
{Next} we assume {the inclusion}
\begin{equation}\label{assform868}
\frak I^{\sigma}_{\Xi^{(a)},i}(K_i^{(a)}) \subset \frak I^{\sigma'}_{\Xi^{(b)},i}(K_i^{(b)}).
\end{equation}
We then obtain a smooth embedding
\begin{equation}\label{map872}
\Psi_{{i},\sigma,T} = (\frak I^{\sigma'}_{\Xi^{(b)},i})^{-1}\circ
 \frak I^{\sigma}_{\Xi^{(a)},i}: K_i^{(a)} \to K_i^{(b)}.
\end{equation}
This defines a $(\sigma,T)$-family of maps which we regard as the map
\begin{equation}\label{map854}
\Psi_i :
\mathcal V^{(a),\R}_{1} \times \mathcal V^{(a),\R}_{2} \times (T_a,\infty]
\times K_i^{(a)} \to K_i^{(b)}.
\end{equation}
\begin{prop}\label{prof825}
There exists $\delta_4 > 0$ with the following properties.
\begin{enumerate}
\item
Consider the map $\Phi$ defined by
$\Phi(\sigma;T) = (\sigma'(\sigma;T),T'(\sigma;T))$.
Then we have the following estimate.
\begin{equation}\label{form8422}
\aligned
&\left\vert \nabla_{\sigma}^n \frac{d^{\ell}}{dT^{\ell}} \sigma'(\sigma;T) \right\vert
\le C_{{n},\ell,(\ref{form8422})} e^{-\delta_4 T}, \\
&\left\vert \nabla_{\sigma}^n \frac{d^{\ell}}{dT^{\ell}}
(T'(\sigma;T) - T)\right\vert
\le C_{{n},\ell,(\ref{form8422})} e^{-\delta_4 T}
\endaligned
\end{equation}
for  $\ell  >0$.
\item
We assume (\ref{assform868}).
Let $\Psi_{{i}}$ be the map given in \eqref{map854}. Then
\begin{equation}\label{form84223}
\left\Vert \nabla_{\sigma}^n \frac{d^{\ell}}{dT^{\ell}} \Psi_{{i}} \right\Vert_{C^{\ell}
(K_i^{(a)},K_i^{(b)})}
\le C_{{n},\ell,(\ref{form84223})} e^{-\delta_4 T}
\end{equation}
for $\ell  >0$.
\end{enumerate}
\end{prop}
\begin{proof}
We first prove (\ref{form8422}).
By taking {the} double of (\ref{diag841222}) we obtain the next diagram.
\begin{equation}\label{diag841222c}
\begin{CD}
\mathcal V^{(a)}_{1} \times \mathcal V^{(a)}_{2} \times D^2(\epsilon)
 @ >{{\rm Glusoc}^{\C}_{\Xi^{(a)}}}>>
\mathcal{CM}^{\rm cl}_{g_1+g_2,k_1+k_2}  \\
@ V{\Phi^{\C}}VV @ V{\rm id}VV\\
\mathcal V^{(b)}_{1} \times \mathcal V^{(b)}_{2} \times D^2(\epsilon)
@ >{{\rm Glusoc}^{\C}_{\Xi^{(b)}}}>>
\mathcal{CM}_{g_1+g_2,k_1+k_2}
\end{CD}
\end{equation}
Since $\varphi_{i}^{(j),\R}$ are analytic coordinates
${\rm Glusoc}^{\C}_{\Xi^{(a)}}$ and ${\rm Glusoc}^{\C}_{\Xi^{(b)}}$
are holomorphic.
Therefore $\Phi^{\C}$ is also holomorphic.
We remark that
$$
\Phi^{\C}(\mathcal V^{(a)}_{1} \times \mathcal V^{(a)}_{2} \times \{0\})
\subseteq
(\mathcal V^{(b)}_{1} \times \mathcal V^{(b)}_{2} \times \{0\}).
$$
\par
Therefore there exist holomorphic maps
$\mathcal F  : \mathcal V^{(a)}_{1} \times \mathcal V^{(a)}_{2}
\to \mathcal V^{(b)}_{1} \times \mathcal V^{(b)}_{2} $
and
$\mathcal G = (\mathcal G_1,\mathcal G_2) : \mathcal V^{(a)}_{1} \times \mathcal V^{(a)}_{2}
\times D^2(1)
\to \C^{\dim \mathcal V^{(b)}_{1} + \dim \mathcal V^{(b)}_{2}} \times \C
$
such that
$$
\Phi^{\C}(\sigma_1,\sigma_2,\frak r)
=
(\mathcal F(\sigma_1,\sigma_2),0) + \frak r\mathcal G(\sigma_1,\sigma_2,\frak r)
$$
Moreover $\mathcal G_2(\sigma_1,\sigma_2,0) \ne 0$.
Here we regard $\mathcal V^{(j)}_{i}$ as an open set of $\C^{\dim \mathcal V^{(j)}_{i}}$.
\par
Note $\Phi$ is a restriction of $\Phi^{\C}$ where the 
$T$ coordinate is identified  with a part of the standard coordinate $\frak r$ of $D^2({\epsilon})$
by the map $T \mapsto \frak r = -e^{-10\pi T} \in D^2({\epsilon})$.
(See Diagram (\ref{diag862}) and Formula (\ref{form812}).)
Therefore
$$
\aligned
\sigma'(\sigma,T) &= \mathcal F(\sigma) + e^{-10\pi T}\mathcal G_1(\sigma,-e^{-10\pi T}),
\\
T'(\sigma,T) &=
T - \log \mathcal G_2(\sigma,-e^{-10\pi T}) / 10\pi.
\endaligned
$$
Here $\sigma = (\sigma_1,\sigma_2)$.
Now, {the estimate (\ref{form8422}) follows from the holomorphicity of $\mathcal F$ and $\mathcal G_1,\mathcal G_2$}.
\par
We next prove (2). By taking the double of (\ref{map854}) we obtain a map:
\begin{equation}\label{map854c}
\Psi^{\C}_{{i}} :
\mathcal V^{(a)}_1 \times \mathcal V^{(a)}_2 \times D^2({\epsilon})
\times K_i^{(a)} \to K_i^{(b)}.
\end{equation}
Using the trivialization of the universal bundle which is a part of the data
$\Xi^{(j)}_i$ ($j=a,b$) we identify
$$
\mathcal V^{(j)}_i \times K_i^{(j)} \subset
\mathcal C(\mathcal V^{(j)}_i).
$$
We use this embedding to define a complex structure of
$\mathcal V^{(j)}_1 \times \mathcal V^{(j)}_2 \times K_i^{(j)}$.
\par
We consider the map
$$
\Phi^{\C} \times \Psi^{\C}_{{i}} :
\mathcal V^{(a)}_1 \times \mathcal V^{(a)}_2 \times D^2({\epsilon})
\times K_i^{(a)}
\to
\mathcal V^{(b)}_1 \times \mathcal V^{(b)}_2 \times D^2({\epsilon})
\times K_i^{(b)}.
$$
The map $\Phi^{\C} \times \Psi^{\C}_{{i}}$ is holomorphic
with respect to the above complex structures because
the left (resp. right) hand side is an open set of
the universal family $\mathcal C(\mathcal V^{(a)}_1 \times \mathcal V^{(a)}_2 \times D^2({\epsilon}))$
(resp. $\mathcal C(\mathcal V^{(b)}_1 \times \mathcal V^{(b)}_2 \times D^2({\epsilon}))$).
Note neither the left nor the right hand side is the direct product
as a complex manifold with respect to this complex structure.
\par
On the other hand, the projection
$$
{\rm Pr} :
\mathcal V^{(b)}_1 \times \mathcal V^{(b)}_2 \times D^2({\epsilon})
\times K_i^{(b)}
\to K_i^{(b)}
$$
is certainly of $C^{\infty}$ class. Now we have
\begin{equation}
\Psi_{{i}}(\sigma,T,z)
=
{\rm Pr}(\Phi^{\C}(\sigma,
-e^{-10\pi T})
, \Psi_{{i}}^{\C}(\sigma,
-e^{-10\pi T},z)
).
\end{equation}
Using this holomorphicity of $(\Phi^{\C},\Psi_{{i}}^{\C})$ and smoothness of
${\rm Pr}$
we obtain the estimate (\ref{form84223}).
\end{proof}
\begin{proof}[Proof of Proposition \ref{prop819}]
Statement (1) is already proved.
Replacing $a,b$ by $1,2$ we apply Proposition \ref{prof825} (1).\footnote{During the proof of (\ref{form8600})
we apply Proposition \ref{prof825} (2)
replacing $a$,$b$ by 2,1. We then change the variables by the map obtained from 
 Proposition \ref{prof825} (1)
replacing $a$,$b$ by 1,2.}
Note we use Proposition \ref{prof825} (1) only here. So (\ref{assform868}) is not 
necessary. Then (\ref{form842}) follows from (\ref{form8422}).
\end{proof}
\begin{proof}[Proof of Proposition \ref{prop816}]
The map
$I^{u_i}_{u_i^{\sigma_i}} \circ I_{u'}^{u_i^{\sigma_i}} \circ I_{u',i}$
(see (\ref{form827plus})), which we use to define
${\bf e}'_{i,a,(\kappa)}(\sigma,\rho,T)$ in (\ref{newname840}),
is a tensor product of two $\C$-linear maps:
one is a map between $\Lambda^{0,1}$ bundles:
the other is a map between sections of the pull back bundles of $TX$.
We denote the former by $I_{\Sigma}$ and the latter by $I_{X}$.
\par
$I_{\Sigma}$ is a composition of (\ref{form818}),
(\ref{map828}) and (\ref{map835}).
Note this map does not depend on $u'$ but depends only on
$\sigma$ and $T$, (which determine the source curve of $u'$).
(\ref{form818}) and (\ref{map835}) are independent of $T$ and depend smoothly
on $\sigma$.
\par
We apply Proposition \ref{prof825} to
$\Sigma_i^{(a)} = \Sigma_i^{\frak{ob}}$, etc.
and $\Sigma_i^{(b)} = \Sigma_i$ etc..
We can then apply (\ref{form84223})
to estimate the map $\Psi_{{i},\rho,T}$ below.
$$
\Psi_{{i},\rho,T} = (\frak I^{\sigma}_{\Xi_i})^{-1} \circ \frak I^{\sigma'}_{\Xi^{\frak{ob}}_i}:
K_i^{\frak{ob}} \to K_i.
$$
We then obtain
\begin{equation}\label{form861}
\left\Vert \nabla_{\sigma,\rho}^n\frac{\partial^{\ell}}{\partial T^{\ell}}\Psi_{{i},\rho,T}
\right\Vert_{C^m}
\le
C_{m,(\ref{form861})} e^{-\delta_4 T},
\end{equation}
for $m \ge n$, $m \ge \ell  >0$.
Note (\ref{assform868}) follows from
Condition \ref{cond813} (3).
\par
We remark that in the definition of (\ref{map828}) the process of pulling back the differential form
by $\Psi_{\rho,T}$ is included. This is usually a difficult process to study in Sobolev spaces.
However in our situation we apply it to smooth forms ${\bf e}_{i,a}$ which are
fixed during the construction. So we can use (\ref{form861}) to deduce the next
inequality
\begin{equation}\label{form862}
\left\Vert \nabla_{\sigma,\rho}^n\frac{\partial^{\ell}}{\partial T^{\ell}}I_{\Sigma}({\bf e}_{i,a})
\right\Vert_{C^m}
\le
C_{m,(\ref{form862})} e^{-\delta_4 T}.
\end{equation}
\par
We next discuss $I_{X}$. By definition $I_X$ is induced by the composition
of the parallel transportations (\ref{form817}), (\ref{map8290})
and (\ref{map829}).
In the case when $u' =u^{\sigma,\rho}_{T,(\kappa)}$
we can estimate it by using  induction hypothesis
(that is, the version of (\ref{form184}) including $\sigma$ derivatives).
We obtain this estimate in the same way as in the proof of
Lemma \ref{lemma615} given in Appendix \ref{appendixA2bisbis}.
\par
Thus together with (\ref{form862}) we obtain the required estimate
(\ref{prop816sformula}).
\end{proof}
\begin{rem}
Proposition \ref{prof825} (1) is a {variation} of \cite[Proposition 16.11]{fooo:techI}
and
Proposition \ref{prof825} (2) is a {variation} of \cite[Proposition 16.15]{fooo:techI}.
Actually the assumption{s} of \cite[Proposition 16.11]{fooo:techI}
and of \cite[Proposition 16.15]{fooo:techI} are weaker than
that of Proposition \ref{prof825}.
Namely in \cite{fooo:techI} we studied a {\it smooth} family of
coordinates at the 0-th marked point.
Here we consider an {\it analytic} family of
coordinates at the 0-th marked point.
\par
The proof of \cite[Propositions 16.11 and 16.15]{fooo:techI}
is given in \cite[Section 25]{fooo:techI} and uses a method similar to the
proof of Theorems \ref{gluethm1} and \ref{exdecayT} of this paper to
find a biholomorphic map between Riemann surfaces with appropriate estimate.
In other words it is based on the study of non-linear partial differential
equation.
The proof of Proposition \ref{prof825} we provide in this chapter is
based on the complex geometry and is shorter than \cite[Section 25]{fooo:techI}.
\par
In case when the almost complex structure of the target space $X$ is integrable,
we may prove a similar estimate as Theorem \ref{exdecayT}  for the
moduli space of stable maps without boundary, by using the complex
geometry in a similar way as the proof of Proposition \ref{prof825}.
Namely we may use the existence of the universal family of stable maps
(with obstruction bundles) in the complex analytic category and
translate the complex analyticity of this moduli space into an
exponential decay estimate.\footnote{Because we need to
study the case when obstruction bundle is present, it is nontrivial
to work out the proof of exponential decay in this way.}
\par
In the situation of Proposition \ref{prof825}, the target space is not
involved. So all the complex structures involved are integrable.
This is the reason  we can use the complex geometry to find a shorter proof.
Since for the purpose of this paper (and all the applications we can see at this
stage), we can always restrict ourselves to an analytic family of coordinates,
we provide this shorter proof in this paper.
\par
The existence of smooth coordinate change between Kuranishi charts in this chapter
and those in \cite{fooo:techI}  can be proved by using
\cite[Propositions 16.11 and 16.15]{fooo:techI}.
\par
In the genus $0$ case, a result corresponding to Proposition \ref{prof825} is proved in
\cite[Lemma A1.59]{fooo:book1}. The proof there uses hyperbolic geometry
and is different from both of this paper and \cite{fooo:techI}.
\end{rem}
\begin{rem}
As we mentioned already the proof of Proposition \ref{prop816} is the
main extra point in the proof of Theorem \ref{gluethm12}
other than those appearing in the proof of Theorems \ref{gluethm1} and \ref{exdecayT}.
In \cite{fooo:techI}, this point was discussed in detail as the proof of  \cite[Lemma 19.14]{fooo:techI}.
\par
The contents of this chapter are taken from \cite[Part 4]{fooo:techI}.
The contents of other chapters of this paper are taken from \cite[Part 3]{fooo:techI}.
\end{rem}
\appendix

\section{Error term estimate of non-linear Cauchy-Riemann equation I}
\label{appendixA}

Let $\Omega$ be an open subset of a bordered Riemann surface $\Sigma$
and $u_1 : (\Omega,\Omega\cap \partial \Sigma) \to (X,L)$  a
pseudoholomorphic map.
Consider two smooth sections
$V^0,W^0 \in \Gamma(\Omega,u_1^*TX)$ such that
their restrictions to $\Omega\cap \partial \Sigma$
are in $\Gamma(\Omega\cap \partial \Sigma,u_1^*TX)$.
We study the maps
$$
u(z) = {\rm Exp}(u_1(z),W^0(z))
$$
and
\begin{equation}\label{A1000}
v_r(z) = {\rm Exp}(u(z),r{\Pal}_{u_1}^u(V^0)(z)).
\end{equation}
\par
We take a trivialization
of $u_1^*TX$ on $\Omega$
and identify
$$
u_1^*TX \cong \Omega \times \R^n.
$$
We write an element (of the total space of) $u_1^*TX$ as $(z,(\xi_1,\dots,\xi_n))$.
So $V^0$, $W^0$ are regarded as $(V^0_j)_{j=1}^n : \Omega \to \R^n$,  $(W^0_j)_{j=1}^n : \Omega \to \R^n$.
Let $z = x+ \sqrt{-1}y$ be a complex coordinate of $\Omega$.
\par
Let $R>0$ be a number smaller than $\iota'_X/10$.
We denote by $D^n(R)$ the ball of radius $R$ centered at $0$ in $\R^n$.
\par
We define $\hat F : \Omega \times  D^n(R) \times  D^n(R) \to X$ by
$$
\hat F(z,\frak v,\frak w)
= {\rm Exp}\left({\rm Exp}(u_1(z),\frak w),{\rm Pal}_{u_1(z)}^{{\rm Exp}(u_1(z),\frak w)}(\frak v)\right)
$$
where $\Omega \times \R^n \times \R^n \supset \Omega \times  D^n(R) \times  D^n(R)$
is identified with the total space of
the direct sum bundle $u_1^*(TX \oplus TX)$.
See Figure \ref{FigureinA}.
\par
We also define $F : \Omega \times  D^n(R) \times  D^n(R) \to \R^n$ by
$$
F(z,\frak v,\frak w) = E(u_1(z),\hat F(z,\frak v,\frak w)).
$$
\begin{figure}
\centering
\includegraphics{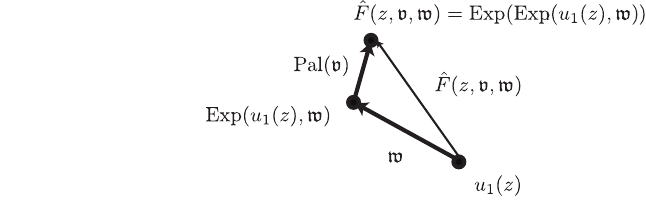}
\caption{$F(z,\frak v,\frak w)$ and $\hat F(z,\frak v,\frak w)$}
\label{FigureinA}
\end{figure}

We denote
$$
\aligned
\mathcal{PP}(z,\frak v,\frak w) =
(\left({\rm Pal}^{{\rm Exp}(u_1(z),\frak w)}_{u_1(z)}\right)^{(0,1)})^{-1}
\circ
&\left(\left({\rm Pal}^{\hat F(z,\frak v,\frak w)}_{{\rm Exp}(u_1(z),\frak w)}\right)^{(0,1)}\right)^{-1}
\\
&:
T_{\hat F(z,\frak v,\frak w)}X\otimes\Lambda^{0,1} \Omega \to T_{u_1(z)}X
\otimes\Lambda^{0,1} \Omega.
\endaligned
$$
We remark
$$
v_r(z) = \hat F(z,rV^0(z),W^0(z)).
$$
We study
\begin{equation}\label{formA4}
\mathcal{PP}(z,rV^0(z),W^0(z))(\overline{\partial} v_r).
\end{equation}

\begin{lem}\label{lemA1}
There exist smooth maps
$$
\aligned
&\frak G : \Omega \times D^n(R) \times D^n(R) \to \R^n,
\\
&\frak H^x_{\frak v,j}, \frak H^x_{\frak w,j},\frak H^y_{\frak v,j}, \frak H^y_{\frak w,j}
 : \Omega \times D^n(R)\times D^n(R) \to \R^n
 \endaligned
$$
$j=1,\dots,n$ such that
\begin{equation}\label{maineqationddd}
\aligned
({\rm \ref{formA4}})
= \frak G(z,rV^0(z),W^0(z)) &+ r \sum_j \frac{\partial V^0_j}{\partial x}(z) \frak H^x_{\frak v,j}(z,rV^0(z),W^0)
\\&+
\sum_j \frac{\partial W^0_j}{\partial x}(z) \frak H^x_{\frak w,j}(z,rV^0(z),W^0(z))
\\
 &+ r \sum_j \frac{\partial V^0_j}{\partial y}(z) \frak H^y_{\frak v,j}(z,rV^0(z),W^0(z))
\\
&+
\sum_j \frac{\partial W^0_j}{\partial y}(z) \frak H^y_{\frak w,j}(z,rV^0(z),W^0(z)).
\endaligned
\end{equation}
\end{lem}
We remark that the maps $\frak G$,
$\frak H^x_{\frak v,j}, \frak H^x_{\frak w,j},\frak H^y_{\frak v,j}, \frak H^y_{\frak w,j}$
depend only on $\Omega$, $u_1$, $X$  and are independent of $V^0$, $W^0$.
\begin{proof}
We emphasize that we do not need to obtain an explicit form of the smooth maps appearing
in (\ref{maineqationddd}).
\par
We first observe that there exists an $n\times n$ matrix valued smooth map
$H$ on $\Omega \times  D^n(R) \times  D^n(R)$
such that
\begin{equation}\label{formA6}
\aligned
\frac{\partial F(z,\frak v,\frak w)}{\partial x}
&=
H(z,\frak v,\frak w) \left(
\mathcal{PP}(z,\frak v,\frak w)\Big(\frac{\partial \hat F(z,\frak v,\frak w)}{\partial x}\Big)\right) \\
&\quad + \left.\frac{\partial}{\partial x}{\rm E}(u_1(z),\xi)\right\vert_{\xi = \hat F(z,\frak v,\frak w)} \\
\frac{\partial F(z,\frak v,\frak w)}{\partial y}
&=
H(z,\frak v,\frak w) \left(
\mathcal{PP}(z,\frak v,\frak w)\Big(\frac{\partial \hat F(z,\frak v,\frak w)}{\partial y}\Big)\right)\\
&\quad + \left.\frac{\partial}{\partial y}{\rm E}(u_1(z),\xi)\right\vert_{\xi = \hat F(z,\frak v,\frak w)} \\
\\
\frac{\partial F(z,\frak v,\frak w)}{\partial \frak v_j}
&=
H(z,\frak v,\frak w) \left(
\mathcal{PP}(z,\frak v,\frak w)\Big(\frac{\partial \hat F(z,\frak v,\frak w)}{\partial \frak v_j}\Big)\right) \\
\frac{\partial F(z,\frak v,\frak w)}{\partial \frak w_j}
&=
H(z,\frak v,\frak w) \left(
\mathcal{PP}(z,\frak v,\frak w)\Big(\frac{\partial \hat F(z,\frak v,\frak w)}{\partial \frak w_j}\Big)\right).
\endaligned
\end{equation}
In fact we  have
$$
H(z,\frak v,\frak w) =  (D_{\hat F(z,\frak v,\frak w).}{\rm E}(u_1(z),\cdot))^{-1}\circ \mathcal{PP}(z,\frak v,\frak w)^{-1}.
$$
Here $(D_{\hat F(z,\frak v,\frak w)}{\rm E})(u_1(z),\cdot)$ is the differential of the map
$\xi \mapsto {\rm E}(u_1(z),\xi)$ at $\xi = \hat F(z,\frak v,\frak w) \in X$.
\par
We next observe that there exists a matrix valued smooth function $\frak J$ on
$\Omega \times  D^n(R) \times  D^n(R)$
such that
\begin{equation}
\frak J(z,\frak v,\frak w)(\mathcal{PP}(z,\frak v,\frak w)(\xi))
=
\mathcal{PP}(z,\frak v,\frak w)(J_{\hat F(z,\frak v,\frak w)}(\xi))
\end{equation}
for any $\xi \in T_{\hat F(z,\frak v,\frak w)}X$. Here
$J_{\hat F(z,\frak v,\frak w)}$
$: T_{\hat F(z,\frak v,\frak w)}X \to T_{\hat F(z,\frak v,\frak w)}X$  is the almost complex structure of $X$.
\par
By definition
$$
\overline{\partial} v_r
=
\frac{\partial}{\partial x}\hat F(z,rV^0(z),W^0(z))
+ J_{\hat F(z,rV^0(z),W^0(z))}
\left(\frac{\partial}{\partial y}\hat F(z,rV^0(z),W^0(z))\right).
$$
Note
$$
\aligned
\frac{\partial}{\partial x}\hat F(z,rV^0(z),W^0(z))
&=
\frac{\partial \hat F}{\partial x}
+\sum_jr \frac{\partial V^0_j}{\partial x}\frac{\partial \hat F}{\partial \frak v_j}
+
\sum_j \frac{\partial W^0_j}{\partial x} \frac{\partial \hat F}{\partial \frak w_j}
\\
\frac{\partial}{\partial y}\hat F(z,rV^0(z),W^0(z))
&=
\frac{\partial \hat F}{\partial y}
+\sum_jr\frac{\partial V^0_j}{\partial y}\frac{\partial \hat F}{\partial \frak v_j}
+
\sum_j \frac{\partial W^0_j}{\partial y}\frac{\partial \hat F}{\partial \frak w_j}.
\endaligned
$$
Now it is fairly obvious that we can find
$\frak G, \frak H^x_{\frak v,j}, \frak H^x_{\frak w,j},\frak H^y_{\frak v,j}, \frak H^y_{\frak w,j}$ from
$H$, $\frak J$, the $x$,$y$ derivatives of
${\rm E}(u_1(z),\xi)$  with $\xi = \hat F$, and
\begin{equation}\label{formA8}
\aligned
&\mathcal{PP}(z,\frak v,\frak w)\left(\frac{\partial \hat F}{\partial x}\right),
&\mathcal{PP}(z,\frak v,\frak w)\left(\frac{\partial \hat F}{\partial y}\right),
\\
&\mathcal{PP}(z,\frak v,\frak w)\left(\frac{\partial \hat F}{\partial \frak v_j}(z,\frak v,\frak w)\right),
&\mathcal{PP}(z,\frak v,\frak w)\left(\frac{\partial \hat F}{\partial \frak w_j}(z,\frak v,\frak w)\right).
\endaligned
\end{equation}
(\ref{formA8}) are also $\R^n$ valued smooth functions of $(z,\frak v,\frak w)$
because of (\ref{formA6}).
\par
The proof of Lemma \ref{lemA1} is complete.
\end{proof}
By (\ref{maineqationddd}) we have
\begin{equation}\label{formA9}
\aligned
\frac{d^2}{dr^2}({\rm \ref{formA4}})
= &\sum_{ij}V^0_iV^0_j \frak F_{ij}\left(r,z,W^0,V^0,\frac{\partial W^0}{\partial x},\frac{\partial W^0}{\partial y}
,\frac{\partial V^0}{\partial x},\frac{\partial V^0}{\partial y}\right) \\
&+
\sum_{ij}V^0_i\frac{\partial V^0_j}{\partial x} \frak F^x_{ij}\left(r,z,W^0,V^0\right) +
\sum_{ij}V^0_i\frac{\partial V^0_j}{\partial y} \frak F^y_{ij}\left(r,z,W^0,V^0\right)
\endaligned
\end{equation}
where $\frak F_{ij}$, $ \frak F^x_{ij}$, $ \frak F^y_{ij}$ are smooth maps independent of $V^0$,$W^0$.
\par
\begin{proof}[Proof of Lemma \ref{lem:mgeq2}]
We take:
$\Sigma = \Sigma_1$, and $\Omega_{\alpha}$
($\alpha = 1,\dots,\frak A$) open subsets of $\Sigma_1$ such that
$\bigcup_{\alpha = 1}^{\frak A}\Omega_{\alpha} \supset K_1$.
We also put
$u_1 = u_1$ and
\begin{equation}\label{formulaA2}
\aligned
W^0(z) &= {\rm E}(u_1(z), u(z)) \in T_{u_1(z)}(X),\\
V^0(z)  &= {\rm Pal}_{u(z)} ^{u_1(z)}(V(z)) \in T_{u_1(z)}(X).
\endaligned
\nonumber
\end{equation}
Then by (\ref{formA9}), we obtain:
\begin{equation}\label{formA10}
\left\Vert \int_0^1 ds \int_0^s \frac{d^2}{dr^2}({\rm \ref{formA4}}) \,dr
\right\Vert_{L^2_m(\Omega_{\alpha})}
\le
C_{({\rm\ref{formA10}})}\left\Vert V^0 \right\Vert^2_{L^2_{m+1}(\Omega_{\alpha})}
\end{equation}
where $C_{({\rm\ref{formA10}})}$ depends on $\max\{\left\Vert V^0 \right\Vert_{L^2_{m+1}(\Omega_{\alpha})},
\left\Vert W^0 \right\Vert_{L^2_{m+1}(\Omega_{\alpha})}\}$.
\par
By taking the sum over $\alpha$, we obtain the inequality (\ref{152ff}) to be proven.
\end{proof}
\begin{proof}[Proof of (\ref{152ffg})]
By putting
$\Sigma = \Sigma(S)$ the proof is the same as the proof of Lemma \ref{lem:mgeq2}.
\end{proof}
\begin{proof}[Proof of the first inequality of (\ref{2ff160})]
We put
$\Sigma = \Sigma_1$, $\Omega$ a neighborhood of $\mathcal A_T$,
$u_1 =u_1$,
$$
\aligned
&u(z) = {\rm Exp}(u^{\rho}_{T,0}(z),V_{T,1,(1)}^{\rho}), \\
&
v_r(z) = {\rm Exp}(u^{\rho}_{T,0}(z),V_{T,1,(1)}^{\rho} + r
 \chi_{\mathcal A}^{\rightarrow}(V^{\rho}_{T,2,(1)}
-(\Delta p^{\rho}_{T,(1)})^{\rm Pal})).
\endaligned
$$
We then define
\begin{equation}
\aligned
W^0(z) &= E(u_1(z),u(z)) \in T_{u_1(z)}(X), \\
V^0(z)  &= {\rm Pal}_{u(z)} ^{u_1(z)}( \chi_{\mathcal A}^{\rightarrow}(V^{\rho}_{T,2,(1)}
-(\Delta p^{\rho}_{T,(1)})^{\rm Pal})) \in T_{u_1(z)}(X).
\endaligned
\nonumber
\end{equation}
We have
\begin{equation}\label{formulaA22}
\mathcal P \overline{\partial} v_{r_0} - \overline{\partial} u
=
\int_{0}^{r_0} \left(\frac{\partial}{\partial r}\right) \mathcal P  (\overline{\partial} v_{r}) dr
\end{equation}
Here $ \mathcal P$ is the inverse of the complex linear part of the
parallel transport along the path $r \mapsto v_r(z)$.
This path is not a geodesic. However we can apply the same argument as the proof of
Lemma \ref{lemA1} to obtain
\begin{equation}\label{formulaA223}
\aligned
\frak P\left(\frac{\partial}{\partial r}\right) \mathcal P  (\overline{\partial} v_{r})
= &\sum_{i}V^0_i \frak F_{i}\left(r,z,W^0,V^0,\frac{\partial W^0}{\partial x},\frac{\partial W^0}{\partial y}
,\frac{\partial V^0}{\partial x},\frac{\partial V^0}{\partial y}\right) \\
&+
\sum_{i}\frac{\partial V^0_i}{\partial x} \frak F^x_{i}\left(r,z,W^0,V^0\right)
\\
&+
\sum_{i}\frac{\partial V^0_i}{\partial y} \frak F^y_{i}\left(r,z,W^0,V^0\right),
\endaligned
\end{equation}
where $\frak F_{i}$, $\frak F^x_{i}$, $\frak F^y_{i}$ are smooth maps.
Here $\frak P = (({\rm Pal}_{u_1}^u)^{(0,1)})^{-1}$.
The first inequality of (\ref{2ff160}) follows from (\ref{formulaA22})
and (\ref{formulaA223}).
\end{proof}
\begin{rem}\label{remarkA2}
In the argument of this appendix or anywhere in this paper
we {\it never} use the fact that we take the parallel transport
with respect to the Levi-Civita connection.
We can actually use any linear connection which preserves
$TL \subset TX$ in place of {the} Levi-Civita connection
of the metric given in Lemma  \ref{lem:g0}, as we did in
\cite[Chapter 7]{fooo:book1}.
\par
In the case when there are several (finitely many) Lagrangian submanifolds $L_{\alpha}$
so that they have mutually clean intersection,
we cannot generalize Lemma  \ref{lem:g0} to find a Hermitian metric
such that all the $L_{\alpha}$ are totally geodesic.
However it is easy to see that there exists a linear connection which
preserves all the subspaces $TL_{\alpha} \subset TX$.
\end{rem}

\section{Estimate of Parallel transport 1}
\label{appendixA2}
Let $X$ be a Riemannian manifold.
We put
\begin{equation}\label{formA11}
\frak U
=
\{ (x,y,\frak v,\frak w) \mid
x,y \in X, d(x,y) \le \iota'_X/2, \frak v,\frak w \in T_xX,
\vert \frak v\vert, \vert \frak w\vert < \iota'_X/10\},
\end{equation}
which is a smooth manifold.
We consider a smooth vector bundle
$\frak L$ on $\frak U$ whose fiber at
$(x,y,\frak v,\frak w)$ is
\begin{equation}\label{formA12}
\frak L_{(x,y,\frak v,\frak w)}
= {\rm Hom}_{\R} (T_yX,T_xX).
\end{equation}
We define its smooth section $\frak P$ whose
value at $(x,y,\frak v,\frak w)$  is as follows.
We put
$$
\frak p = {\rm Exp}(x,\frak w),
\qquad
\frak q = {\rm Exp}(\frak p,{\rm Pal}_{x}^{\frak p}(\frak v)),
$$
then
\begin{equation}\label{sectionL}
\frak P(x,y,\frak v,\frak w)
=
(({\rm Pal}_{x}^{\frak p})^J)^{-1}
\circ (({\rm Pal}_{\frak p}^{\frak q})^J)^{-1}
\circ
({\rm Pal}_{y}^{\frak q})^J.
\end{equation}
The smoothness of $\frak P$ is obvious from the definition.
\begin{lem}\label{lemB2}
Let $\mathscr S$ be a 2 dimensional manifold and
$\frak v(s)$ and $\frak w(s)$ be an $s \in \mathcal S$ parameterized family of smooth maps
$\mathscr S \to T_xX$ with $\vert \frak v(s)\vert, \vert \frak w(s)\vert \le i'_X/10$.
If
$$
\aligned
&\left\Vert \frac{\partial^{n_1}}{\partial s^{n_1}}\frak v\right\Vert_{L^2_m(\mathscr S)}
\le 1
,
\qquad
\left\Vert \frac{\partial^{n_2}}{\partial s^{n_2}}\frak w\right\Vert_{L^2_m(\mathscr S)}
\le 1
\endaligned
$$
for $n_1,n_2 \le n$,
then
\begin{equation}\label{lemBB}
\aligned
&\left\Vert \frac{\partial^{n}}{\partial s^n}
\Big(\frak P(x,y,\frak v,\frak w) - \frak P(x,y,\frak v,0) \Big)\right\Vert_{L^2_m(\mathscr S)}
\\
&\le
C_{m,{\rm(\ref{lemBB})}} \sum_{n'\le n}
\left\Vert \frac{\partial^{n'}}{\partial s^{n'}}\frak w\right\Vert_{L^2_m(\mathscr S)}.
\endaligned
\end{equation}
\end{lem}
\begin{proof}
The lemma is an immediate consequence of the smoothness of $\frak P$.
\end{proof}
\begin{proof}[Proof of Lemma \ref{triangleholonomy}]
We consider $\frak e = \frak e^{\rho}_{1,T,(0)} \in \mathcal E_1$.
When we defined $\mathcal E_1$ it was a subset $\mathcal E_1 \subset L^2_m(\Sigma_1;(
u_1^{\frak{ob}})^*TX \otimes \Lambda^{0,1})$.
We used the parallel transport to regard
$$
\frak e \in L^2_m(\Sigma_1;(\hat u^{\rho}_{1,T,(0)})^*TX \otimes \Lambda^{0,1})
= L^2_m(\Sigma_1;u^*TX \otimes \Lambda^{0,1}).
$$
It appears in (\ref{156ff}). So more precisely $\frak e$ in  (\ref{156ff}) is:
$$
z \mapsto  ({\rm Pal}_{u_1^{\frak{ob}}(z)}^{\hat u^{\rho}_{1,T,(0)}(z)})^{(0,1)}( \frak e^{\rho}_{1,T,(0)}(z)).
$$
(Here $z \in K_1^{\frak{ob}}$.)
The expression $\mathcal P^{-1} \frak e$ in (\ref{156ff}) is then equal to
\begin{equation}\label{formulaB5}
z \mapsto
\big({\rm Pal}^{\Exp (\hat u^{\rho}_{1,T,(0)}(z),V^{\rho}_{T,1,(1)}(z))}_{\hat u^{\rho}_{1,T,(0)}(z)}\big)^{(0,1)}
\circ
\big({\rm Pal}_{u_1^{\frak{ob}}(z)}^{\hat u^{\rho}_{1,T,(0)}(z)}\big)^{(0,1)}( \frak e^{\rho}_{1,T,(0)}(z)).
\end{equation}
On the other hand
$\frak e \in L^2_m(\Sigma_i;\Exp (\hat u^{\rho}_{1,T,(0)},V^{\rho}_{T,1,(1)})^*TX \otimes \Lambda^{0,1})$
appearing in Lemma \ref{triangleholonomy} is, by definition,
\begin{equation}\label{formulaB6}
z \mapsto ({\rm Pal}_{u_1^{\frak{ob}}(z)}^{\Exp (\hat u^{\rho}_{{T,1,(0)}}(z),V^{\rho}_{T,1,(1)}(z)})^{(0,1)}
( \frak e^{\rho}_{1,T,(0)}(z)).
\end{equation}
\par
Using the section $\frak P$ in (\ref{sectionL}),  the sections
(\ref{formulaB5}) and (\ref{formulaB6}) are written as follows.
$$
\aligned
{\rm (\ref{formulaB5})}
&=
\frak P(\hat u^{\rho}_{{T,1,(0)}}(z),u_i^{\frak{ob}}(z),0,V^{\rho}_{T,1,(1)}(z)))
(\frak e^{\rho}_{1,T,(0)}(z)),
\\
{\rm (\ref{formulaB6})}
&=
\frak P(\hat u^{\rho}_{{T,1,(0)}}(z),u_i^{\frak{ob}}(z),0,0))
(\frak e^{\rho}_{1,T,(0)}(z)).
\endaligned
$$
\par
Therefore \eqref{formulaB5}-\eqref{formulaB6} are estimated by $Ce^{-\delta_1 T}$ by Lemma \ref{lemB2}.
Then the estimate (\ref{triangleholonomy2}) easily follows from this.
\end{proof}

\section{Error term estimate of non-linear Cauchy-Riemann equation II}
\label{appendixB}
In this appendix, we give a proof of Proposition \ref{mainestimatestep13kappa}.
Certain estimates of the parallel transport are  postponed till the next appendix.

\begin{proof}[Proof of Proposition \ref{mainestimatestep13kappa}]
The proof is similar to that of Proposition \ref{mainestimatestep13} and proceeds as follows.

\par
We first perform the estimates on $K_1$.
We use the simplified notation:
\begin{equation}\label{simplenote2}
\aligned
u &= \hat u^{\rho}_{1,T,(\kappa-1)}, \qquad &V=V^{\rho}_{T,1,(\kappa)}, \\
\mathcal P &= (\Pal_1^{(0,1)})^{-1},  \qquad &
{\frak e =\sum_{a=0}^{\kappa-1}\frak e^{\rho} _{1,T,(a)}}.
\endaligned
\end{equation}
\par
Here $\Pal_1^{(0,1)}$ is the $(0,1)$ part of the parallel translation along the
map $s \mapsto \Exp (u,sV)$, $s \in [0,s_0]$ for some $s_0 \in [0,1]$.
\par
We obtain the same formula as (\ref{151ffff})
\begin{equation}\label{2151ffff}
\aligned
&\mathcal P\overline \partial (\Exp (u,V))  \\
& =
\overline\partial (\Exp (u,0)) +\int_0^1 \frac{\del}{\del s}\left(\mathcal P\overline\partial (\Exp (u,sV))\right) ds
 \\
& =  \overline\partial (\Exp (u,0))
+ (D_{u}\overline\partial)(V)
 +\int_0^1 ds \int_0^s
\left(\frac{\del}{\del r}\right)^2 \left(\mathcal P \overline\partial (\Exp (u,rV)\right)dr.
\endaligned
\end{equation}
Then we also obtain
\begin{equation}\label{2152ff}
\aligned
&\left\|\int_0^1 ds \int_0^s
\left(\frac{\del}{\del r}\right)^2 \left(\mathcal P \overline\partial (\Exp (u,rV)\right)\, dr
\right\|_{L^2_m(K_1)}\\
&\qquad\le C_{m,{\rm(\ref{2152ff}})} \| V\|_{L^2_{m+1,\delta}}^2 \le C'_{m,{\rm(\ref{2152ff})}}e^{-2\delta_1 T}\mu^{2(\kappa-1)}
\endaligned
\end{equation}
provided $m \geq 2$.
One can prove (\ref{2152ff}) by applying Lemma \ref{lemA1} to
\begin{equation}\label{formD4}
W^0(z) = {\rm E}(u_1(z),\hat u^{\rho}_{1,T,(\kappa-1)}(z)), \quad
V^0(z) = {\rm Pal}_{\hat u^{\rho}_{1,T,(\kappa-1)}(z)}^{u_1(z)}
\left(V^{\rho}_{T,1,(\kappa)}(z)
\right)
\end{equation}
and applying the Sobolev inequality etc. to the right hand side of (\ref{formA9}).
\begin{rem}\label{remD1D1}
We remark that we are working by induction on $\kappa$ and
take infinitely many steps $\kappa = 1,2,\dots$.
For this proof to work we need the constants $C_{m,{\rm(\ref{2152ff})}}$
and $C'_{m,{\rm(\ref{2152ff})}}$ in (\ref{2152ff}) to be independent of
$\kappa$.
\par
The reason  we can take $C_{m,{\rm(\ref{2152ff})}}$
and $C'_{m,{\rm(\ref{2152ff})}}$ to be independent of $\kappa$ is as follows.
As we mentioned right after (\ref{formA10}) the constant
$C_{{\rm(\ref{formA10})}}$ there depends only on $\max\{\left\Vert V^0 \right\Vert_{L^2_{m+1}(\Omega)},
\left\Vert W^0 \right\Vert_{L^2_{m+1}(\Omega)}\}$.
In our situation where $W^0$ and $V^0$ are given by (\ref{formD4}),
their local $L^2_{m+1}$ norms are bounded by the induction hypothesis
((\ref{form182}) and (\ref{form184})) and Lemma \ref{Lemma6.9}, by a number independent of $\kappa$.
\par
Therefore $C_{m,{\rm(\ref{2152ff})}}$
and $C'_{m,{\rm(\ref{2152ff})}}$ can be taken to be independent of $\kappa$.
It can be taken to be independent of $T$ since we are working on $K_1$ which is
independent of $T$.
\par
Independence of the constants of $\kappa$ or $T$ appears
in other part of the proof, which can be proved in the same way.
So we do not mention it usually.
\end{rem}

Next we have
\begin{equation}\label{2155ff}
\aligned
& \mathcal P \circ \Pi_{\mathcal E_1(\Exp (u,V))}^{\perp} \circ \mathcal P^{-1}\\
& =
\Pi_{\mathcal E_1(u)}^{\perp} +
\int_0^1 \frac{d}{ds}\left(\mathcal P\circ \Pi_{\mathcal E_1(\Exp (u,sV)}^{\perp}
\circ \mathcal P^{-1}\right)\, ds
\\
& =
\Pi_{\mathcal E_1(u)}^{\perp} -
(D_{u}\mathcal E_1)(\cdot,V) +
\int_0^1 ds \int_0^s \frac{d^2}{dr^2}\left(\mathcal P\circ
\Pi_{\mathcal E_1(\Exp (u,rV))}^{\perp}\circ \mathcal P^{-1}
\right)\, dr,
\endaligned\end{equation}
in which we can estimate the third term in the same way as (\ref{2152ff}).
(See Appendix \ref{appendixA2bis}.)
\par
On the other hand, (\ref{frakeissmall2kappa0}) for $\le \kappa-1$, (\ref{form0182})
for $\le \kappa$ and (\ref{2151ffff}), (\ref{2152ff}) imply 
\begin{equation}\label{2156ff}
\left\|\overline\partial (\Exp (u,V))
- \mathcal P^{-1}\frak{e}\right\|_{L^2_{m}(K_1)}
\le C_{m,{\rm(\ref{2156ff})}}e^{-\delta_1 T}\mu^{\kappa-1}.
\end{equation}
We also use the next inequality
\begin{equation}\label{form5505550ap}
\left\Vert\mathcal P ^{-1}(D_{u}\mathcal E_1)(\mathcal P Q,V)\right\Vert_{L^2_{m}(K_1)}
\le
C_{m,{\rm(\ref{form5505550ap})}}\left\|Q\right\|_{L^2_{m}(K_1)}\left\| V\right\|_{L^2_{m}(K_1)}.
\end{equation}
Here $Q$ is a section of $u^*TX \otimes \Lambda^{0,1}$
of $L^2_{m}$ class.
(\ref{form5505550ap}) is a version of (\ref{form5505550}) and is proved in
Appendix \ref{appendixA2bis}.
Using (\ref{2156ff})  and (\ref{form5505550ap}) we can show
\begin{equation}\label{A122form}
\aligned
\Big\|\Pi_{\mathcal E_1(\Exp (u,V))}^{\perp}&\overline\partial (\Exp (u,V))
-\mathcal P^{-1}
\Pi_{\mathcal E_1(u)}^{\perp}(\mathcal P \overline\partial (\Exp (u,V))\\
&+
\mathcal P^{-1}(D_{u}\mathcal E_1)\left(\frak e,V\right)\Big\|_{L^2_{m}(K_1^+)}\le
C_{m,{\rm(\ref{A122form})}}e^{-2\delta_1 T}\mu^{\kappa-1},
\endaligned\end{equation}
in the same way as (\ref{160ff}).
\par
By (\ref{formula158}) we have:
\begin{equation}
\overline\partial u
+ (D_{u}\overline\partial)(V)- (D_{u}\mathcal E_1)(\frak{e},V)
\in  \mathcal E_1(u)
\end{equation}
on $K_1$.
\par
Summing up we have derived
\begin{equation}\label{formB8}
\aligned
\|\Pi_{\mathcal E_1(\Exp (u,V))}^{\perp}(\overline\partial
(\Exp (u,V)))\|_{L^2_m(K_1)}
&\le C_{m,(\ref{formB8})}e^{-2\delta_1 T}\mu^{\kappa-1} \\
&\le C_{1,m}e^{-\delta_1 T}\epsilon(5)\mu^{\kappa}/10
\endaligned
\end{equation}
for $T>T_{m,\epsilon(5),({\rm\ref{formB8})}}$, as long as we choose
$T_{m,\epsilon(5),({\rm\ref{formB8}})}$ so that
$$
C_{m,({\rm\ref{formB8}})}e^{-\delta_1  T_{m,\epsilon(5),({\rm\ref{formB8}})} }
\le C_{1,m}\epsilon(5) \mu/10.
$$
We emphasize that this choice of $ T_{m,\epsilon(5),(\ref{formB8})}$ does not depend on $\kappa$ but depends only on $m$, $\epsilon(5)$ and $\mu$.
\par
It follows from (\ref{2156ff}) that
\begin{equation}\label{formB822}
\|\Pi_{\mathcal E_1(\Exp (u,V))}\left(\overline\partial (\Exp (u,V))\right)
 - \frak{e}\|_{L^2_{m}(K_1)}
\le  C_{m,{\rm(\ref{formB822})}}e^{-\delta T}\mu^{\kappa-1},
\end{equation}
in the same way as the proof of Lemma \ref{triangleholonomy} given in Appendix
\ref{appendixA2}.
\par
Then we can prove (\ref{frakeissmall2kappaa}) by putting
\begin{equation}\label{formD9atoato}
\aligned
\frak e^{\rho}_{1,T,(\kappa)}
&=
\Pi_{\mathcal E_1(\Exp (u,V))}(\overline\partial (\Exp (u,V))
- \frak{e}\in \mathcal E_1(\Exp (u,V))
\cong \mathcal E^{\frak{ob}}_1.
\endaligned
\end{equation}
\par
We can perform the estimate on $[-5T,-T-1] \times [0,1]$
modifying the proof of (\ref{form56156})
in the same way as follows.
For $[S,S+1] \times [0,1] \subset [-5T,-T-1]_{\tau}\times [0,1]$
the inequality
\begin{equation}\label{newformA260}
\left\|\overline\partial u^{\rho} _{T,(\kappa)} \right\|_{L_{m}^2([S,S+1]\times [0,1]\subset \Sigma_T)}
< C_{m,({\rm\ref{newformA260}})}
\mu^{\kappa-1} e^{-2 \delta_1 T}
\end{equation}
can be proved in the same way as
(\ref{formB8}).
Therefore
\begin{equation}\label{newformA2602}
\aligned
\left\|\overline\partial u^{\rho} _{T,(\kappa)} \right\|_{L_{m,\delta}^2( [-5T,-T-1]\times [0,1]\subset \Sigma_T)}
&\le 40T e^{4T\delta}C_{m,({\rm\ref{newformA260}})}
\mu^{\kappa-1} e^{-2 \delta_1 T} \\
&\le C_{1,m}e^{-\delta_1 T}\epsilon(5)\mu^{\kappa}/10,
\endaligned
\end{equation}
for $T > T_{m,\epsilon(5),({\rm\ref{newformA2602})}}$.
(Here we use the fact that the weight function $e_{T,\delta}$
is smaller than $10 e^{4T\delta}$ on our domain.)
\par
The estimates on $K_2$ and $[T+1,5T] \times [0,1]$
are the same.
Notation (\ref{simplenote2}) is used up to here.
\par\medskip
The estimate $\overline\partial u^{\rho} _{T,(\kappa)}$ on $[-T+1,T-1]\times [0,1]$
is as follows.
Note the bump functions $\chi_{\mathcal B}^{\leftarrow}$ and $\chi_{\mathcal A}^{\rightarrow}$ are $\equiv 1$ there and
$$
\overline\partial u^{\rho} _{T,(\kappa-1)}
+
(D_{u^{\rho} _{T,(\kappa-1)}}\overline\partial)
(V^{\rho}_{T,1,(\kappa)}+V^{\rho}_{T,2,(\kappa)})
= 0.
$$
Therefore the inequality
\begin{equation}\label{newformA26}
\left\|\overline\partial u^{\rho} _{T,(\kappa)} \right\|_{L_{m}^2([-T+1,T-1]\times [0,1]\subset \Sigma_T)}
< C_{m,({\rm\ref{newformA26}})}e^{-2\delta_1 T}\mu^{\kappa-1}
\end{equation}
can be proved in the same way as
(\ref{formB8}).
Since $e_{T,\delta} \le 10 e^{5\delta T} \le 10 e^{\delta_1 T/2}$,  it implies
\begin{equation}\label{newformA262}
\left\|\overline\partial u^{\rho} _{T,(\kappa)} \right\|_{L_{m,\delta}^2([-T+1,T-1]\times [0,1]\subset \Sigma_T)}
< e^{-\delta_1 T} \epsilon(5) \mu^{\kappa}/10
\end{equation}
for $T > T_{m,\epsilon(5),({\rm\ref{newformA262}})}$.

On $\mathcal A_T$ we have
\begin{equation}\label{2estimateatA1}
u^{\rho} _{T,(\kappa)}
=
{\rm Exp}(u^{\rho}_{T,(\kappa-1)},\chi_{\mathcal A}^{\rightarrow}
(V^{\rho}_{T,2,(\kappa)}-\Delta p^{\rho}_{T,(\kappa)})+V^{\rho}_{T,1,(\kappa)}).
\end{equation}
Note
\begin{equation}\label{newformA28}
\aligned
&\big\| \chi_{\mathcal A}^{\rightarrow}(V^{\rho}_{T,2,(\kappa)}-\Delta p^{\rho}_{T,(\kappa)})
\big\|_{L^2_{m+1}(\mathcal A_T)} \\
&\le
C_{m,({\rm\ref{newformA28}})}e^{-6T\delta}\big\| V^{\rho}_{T,2,(\kappa)}-\Delta p^{\rho}_{T,(\kappa)}
\big\|_{L^2_{m+1,\delta}(\mathcal A_T \subset \Sigma_{2})}
\\
&\le C'_{m,({\rm\ref{newformA28}})} e^{-6T\delta - T\delta_1}\mu^{\kappa-1}.
\endaligned
\end{equation}
The first inequality follows from  the fact that the weight function $e_{2,\delta}$
is around $e^{6T\delta}$ on $\mathcal A_T$. The second inequality follows from (\ref{142form23}).
On the other hand the weight function $e_{T,\delta}$ is around $e^{4T\delta}$ at $\mathcal A_T$.
(\ref{newformA28}) implies
\begin{equation}\label{ff160}
\big\| \chi_{\mathcal A}^{\rightarrow}(V^{\rho}_{T,2,(\kappa)}-\Delta p^{\rho}_{T,(\kappa)})
\big\|_{L^2_{m+1,\delta}(\mathcal A_T\subset \Sigma_T)}
\le
C_{m,({\rm\ref{ff160})}} e^{-2T\delta - T\delta_1}\mu^{\kappa-1}.
\end{equation}
Therefore in the same way as the proof of (\ref{2ff160})
given at the end of  Appendix \ref{appendixA},
we obtain
\begin{equation}\label{2ff1602}
\aligned
&
\left\vert\Vert \overline\partial u^{\rho} _{T,(\kappa-1)}\Vert_{L^2_{m,\delta}(\mathcal A_T\subset \Sigma_T)} -
\Vert\overline\partial(\Exp(u^{\rho}_{T,(\kappa-1)},V^{\rho}_{T,1,(\kappa)})
\Vert_{L^2_{m,\delta}(\mathcal A_T\subset \Sigma_T)}
\right\vert
\\
&\qquad\le
C_{m,({\rm\ref{2ff1602}})} e^{-2\delta T-\delta_1T}
\le
C_{1,m}\epsilon(5) e^{-\delta_1 T}\mu^{\kappa}/20
\endaligned
\end{equation}
for $T > T_{m,\epsilon(5),({\rm\ref{2ff1602}})}$.
\par
Since
$
{\rm Err}^{\rho}_{2,T,(\kappa-1)}   = 0
$
on $\mathcal A_T$
we have
\begin{equation}
\overline\partial u^{\rho}_{T,(\kappa-1)} + (D_{u^{\rho}_{T,(\kappa-1)}}\overline\partial)(V^{\rho}_{T,1,(\kappa)}) = 0.
\end{equation}
Therefore in the same way as we did on $K_1$ we can show
\begin{equation}\label{ff161}
\aligned
\big\|\overline\partial( {\rm Exp}(u^{\rho}_{T,(\kappa-1)},V^{\rho}_{T,1,(\kappa)}
)\big\|_{L^2_{m,\delta}(\mathcal A_T \subset \Sigma_T)}
&\le
C_{m,({\rm\ref{ff161}})}\mu^{2(\kappa-1)}e^{4\delta T} e^{-2\delta_1T} \\
&\le C_{1,m}\epsilon(5) e^{-\delta_1 T}\mu^{\kappa}/20
\endaligned
\end{equation}
for $T > T_{m,\epsilon(5),(\ref{ff161})}$.
(Here we use the fact that $e_{T,\delta} \sim e^{4T\delta}$
on $\mathcal A_T$.)
\par
(\ref{2ff1602}) and (\ref{ff161}) imply
\begin{equation}\label{ff162}
\big\|\overline\partial u^{\rho} _{T,(\kappa)} \big\|_{L_{m}^2(\mathcal A_T\subset \Sigma_T)} <
C_{1,m}\epsilon(5) e^{-\delta_1 T}\mu^{\kappa}/10
\end{equation}
\par
The estimate on $\mathcal B_T$ is similar.
The proof of Proposition \ref{mainestimatestep13kappa} is complete
except the estimate of the parallel transport given in the next appendix.
\end{proof}

\section{Estimate of Parallel transport 2}
\label{appendixA2bis}

\begin{proof}
[Proof of (\ref{form5505550ap})]
Let $\Omega$ {be} a small neighborhood of $K_1$.
We consider the vector bundle $\frak L$ on $\frak U$
defined by (\ref{formA11}) and (\ref{formA12}).
We put
$$
\frak V =
\{(z,\frak v,\frak w) \mid z \in \Omega,
\, \frak v,\frak w \in T_{u_1(z)}X, \,  \vert\frak v\vert, \vert\frak w\vert < R\},
$$
with $R < \iota'_X/10$.
We pull $\frak L$ back
by the map
$
\frak V\to \frak U
$
defined by
$$
(z,\frak v,\frak w) \mapsto (u_1(z),u_1^{\frak{ob}}(z),\frak v,\frak w)
$$
to obtain a vector bundle $\hat{\frak L}$ on $\frak V$.
\par
We pull back the section $\frak P$ in
(\ref{sectionL}) by this map
and tensor it with the identity in ${\rm End}(\Lambda^{0,1}(\Omega))$.
We then obtain a section
$P$ of $\hat{\frak L} \otimes {\rm End}(\Lambda^{0,1}(\Omega))$. $P$ is written as:
$$
P(z,\frak v,\frak w) =
\Big(({\rm Pal}^{{\rm Exp}(u_1(z),\frak w)}_{u_1(z)})^{(0,1)}\Big)^{-1}
\circ
\Big(({\rm Pal}_{{\rm Exp}(u_1(z),\frak w)}^{\hat v(z,\frak v,\frak w)})^{(0,1)}\Big)^{-1}
\circ
({\rm Pal}_{u_1^{\frak{ob}}(z)}^{\hat v(z,\frak v,\frak w)})^{(0,1)}
$$
where
\begin{equation}\label{defofvinE}
\hat v(z,\frak v,\frak w) =  {\rm Exp}({\rm Exp}(u_1(z),\frak w),
{\rm Pal}_{u_1(z)}^{{\rm Exp}(u_1(z),\frak w)}(\frak v)).
\end{equation}
See Figure \ref{FigureappE0}.
\begin{figure}
\centering
\includegraphics{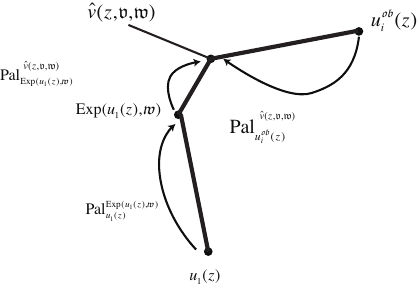}
\caption{Map $P$}
\label{FigureappE0}
\end{figure}
\par
Let $\hat{\bf e}_{i}(z)$  ($i=1,\dots,\dim \mathcal E_1$)
be a basis of $\mathcal E_1 = \mathcal E_1(u_1^{\frak{ob}})$.
We define
$$
\hat{\bf e}_{i}(z,\frak v,\frak w) = P(z,\frak v,\frak w)(\hat{\bf e}_{i}(z)),
$$
which is a smooth section of the bundle
$u_1^*TX \otimes \Lambda^{0,1}(\Omega)$ on $\frak V$.
(Here we denote the map
of $(z,\frak v,\frak w) \mapsto u_1(z)$ by $u_1$
by a slight abuse of notation.)
We consider
\begin{equation}\label{forme'e'e'}
\aligned
{\bf e}'_{i}(z,r)
&= P(z,rV^0(z),W^0(z))(\hat{\bf e}_{i}(z)) \\
& =\hat{\bf e}_{i}(z,rV^0(z),W^0(z))
 \in T_{u_1(z)}X\otimes \Lambda^{0,1}(\Omega).
 \endaligned
\end{equation}
The sections
${\bf e}'_{i} \in \Gamma(u_1^*TX\otimes \Lambda^{0,1}(\Omega))$,
$(i=1,\dots,\dim\mathcal E_1)$  form a basis of
$
P(\mathcal E_1).
$
\par
In the next step, we will use the Gram-Schmidt process to modify it to an
orthonormal basis with respect to the $L^2$ metric which is induced
from one on $v_r^*TX$, where
\begin{equation}\label{vrzvrz}
v_r(z) = \hat v(z,rV_0(z),W_0(z)).
\end{equation}
Let ${\rm Quad}_{+}$ be {an open subset of a vector bundle over $\frak V$
such that the fiber of ${\rm Quad}_{+}$} at $(z,\frak v,\frak w)$ is the
positive definite quadratic form on $T_{u_1(z)}X$.
We define its section
$g$
as follows.
We  define
$
\mathcal{PP} : T_{u_1(z)}X \otimes \Lambda^{0,1}_z \to T_{\hat v(z,\frak v,\frak w)}X\otimes \Lambda^{0,1}_z
$
by
$$
\mathcal{PP} =
\left({\rm Pal}^{\hat v(z,\frak v,\frak w)}_{{\rm Exp}(u_1(z),\frak w)}\right)^{(0,1)}
\circ
\left({\rm Pal}_{u_1(z)}^{{\rm Exp}(u_1(z),\frak w)}\right)^{(0,1)}.
$$
Then
$$
g(z,\frak v,\frak w)(\vec v,\vec w)
=
g_{\hat v(z,\frak v,\frak w)}(\mathcal{PP}(\vec v),\mathcal{PP}(\vec w))).
$$
Here $\hat v$ is as in  (\ref{defofvinE}) and $g_{\hat v(z,\frak v,\frak w)}$ is the Riemann metric tensor at $\hat v(z,\frak v,\frak w) \in X$.
It is also obvious that $g$ is a smooth section.
\par
We define ${\bf e}_{i}(z,r)$ by induction on $i$ as follows.
Suppose ${\bf e}_{j}(z,r)$ is defined for $j<i$.
We put
\begin{equation}\label{formE92}
\aligned
{\bf e}''_{i}(z,r)
=
&{\bf e}'_{i}(z,r) \\
&- \sum_{j=1}^{i-1}\int_{z \in \Omega}
g(z,rV^0(z),W^0(z))
({\bf e}'_{i}(z,r),{\bf e}_{j}(z,r)) d{\rm vol}
\times {\bf e}_{j}(z,r),
\endaligned
\end{equation}
and
\begin{equation}\label{formE93}
{\bf e}_{i}(z,r)
= \frac{{\bf e}''_{i}(z,r)}
{\left(\int_{z \in \Omega} g(z,rV^0(z),W^0(z))
({\bf e}''_{i}(z,r),{\bf e}''_{i}(z,r))d{\rm vol}\right)^{1/2}}.
\end{equation}
Then ${\bf e}_{i}(z,r)$ ($i=1,\dots,\dim \mathcal E_1$) is the
orthonormal basis we look for.
\begin{rem}
Since {the} $L^2$ norm is not defined pointwise,
we can not write ${\bf e}_{i}(z,r)$ in a form
$\widehat{\bf e}_i(z,rV^0(z),W^0(z))$ with some smooth map
$\widehat{\bf e}_i(z,\frak v,\frak w)$.
In fact (\ref{formE92}) and (\ref{formE93}) involve
integration.
On the other hand, as far as smoothness in the sense of Sobolev space
concerns, integration behaves nicely.
\end{rem}
Now
the linear map
\begin{equation}
\aligned
&(({\rm Pal}^{u}_{u_1})^{(0,1)})^{-1}
\circ
(({\rm Pal}_{u}^{v_r})^{(0,1)})^{-1}
\circ \Pi^{\perp}_{\mathcal E_1(v_r)}  \circ ({\rm Pal}_{u}^{v_r})^{(0,1)}
\circ ({\rm Pal}_{u_1}^{u})^{(0,1)}
\\
&\qquad: \Gamma(\Omega;u_1^*TX \otimes \Lambda^{0,1})
\to \Gamma(\Omega;u_1^*TX \otimes \Lambda^{0,1})
\endaligned
\end{equation}
is written as
\begin{equation}\label{formnewQ12}
\aligned
\mathcal Q
\mapsto
\mathcal Q
-
\sum_{i=1}^{\dim \mathcal E_1}
\int_{z \in \Omega}
g(z,rV^0(z),W^0(z))
(\mathcal Q(z),{\bf e}_{i}(z,r)) d{\rm vol}\times {\bf e}_{i}(\cdot,r).
\endaligned
\end{equation}
We write the right hand side of (\ref{formnewQ12}) by
$$
H(r,\mathcal Q)(\cdot)
$$
where $\cdot \in \Omega$. Namely $(r,w) \mapsto H(r,\mathcal Q)(w)$ is a
section of the pullback of $u_1^*TX \otimes \Lambda^{0,1}$  to
$[0,\epsilon) \times \Omega$.
\par
By definition of $D_{u}\mathcal E_1$ in (\ref{DEidef}), we have
\begin{equation}\label{formE12}
\mathcal P ^{-1}(D_{u}\mathcal E_1)(\mathcal P Q,V)(w)
=
\left.\frac{d}{dr}\right\vert_{r=0}H(r,\overline Q)(w),
\end{equation}
where
$$
\overline Q = \left( ({\rm Pal}_{u_1}^u)^{(0,1)}\right)^{-1}(Q).
$$
(Note $Q(z) \in T_{u(z)}X \otimes \Lambda^{0,1}_z$.)
\par
Using (\ref{formE12}) we can show (\ref{form5505550ap}) as follows.
\par
Smoothness of ${\bf e}'_{i}$ and (\ref{forme'e'e'}) imply
\begin{equation}\label{formA45}
\left\Vert \frac{d}{dr}{\bf e}'_{i} \right\Vert_{L^2_{m}(\Omega)}
\le
C_{m,{\rm(\ref{formA45})}}\Vert  V^0\Vert_{L^2_{m}}.
\end{equation}
Then using (\ref{formE92}) and (\ref{formE93})
we can prove
\begin{equation}\label{formA46}
\left\Vert \frac{d}{dr}{\bf e}_{i} \right\Vert_{L^2_{m}(\Omega)}
\le
C_{m,{\rm(\ref{formA46})}}\Vert  V^0\Vert_{L^2_{m}}.
\end{equation}
Then using also (\ref{formnewQ12}) and (\ref{formE12})
we obtain (\ref{form5505550ap}).
\end{proof}
\begin{proof}[Estimate of the third term of (\ref{2155ff})]
Using \eqref{forme'e'e'}, \eqref{formE92}, \eqref{formE93} in the same way, we obtain

$$
\left\Vert \frac{d^2}{dr^2}{\bf e}_{i} \right\Vert_{L^2_{m}(\Omega)}
\le
C \Vert  V^0\Vert^2_{L^2_{m}}.
$$
Then using (\ref{formA46}) and (\ref{formnewQ12}) also we estimate
\begin{equation}\label{formA477}
\left\Vert\left.\frac{d^2}{dr^2}\right\vert_{r=r_0}H(r,\mathcal Q)
\right\Vert_{L^2_{m}}\le C_{m,{\rm(\ref{formA477})}}\Vert  V^0\Vert^2_{L^2_{m}}\Vert  \mathcal Q\Vert_{L^2_{m}}
\end{equation}
for any $\mathcal Q \in L^2_{m}(\Omega;u_1^*TX \otimes \Lambda^{0,1})$.
This gives the required estimate of the third term of (\ref{2155ff})
\end{proof}
\begin{proof}[Estimate of (\ref{eq18122})]
We put
$$
W_0 = {\rm E}(u_1,\hat u^{\rho}_{1,T,(0)}).
$$
Then using (\ref{formA46}) and (\ref{formE12})
we obtain
\begin{equation}\label{newnewA50}
\Vert (D_{\hat u^{\rho}_{1,T,(0)}}\mathcal E_1)(V,\mathcal Q)\Vert _{L^2_m}
\le
C_{m,{\rm (\ref{newnewA50})}} \Vert V \Vert_{L^2_m} \Vert\mathcal Q\Vert_{L^2_m}
\end{equation}
for any $V, \mathcal Q \in L^2_{m}(\Omega;(\hat u^{\rho}_{1,T,(0)})^*TX \otimes \Lambda^{01})$.

Using (\ref{eq181})
also we can estimate
\begin{equation}\label{eq1812200}
\aligned
\Big\Vert&(D_{\hat u^{\rho}_{1,T,(0)}}\mathcal E_1)
(({\rm Pal}_{u_1^{\frak{ob}}}
^{\hat u^{\rho}_{1,T,(0)}})^{0,1}((\frak {se})^{\rho} _{i,T,(\kappa-1)}),\mathcal Q)\\
&-
(D_{\hat u^{\rho}_{1,T,(0)}}\mathcal E_1)(
(({\rm Pal}_{u_1^{\frak{ob}}}
^{\hat u^{\rho}_{1,T,(0)}})^{(0,1)}(\frak e^{\rho} _{1,T,(0)}),\mathcal Q)
\Big\Vert_{L^2_{m}}
\le
C_{3,m}\frac{e^{-\delta_1 T}}{1-\mu}\Vert \mathcal Q\Vert_{L^2_{m}}
\endaligned
\end{equation}
for $L^2_{m}$ section $\mathcal Q$ of $(\hat u^{\rho}_{1,T,(0)})^*TX
\otimes \Lambda^{0,1}(\Omega)$.
\par
We put
$({\rm Pal}_{u_1^{\frak{ob}}}
^{\hat u^{\rho}_{1,T,(0)}})^{(0,1)}((\frak {se})^{\rho} _{i,T,(\kappa-1)})
= \frak {se}$.
\par
Then using (\ref{formA46}) to
$V_0 = {\rm E}(\hat u^{\rho}_{1,T,(0)},\hat u^{\rho}_{1,T,(\kappa-1)})$
and integrating along the curve
$r \mapsto {\rm Exp}(\hat u^{\rho}_{1,T,(0)},rV_0)$
we obtain
\begin{equation}\label{formA50}
\aligned
&
\Big\Vert
({\rm Pal}_{\hat u^{\rho}_{1,T,(0)}}^{\hat u^{\rho}_{1,T,(\kappa-1)}})^{(0,1)}
(D_{\hat u^{\rho}_{1,T,(0)}}\mathcal E_1)
(\frak {se},\mathcal Q)
 \\
&-
(D_{\hat u^{\rho}_{1,T,(\kappa-1)}}\mathcal E_1)
(({\rm Pal}_{\hat u^{\rho}_{1,T,(0)}}^{\hat u^{\rho}_{1,T,(\kappa-1)}})^{(0,1)}(\frak {se}),
(({\rm Pal}_{\hat u^{\rho}_{1,T,(0)}}^{\hat u^{\rho}_{1,T,(\kappa-1)}})^{(0,1)}(\mathcal Q))
\Big\Vert_{L^2_{m}}
\\
&
\le C_{m,{\rm(\ref{formA50})}}\Vert  V^0\Vert_{L^2_{m}}\Vert  \mathcal Q\Vert_{L^2_{m}}
\le
C'_{m,{\rm(\ref{formA50})}}e^{-\delta_1 T}\Vert  \mathcal Q\Vert_{L^2_{m}}.
\endaligned
\end{equation}
Note we use {the} induction hypothesis (\ref{form0184a}) here.
\par
In the same way as the proof of Lemma \ref{triangleholonomy}
we can show
\begin{equation}\label{formA51}
\aligned
&\Vert
(({\rm Pal}_{\hat u^{\rho}_{1,T,(0)}}^{\hat u^{\rho}_{1,T,(\kappa-1)}})^{(0,1)}
\circ
({\rm Pal}_{u_1^{\frak{ob}}}
^{\hat u^{\rho}_{1,T,(0)}})^{(0,1)}
-
({\rm Pal}_{u_1^{\frak{ob}}}
^{\hat u^{\rho}_{1,T,(\kappa-1)}})^{(0,1)})(\frak f)
\Vert_{L^2_m}\\
&\le
C_{m,{\rm(\ref{formA51})}}\Vert {\rm E}(\hat u^{\rho}_{1,T,(0)},\hat u^{\rho}_{1,T,(\kappa-1)})\Vert_{L^2_m}
\Vert \frak f\Vert_{L^2_m}
\le
C'_{m,{\rm(\ref{formA51})}} e^{-\delta_1 T}\Vert \frak f\Vert_{L^2_m}
\endaligned
\end{equation}
and
\begin{equation}\label{formA52}
\aligned
&\Vert
(({\rm Pal}_{\hat u^{\rho}_{1,T,(0)}}^{\hat u^{\rho}_{1,T,(\kappa-1)}})^{(0,1)}
\circ
({\rm Pal}_{u_1}
^{\hat u^{\rho}_{1,T,(0)}})^{(0,1)}
-
({\rm Pal}_{u_1}
^{\hat u^{\rho}_{1,T,(\kappa-1)}}))^{(0,1)}(Y)
\Vert_{L^2_m}\\
&\le
C_{m,{\rm(\ref{formA52})}}\Vert {\rm E}(\hat u^{\rho}_{1,T,(0)},\hat u^{\rho}_{1,T,(\kappa-1)})\Vert_{L^2_m}
\Vert Y\Vert_{L^2_m}
\le
C'_{m,{\rm(\ref{formA52})}} e^{-\delta_1 T}\Vert  Y\Vert_{L^2_m}.
\endaligned
\end{equation}
Combining (\ref{newnewA50})-(\ref{formA52})
we obtain the required estimate:
\begin{equation}\label{formA53}
\aligned
\Big\Vert&(D_{\hat u^{\rho}_{1,T,(\kappa-1)}}\mathcal E_1)
(({\rm Pal}_{u_1^{\frak{ob}}}
^{\hat u^{\rho}_{1,T,(\kappa-1)}})^{(0,1)}((\frak {se})^{\rho} _{i,T,(\kappa-1)}),
(({\rm Pal}_{u_1}
^{\hat u^{\rho}_{1,T,(\kappa-1)}})^{(0,1)}(Y))\\
&-
({\rm Pal}_{\hat u^{\rho}_{1,T,(0)}}^{\hat u^{\rho}_{1,T,(\kappa-1)}})^{(0,1)}(D_{\hat u^{\rho}_{1,T,(0)}}\mathcal E_1)(
({\rm Pal}_{u_1^{\frak{ob}}}
^{\hat u^{\rho}_{1,T,(0)}}(\frak e^{\rho} _{1,T,(0)}),(({\rm Pal}_{u_1}
^{\hat u^{\rho}_{1,T,(0)}})^{0,1}(Y))
\Big\Vert_{L^2_{m}}\\
&\le
C_{m,{\rm(\ref{formA53})}}e^{-\delta_1 T}\Vert Y\Vert_{L^2_{m}}.
\endaligned
\end{equation}
\end{proof}

\section{Estimate of the non-linearity of Exponential map}
\label{appendixB-}

We use the next lemma for the proof of Proposition \ref{prop:inequalitieskappa} (3).

\begin{lem}\label{expest}
Let $\mathscr S$ be a 2 dimensional manifold and $u : \mathscr S \to X$ a smooth map.
Let $v : \mathscr S \to X$ be a map of  $L^2_{m+1}$ class with $m > 2$.
We assume $d(u(z),v(z))$ is smaller than $\iota_X/2$ for all $z \in \mathscr S$.
We define a map
$$
\mathcal{EP}_v : L^2_{m+1}(\mathscr S,u^*TX) \to L^2_{m+1}(\mathscr S,u^*TX)
$$
by the next formula
\begin{equation}
\mathcal{EP}_v(V)(z) = {\rm E}\left(u(z),{\rm Exp}(v(z),{\rm Pal}_{u(z)}^{v(z)}(V(z)))\right).
\end{equation}
We assume
$
\Vert V\Vert_{L^2_{m+1}},  \Vert {\rm E}(u,v)\Vert_{L^2_{m+1}} \le 1.
$
Here ${\rm E}(u,v)(z) = {\rm E}(u(z),v(z))$.
Then
\begin{equation}\label{formA554}
\Vert\mathcal{EP}_v(V)  - {\rm E}(u,v) - V \Vert_{L^2_{m+1}}\le
C_{m,{\rm(\ref{formA554})}}\Vert {\rm E}(u,v)\Vert_{L^2_{m+1}}\Vert V\Vert_{L^2_{m+1}}.
\end{equation}
\end{lem}
\begin{rem}
In the simplified notation mentioned in Remark \ref{rem2222}, we have
$
\mathcal{EP}_v(V) = (v-u)+V.
$
This way of writing is a bit confusing here since
(\ref{formA554}) implies that
this `equality' holds modulo quadratic order term.
\end{rem}
\begin{proof}
Let
$$
\Omega = \{(V,W) \in TX\oplus TX \mid \vert V\vert, \vert W\vert \le \iota'_X/10 \}.
$$
We define a map $F : \Omega \to TX$ as follows. Let $(V,W) \in T_xX \oplus T_xX$.
We put $y = {\rm Exp}(x,W)$ and
$$
F(V,W) = {\rm E}\left(x,{\rm Exp}(y,{\rm Pal}_{x}^{y}(V))\right).
$$
This map is obviously smooth. So by the compactness of $\Omega$ its
$C^k$ norm is uniformly bounded for any $k$.
Moreover we have
$$
F(0,0) = 0, \qquad F(V,0) = V, \qquad F(0,W) = W.
$$
We observe
$$
\mathcal{EP}_v(V)(z) = F(V(z),{\rm E}(u(z),v(z))).
$$
Therefore the lemma follows from the standard fact that the left composition with smooth map
defines a smooth map between $L^2_{m+1}$ spaces.
In fact
\begin{equation}\label{newE3new}
\aligned
\mathcal{EP}_v(V) - {\rm E}(u,v) -V
&=
F(V,{\rm E}(u,v)) - F(V,0) - F(0,{\rm E}(u,v))\\
&=
\int_{0}^1 \frac{\partial}{\partial r} (F(V,r{\rm E}(u,v)) - F(0,r{\rm E}(u,v)))dr
\\
&=
\int_{0}^1\int_0^1 \frac{\partial^2}{\partial r\partial t} F(tV,r{\rm E}(u,v)) drdt.
\endaligned
\end{equation}
Therefore we can estimate
$\mathcal{EP}_v(V) - {\rm E}(u,v) -V$ as (\ref{formA554}),
by using
\begin{equation}\label{newnewE4}
\frac{\partial^2}{\partial r\partial t} (F(tV,r{\rm E}(u,v))
=
\sum_{a,b} V_a {\rm E}_b(u,v)
G_{a,b}(r,t,tV,r{\rm E}(u,v)),
\end{equation}
where $V_a$, ${\rm E}_b(u,v)$ are components of $V$, ${\rm E}(u,v)$,
respectively, 
and $G_{a,b}$ are smooth.
\end{proof}
The following variant of Lemma \ref{expest} can be proved in the same way.
\begin{lem}\label{expest2}
In the situation of Lemma \ref{expest} we assume $v= v(s)$, $V = V(s)$ are  families of $C^{\ell}$ class parameterized by
$s \in \mathcal S$ for some parameter space $\mathcal S \subset \R^N$.
We assume
$
\left\Vert \frac{d^{n'}}{ds^{n'}}V \right\Vert_{L^2_{m+1}},
\left\Vert \frac{d^{n'}}{ds^{n'}}{\rm E}(u,v) \right\Vert_{L^2_{m+1}} \le 1
$
for $n'\le n$.
Then
\begin{equation}\label{A56form}
\aligned
&\left\Vert \frac{d^n}{ds^{n}} \left(\mathcal{EP}_{v(s)}(V(s))- {\rm E}(u,v(s))
- V(s) \right)\right\Vert_{L^2_{m+1}}
\\
&\le
C_{m,{\rm(\ref{A56form})}} \sum_{k_1+k_2\le n}
\left\Vert \frac{d^{k_1}}{ds^{k_1}}V(s)\right\Vert_{L^2_{m+1}}
\left\Vert \frac{d^{k_2}}{ds^{k_2}}{\rm E}(u,v(s))\right\Vert_{L^2_{m+1}}.
\endaligned
\end{equation}
\end{lem}
\begin{proof}
The lemma follows easily by taking $s$ differentiation of (\ref{newE3new}) and (\ref{newnewE4}).
\end{proof}

We can prove Proposition  \ref{prop:kappatokappa+1} (3) by using Lemma \ref{expest2}.
Actually Proposition  \ref{prop:kappatokappa+1} (3) follows from (\ref{form184})
in the same way as the proof of Lemma \ref{Lemma6.9}, that is, by integration over $T$.

\begin{proof}[Proof of Proposition \ref{prop:inequalitieskappa} (3)]
Using the simplified notation as in Remark \ref{rem2222}
we have
$$
u^{\rho}_{T,(\kappa)}
``=\text{''}
u^{\rho_1} + \sum_{j=0}^{\kappa} V_{T,1,(j)}^{\rho}
$$
on $K_1$.
Therefore
$$
\aligned
\Vert u^{\rho}_{T,(\kappa)} - u^{\rho}_1\Vert_{L^2_{m+1}}
&``\le\text{''}
\sum_{j=0}^{\kappa} \Vert V_{T,1,(j)}^{\rho} \Vert_{L^2_{m+1}}
\\
&``\le\text{''}
C_{5,m} \sum_{j=0}^{\kappa} \mu^{j-1} e^{-\delta_1 T}
\le
C_{5,m}  \mu^{-1}\frac{1 - \mu^{\kappa+1}}{1-\mu} e^{-\delta_1 T}.
\endaligned
$$
This would prove (\ref{form184}) on $K_1$.
However here the above notation $+$ is rather imprecise
and we need to use exponential map and parallel transport.  In other
words we
need to estimate the effect of the nonlinearity of the exponential map to
work out the above inequality.
We use Lemma \ref{expest2} for this purpose.
The above inequality suggests that $C_{6,m}$
depends on $C_{5,m}$.
\par
In the proof of  Proposition \ref{prop:inequalitieskappa} (3),
we need to decide $C_{6,m}$ and $C_{7,m}$
depending on $C_{5,m}$, $C_{8,m}$, $C_{9,m}$.
So the dependence of the constants appearing
during the proof on $C_{5,m}$, $C_{6,m}$, $C_{7,m}$, $C_{8,m}$, $C_{9,m}$
need to be carefully examined.
\par
Hereafter during the proof of Proposition \ref{prop:inequalitieskappa} (3)
we use the notation of the constants $C_{m,*}$
and $D_{m,*}$ as follows. (Here $*$ is the number of the formula where
the constants appear.)
$C_{m,*}$ is a constant depending on all of
$C_{5,m}$, $C_{6,m}$, $C_{7,m}$, $C_{8,m}$, $C_{9,m}$ and on $m$.
$D_{m,*}$ is a constant depending $C_{5,m}$, $C_{8,m}$, $C_{9,m}$, $m$
{\it but is independent of $C_{6,m}$ and $C_{7,m}$}.
$T_{m,*}$ is a constant depending on all of
$C_{5,m}$, $C_{6,m}$, $C_{7,m}$, $C_{8,m}$, $C_{9,m}$ and on $m$.
\par\smallskip
We now start the proof.
{During the proof we assume $m-2 \ge n,\ell \ge 0$. 
.}
\par
On $K_1$ we have
$
u^{\rho}_{T,(\kappa)} = \Exp(u^{\rho}_{T,(\kappa-1)},V^{\rho}_{T,1,(\kappa)}).
$
Therefore by (\ref{form182}) for $\le \kappa$, (\ref{form184})
for $\le \kappa -1$,
and Lemma \ref{expest2} we have
\begin{equation}\label{eqE41}
\aligned
&\left\| \nabla_{\rho}^n \frac{\partial^{\ell}}{\partial T^{\ell}}
{\mathscr P_{1,\rho_1}{\rm E}(u^{\rho}_1,u^{\rho}_{T,(\kappa)})}\right\|_{L^2_{m+1-\ell,\delta}(K_1)}\\
&<
\left\| \nabla_{\rho}^n \frac{\partial^{\ell}}{\partial T^{\ell}}
{\mathscr P_{1,\rho_1}{\rm E}(u^{\rho}_1,u^{\rho}_{T,(\kappa-1)})}\right\|_{L^2_{m+1-\ell,\delta}(K_1)}\\
&\quad + C_{5,m}\mu^{\kappa-1}e^{-\delta_1 T} +
C_{m,{\rm(\ref{eqE41})}}\mu^{\kappa-1}e^{-2\delta_1 T}.
\endaligned
\end{equation}
Taking $T_{m,{\rm(\ref{form1842})}}$ so that $C_{m,{\rm(\ref{eqE41})}}e^{-\delta_1
T_{m,{\rm(\ref{form1842})}}} \le 1$
we have
\begin{equation}\label{form1842}
\aligned
&\text{LHS of (\ref{eqE41})} \\
&\le \left\| \nabla_{\rho}^n \frac{\partial^{\ell}}{\partial T^{\ell}}
{\mathscr P_{1,\rho_i}{\rm E}(u^{\rho}_1,u^{\rho}_{T,(\kappa-1)})}\right\|_{L^2_{m+1-\ell,\delta}(K_1)}
+ D_{m,{\rm(\ref{form1842})}}\mu^{\kappa-1}e^{-\delta_1 T}
\endaligned
\end{equation}
for $T \ge T_{m,{\rm(\ref{form1842})}}$.
{Here $\mathscr P_{1,\rho_i}$ is defined in (\ref{mathscrPform}).}
\par
We next study at the neck region.
We consider $m -2\ge n,\ell \ge 0$.
(The case $\ell = 0$ is included.)
We put
\begin{equation}
U^{\rho}_{T,i,(\kappa)}(\tau,t)
= V^{\rho}_{T,i,(\kappa)}(\tau,t) -  (\Delta p^{\rho}_{T,(\kappa)})^{\rm Pal}
\end{equation}
and
\begin{equation}
U(\tau',t) = U^{\rho}_{T,1,(\kappa)}(\tau',t)  + \chi(\tau'-4T)
U^{\rho}_{T,2,(\kappa)}(\tau'-10T,t).
\end{equation}
On $(\tau',t) \in [0,5T+1]_{\tau'} \times [0,1]$, we have
\begin{equation}\label{formA60}
\aligned
u^{\rho}_{T,(\kappa)}(\tau',t)
&=\Exp\Big(u^{\rho}_{T,(\kappa-1)}(\tau',t), V^{\rho}_{T,1,(\kappa)}(\tau',t) \\
&\qquad\quad+ \chi(\tau'-4T)(U^{\rho}_{T,2,(\kappa)}(\tau'-10T,t)\Big) \\
&=\Exp\Big(u^{\rho}_{T,(\kappa-1)}(\tau',t), U(\tau',t)+ (\Delta p^{\rho}_{T,(\kappa)})^{\rm Pal}).
\endaligned
\end{equation}

Suppose $[S',S'+1]_{\tau'} \subseteq  [0,5T+1]_{\tau'} \times [0,1]$.
We put
$$
\frak P_i
= \Big({\rm Pal}_{u_i}^{u^{\rho}_{T,(\kappa-1)}}\Big)^{-1}
= {\rm Pal}^{u_i}_{u^{\rho}_{T,(\kappa-1)}}
$$
and
$$
\Sigma(S') = [S',S'+1]_{\tau'} \times [0,1] \subset [0,5T+1]_{\tau'} \times [0,1].
$$
When we use {the} $\tau'$ (resp. $\tau''$) coordinate
we write $\Sigma(S')_{\tau'}$ (resp. $\Sigma(S')_{\tau''}$).
Using (\ref{1106}) we can estimate
\begin{equation}\label{newD5form}
\aligned
 &\left\Vert
 \nabla_{\rho}^n \frac{\partial^{\ell}}{\partial T^{\ell}}
 \frak P_1  U(\tau',t)\right\Vert_{L^2_{m+1-\ell}(\Sigma(S')_{\tau'})}
\\
&\le
\sum_{l=1}^{\ell}\sum_{n'=0}^{n}
D_{m,{\rm(\ref{newD5form})}} \left(\left\Vert
 \left(
 \nabla_{\rho}^{n'} \frac{\partial^{l}}{\partial T^{l}}\right)(\frak P_1 U^{\rho}_{T,1,(\kappa)}(\tau',t))
 \right\Vert_{L^2_{m-l+1}(\Sigma(S')_{\tau'})}
 \right.
 \\
& \qquad\qquad\qquad\qquad+
\left.
\left\Vert
 \left(
 \nabla_{\rho}^{n'} \frac{\partial^{l}}{\partial T^{l}}\right)(\frak P_1 U^{\rho}_{T,2,(\kappa)}(\tau'',t))
 \right\Vert_{L^2_{m+1}(\Sigma(S')_{\tau''})}
 \right)
\endaligned
\end{equation}
\par
We remark that the second term of (\ref{newD5form})
drops if $S' \le T$. (This is because $\chi(\tau'-4T) = 0$ on $\Sigma(S')_{\tau'}$ in that case.)
\par
By (\ref{form182}) we have
\begin{equation}\label{formE9E9}
\sum_{S'=-5T}^0\left\Vert\left(
 \nabla_{\rho}^{n'} \frac{\partial^{l}}{\partial T^{l}}\right)(\frak P_i U^{\rho}_{T,i,(\kappa)}(\tau',t))
\right\Vert_{L^2_{m-l+1,\delta}(\Sigma(S')_{\tau'})} \le
C_{5,m}  \mu^{\kappa-1} e^{-\delta_1 T},
\end{equation}
for $i=1,2$.
In case $i=2$ we need to replace $\frak P_2$ by $\frak P_1$ to apply (\ref{formE9E9})
to estimate (\ref{newD5form}).
\begin{sublem}\label{sublemE4}
We have
\begin{equation}\label{formE1121}
\aligned
&
\left\Vert
 \left(
 \nabla_{\rho}^{n'} \frac{\partial^{l}}{\partial T^{l}}\right)(\frak P_1 U^{\rho}_{T,2,(\kappa)}(\tau'',t))
 \right\Vert_{L^2_{m+1-l}(\Sigma(S')_{\tau''})} \\
&\le
D_{m,{\rm(\ref{formE1121})}}  \sum_{l'\le l}\sum_{n''\le n'}
\left\Vert
 \left(
 \nabla_{\rho}^{n''} \frac{\partial^{l'}}{\partial T^{l'}}\right)(\frak P_2 U^{\rho}_{T,2,(\kappa)}(\tau'',t))
 \right\Vert_{L^2_{m+1-l'}(\Sigma(S')_{\tau''})}
\endaligned
\end{equation}
if $T > T_{m,{\rm(\ref{formE1121})}}$ and $S' \ge T$.
\end{sublem}
Postponing the proof of the sublemma we continue the proof.
Note the ratio between the weights $e_{1,\delta}$ and $e_{T,\delta}$ are bounded on  $\tau' \in [0,5T]_{\tau'} \times [0,1]$
and the weight $e_{2,\delta}$ is larger than $e_{T,\delta}$.
Therefore taking the weighted sum of the square of (\ref{newD5form}) for
$S'=0,\dots,[5T],5T$ (where $[5T]$ is the largest integer with $\le 5T$)
and using (\ref{formE9E9}), (\ref{formE1121}) we obtain
\begin{equation}\label{newD5form2}
\aligned
 &\left\Vert
 \nabla_{\rho}^n \frac{\partial^{\ell}}{\partial T^{\ell}}
\frak P_1 U(\tau',t)\right\Vert_{L^2_{m+1-\ell,\delta}(K_1^{5T+1}
\subset \Sigma_T)}
\le
D_{m,{\rm(\ref{newD5form2})}} \mu^{\kappa-1} e^{-\delta_1 T}.
\endaligned
\end{equation}
We put
$
q = u_1(0,1/2)
$
and
$$
\frak P'_1
= {\rm Pal}^{u_1}_{q}
\circ {\rm Pal}^{q}_{u^{\rho}_{T,(\kappa-1)}}
$$

\begin{sublem}\label{sublemE5}
We have
\begin{equation}\label{newform71newmae}
\aligned
&\Big\Vert
\nabla_{\rho}^{n} \frac{\partial^{\ell}}{\partial T^{\ell}}
\Big(\frak P_1 \big((\Delta p^{\rho}_{T,(\kappa)})^{\rm Pal}\big)
\\
&\qquad\qquad- \frak P'_1\big((\Delta p^{\rho}_{T,(\kappa)})^{\rm Pal}\big)
\Big)
\Big\Vert_{L^2_{m+1-\ell,\delta}([0,5T+1]_{\tau'} \times [0,1])}\\
&\le
D_{m,{\rm(\ref{newform71newmae})}} \mu^{\kappa-1} e^{-\delta_1 T}.
\endaligned
\end{equation}
\end{sublem}
Here we regard $(\Delta p^{\rho}_{T,(\kappa)})^{\rm Pal}$ as a section on
$(u^{\rho}_{T,(\kappa-1)})^*TX$ (as we do in (\ref{formA60})).
So  (\ref{newform71newmae}) is
\begin{equation}\label{formformE14}
\aligned
&\left\Vert
\nabla_{\rho}^{n} \frac{\partial^{\ell}}{\partial T^{\ell}}
\Big(({\rm Pal}^{u_1}_{u^{\rho}_{T,(\kappa-1)}} \circ
{\rm Pal}_{p_{T,(\kappa-1)}^{\rho}}^{u^{\rho}_{T,(\kappa-1)}})
(\Delta p^{\rho}_{T,(\kappa)})
\right. \\
&- ({\rm Pal}^{u_1}_{q}
\circ {\rm Pal}^{q}_{u^{\rho}_{T,(\kappa-1)}}
\circ
{\rm Pal}_{p_{T,(\kappa-1)}^{\rho}}^{u^{\rho}_{T,(\kappa-1)}})(\Delta p^{\rho}_{T,(\kappa)})
\Big)
\Big\Vert_{L^2_{m+1-\ell}([0,5T+1]_{\tau'} \times [0,1])} \\
&\le
D_{m,{\rm(\ref{newform71newmae})}} \mu^{\kappa-1} e^{-\delta_1 T}.
\endaligned
\end{equation}
Postponing the proof of the sublemma we continue the proof.
\par
Lemma \ref{expest2} and (\ref{formA60}),
(\ref{newD5form2}), (\ref{newform71newmae}),
imply
\begin{equation}\label{formA462}
\aligned
&\left\Vert
 \nabla_{\rho}^n \frac{\partial^{\ell}}{\partial T^{\ell}}
 \Big(
\frak P_1 {\rm E}(u^{\rho}_{T,(\kappa-1)},u^{\rho}_{T,(\kappa)})
\right.\\
&\qquad\qquad
\left.
-  \frak P'_1\Big((\Delta p^{\rho}_{T,(\kappa)})^{\rm Pal}\Big)
\Big)
\right\Vert_{L^2_{m+1-\ell,\delta}([0,5T+1]_{\tau'} \times [0,1]
\subset \Sigma_T)}
\\
&\le
D_{m,{\rm(\ref{formA462})}} \mu^{\kappa-1} e^{-\delta_1 T}.
\endaligned
\end{equation}
Using also Lemma \ref{expest2}
we can show the next sublemma.
\begin{sublem}\label{SublemE6}
The next inequality holds.
\begin{equation}\label{formA63}
\aligned
&\left\Vert
 \nabla_{\rho}^n \frac{\partial^{\ell}}{\partial T^{\ell}}
{\mathscr P_{1,\rho_1} {\rm E}(u^{\rho}_1,u^{\rho}_{T,(\kappa)})}
 \right\Vert_{W^2_{m+1-\ell,\delta}([0,5T+1]_{\tau'} \times [0,1]
\subset \Sigma_T)}
\\
&
\le
\left\Vert
 \nabla_{\rho}^{n} \frac{\partial^{\ell}}{\partial T^{\ell}}
{\mathscr P_{1,\rho_1}{\rm E}(u^{\rho}_1,u^{\rho}_{T,(\kappa-1)})}
 \right\Vert_{W^2_{m+1-\ell,\delta}([0,5T+1]_{\tau'} \times [0,1]
\subset \Sigma_T)}
\\
&
\quad+ D_{m,{\rm(\ref{formA63})}} \mu^{\kappa-1} e^{-\delta_1 T}.
\endaligned
\end{equation}
\end{sublem}
Note the term corresponding to
$ \frak P'_1((\Delta p^{\rho}_{T,(\kappa)})^{\rm Pal})
= ({\rm Pal}^{u_1}_{q}
\circ {\rm Pal}^{q}_{u^{\rho}_{T,(\kappa-1)}})((\Delta p^{\rho}_{T,(\kappa)})^{\rm Pal})$
disappears in (\ref{formA63}) since we take $W^{2}_{m+1-\ell',\delta}$
norm in place of $L^{2}_{m+1-\ell',\delta}$ norm.
See (\ref{normformjula5}) and we recall $q = u_1(0,1/2)$.
Postponing the detail of the proof of the sublemma we continue the proof.
\par
We choose $C_{6,m}$ such that
$$
D_{m,{\rm(\ref{formA63})}}
 \le C_{6,m}(1-\mu)/10,
\quad
D_{m,{\rm(\ref{form1842})}}
\le C_{6,m}(1-\mu)/10.
$$
Then (\ref{form1842}), (\ref{formA63})
and the {$(\kappa-1)$-th} case of {(\ref{form185})} implies that for $\ell > 0$ we have
\begin{equation}
\aligned
&\left\Vert
 \nabla_{\rho}^n \frac{\partial^{\ell}}{\partial T^{\ell}}
 {\mathscr P_{1,\rho_1}{\rm E}(u^{\rho}_1,u^{\rho}_{T,(\kappa)})}
 \right\Vert_{W^2_{m+1-\ell,\delta}(K_1^{5T+1})} \\
 &
 \le
\left\Vert
 \nabla_{\rho}^n \frac{\partial^{\ell}}{\partial T^{\ell}}
 {\mathscr P_{1,\rho_1}{\rm E}(u^{\rho}_1,u^{\rho}_{T,(\kappa-1)})}
 \right\Vert_{W^2_{m+1-\ell,\delta}(K_1^{5T+1})}
 \\
 &\quad +
D_{m,{\rm(\ref{formA63})}} \mu^{\kappa-1} e^{-\delta_1 T}
 +
 D_{m,{\rm(\ref{form1842})}} \mu^{\kappa-1} e^{-\delta_1 T}
 \\
 &\le
 C_{6,m} (2 - \mu^{\kappa-1})e^{-\delta_1 T}
 + \frac{C_{6,m}(\mu^{\kappa-1}-\mu^{\kappa})e^{-\delta_1 T}}{5}
 \\
& \le
 C_{6,m} (2 - \mu^{\kappa})e^{-\delta_1 T}.
 \endaligned
\end{equation}
We thus proved {the $\kappa$-th case} of (\ref{form184}).

We next prove (\ref{form185}). We consider $m -2\ge n,\ell \ge 0$.
(The case $\ell = 0$ is included.)
We denote
$$
d_{\ell,n}(u,v)
=
\sum_{n'\le n
\atop \ell' \le \ell}
\left\Vert \nabla_{\rho}^{n'} \frac{\partial^{\ell'}}{\partial T^{\ell'}}
{\rm Pal}^{u_1}_u {\rm E}(u,v)
\right\Vert
_{L^2_{m+1-\ell'}(K_1^{9T} \setminus K_1^T)}.
$$
Here $u,v$ are $(T,\rho)$ dependent family of maps on $\Sigma_1$.
(We assume $d(u_1(z),u(z)) \le \iota_X$ etc. so that
${\rm Pal}_{u_1}^u$ etc. is defined.)
We remark that we take $L^2_{m+1-\ell'}$ norm without weight
in the right hand side.
\par
Note $d_{\ell,n}$ is not a metric. In particular it may not satisfy triangle inequality.
However Lemma \ref{expest2} implies
\begin{equation}\label{weaktriangleeq}
d_{\ell,n}(v_1,v_3) \le d_{\ell,n}(v_1,v_2) + d_{\ell,n}(v_2,v_3) +
C_{m,\rm(\ref{weaktriangleeq})}d_{\ell,n}(v_1,v_2)d_{\ell,n}(v_2,v_3)
\end{equation}
if
$$
\left\Vert \nabla_{\rho}^{n'} \frac{\partial^{\ell'}}{\partial T^{\ell'}} {\rm E}(u_1,v_i)
\right\Vert_{L^2_{m+1-\ell'}(K_1^{9T} \setminus K_1^T)}
\le 1
$$
for $i=1,2,3$, $n' \le n$, $\ell' \le \ell$.
\par
Note (\ref{formA462}) holds when we replace $\frak P'_1$ there by $\frak P_1$.
It then implies
\begin{equation}\label{formA666}
d_{\ell,n}(u^{\rho}_{T,(\kappa)},{\rm Exp}(u^{\rho}_{T,(\kappa-1)},(\Delta p^{\rho}_{T,(\kappa)})^{\rm Pal}))
\le
D_{m,{\rm(\ref{formA666})}} \mu^{\kappa-1} e^{- \delta_1 T}.
\end{equation}
We put
$$
\square_{\kappa-1} = {\rm E}(p_{0}^{\rho},u^{\rho}_{T,(\kappa-1)}).
$$
The induction hypothesis, which is $\kappa-1$ case of (\ref{form185}), implies
\begin{equation}\label{formA67-}
\left\Vert\nabla_{\rho}^{n} \frac{\partial^{\ell}}{\partial T^{\ell}}
{\rm Pal}^{p_0}_{p_0^{\rho}}
(\square_{\kappa-1})
\right\Vert
_{L^2_{m+1-\ell}(K_1^{9T} \setminus K_1^T)}
\le C_{7,m} (2- \mu^{\kappa-1})e^{-\delta_1 T}.
\end{equation}
Note (\ref{form185}) is actually (\ref{form185real}).
\par
Then Lemma \ref{expest2} implies
\begin{equation}\label{formA67}
\aligned
&d_{\ell,n}({\rm Exp}(u^{\rho}_{T,(\kappa-1)},(\Delta p^{\rho}_{T,(\kappa)})^{\rm Pal}),
{\rm Exp}(p_{0}^{\rho},\square_{\kappa-1}+{\rm Pal}^{p_0^{\rho}}_{p^{\rho}_{T,(\kappa-1)}}(\Delta p^{\rho}_{T,(\kappa)})) \\
&\le
C_{m,{\rm(\ref{formA67})}} \mu^{\kappa-1} e^{-2 \delta_1 T}.
\endaligned
\end{equation}
On the other hand, when we regard
$(\Delta p^{\rho}_{T,(\kappa)})^{\rm Pal}$ as an element of
$T_{p_{0}^{\rho}}X$ we have
\begin{equation}\label{newform71new}
\aligned
&\left\Vert
\nabla_{\rho}^{n} \frac{\partial^{\ell}}{\partial T^{\ell}}
{\rm Pal}^{p_0}_{p_0^{\rho}}({\rm Pal}^{p_0^{\rho}}_{p^{\rho}_{T,(\kappa-1)}}(\Delta p^{\rho}_{T,(\kappa)}))
\right\Vert_{L^2_{m+1-\ell}(K_1^{9T} \setminus K_1^T)}
\\
&\le
D_{m,{\rm(\ref{newform71new})}} \mu^{\kappa-1} e^{-\delta_1 T}.
\endaligned
\end{equation}
This follows from (\ref{form182}), Lemma \ref{lemB2} and
\begin{equation}\label{newnewnewAE}
\Big\Vert
\nabla_{\rho}^{n} \frac{\partial^{\ell}}{\partial T^{\ell}}
{\rm Pal}_{p_0^{\rho}}^{p_0}({\rm E}(p_0^{\rho},p^{\rho}_{T,(\kappa)}))
\Big\Vert_{L^2_{m+1-\ell}(K_1^{9T} \setminus K_1^T)}
\le C_{m,{\rm(\ref{newnewnewAE})}} \mu^{\kappa-1} e^{-\delta_1T}.
\end{equation}
\par
(\ref{formA666}), (\ref{formA67})
and (\ref{weaktriangleeq})  imply
\begin{equation}\label{formA71}
d_{\ell,n}(u^{\rho}_{T,(\kappa)},
{\rm Exp}(p_{0}^{ \rho},\square_{\kappa-1}+{\rm Pal}^{p_0^{\rho}}_{p^{\rho}_{T,(\kappa-1)}}(\Delta p^{\rho}_{T,(\kappa)}))
\le D_{m,\rm{(\ref{formA71})}} \mu^{\kappa-1} e^{-\delta_1 T}
\end{equation}
if $T > T_{m,{\rm(\ref{formA71})}}$.
We use the fac{t} that the exponent of $e$ in
(\ref{formA67}) is $-2\delta_1 T$,
to change  $C_{m,{\rm(\ref{formA67})}}$ there to
$ D_{m,\rm{(\ref{formA71})}}$.
The third term of (\ref{weaktriangleeq}) is estimated by
$C \mu^{2(\kappa-1)}e^{-2\delta_1 T}$ in our case. So
it is estimated by $D \mu^{\kappa-1}e^{-\delta_1 T}$
if $T$ is sufficiently large.
\par
Using Lemma \ref{expest2}, (\ref{newform71new}),
(\ref{formA67-}),
(\ref{formA71}) imply
\begin{equation}\label{formA699}
\aligned
&\left\Vert\nabla_{\rho}^{n} \frac{\partial^{\ell}}{\partial T^{\ell}}
{\rm Pal}^{p_0}_{p_0^{\rho}}({\rm E}(p_{0}^{\rho},u^{\rho}_{T,(\kappa)}))
\right\Vert
_{L^2_{m+1-\ell}(K_1^{9T} \setminus K_1^T)}
\\
&\le C_{7,m} (2- \mu^{\kappa-1})e^{-\delta_1 T}
+
D_{m,\rm(\ref{formA699})} \mu^{\kappa-1} e^{ - \delta_1 T}.
\endaligned
\end{equation}
We choose $C_{7,m}$ such that
$$
C_{7,m}(1 - \mu)
>
D_{m,\rm(\ref{formA699})}.
$$
Then
\begin{equation}\label{formA6990}
\aligned
&\left\Vert\nabla_{\rho}^{n} \frac{\partial^{\ell}}{\partial T^{\ell}}
{\rm Pal}^{p_0}_{p_0^{\rho}}{\rm E}(p_{0}^{\rho},u^{\rho}_{T,(\kappa)})
\right\Vert
_{L^2_{m+1-\ell}(K_1^{9T} \setminus K_1^T)}
\le C_{7,m}(2-\mu^{\kappa})e^{-\delta_1 T}.
\endaligned
\end{equation}
\begin{rem}
We remark that when we consider the domain $K^{9T}_1 \setminus K_1$ in place of
 $K^{9T}_1 \setminus K_1^T$ and $\ell = 0$ then
 (\ref{formA6990}) fails in the first step, that is, $\kappa = 0$.
\end{rem}
\par
To complete the proof of  Proposition \ref{prop:inequalitieskappa} (3) it remains to prove
sublemmata.
\end{proof}

\begin{proof}[Proof of Sublemma \ref{sublemE4}]
We take and fix $\rho_0'$ and prove this estimate under
the assumption that $d(\rho,\rho_0') \le e^{-\delta_1 T}$.
As usual we use {the} $\tau'$ (resp $\tau''$) coordinate in the
definition of {the} $L^2_{m+1}$ norm for $i=1$ (resp. $i=2$).
We consider maps
$$
\frak P''_i = {\rm Pal}^{u_i}_{p^{\rho_0'}} : \Gamma(\Sigma(S');T_{p^{\rho_0'}}X)
\to \Gamma(\Sigma(S');u_i^*TX).
$$
Let $W$ be a $(T,\rho)$ parameterized family of sections of
$\Gamma(\Sigma(S');T_{p^{\rho_0'}}X)$.
By using $(T,\rho)$ independence of $\frak P''_i$  we can prove
\begin{equation}\label{newformE19}
\aligned
&D^{-1}_{m,{\rm (\ref{newformE19})}}\left\Vert\nabla_{\rho}^n
\frac{\partial^{l}}{\partial T^{l}} \frak P''_i (W)\right\Vert_{L^2_{m+1-l}(\Sigma(S'))}\\
&\le
\left\Vert\nabla_{\rho}^n \frac{\partial^{l}}{\partial T^{l}} W
\right\Vert_{L^2_{m+1-l}(\Sigma(S'))}
\\
&\le
D_{m,{\rm (\ref{newformE19})}}
\left\Vert\nabla_{\rho}^n \frac{\partial^{l}}{\partial T^{l}} \frak P''_i (W)
\right\Vert_{L^2_{m+1-l}(\Sigma(S'))}
\endaligned
\end{equation}
for $i=1,2$.
\par
Since $S' \ge T$ we have
\begin{equation}\label{formE18}
\Vert {\rm E}(p_0^{\rho_0'},u^{\rho}_{T,(\kappa-1)})\Vert_{L^2_{m+1}([S',S'+1]_{\tau'} \times [0,1])}
\le C_{m,{\rm (\ref{formE18})}} e^{-\delta_1 T}.
\end{equation}
Here $C_{m,{\rm (\ref{formE18})}}$ depends on $C_{7,m}$.
\par
We put:
$$
\frak P'''
= {\rm Pal}_{p_0^{\rho_0'}}^{u^{\rho}_{T,(\kappa-1)}} : \Gamma(\Sigma(S');T_{p_0^{\rho_0'}}X)
\to \Gamma(\Sigma(S');(u_{T,(\kappa-1)}^{\rho})^*TX).
$$
If $\hat W$ is a $(T,\rho)$ independent section of
$u_2^*TX$ on $\Sigma(S')$, Lemma \ref{lemB2} implies
\begin{equation}\label{formE20221}
\aligned
&\Vert
\nabla_{\rho}^n \frac{\partial^{l}}{\partial T^{l}}(\frak P''_1 \circ (\frak P''_2)
^{-1} - \frak P_1 \circ (\frak P_2)
^{-1})
(\hat W)
\Vert_{L^2_{m+1-l}(\Sigma(S'))}
\\
&\le
\Vert
\nabla_{\rho}^n \frac{\partial^{l}}{\partial T^{l}}(\frak P''_1 \circ (\frak P''_2)
^{-1} - \frak P_1 \circ \frak P''' \circ (\frak P''_2)
^{-1})
(\hat W)
\Vert_{L^2_{m+1-l}(\Sigma(S'))}
\\
&\quad+
\Vert
\nabla_{\rho}^n \frac{\partial^{l}}{\partial T^{l}}(
 \frak P_1 \circ \frak P''' \circ (\frak P''_2)
^{-1} - (\frak P_1 \circ (\frak P_2)
^{-1})
(\hat W)
\Vert_{L^2_{m+1-l}(\Sigma(S'))}
\\
&\le
C_{m,{\rm(\ref{formE20221})}} e^{-\delta_1 T} \Vert \hat W \Vert_{L^2_{m+1-l}(\Sigma(S'))}
\endaligned
\end{equation}
See Figure \ref{FigureinE1}.
\begin{figure}
\centering
\includegraphics{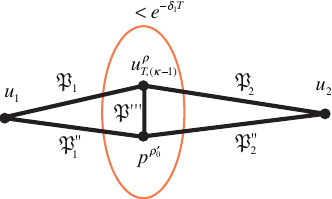}
\caption{$\frak P'''$, $\frak P''_i$ and $\frak P_i$}
\label{FigureinE1}
\end{figure}
\par
We put
$$
\hat W(T,\rho) = \frak P_2(U^{\rho}_{T,2,(\kappa)}).
$$
Then
using (\ref{formE9E9}), (\ref{formE20221})
we can derive
\begin{equation}\label{formE112100}
\aligned
&\left\Vert\nabla_{\rho}^{n'} \frac{\partial^{l}}{\partial T^{l}}
\frak P_1(U^{\rho}_{T,2,(\kappa)}) - \frak P''_1(\frak P''_2)^{-1}(\hat W(T,\rho)))
\right\Vert_{L^2_{m+1-l}(\Sigma(S')_{\tau''})}
\\
&=
\left\Vert\nabla_{\rho}^{n'} \frac{\partial^{l}}{\partial T^{l}}
(\frak P_1\circ \frak P_2^{-1} - \frak P''_1\circ (\frak P''_2)^{-1}) (\hat W(T,\rho))
\right\Vert_{L^2_{m+1-l}(\Sigma(S')_{\tau''})}
\\
&
\le  \sum_{n''\le n'}\sum_{l'\le l}
C_{m,{\rm(\ref{formE112100})}} e^{-\delta_1 T} \left\Vert
\nabla_{\rho}^{n''} \frac{\partial^{l'}}{\partial T^{l'}} \hat W(T,\rho)
\right\Vert_{L^2_{m+1-l'}(\Sigma(S')_{\tau''})}.
\endaligned
\end{equation}
Therefore using (\ref{newformE19}) also we have
\begin{equation}\label{formE11210}
\aligned
&\left\Vert
 \left(
 \nabla_{\rho}^{n'} \frac{\partial^{l}}{\partial T^{l}}\right)(\frak P_1 U^{\rho}_{T,2,(\kappa)}(\tau'',t))
 \right\Vert_{L^2_{m+1-l}(\Sigma(S')_{\tau''})} \\
&\le
(D_{m,{\rm(\ref{formE11210})}}
+  C_{m,{\rm(\ref{formE11210})}}e^{-\delta_1 T}) \\
&\qquad \sum_{n''\le n'}\sum_{l'\le l}\left\Vert
\nabla_{\rho}^{n''} \frac{\partial^{l'}}{\partial T^{l'}} \hat W(T,\rho)
\right\Vert_{L^2_{m+1-l'}(\Sigma(S')_{\tau''})}.
\endaligned
\end{equation}
\par
We take $T_{m,{\rm(\ref{formE1121})}}$ such that
$e^{-\delta_1 T}C_{m,{\rm(\ref{formE11210})}} \le 1$
if $T > T_{m,{\rm(\ref{formE1121})}}$.
The sublemma is proved.
\end{proof}
\begin{proof}[Proof of Sublemma \ref{sublemE5}]
The proof is similar to the proof of Sublemma \ref{sublemE4}.
We take and fix $\rho_0'$ and prove this estimate under
the assumption that $d(\rho,\rho_0') \le e^{-\delta_1 T}$.
\par
We observe
\begin{equation}\label{newformE23}
\aligned
\left\Vert
\nabla_{\rho}^n \frac{\partial^{\ell}}{\partial T^{\ell}}
{\rm E}(u^{\rho_0'}_1,u_{T,(\kappa-1)}^{\rho})
\right\Vert_{L^2_{m+1-\ell}(\Sigma(S')_{\tau'})}
&\le
C_{m,{\rm(\ref{newformE23})}} e^{-\delta_1 T},
\\
\left\vert
\nabla_{\rho}^n \frac{\partial^{\ell}}{\partial T^{\ell}}
{\rm E}(p^{\rho_0'}_0,p_{T,(\kappa-1)}^{\rho})
\right\vert
&\le
C'_{m,{\rm(\ref{newformE23})}} e^{-\delta_1 T}.
\endaligned
\end{equation}
Therefore using an obvious variant of Lemma \ref{lemB2} we can prove
\begin{equation}\label{newformE24}
\aligned
&\left\Vert
\nabla_{\rho}^n \frac{\partial^{\ell}}{\partial T^{\ell}}
\Big(
\frak P_1 \circ {\rm Pal}_{p_{T,(\kappa-1)}^{\rho}}^{u_{T,(\kappa-1)}^{\rho}}
\circ
{\rm Pal}^{p_{T,(\kappa-1)}^{\rho}}_{p_0}
\right.
\\
&\qquad\qquad\quad
\left.-
{\rm Pal}^{u_1}_{u_1^{\rho'_0}}
\circ
{\rm Pal}^{u_1^{\rho'_0}}_{p^{\rho'_0}_0}
\circ
{\rm Pal}_{p_0}^{p^{\rho'_0}_0}
\Big)(W)
\right\Vert_{L^2_{m+1-\ell}(\Sigma(S')_{\tau'})}
\\
&\le
C_{m,{\rm(\ref{newformE24})}} e^{-\delta_1 T}\Vert W\Vert
\endaligned
\end{equation}
and
\begin{equation}\label{newformE25}
\aligned
&\left\Vert
\nabla_{\rho}^n \frac{\partial^{\ell}}{\partial T^{\ell}}
\Big(
\frak P'_1 \circ {\rm Pal}_{p_{T,(\kappa-1)}^{\rho}}^{u_{T,(\kappa-1)}^{\rho}}
\circ
{\rm Pal}^{p_{T,(\kappa-1)}^{\rho}}_{p_0}
\right.
\\
&\qquad\qquad\quad
\left.-
{\rm Pal}^{u_1}_{q}
\circ
{\rm Pal}^{q}_{u_1^{\rho'_0}}
\circ
{\rm Pal}^{u_1^{\rho'_0}}_{p^{\rho'_0}_0}
\circ
{\rm Pal}_{p_0}^{p^{\rho'_0}_0}
\Big)(W)
\right\Vert_{L^2_{m+1-\ell}(\Sigma(S')_{\tau'})}
\\
&\le
C_{m,{\rm(\ref{newformE25})}} e^{-\delta_1 T}\Vert W\Vert
\endaligned
\end{equation}
for $W \in T_{p_0}X$.
See Figure \ref{FigureinE2}.
\begin{figure}
\centering
\includegraphics{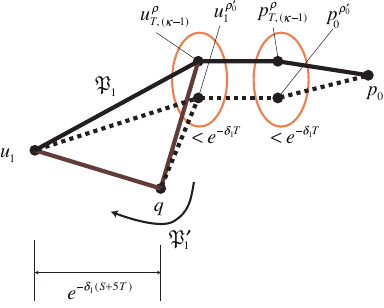}
\caption{Parallel transportations in the proof of
Sublemma \ref{sublemE5}}
\label{FigureinE2}
\end{figure}
\par
Note (\ref{form182}) and (\ref{form185real}) imply
\begin{equation}\label{newformE26}
\left\Vert
\nabla_{\rho}^n \frac{\partial^{\ell}}{\partial T^{\ell}}
\Big(
{\rm Pal}_{p_{T,(\kappa-1)}^{\rho}}^{p_0}
(\Delta p^{\rho}_{T,(\kappa)})
\Big)
\right\Vert_{L^2_{m+1-\ell}(\Sigma(S')_{\tau'})}
\le
C_{m,{\rm(\ref{newformE26})}} \mu^{\kappa-1}e^{-\delta_1 T}.
\end{equation}
On the other hand a similar formula as
(\ref{newformE19}) holds with
$\frak P''_i$ replaced by
${\rm Pal}_{p_0}^{p^{\rho'_0}_0}$ or by
${\rm Pal}^{u^{\rho'_0}_i}_{p^{\rho'_0}_0}$.
Therefore we obtain
\begin{equation}\label{newformE27}
\aligned
&\left\Vert
\nabla_{\rho}^n \frac{\partial^{\ell}}{\partial T^{\ell}}
\Big(
({\rm Pal}^{u_1^{\rho'_0}}_{p^{\rho'_0}_0}\circ
{\rm Pal}_{p_0}^{p^{\rho'_0}_0}\circ
{\rm Pal}_{p_{T,(\kappa-1)}^{\rho}}^{p_0})
(\Delta p^{\rho}_{T,(\kappa)})
\Big)
\right\Vert_{L^2_{m+1-\ell}(\Sigma(S'))}
\\
&\le
C_{m,{\rm(\ref{newformE27})}} \mu^{\kappa-1}e^{-\delta_1 T}.
\endaligned
\end{equation}
By Lemma \ref{lem:Ckexpdecay}
and Condition \ref{conds:smalldiam}, we derive
\begin{equation}\label{formnewE28}
\Vert
{\rm E}(q,u_1)
\Vert_{L^2_{m+1-\ell}(\Sigma(S')_{\tau'})}
\le D_{m,{\rm(\ref{formnewE28})}} e^{-\delta_1 S'}.
\end{equation}
Therefore by Lemma \ref{lemB2}
\begin{equation}\label{formnewE29}
\aligned
&\left\Vert
\Big(
{\rm Pal}_{u_1^{\rho'_0}}^{u_1}
-
{\rm Pal}_q^{u_1}
\circ
{\rm Pal}^q_{u_1^{\rho'_0}}
\Big)
(W)
\right\Vert_{L^2_{m+1-\ell}(\Sigma(S')_{\tau'})}
\\
&\le
D_{m,{\rm(\ref{formnewE29})}} e^{-\delta_1S'}
\Vert W\Vert_{L^2_{m+1-\ell}(\Sigma(S')_{\tau'})}
\endaligned
\end{equation}
for $W \in \Gamma(\Sigma(S'),({u_1^{\rho'_0}})^*TX)$.
\par
Now  (\ref{newformE24}), (\ref{newformE25}), (\ref{newformE27}),
(\ref{formnewE29})
imply
\begin{equation}\label{aaaE25}
\aligned
&\left\Vert
\nabla_{\rho}^n \frac{\partial^{\ell}}{\partial T^{\ell}}
\Big(\frak P_1 \big((\Delta p^{\rho}_{T,(\kappa)})^{\rm Pal}\big)
- \frak P'_1((\Delta p^{\rho}_{T,(\kappa)})^{\rm Pal})
\Big)
\right\Vert_{L^2_{m+1-\ell}(\Sigma(S')_{\tau'})}
\\
&
\le (C_{m,{\rm(\ref{aaaE25})}} e^{-\delta_1 T} + D_{m,{\rm(\ref{aaaE25})}} e^{-\delta_1 S'} )
\mu^{\kappa-1} e^{-\delta_1 T}.
\endaligned
\end{equation}
By taking weighted sum of the square of (\ref{aaaE25}),
Sublemma \ref{sublemE5} follows easily.
\end{proof}
\begin{proof}[Proof of Sublemma \ref{SublemE6}]
We take and fix $\rho_0'$ and prove this estimate under
the assumption that $d(\rho,\rho_0') \le e^{-\delta_1 T}$.
\par
We put
$$
(\Delta q_{T,(\kappa)}^{\rho})'
=
({\rm Pal}_{u_1^{\rho'_0}}^q
\circ
{\rm Pal}^{u_1^{\rho'_0}}_{p^{\rho'_0}_0}
\circ
{\rm Pal}^{p^{\rho'_0}_0}_{p_0}
\circ
{\rm Pal}_{p_{T,(\kappa-1)}^{\rho}}^{p_0})
(\Delta p_{T,(\kappa)}^{\rho})
\in \Gamma(\Sigma(S'),T_qX)
$$
and
$$
\Delta q_{T,(\kappa)}^{\rho}
=
({\rm Pal}_{p^{\rho'_0}_0}^q
\circ
{\rm Pal}^{p^{\rho'_0}_0}_{p_0}
\circ
{\rm Pal}_{p_{T,(\kappa-1)}^{\rho}}^{p_0})
(\Delta p_{T,(\kappa)}^{\rho})
\in T_qX.
$$
(See Figure \ref{FigureinE2}.)
\par
Note $(\Delta q_{T,(\kappa)}^{\rho})'$ is $z\in \Sigma(S')$ dependent
but $\Delta q_{T,(\kappa)}^{\rho}$ is a $(T,\rho)$ dependent
family of constant maps on  $\Sigma(S')$.
\par
By (\ref{newformE25}) we obtain:
\begin{equation}\label{aaaE31}
\aligned
&\Big\Vert
\nabla_{\rho}^n \frac{\partial^{\ell}}{\partial T^{\ell}}
\Big(
({\rm Pal}_{u_{T,(\kappa-1)}^{\rho}}^{q}
\circ {\rm Pal}^{u_{T,(\kappa-1)}^{\rho}}_{p_{T,(\kappa-1)}^{\rho}})
(\Delta p_{T,(\kappa)}^{\rho}))
-
(\Delta q_{T,(\kappa)}^{\rho})'
\Big)
\Big\Vert_{L^2_{m+1-\ell}(\Sigma(S'))}
\\
&\le C_{m,{\rm(\ref{aaaE31})}} \mu^{\kappa-1}e^{-2\delta_1  T}.
\endaligned
\end{equation}
We use Lemma \ref{lemB2} and
\begin{equation}\label{aaaE3100}
\Big\Vert
{\rm E}(p^{\rho'_0}_0,u_1^{\rho'_0})
\Big\Vert_{L^2_{m+1-\ell}(\Sigma(S'))}
\le D_{m,{\rm(\ref{aaaE3100})}} e^{-\delta_1 S'}
\end{equation}
to derive
\begin{equation}\label{aaaE31222}
\aligned
&\Big\Vert
\nabla_{\rho}^n \frac{\partial^{\ell}}{\partial T^{\ell}}
\Big(
\Delta q_{T,(\kappa)}^{\rho}
-
(\Delta q_{T,(\kappa)}^{\rho})'
\Big)
\Big\Vert_{L^2_{m+1-\ell}(\Sigma(S'))}
\\
&\le D_{m,{\rm(\ref{aaaE31222})}} \mu^{\kappa-1}e^{-\delta_1 (T + S')}.
\endaligned
\end{equation}
On the other hand,  (\ref{formA462}) and Lemma \ref{expest2} imply

\begin{equation}\label{aaaE32}
\aligned
&
\Big\Vert
\nabla_{\rho}^n \frac{\partial^{\ell}}{\partial T^{\ell}}
\Big(
{\rm E}(u_1,u^{\rho}_{T,(\kappa)})
-
{\rm E}(u_1,u^{\rho}_{T,(\kappa-1)})
\\
&-
({\rm Pal}^{u_1}_{q}
\circ
{\rm Pal}_{u_{T,(\kappa-1)}^{\rho}}^{q}
\circ {\rm Pal}^{u_{T,(\kappa-1)}^{\rho}}_{p_{T,(\kappa-1)}^{\rho}})
(\Delta p_{T,(\kappa)}^{\rho})
\Big)\Big\Vert_{L^2_{m+1-\ell,\delta}([0,5T+1]_{\tau'} \times [0,1]
\subset \Sigma_T)}
\\
&\le D_{m,{\rm(\ref{aaaE32})}} \mu^{\kappa-1}e^{-\delta_1 T}.
\endaligned
\end{equation}
By (\ref{aaaE31}), (\ref{aaaE31222}) and (\ref{aaaE32}) we obtain
\begin{equation}\label{newnewE38}
\aligned
&
\Big\Vert
\nabla_{\rho}^n \frac{\partial^{\ell}}{\partial T^{\ell}}
\Big(
{\rm E}(u_1,u^{\rho}_{T,(\kappa)})
-
{\rm E}(u_1,u^{\rho}_{T,(\kappa-1)})
\\
&\qquad\qquad-
{\rm Pal}^{u_1}_{q}
(\Delta q_{T,(\kappa)}^{\rho})\Big)
\Big\Vert_{L^2_{m+1-\ell,\delta}([0,5T+1]_{\tau'} \times [0,1]
\subset \Sigma_T)}
\\
&\le D_{m,{\rm(\ref{newnewE38})}} \mu^{\kappa-1}e^{-\delta_1 T}.
\endaligned
\end{equation}
Using Lemma \ref{expest2} we can also show
\begin{equation}\label{aaaE323}
\aligned
&\Big\Vert
\nabla_{\rho}^n \frac{\partial^{\ell}}{\partial T^{\ell}}
\Big(
{\rm E}(u_1,u^{\rho}_{T,(\kappa)})
-
{\rm E}(u_1,u^{\rho}_{T,(\kappa-1)})
\\
&\qquad-
({\rm Pal}^{u_1}_{q}
\circ
{\rm Pal}_{u_{T,(\kappa-1)}^{\rho}}^{q}
\circ {\rm Pal}^{u_{T,(\kappa-1)}^{\rho}}_{p_{T,(\kappa-1)}^{\rho}})
(\Delta p_{T,(\kappa)}^{\rho})\Big)
\Big\Vert_{L^2_{m+1-\ell,\delta}(\Sigma(S')_{\tau'}
\subset \Sigma_T)}
\\
&\le C_{m,{\rm(\ref{aaaE323})}} \mu^{\kappa-1}e^{-2\delta_1 T}.
\endaligned
\end{equation}
for $S' \ge T$.
In fact $u^{\rho}_{T,(\kappa-1)}$, $u^{\rho}_{T,(\kappa)}$,
$p^{\rho}_{T,(\kappa-1)}$ are all close to each other by the order of
$e^{-\delta_1 T}$ there, including their $(T,\rho)$ derivatives.
Also $u_1$ and $q$ are close by the order of
$e^{-\delta_1 T}$, there.
Therefore all the error terms
appearing while applying Lemma \ref{expest2} are of the order $ \mu^{\kappa-1}e^{-2\delta_1 T}$.
\par
(\ref{aaaE31}), (\ref{aaaE31222}) and (\ref{aaaE323}) imply
\begin{equation}\label{aaaE3234}
\aligned
&\Big\Vert
\nabla_{\rho}^n \frac{\partial^{\ell}}{\partial T^{\ell}}
\Big(
{\rm E}(u_1,u^{\rho}_{T,(\kappa)})(0,1/2)
-
{\rm E}(u_1,u^{\rho}_{T,(\kappa-1)})(0,1/2)
\\
&\qquad\qquad\qquad\qquad\qquad\qquad\,\,\,\,\,\,-
\Delta q_{T,(\kappa)}^{\rho}\Big)
\Big\Vert_{L^2_{m+1-\ell,\delta}(\Sigma(S')_{\tau'}
\subset \Sigma_T)}
\\
&\le C_{m,{\rm(\ref{aaaE3234})}} \mu^{\kappa-1}e^{-2\delta_1 T}.
\endaligned
\end{equation}

(\ref{formA63}) follows easily from  (\ref{newnewE38}) and (\ref{aaaE3234}).
\end{proof}

\section{Estimate of Parallel transport 3}
\label{appendixA2bisbis}

\begin{proof}[Proof of Lemma \ref{lemma615}]
The proof is similar to the argument of  Appendix \ref{appendixA2bis}.
We put
$$
W_0(z) = {\rm E}(u_1(z), \hat u^{\rho}_{1,T,(\kappa)})
$$
then we estimate  ${\bf e}'_{i}(z,0)$ in (\ref{forme'e'e'}) and obtain
\begin{equation}\label{formA77}
\aligned
&\left\Vert
\nabla_{\rho}^n \left(\frac{\partial^{\ell}}{\partial T^{\ell}}\right) {\bf e}'_{i}(z,0)
\right\Vert_{L^2_{m+1-\ell}} \\
&\le
C_{m,{\rm(\ref{formA77})}} \left\Vert \nabla_{\rho}^n \left(\frac{\partial^{\ell}}{\partial T^{\ell}}\right)(W_0)
\right\Vert_{L^2_{m+1-\ell}}
\le C'_{m,{\rm(\ref{formA77})}} e^{- \delta_1 T}.
\endaligned
\end{equation}
Here the second inequality follows from (\ref{form184})
and (\ref{form185}).
\par
Then (\ref{formE92}) and (\ref{formE93}) imply the required inequality
(\ref{form6321}).
\end{proof}

\section[Estimate of $T$ derivative of the error term]{Estimate of $T$ derivative of the error term of
non-linear Cauchy-Riemann equation}
\label{appendixC}
\begin{proof}
[Proof of Lemma \ref{mainestimatestep13kappaT}]
We discuss estimate on $K_1$ only.
Estimate on $K_1^{4T-1} \setminus K_1$
is similar.
We put $K_1^+ = K_1^1 = K_1 \cup [-5T,-5T+1]_{\tau}\times [0,1]$.
\par
We use the simplified notation:
\begin{equation}\label{simplenote2onrji}
\aligned
u &= u^{\rho}_{1,T,(\kappa-1)}, \qquad &V=V^{\rho}_{T,1,(\kappa)}, \\
\mathcal P &= (\Pal_1^{(0,1)})^{-1},  \qquad &\frak e =
{\sum_{a=0}^{\kappa-1}\frak e^{\rho} _{1,T,(a)}}.
\endaligned
\end{equation}
where $\Pal_1^{(0,1)}$ is $0,1$ part of the parallel transport
$r \to {\rm Exp}(u(z),rV)$.
\par
We apply Lemma \ref{lemA1}
to
\begin{equation}\label{VandWappendixE}
W^0(z) = {\rm E}(u_1(z),u(z)), \qquad
V^0(z) = {\rm Pal}_{u(z)}^{u_1(z)}
\left(V(z)
\right)
\end{equation}
and use induction hypothesis.
We remark that $V^0$, $W^0$, $\frak e$ are $T$ dependent.
We divide $K_1 \subset \bigcup \Omega_a$ such that
$u_1^*TX$ etc. are trivial on $\Omega_a$.
(See the beginning of Appendix \ref{appendixA}.)
\par
We put
$$
\frak P = ({\rm Pal}_{u_1}^{u})^{(0,1)}.
$$
We  then can show:
\begin{lem}
\begin{equation}\label{estimateE3}
\aligned
&\left\| \nabla_{\rho}^n\frac{d^{\ell}}{dT^{\ell}} \frak P^{-1}\int_{0}^1
 ds\int_0^s
\frac{d^2}{dr^2} \left(\mathcal P \overline\partial (\Exp (u,rV)\right)dr
\right\|_{L^2_{m-\ell}(K_1)}
\\
& \le C_{m,{\rm(\ref{estimateE3})}} \left(
\sum_{n'=0}^n\sum_{j=0}^{\ell}\left\| \nabla_{\rho}^{n'}\frac{d^j}{dT^j}V\right\|_{L^2_{m-j+1}(K_1^+)}\right)^2
\\
&\le C'_{m,{\rm(\ref{estimateE3})}}e^{-2\delta_1 T}\mu^{2(\kappa-1)}.
\endaligned
\end{equation}
\end{lem}
\begin{proof}
For simplicity of notation we consider the case $n=0$.
In the case $n\ne 0$ (that is, the case when the $\rho$ derivative is included)
the proof is the same.
\par
We take $\ell$-th $T$ derivative of (\ref{formA9}) and find that
$$
\aligned
&
\frac{d^{\ell}}{dT^{\ell}} {\rm (\ref{formA9})} \\
&=
\sum_{ij,\ell_1,\ell_2, \ell_3
\atop \ell_1+\ell_2+\ell_3 = \ell}\frac{\partial^{\ell_1} V^0_i}{\partial T^{\ell_1}}\frac{\partial^{\ell_2} V^0_j}{\partial T^{\ell_2}}\frak G^{\ell_1,\ell_2,\ell_3}_{ij}
\Big(r,z,W^0,\dots,\frac{\partial^{\ell_3} W^0}{\partial T^{\ell_3}},
V^0,\dots,\frac{\partial^{\ell_3} V^0}{\partial T^{\ell_3}},
\\
&\qquad\qquad\qquad\qquad\qquad\qquad\qquad\qquad\frac{\partial W^0}{\partial x},\dots, \frac{\partial^{\ell_3+1} W^0}
{\partial x\partial T^{\ell_3}},\frac{\partial W^0}{\partial y},\dots, \frac{\partial^{\ell_3+1} W^0}
{\partial y\partial T^{\ell_3}},
\\
&\qquad\qquad\qquad\qquad\qquad\qquad\qquad\qquad\frac{\partial V^0}{\partial x},\dots, \frac{\partial^{\ell_3+1} V^0}
{\partial x\partial T^{\ell_3}},\frac{\partial V^0}{\partial y},\dots, \frac{\partial^{\ell_3+1} V^0}
{\partial y\partial T^{\ell_3}}\Big) \\
&\quad+
\sum_{ij,\ell_1,\ell_2, \ell_3
\atop \ell_1+\ell_2+\ell_3 = \ell}\frac{\partial^{\ell_1} V^0_i}{\partial T^{\ell_1}}
\frac{\partial^{\ell_2+1} V^0_j}{\partial x\partial T^{\ell_2}} \frak G^{x,\ell_1,\ell_2,\ell_3}_{ij}\left(r,z,
W^0,\dots,\frac{\partial^{\ell_3} W^0}{\partial T^{\ell_3}},
V^0,\dots,\frac{\partial^{\ell_3} V^0}{\partial T^{\ell_3}}\right)
\\
&\quad+
\sum_{ij,\ell_1,\ell_2, \ell_3
\atop \ell_1+\ell_2+\ell_3 = \ell}\frac{\partial^{\ell_1} V^0_i}{\partial T^{\ell_1}}
\frac{\partial^{\ell_2+1} V^0_j}{\partial y\partial T^{\ell_2}} \frak G^{y,\ell_1,\ell_2,\ell_3}_{ij}\left(r,z,
W^0,\dots,\frac{\partial^{\ell_3} W^0}{\partial T^{\ell_3}},
V^0,\dots,\frac{\partial^{\ell_3} V^0}{\partial T^{\ell_3}}\right),
\endaligned
$$
where $\frak G^{\ell_1,\ell_2,\ell_3}_{ij}$, $\frak G^{x,\ell_1,\ell_2,\ell_3}_{ij}$
and $\frak G^{y,\ell_1,\ell_2,\ell_3}_{ij}$ are smooth maps of the variables in the
parentheses.
Note $V^0$ there, for example,  means $V^0(z) \in \R^n$.
\par
By induction hypothesis (on $\ell$), that is an estimate of
$\frac{\partial V^0_i}{\partial x}$, $\frac{\partial V^0_i}{\partial y}$,
$\frac{\partial W^0_i}{\partial x}$, $\frac{\partial W^0_i}{\partial y}$
and their $\ell'$-derivatives (for $\ell' \le \ell$) with respect to $T$,
we can show that  $L^2_{m-\ell_3}(\Omega_a)$ norms of the maps $\frak G^{\ell_1,\ell_2,\ell_3}_{ij}
(r,z,\dots,\frac{\partial^{\ell_3+1} V^0}
{\partial y\partial T^{\ell_3}})$, $\frak G^{x,\ell_1,\ell_2,\ell_3}_{ij}(r,z,\dots,\frac{\partial^{\ell_3} V^0}{\partial T^{\ell_3}})$
and 
\linebreak
$\frak G^{y,\ell_1,\ell_2,\ell_3}_{ij}(r,z,\dots,\frac{\partial^{\ell_3} V^0}{\partial T^{\ell_3}})$ are bounded.
\par
Therefore
\begin{equation}\label{newforma81}
\left\Vert \frac{d^{\ell}}{dT^{\ell}} {\rm (\ref{formA9})} \right\Vert_{L^2_{m-\ell}(\Omega_a)}
\le C_{m,{\rm(\ref{newforma81})}} \left( \sum_{j=0}^{\ell}\left\| \frac{d^j}{dT^j}V\right\|_{L^2_{m-j+1}(K_1^+)}\right)^2.
\end{equation}
Taking the sum over $a$,
we obtain the first inequality.
\par
The second inequality is the consequence of induction hypothesis.
\end{proof}
\begin{lem}\label{lemE2E2}
We have the next two inequalities.
\begin{equation}\label{formE6}
\aligned
&\left\|
\nabla_{\rho}^n\frac{d^{\ell}}{dT^{\ell}}\frak P^{-1}(D_{u}\mathcal E_1)(\mathcal P Q,V)\right\|_{L^2_{m-\ell}(K_1)}\\
&\le C_{m,{\rm(\ref{formE6})}}
\left(
\sum_{n'=0}^{n}\sum_{\ell'=0}^{\ell}\left\|  \nabla_{\rho}^{n'}\frac{d^{\ell'}}{dT^{\ell'}}Q \right\|_{L^2_{m-\ell'}(K_1^+)}\right)
\\
&\qquad\qquad \times
\left(
\sum_{n'=0}^{n}\sum_{\ell'=0}^{\ell}\left\|  \nabla_{\rho}^{n'}\frac{d^{\ell'}}{dT^{\ell'}}V \right\|_{L^2_{m-\ell'}(K_1^+)}
\right).
\endaligned
\end{equation}
\begin{equation}\label{2152ffT}
\aligned
&\left\| \nabla_{\rho}^n\frac{d^{\ell}}{dT^{\ell}}
\frak P^{-1}\int_0^1 ds \int_0^s
\left(\frac{\del}{\del r}\right)^2 \left({\mathcal P\circ \Pi_{\mathcal E_1(\Exp (u,sV))}^{\perp}
\circ \mathcal P^{-1}}\right)\, dr
\right\|_{L^2_m(K_1)}\\
&\le C_{m,{\rm(\ref{2152ffT})}}
\left(
\sum_{n'=0}^{n}\sum_{\ell'=0}^{\ell}\left\|  \nabla_{\rho}^{n'}\frac{d^{\ell'}}{dT^{\ell'}}V \right\|_{L^2_{m-\ell'}(K_1^+)}\right)^2.
\endaligned
\end{equation}
{Here $Q,V$ are $T,\rho$ dependent sections of $u^*TX \times \Lambda^{0,1}$, $u^*TX$,
respectively.}
\end{lem}
\begin{proof}
We prove this lemma in a similar way to (\ref{estimateE3}) and also to
Lemma \ref{expest2} as follows.
For simplicity of formula we consider the case $n=0$.
\par
We divide $K_1 \subset \bigcup \Omega_a$ as above.
We define {$W^0$, $V^0$} as in (\ref{VandWappendixE}) and apply the calculation
in Appendix \ref{appendixA2bis}
to ${\bf e}'_{i}(z,r) $ and ${\bf e}_{i}(z,r) $.
\par
By the smoothness of  $\hat{\bf e}_{i}$ and (\ref{forme'e'e'})
the next inequality holds for $\ell > 0$.
\begin{equation}\label{newformA84}
\aligned
&\left\Vert \left(\frac{d}{dT}\right)^{\ell}{\bf e}'_{i} \right\Vert_{L^2_{m-\ell}(\Omega_a)} \\
&\le
C_{m,{\rm(\ref{newformA84})}}\sum_{k=1}^{\ell}\left(\left\Vert \frac{d^{k}V^0}{dT^{k}} \right\Vert_{L^2_{m-k}(K_1^+)}
+ \left\Vert \frac{d^{k}W^0}{dT^{k}} \right\Vert_{L^2_{m-k}(K_1^+)}\right)
\\
& \le C'_{m,{\rm(\ref{newformA84})}} e^{-\delta_1 T}.
\endaligned
\end{equation}
We also have the next formula for $\ell \ge 0$.
\begin{equation}\label{newformA85}
\left\Vert \left(\frac{d}{dT}\right)^{\ell}\frac{d}{dr}{\bf e}'_{i} \right\Vert_{L^2_{m-\ell}(\Omega_a)}
\le C_{m,{\rm(\ref{newformA85})}}\left(\sum_{k=0}^{\ell}\left\Vert \left(\frac{d}{dT}\right)^{k} V^0\right\Vert_{L^2_{m-k}(K_1^+)}\right).
\end{equation}
Then using (\ref{formE92}) and (\ref{formE93})
we can prove
\begin{equation}\label{newformA86}
\aligned
&\left\Vert \left(\frac{d}{dT}\right)^{\ell}{\bf e}_{i} \right\Vert_{L^2_{m-\ell}(\Omega_a)}
\le C_{m,{\rm(\ref{newformA86})}}e^{-\delta_1 T}, \\
&\left\Vert \left(\frac{d}{dT}\right)^{\ell}\frac{d}{dr}{\bf e}_{i} \right\Vert_{L^2_{m-\ell}(\Omega_a)}
\le C'_{m,{\rm(\ref{newformA86})}}\left(\sum_{k=0}^{\ell}\left\Vert \left(\frac{d}{dT}\right)^{k} V^0\right\Vert_{L^2_{m-k}(K_1^+)}\right),
\endaligned
\end{equation}
by induction on $i$.
Using this formula and (\ref{formnewQ12}), (\ref{formE12})
it is easy to show  (\ref{formE6}).
The proof of (\ref{2152ffT}) is similar.
\end{proof}
We  use the inequalities (\ref{estimateE3}), (\ref{formE6}) in place of (\ref{2152ff}), (\ref{form5505550ap}),
respectively, for our estimate of $T$-derivatives.
We use (\ref{2152ffT}) to estimate the $T$ derivatives of the third term of (\ref{2155ff}).
We can then prove
\begin{equation}\label{2156ffT}
\left\|
\frac{d^{\ell}}{dT^{\ell}}(\frak P')^{-1}
\left(\overline\partial (\Exp (u,V))
- \mathcal P^{-1}\frak{e}
\right)\right\|_{L^2_{m}(K_1)} \\
\le C_{m,{\rm(\ref{2156ffT})}}e^{-\delta T}\mu^{\kappa-1}
\end{equation}
in the sam{e} way as (\ref{2156ff}).
(Here we put
$
\frak P' = ({\rm Pal}_{u_1}^{\Exp (u,V)})^{(0,1)}
$.)
\par
We then deduce
\begin{equation}\label{formA8888}
\aligned
\Big\|\frac{d^{\ell}}{dT^{\ell}}(\frak P')^{-1}\Big(\Pi_{
\mathcal E_1(\Exp (u,V))}^{\perp}&\overline\partial (\Exp (u,V))
-\mathcal P^{-1}
\Pi_{\mathcal E_1(u)}^{\perp}(\mathcal P \overline\partial (\Exp (u,V))\\
&+
\mathcal P^{-1}(D_{u}\mathcal E_1)\left(\mathcal P \frak e,V\right)\Big)\Big\|_{L^2_{m}(K_1^+)}\\
&\le C_{m,{\rm(\ref{formA8888})}}e^{-2\delta T}\mu^{\kappa-1},
\endaligned\end{equation}
in a similar way as ({\ref{form5500}}).

The rest of the proof of Lemma \ref{mainestimatestep13kappaT} is a straightforward modification of the proof of
(\ref{formB8}) in the same way as above and so is omitted.
\end{proof}
\begin{proof}
[Proof of (\ref{newform6240})]
We take $\rho_0'$ and prove the inequality for $\rho$ with 
$d(\rho,\rho'_0) \le e^{-\delta_1 T}$.
Let $\Omega = [S',S'+1]_{\tau'} \times [0,1] \subset [T,9T]_{\tau'}\times [0,1]$.
We consider
$u'_1= p^{\rho_0'}_0$ (constant map)
and
\begin{equation}\label{simplenote2onrji223}
\aligned
&W^{0}(z) ={\rm E}(p^{\rho_0'}_0,u^{\rho}_{T,(\kappa-1)}(z)),  \\
&V^{0}(z) = \Big({\rm Pal}_{u^{\rho}_{T,(\kappa-1)}(z))}^{p^{\rho_0'}_0}
\circ {\rm Pal}^{u^{\rho}_{T,(\kappa-1)}(z))}_{p^{\rho}_{T,(\kappa-1)}}\Big)(\Delta p^{\rho}_{T,(\kappa)}).
\endaligned
\end{equation}
We then apply Lemma \ref{lemA1}.
Note
$$
v_r =  {\rm Exp}(u^{\rho}_{T,(\kappa-1)},
r(\Delta p^{\rho}_{T,(\kappa)})^{\rm Pal}).
$$
We put
$$
\mathcal{P'P}
= (({\rm Pal}_{p_0^{\rho_0'}}^{v_0})^{(0,1)})^{-1}
\circ (({\rm Pal}_{v_0}^{v_r})^{(0,1)})^{-1}.
$$
\begin{lem}
\begin{equation}\label{A91form}
\Big\Vert
\nabla_{\rho}^n\frac{\partial^{\ell}}{\partial T^{\ell}}
\left.\frac{\partial}{\partial r}\right\vert_{r=0}
\mathcal{P'P}\overline{\partial}v_r
\Big\Vert_{L^2_{m-\ell}(\Omega)}
\le C_{m,{\rm{(\ref{A91form})}}}\mu^{\kappa-1}e^{-2T\delta_1}.
\end{equation}
\end{lem}
\begin{proof}
By differentiating (\ref{maineqationddd}) once by $r$ we obtain
\begin{equation}\label{formulaA2232}
\aligned
\left.\frac{\partial}{\partial r}\right\vert_{r=0}
\mathcal{P'P}\overline{\partial}v_r
&= \sum_{i}V^0_i \frak F_{i}\left(z,W^0,\frac{\partial W^0}{\partial x},\frac{\partial W^0}{\partial y}\right) \\
&+
\sum_{i}\frac{\partial V^0_i}{\partial x} \frak F^x_{i}\left(z,W^0\right) +
\sum_{i}\frac{\partial V^0_i}{\partial y} \frak F^y_{i}\left(z,W^0\right).
\endaligned
\end{equation}
Note
\begin{equation}\label{newA930}
\aligned
&\Big\Vert
\nabla_{\rho}^n\frac{\partial^{\ell}}{\partial T^{\ell}}
\Big({\rm Pal}_{u^{\rho}_{T,(\kappa-1)}(z))}^{p^{\rho_0'}_0}
\circ {\rm Pal}^{u^{\rho}_{T,(\kappa-1)}(z))}_{p^{\rho}_{T,(\kappa-1)}}-{\rm Pal}_{p^{\rho}_{T,(\kappa-1)}}^{p^{\rho_0'}_0}\Big)
\Big\Vert_{L^2_{m-\ell}(\Omega)} \\
&\le C_{m,{\rm(\ref{newA930})}}\mu^{\kappa-1}e^{-\delta_1 T}
\endaligned
\end{equation}
follows from (\ref{form185}) and (\ref{form182})
in the same way as in Appendix \ref{appendixA2}.
Using also the fact that
${\rm Pal}_{p^{\rho}_{T,(\kappa-1)}}^{p^{\rho_0'}_0}(\Delta p^{\rho}_{T,(\kappa)})$
is constant on $\Omega$, we can use Lemma \ref{lemB2}  to estimate the 2nd and 3rd terms of
(\ref{formulaA2232}) by
$C \mu^{\kappa-1} e^{-2\delta_1 T}$.
\par
We also remark that when we substitute $W^0 = 0$
then (\ref{formulaA2232}) vanishes. In fact
if $W^0 = 0$ then $v_r(z) = {\rm Exp}(p_0^{\rho_0'},r{\rm Pal}_{p^{\rho}_{T,(\kappa-1)}}^{p^{\rho_0'}_0}(\Delta p^{\rho}_{T,(\kappa)}))$,
which is a constant map. So $\overline{\partial}v_r = 0$.
\par
Therefore
\begin{equation}\label{newformulaA2232}
\aligned
&\Big\Vert
\left.\frac{\partial}{\partial r}\right\vert_{r=0}
\mathcal{P'P}\overline{\partial}v_r
\Big\Vert_{L^2_{m}(\Omega)}
-  C_{m,{\rm(\ref{newA930})}}\mu^{\kappa-1}e^{-2\delta_1 T}\\
&\le
C_{m,{\rm(\ref{newformulaA2232})}}
\Vert V_0 \Vert_{L^2_m(\Omega)}
\left( \Vert W_0 \Vert_{L^2_m(\Omega)}
+ \Big\Vert \frac{\partial W^0}{\partial x} \Big\Vert_{L^2_m(\Omega)} +
\Big\Vert \frac{\partial W^0}{\partial y} \Big\Vert_{L^2_m(\Omega)} \right)
\\
&\le
C'_{m,{\rm(\ref{newformulaA2232})}}
\Vert V_0 \Vert_{L^2_m(\Omega)}
\Vert W_0 \Vert_{L^2_{m+1}(\Omega)}
\\
&\le C''_{m,{\rm(\ref{newformulaA2232})}}
\mu^{\kappa-1}e^{-2T\delta_1}.
\endaligned
\end{equation}
Here the last inequality follows from (\ref{form182}) and (\ref{form185}).
(Note we are working on $[T,9T]_{\tau'} \times [0,1]$.)
\par
The case when $T$ and $\rho$ derivatives are included is similar.
\end{proof}
We put
$$
\mathcal P (\mathcal P^{\prime})^{-1}
=
(({\rm Pal}_{u_1}^{v_0})^{(0,1)})^{-1}\circ
({\rm Pal}_{p_0^{\rho_0'}}^{v_0})^{(0,1)}.
$$
Here $u_1 : \Sigma_1 \to X$ is the map we start with.
(It is different from $u'_1$.)
$v_0 = u^{\rho}_{T,(\kappa-1)}$.
This is a $(\rho, T)$ parameterized family of
sections of the bundle ${\rm Hom}(T_{p_0^{\rho_0'}}X,u_1^*TX)
{\otimes \Lambda^{01}}$
on our domain $\Omega$.
\par
Let $W \in T_{p_0^{\rho_0'}}X{\otimes \Lambda^{01}}$.
Then, by (\ref{form185}) we can show the inequality
\begin{equation}\label{newA93}
\aligned
&\Vert
\nabla_{\rho}^n\frac{\partial^{\ell}}{\partial T^{\ell}}
\big(\mathcal P (\mathcal P^{\prime})^{-1}
- (({\rm Pal}^{p_0^{\rho_0'}}_{u_1})^{(0,1)})^{-1}
\big)(W)
\Vert_{L^2_{m-\ell}(\Omega)}
\\
&\le C_{m,{\rm(\ref{newA93})}}e^{-\delta_1 T}
\vert W\vert
\endaligned
\end{equation}
by using Lemma \ref{lemB2}.
\par
Note $(({\rm Pal}^{p_0^{\rho_0'}}_{u_1})^{(0,1)})^{-1}$
is $(T,\rho)$ independent and is bounded.
\par
Therefore by putting $W = \mathcal{P'P}\overline{\partial}v_r$,
the formulae
$\mathcal P (\mathcal P^{\prime})^{-1}\circ \mathcal{P'P}
= \mathcal{PP}$, (\ref{A91form}) and (\ref{newA93}) imply
\begin{equation}\label{newformulaA22322}
\Big\Vert
\left.
\nabla_{\rho}^n\frac{\partial^{\ell}}{\partial T^{\ell}}\frac{\partial}{\partial r}\right\vert_{r=0}
\mathcal{PP}\overline{\partial}v_r
\Big\Vert_{L^2_{m}(\Omega)} \le C_{m,{\rm(\ref{newformulaA22322})}}\mu^{\kappa-1}e^{-2T\delta_1}.
\end{equation}
where
$$
\mathcal{PP}
= (({\rm Pal}_{u_1}^{v_0})^{(0,1)})^{-1}
\circ (({\rm Pal}_{v_0}^{v_r})^{(0,1)})^{-1}.
$$
(\ref{newform6240}) follows immediately from this formula.
In fact
$$
\left.\frac{\partial}{\partial r}\right\vert_{r=0}\mathcal{PP}\overline{\partial}v_r
=
(({\rm Pal}_{u_1}^{u^{\rho}_{T,(\kappa-1)}})^{(0,1)})^{-1}(D_{u^{\rho}_{T,(\kappa-1)}}
\overline{\partial}) ((\Delta^{\rho}_{T,(\kappa)})^{\rm Pal}).
$$
\end{proof}

\section[Proof of Lemma 8.4]
{Proof of Lemma \ref{lem8484}}
\label{appendixF}

Let $\Gamma_+$ be a finite group and $\Gamma$ be its index $2$ subgroup.
We assume that there exists $\tau \in \Gamma_+$ of order 2
such that $\Gamma_+ = \Gamma_ \cup \tau \Gamma$.
\par
A smooth action of $\Gamma_+$ to a complex manifold $M$ is said to be {\it holomorphic}
if all the action of elements of $\Gamma$ is  holomorphic and $\tau$'s action is by an anti-holomorphic involution.
\begin{lem}\label{LemF1}
Let $p \in M$ be a fixed point of a holomorphic action of $\Gamma_+$
then there exists a complex coordinate of $M$ at $p$ such that the
$\Gamma_+$ action is linear in this coordinate.
\end{lem}
\begin{proof}
We fix a coordinate of $M$ at $p$.
Then $\Gamma_+$ induces a holomorphic action $\rho$ on $U \subset \C^n$
such that $0$ is a fixed point.
We consider the $\Gamma_+$ action on the tangent space {at}  $0$ which we denote by
$\rho_0$.
For $t \in (0,1]$ we define the action $\rho_t$ by
$$
\rho_t(g,x) = \rho(g,tx)/t.
$$
Together with $\rho_0$ it defines a smooth one parameter family of holomorphic actions.
Consider the map
$$
g \mapsto \left.\frac{d}{dt}\right\vert_{t=t_0} \rho_t(g,
{\rho_{t_0}(g^{-1},x)})\in {\rm Hol}(U,T\C^n).
$$
Here ${\rm Hol}(U,TX)$ is the vector space of holomorphic vector fields on $U$.
For each $t_0$ the map defines a 1 cycle of the complex
$C^*(\Gamma_+,{\rm Hol}(U,T\C^n))$ which calculates the group cohomology with local coefficient{s}  
induced by the
conjugate action associated to $\rho_{t_0}$.
(Note the conjugate action by the anti-holomorphic involution $\tau$ preserves the set of
holomorphic vector fields.)
\par
Using the fact that $\Gamma_+$ is a finite group we can  show that the first cohomology group
vanishes.
In fact the Eilenberg-MacLane space $K(\Gamma_+,1)$  has {a} finite cover $
\widetilde{K(\Gamma_+,1)}$ which is contractible.
All the local systems on $K(\Gamma_+,1)$ pulled back to $\widetilde{K(\Gamma_+,1)}$ are trivial and the first
cohomology groups of the pulled back are trivial. Since $\Q$ is contained in the coefficient ring of our
local system, {the} Gysin map induces {an} isomorphism between cohomologi{es}  {of}  
$K({\Gamma_+},1)$ and {of}  $\widetilde{K(\Gamma_+,1)}$.
\par
So there exists {a} one parameter family of holomorphic vector fields $V_t$
such that
$$
 \left.\frac{d}{dt}\right\vert_{t=t_0} \rho_t(g,{\rho_{t_0}(g^{-1},x)})
 = \rho_{t_0}(g,\cdot)_* V_t - V_t.
$$
By integrating $V_t$ we obtain a biholomorphic map
which interpolates $\rho_0$ and $\rho$.
The lemma follows.
\end{proof}
\begin{lem}\label{lemF2}
Let $M$, $N$ be complex manifolds on which $\Gamma_+$ has holomorphic actions.
Let $p\in M$ and $q \in N$ be fixed points of $\Gamma_+$.
Assume $F : M \to N$ is a $\Gamma_+$ equivariant holomorphic map such that $F(p) = q$ and
$d_pF:T_{{p}}M \to T_{{q}}N$ is surjective.
\par
We decompose
$T_pM = T_qN \oplus V$
as a $\Gamma_+$ complex vector space such that $d_pF$ is identity on $T_qN$
and is $0$ on $V$.
\par
Then there exists a $\Gamma_+$ equivariant biholomorphic map
$$
\varphi : B_{\epsilon}(T_qN) \times B_{\epsilon}(V) \to M
$$
sending $(0,0)$ to $p$ such that $F \circ \varphi $ is constant in {the} $V$ direction.
\end{lem}
\begin{proof}
Using Lemma \ref{LemF1} we may assume that $M$, $N$ are open subsets ${W_1}$, ${W_2}$ of $\C^M$, $\C^N$ and
{the} $\Gamma_+$ action is linear. We may assume $p=0$, $q=0$.
We take two $\Gamma_+$ invariant complex linear subspaces $U_1,U_2$ of $\C^M$ such that
$U_1 \oplus U_2 = \C^M$, $(d_0F)(U_1) = 0$ and $d_0F$ induce{s}  an isomorphism between $U_2$ and $\C^N$.
We define $G : {W_1} \to U_1 \times {W_2}$ by $G(x+y)= (x,F(x+y))$, where $x \in U_1, y\in U_2$.
By the inverse mapping theorem $G$ has a local inverse $\varphi $.
Then $F  \circ \varphi$ is the projection to the $W_2$ factor. Therefore $\varphi $ gives the
required coordinate.
(The $\Gamma_+$ equivariance and holomorphicity of $\varphi $ are obvious from construction.)
\end{proof}
\begin{proof}[Proof of Lemma \ref{lem8484}]
Applying Lemma \ref{lemF2} to
$\mathcal C^{\rm cl}(\mathcal V) \to \mathcal V$
we can prove Lemma \ref{lem8484} easily.
\end{proof}

\printindex[syindex]
\printindex
\bibliographystyle{amsalpha}

\end{document}